\documentclass[a4paper, 11pt]{article}
\usepackage[arxiv]{mymacros}

\usepackage{xr-hyper}
\usepackage{hyperref}

\usepackage{tikz}
\usepackage{tikz-cd}
\usepackage[normalem]{ulem}

\title{The Discrete Gaussian model, I.\\ Renormalisation group flow at high temperature}

\author{Roland Bauerschmidt\footnote{University of Cambridge, Statistical Laboratory, DPMMS. E-mail: {\tt rb812@cam.ac.uk}.}
  \and
  Jiwoon Park\footnote{University of Cambridge, Statistical Laboratory, DPMMS. E-mail: {\tt jp711@cam.ac.uk}.}
  \and
  Pierre-Fran\c{c}ois Rodriguez\footnote{Imperial College London, Department of Mathematics. E-mail: {\tt p.rodriguez@imperial.ac.uk}.}
}

\makeatletter
\newcommand{\subjclass}[2][1991]{%
  \let\@oldtitle\@title%
  \gdef\@title{\@oldtitle\footnotetext{#1 \emph{Mathematics subject classification.} #2}}%
}
\makeatother

\subjclass[2020]{82B20, 82B28, 60G15, 60K35}

\date{\vspace*{-2em}} 

\newcommand{\betafree}{\beta_{\operatorname{free}}}
\newcommand{\betaeff}{\beta_{\operatorname{eff}}}

\definecolor{grey}{rgb}{0.55, 0.55, 0.55}
\definecolor{gray}{rgb}{0.55, 0.55, 0.55}
\definecolor{darkmagenta}{rgb}{0.55, 0.0, 0.55}
\definecolor{magenta}{rgb}{0.85, 0.0, 0.55}

\newcommand{\Tay}{\operatorname{Tay}}
\newcommand{\Rem}{\operatorname{Rem}}
\newcommand{\Loc}{\operatorname{Loc}}

\newcommand{\MM}{\mathfrak{M}}
\newcommand{\Eplus}{\mathbb{E}}
\newcommand{\cwone}{{}}

\renewcommand{\bar}[1]{\overline{#1}}
\newcommand{\alphaLoc}{\alpha_{\Loc}}
\newcommand{\xnorm}{|x|}

\usepackage[titles]{tocloft}
\begin{document}
\maketitle

\begin{abstract}
  The Discrete Gaussian model is the lattice Gaussian free field conditioned to be integer-valued.
  In two dimensions, at sufficiently high temperature,
  we show that its macroscopic scaling limit on the torus
  is a multiple of the Gaussian free field.
  Our proof starts from a single renormalisation group step after which the integer-valued field
  becomes a smooth field which we then analyse using the renormalisation group method.
  
  This paper also provides the foundation for the construction of the scaling limit of the infinite-volume
  gradient Gibbs state of the Discrete Gaussian model in the companion paper.
  Moreover, we  develop all estimates
  for general finite-range interaction with sharp dependence on the range.
  We expect these estimates to prepare for  
  a future analysis of the spread-out version of the
  Discrete Gaussian model at its critical temperature.
\end{abstract}

\setcounter{tocdepth}{1}
\tableofcontents


\newpage
\section{Introduction and main results}
\label{sec:intro}

\subsection{Introduction}

Many fundamental models in two-dim\-en\-sional statistical physics are related to
two-di\-men\-sional Coulomb gas models (with charge symmetry).
These models include the Discrete Gaussian model (which is the Gaussian free field conditioned
to be integer-valued),
the Solid-On-Solid model,
the plane rotator (or XY) model
\cite{MR634447}, and many further models \cite{MR496208}.
The relation between the former lattice models and Coulomb gas models
was rigorously used to prove the existence of phase transitions in these models
in the fundamental work of Fr\"ohlich--Spencer \cite{MR634447}.
The Discrete Gaussian model is a model for a crystal interface (in 2+1 dimensions)
and the phase transition is one between a smooth (localised) low-temperature
phase and a rough (delocalised) high-temperature phase.
Its  understanding was pioneered in independent works by Berezinski\u{\i} \cite{MR0314399}
and Kosterlitz--Thouless \cite{0022-3719-6-7-010}; see 
\cite[Chapter~6]{MR1446000} for a textbook treatment
and also the recent survey \cite{10.1088/0034-4885/79/2/026001} and references therein.

The rigorous approach of \cite{MR634447} (see also \cite{MR733469})
uses a multiscale resummation based on conditional expectations and Jensen's inequality.
For a recent exposition as well as recent extensions and applications of this approach, see
\cite{1711.04720,2002.12284,1907.08868,2012.01400},
and for recent alternative approaches to the proof of the existence of the Kosterlitz--Thouless transition, see also \cite{MR4367953,2101.05139,2110.09498,2110.09465}.
These approaches have many appealing features which include that they apply
quite robustly to various  models, but they are not precise enough
to derive scaling limits or sharp asymptotics of correlation functions,
or to study the (expected) critical curve --- the Kosterlitz--Thouless transition line, see Fig.~\ref{fig:phasediagram}.

\begin{figure}[h] 
\centering
\includegraphics[width=0.65\columnwidth]{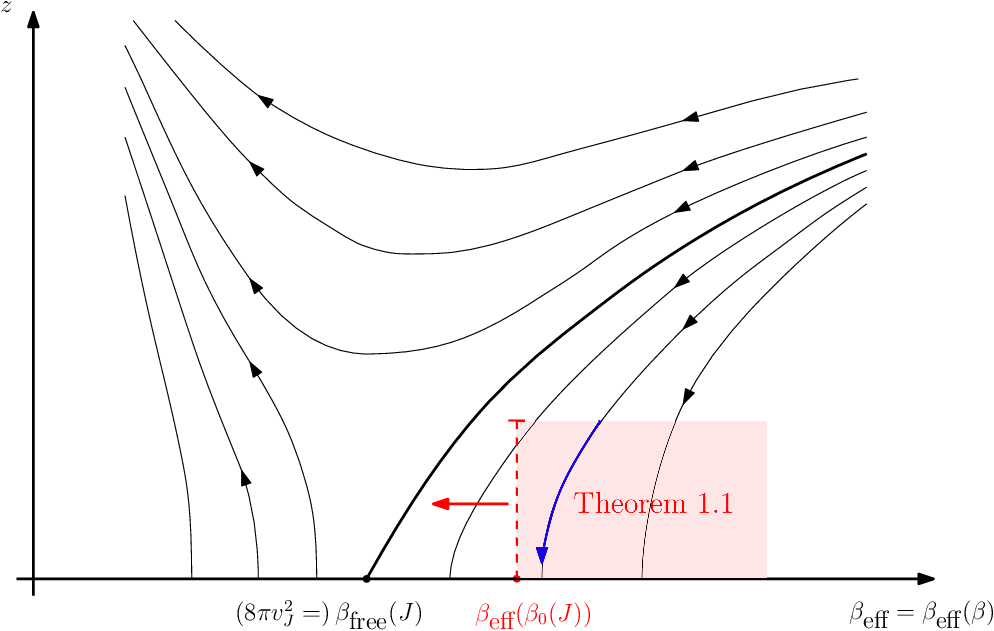}
\caption{The phase diagram of the Discrete Gaussian model.
  The blue arrow represents an initial renormalisation group step which reduces
  the activity 
  of the Coulomb gas associated with the Discrete Gaussian model from order $1$ to $O(e^{-c\beta})$,
  see Section~\ref{sec:first_step}. 
  The region shaded in red comprises values of $\beta$ for which the Discrete Gaussian model is shown to rescale to a Gaussian free field with effective temperature $\beta_{\operatorname{eff}}$.
  For $J=J_{\rho}$ and large enough $\rho$, this region invades an entire ``strip'' to the right of the (boldfaced) separatrix, as indicated by the red arrow. \label{fig:phasediagram}}
\end{figure}

A more precise approach to the low-temperature phase
of the two-dimensional lattice Coulomb gas is based on the renormalisation
group method. This approach is the standard approach in physics and
it provides very important heuristics as well as precise predictions, but for this reason
it is also more challenging to implement rigorously.
For the Coulomb gas at sufficiently low temperatures and small activity,
this approach was first made rigorous by Dimock--Hurd in
\cite{MR1777310,MR1101688} (relying on \cite{MR1048698}),
and an analysis that exhibits a segment of Kosterlitz--Thouless transition line was finally carried out
by Falco in \cite{MR2917175,1311.2237} (relying on the former works and \cite{MR2523458}),
again for small activity.

Discrete models such as the Discrete Gaussian model and the XY model are however related to
Coulomb gases with activity that appears to be \emph{not small} initially, but rather of order $1$.
Thus the relationship between these discrete models and the lattice Coulomb gas is roughly the
same as that between the Ising and the lattice $\varphi^4$ model.
It is widely expected that the phase diagram of the two-dimensional Coulomb gas looks qualitatively as in
Fig.~\ref{fig:phasediagram}, and thus in particular also that large activities have
the large-scale behaviour of an effective Coulomb gas with small activity,
obtained after an appropriate initial renormalisation group step.
Implementing such an initial renormalisation group step is in general difficult
and requires ideas different from the small activity analysis.
There are however two scenarios in which we believe that such an initial renormalisation group step can be
carried out effectively by trading the assumption of small activity
for a smallness assumption concerning another parameter.

The first mechanism is to (effectively) trade the assumption of a small activity of the Coulomb gas
for that of a sufficiently low temperature, an idea that was already used in various forms in the past,
including by Fr\"ohlich--Spencer. In the present article, we implement this idea in the context of the renormalisation group approach, in which it is more
challenging to carry out as sufficient precision must be maintained in particular to identify the effective temperature exactly.
As a result, we prove that the scaling limit of the Discrete Gaussian model is a
multiplicatively renormalised Gaussian free field, at sufficiently high temperature; see Theorem~\ref{thm:highbeta} below
and also~Theorems~1.1 and 1.2 in the companion paper \cite{dgauss2}.
Note that the high-temperature regime of the Discrete Gaussian model
corresponds by duality to low temperatures in the related Coulomb gas.

The second mechanism we believe is possible to carry out concerns models with finite-range step distributions
(also called spread-out models), for which the range can be used as a large parameter.
This idea has also been used previously, in particular in the context of the lace expansion \cite{MR2239599}.
For large range, we expect that it is possible to show that the scaling limit of the Discrete Gaussian model is a renormalised Gaussian free field,
but now not only at sufficiently large temperatures, but for all temperatures up to and including the
critical temperature of the Kosterlitz--Thouless transition, where logarithmic corrections
to various critical exponents expectedly appear, cf.~also Conjecture~\ref{thm:criticalbeta} below.
In this article we prepare  for such an analysis, by carrying out all estimates
with explicit dependence on the range
of the finite-range distribution, but we leave the ultimate second-order analysis of the renormalisation group map itself,
needed to control the approach of the critical point, to the future.
For the closely related lattice Coulomb gas but with small activity such a second-order analysis was carried out by Falco \cite{1311.2237,MR2917175}.

\bigskip

We next define the Discrete Gaussian model precisely and state our results.

\subsection{Discrete Gaussian model with finite-range coupling}

For $J \subset \Z^d \setminus \{0\}$ finite and symmetric under reflections and lattice rotations,
the normalised range-$J$ Laplacian $\Delta_J$ is defined by
\begin{equation}
  (\Delta_J f)(x) = \frac{1}{|J|}\sum_{y\in J} (f(x+y)-f(x)), \label{eq:Delta_J_definition}
\end{equation}
for $f:\Z^d\to \R$, 
where $|J|$ denotes the number of elements of $J$.
The normalised standard nearest-neighbour Laplacian on $\Z^2$ is given by choosing $J=J_{\rm nn} = \{(1,0),(0,1),(-1,0),(0,-1)\}$.
Another particular case of interest is the range-$\rho$ Laplacian, for $\rho \geq 1$, see the discussion below Theorem~\ref{thm:highbeta}.
The Green's function of $\Delta_J$ on $\Z^2$ restricted to functions with sum $0$ satisfies, as $|x-y|\to\infty$,
\begin{equation}
  \label{eq:green_asymptotics}
  (-\Delta_J)^{-1}(x,y) \sim - \frac{1}{2\pi v_J^2} \log |x-y|,
  \qquad
  v_J^2
  = \frac{1}{2|J|} \sum_{x\in J} |x_1|^2,
\end{equation}
where $x= (x_1,x_2)$,
see e.g.~\cite[Theorem~4.4.4]{MR2677157}.
For example, for the nearest-neighbour model one has $v_{J_{\rm nn}}^2=1/4$.  

We consider the two-dimensional Discrete Gaussian model with periodic boundary conditions.
Thus, let $\Lambda=\Lambda_N$ be a two-dimensional discrete torus of side length $L^N$ for integers $L>1,N \geq 1$, and fix a point $0\in \Lambda_N$, the origin.
For a given step distribution $J$,
the \emph{Discrete Gaussian model} on $\Lambda_N$ at temperature $\beta \in (0,\infty)$
has expectation, for any $F: (2\pi\Z)^{\Lambda_N} \to \R$ with $F(\sigma) = F(\sigma+c)$ for any constant $c\in 2\pi\Z$ and such that the following series converges, defined by\begin{equation}
  \avg{F}_{J,\beta}^{\Lambda_{N}}
  \propto \sum_{\sigma \in \Omega^{\Lambda_N}} e^{-\frac{1}{2\beta} (\sigma,-\Delta_J\sigma)} \, F(\sigma)
  = \sum_{\sigma \in \Omega^{\Lambda_N}} e^{-\frac{1}{4\beta |J|} \sum_{x-y \in J} (\sigma_x-\sigma_y)^2} \, F(\sigma)
  \label{eq:DG_model_1}
\end{equation}
where the sum over $x-y\in J$ counts every undirected edge $\{x,y\}$ twice and
\begin{equation} \label{e:Omega-def}
  \Omega^{\Lambda_N} = \{\sigma \in (2\pi\Z)^{\Lambda_N}: \sigma_{x=0} =0 \}.
\end{equation}
The factors of $2\pi$ in the spacing for $\sigma$ will be most convenient for our purposes (but could of course be absorbed by rescaling $\beta$).
To relate better to the Coulomb gas literature,
we use $\frac 1\beta$ rather than $\beta$ to denote the inverse temperature of the Discrete Gaussian model,
so that $\beta$ corresponds to an inverse temperature for the dual Coulomb gas.
For small temperature $\beta$, one can show by a Peierls expansion that the Discrete Gaussian field is localised,
e.g., in the sense of bounded variance and exponential decay of truncated correlations (see, e.g., \cite{2103.11985}),
and even the extremal behaviour is understood very precisely \cite{MR3508158}.
We study the regime of large  $\beta$ where typical fluctuations of the field are unbounded.
The transition between the two regimes is the roughening transition
or Kosterlitz--Thouless transition.

\subsection{Main results}
\label{sec:main_results}

In this first article, we consider the (small mesh) scaling limit of the Discrete Gaussian model
for macroscopic functions on the torus.
The scaling limit of the infinite-volume limit of the Discrete Gaussian model on $\Z^2$,
or the consideration of mesoscopic test functions on the torus, requires additional analysis
and the extension is considered in the companion paper \cite{dgauss2}.

To state the torus scaling limit, let $\T^2=(\R/\Z)^2$
and for $f\in C^\infty(\T^2)$ with $\int_{\T^2} f \, dx =0$, define
$f_N: \Lambda_N\to\R$ by
\begin{equation}
  f_N(x) = \frac{1}{|\Lambda_N|} \bigg( f(L^{-N}x)-\frac{1}{|\Lambda_N|} \sum_{y\in\Lambda_N} f(L^{-N}y) \bigg),
  \qquad x\in\Lambda_N,
  \label{eq:f_N_definition}
\end{equation}
so that $\sum_{x\in\Lambda_N} f_N (x) = 0$. By Fourier analysis, one can verify that (see, e.g., Lemma~\ref{lemma:integral_of_zero_mode})
\begin{equation}
  (f_N,(-\Delta_J)^{-1}f_N)_{\Lambda_N} \to
  \frac{1}{v_J^2}
  (f,(-\Delta_{\T^2})^{-1}f)_{\T^2}
\end{equation}
where on the left-hand side the inner product is $(u,v)_{\Lambda_N} = \sum_{x\in\Lambda_N} u(x) v(y)$ and
$(-\Delta_J)^{-1}$ acts on $\{\varphi\in \R^{\Lambda_N}: \sum_{x\in\Lambda_N}\varphi(x)=0\}$,
and on the right-hand side 
$(f,g)_{\T^2} = \int_{\T^2} f(x)g(x)\,dx$ and $\Delta_{\T^2}$ is the Laplace operator on $\T^2$.
The constant $v_J^2$ is the one defined in \eqref{eq:green_asymptotics}.

\begin{theorem} \label{thm:highbeta}
  Let $J \subset \Z^2 \setminus \{0\}$ be any finite-range step distribution
  that is invariant under lattice rotations and reflections and includes the nearest-neighbour edges.
  Then there exists $\beta_0(J) \in (0,\infty)$ and an integer $L = L(J)$ such that for   
  the Discrete Gaussian Model on the torus $\Lambda_N$  of side length $L^N$
  at temperature $\beta \geq \beta_0(J)$,
  there is $\beta_{\rm eff}(J,\beta) \in (0,\infty)$
  such that for any $f \in C^\infty(\T^2)$
  with $\int f \, dx =0$, as $N\to\infty$,
  \begin{equation} \label{e:highbeta-convergence}
    \log \avgb{e^{(f_N,\sigma)_{\Lambda_N}}}_{J,\beta}^{\Lambda_N}
    \to
    \frac{\beta_{\rm eff}(J,\beta)}{2 v_J^2}
    (f,(-\Delta_{\T^2})^{-1}f)_{\T^2}.
  \end{equation}
  Moreover, $\beta_{\rm eff}(J,\beta) = \beta +O_J(e^{-c\beta})$ for some $c>0$
  (independent of $J$) as $\beta \to \infty$.
\end{theorem}

Nonmatching lower and upper bounds on the left-hand side of \eqref{e:highbeta-convergence}
are known  from  the previously mentioned work of Fr\"ohlich--Spencer \cite{MR634447}
and Ginibre-type correlation inequalities \cite{MR496191},
see also \cite{1711.04720} for a detailed presentation of both
as well as \cite{2012.01400} for improved estimates.
Thus the main contribution of Theorem~\ref{thm:highbeta} is the identification of the exact scaling limit. We expect that the convergence \eqref{e:highbeta-convergence} can be quantified with a rate that is polynomial in $|\Lambda_N|^{-1}$, but we omit the detailed analysis.
In our companion paper \cite{dgauss2},
we extend this scaling limit result, which applies to macroscopic test functions on the torus,
by an infinite-volume version and as a by-product also prove a more precise version for the torus that applies to mesoscopic test functions. 

Even though our main objective is Theorem~\ref{thm:highbeta} (and the mentioned extensions in the companion paper),
which both apply to the Discrete Gaussian
model at sufficiently high temperatures $\beta \geq \beta_0(J)$,
one important motivation for studying general step distributions $J$ 
is that 
sharper results can be obtained when the set $J$ is sufficiently large, i.e., when it is sufficiently spread out.
For concreteness, we focus on the particular choice $J = J_\rho =\{x\in\Z^2 \setminus 0: |x|_\infty \leq \rho\}$,
which we refer to as the standard range-$\rho$ distribution, but any choice satisfying certain regularity conditions would be admissible (see \eqref{eq:frd_ulbds} below).
The following remark is a corollary of our proof of Theorem~\ref{thm:highbeta}
and already indicates that the result
becomes sharper when the range $\rho$ of the distribution is large,
and its proof will also appear in Section~\ref{sec:integration_of_zero_mode}.

\begin{remark}\label{rk:highbeta-rho} 
  For the standard range-$\rho$ distribution $J=J_\rho$,  
  there exists $C>0$ such that
  for any $\delta>0$,
  $\rho^2 \geq C |\log \delta|$ and $L = L(J_\rho,\delta)$,
  the conclusion of Theorem~\ref{thm:highbeta} holds with $\beta_0(J_\rho) =(1+\delta)\betafree(J_\rho)$, 
  where for any step distribution $J$, we define
  \begin{equation} \label{eq:betafree_def}
    \betafree(J) := 8\pi v_{J}^2.
  \end{equation}
  Moreover, $v_{J_\rho}^2 \sim \rho^2/6$ as $\rho\to\infty$. 
\end{remark}

The parameter $\betafree(J)$ is the generalisation of the reference value $8\pi$
that frequently appears in the literature as the approximate value of the transition temperature
of the lattice Coulomb gas.
This value is simply $4=2d$ times the constant in the denominator in front of the
logarithm in \eqref{eq:green_asymptotics}, and turns out to be actually
valid in the $0$ activity limit of the lattice Coulomb gas, but it differs
from the transition temperature for finite values of the activity.

For the Discrete Gaussian model,
we also expect that the critical point $\beta_c(J_\rho)$ is strictly greater than
$\betafree(J_\rho)$ for any finite $\rho$,
but asymptotic to 
$\betafree(J_\rho) \sim (4\pi/3) \rho^2$ as $\rho\to\infty$.
As mentioned above, we develop much of the analysis in sufficient generality and precision that
we expect that it forms the basis for the study of the
Discrete Gaussian model at the critical point $\beta=\beta_c(J_\rho)$ when $\rho \geq \rho_0$
(spread-out interaction).
We state what we expect could be proved by extending our analysis as the following conjecture
(whose proof we  hope to come back to in the future).

\begin{conjecture} \label{thm:criticalbeta}
  Let $J = J_\rho =\{x \in \Z^2 \setminus \{0\}: |x|_\infty \leq \rho\}$ be the standard range-$\rho$ step distribution.
  Then there is $\rho_0 \in [1,\infty)$ such that for $\rho \geq \rho_0$,
  one can choose $\beta_0(J) = \beta_c(\rho)$ in 
  the above theorem where $\beta_c(\rho)$ is such that,
  as $\beta \downarrow \beta_c=\beta_c(\rho)$,
  \begin{equation} \label{eq:criticalbeta_div}
    \betaeff(J_\rho,\beta_c) = \betafree(J_\rho) = 8\pi v_{J_\rho}^2, \qquad
    \frac{\betaeff(J_\rho,\beta) -  \betaeff(J_\rho,\beta_c)}{\beta-\beta_c} \to \infty,
  \end{equation}
  while $\beta_{\rm eff}(J_\rho,\beta)$ is differentiable in $\beta >\beta_c(\rho)$.
  Moreover, this critical temperature satisfies
  \begin{equation}
    \beta_c(\rho) \sim \betafree(J_\rho) \sim \frac{4\pi}{3} \rho^2 \qquad (\rho\to \infty).
  \end{equation}
  In addition, we expect the critical Discrete Gaussian model to exhibit
  various logarithmic corrections, in the approach of the critical point $\beta \downarrow \beta_c$ in \eqref{eq:criticalbeta_div},
  and in pointwise correlations of composite fields exactly at $\beta=\beta_c$.
\end{conjecture}

In fact, we expect the conjecture to be true for any finite-range step distribution,
so in particular also for $\rho \geq 1$, but a proof for small $\rho$ currently appears out of reach.
For the fractional charge correlation functions of the Coulomb gas at small activity, logarithmic corrections
at the critical point were derived in \cite{1311.2237}.

\subsection{Related problems}
\label{sec:related_problems}

There are many interesting related problems we expect to be within reach, some with significant effort.
These include on the one hand
the extension of our theorems to more general 
discrete gradient distributions, such as the height functions from \cite{MR634447}
arising in the study of the XY model,
and on the other hand the 
correlation functions (different from the height correlations studied in this paper)
that are needed for the study of Villain and XY models (both then in the dual low temperature
regime, or at the critical point but then with spread-out interaction $J$).
Further interesting problems include the potential application of the methods to
models of two-dimensional melting \cite{0022-3719-6-7-010} (see also \cite{MR3995896,1907.07923})
and the related analysis of spin systems that take values in more general lattices than $\Z$.

Finally, we mention that gradient models with sufficiently smooth uniformly convex potential have been studied
extensively.
The scaling limit was proven in \cite{MR1461951},
and a very comprehensive picture including stochastic dynamics has also been established \cite{MR1463032,MR2228384,MR1872740}. 
For more recent developments on these models, we mention \cite{ %
MR2251117, 
MR4212193,MR3189075,MR4137943,
MR2855536, 
MR4003143, 
1909.13325, 
2002.02946, 
DR22,
MR4164451,MR3933043}. Smooth gradient models with nonconvex potential have also been considered, see, e.g., \cite{MR2778801,MR2322690, 
  MR2470934,MR2976565}, 
and in relation to our work we mention in particular
\cite{MR1048698} and \cite{1910.13564} 
which also use the renormalisation group approach.
For other discrete height functions, we mention in addition to the works discussed
in the introduction the extensive literature on dimer models, see
\cite{MR2523460} and references therein,
their nonintegrable versions \cite{MR4121614,MR3606736},
as well as recent progress on height functions associated with other six-vertex models
\cite{1909.03436,1911.00092,2012.13750} and graph homomorphisms \cite{MR4315657}.

Among the novelties in our paper compared to earlier works on the renormalisation group
method are, to start with, that we develop a new finite-range decomposition that permits us to integrate out
a preliminary renormalisation group step of range $0$. We expect this to have further applications, e.g., to the analysis
of critical Ising models and $O(n)$ sigma models, or strictly self-avoiding walks
with spread-out interaction, e.g., using the analysis of $|\varphi|^4$ type models
from \cite{MR3339164,MR3269689,MR3332938,MR3332939,MR3332940,MR3332942,MR3332941};
see also \cite{MR4061408,MR3926125} for applications of this idea to dynamics in simpler contexts.
For the main renormalisation group analysis,
compared to \cite{1311.2237,MR2917175}, we incorporate in particular a more
systematic organisation of the contraction mechanisms inspired by \cite{MR3332939},
quantify the dependence on the range of the interaction,
and along the way provide the details for some of the technical properties omitted in \cite{1311.2237,MR2917175} as well as some simpler arguments.


\subsection{Outline of the paper}

We now give a high-level exposition of the proof, with pointers to the relevant sections. We will work with a mass-regularised version of our model, parametrised by $m^2 \in (0,1]$, see Section~2.1, which amounts to replacing $-\Delta_J$ by $-\Delta_J+m^2$ in \eqref{eq:DG_model_1}. All subsequent arguments will be shown to apply uniformly in $m^2$, and our main results will be obtained by letting $m^2 \downarrow 0$ at the end (while considering a suitable class of test functions).

In Section~\ref{sec:first_step} we carry out the preliminary renormalisation group step (Step 0) mentioned above, which involves integrating out i.i.d.~Gaussians with (small) variance $\gamma >0$. The parameter $\gamma$ will be chosen small enough (see Proposition~\ref{prop:decomp}, where it is fixed), but will otherwise not require fine-tuning. The benefit of this simple but important step is to turn the singular $\mathbb Z$-valued conditioning inherent to the Discrete Gaussian model into a spin model with smooth
periodic potential, see \eqref{eq:fourier_tildeV}-\eqref{eq:z0_estimate}, comprising infinitely many cosine modes indexed by an integer $q \geq 1$ (which one may regard as a generalised lattice sine-Gordon model, the latter corresponding to having $q=1$ only). The activities (prefactors) $z=({z}^{(q)}: q \geq 1)$ of each mode decay exponentially in both $q$ and $\beta$. Together with a (scalar) stiffness parameter $s$ discussed shortly, the activities form the coupling constants of our model, whose evolution will be tracked by the renormalisation group map. In particular, for $\beta$ large enough (as in the statement of our main results), this defines a regime of \textit{small} coupling constants after this preliminary step, cf.~also~Fig.~\ref{fig:phasediagram}.

As a result of Step 0, the generalised sine-Gordon potential is now integrated against a Gaussian field with covariance $C= (-\Delta_J + m^2)^{-1} - \gamma$. Crucially, we extract from $C^{-1}$ a GFF part $s\Delta$ (with $\Delta$ the usual nearest-neighbor Laplacian) for small $|s|$, which for approriate choice of $s$ will act as the correct `counterterm' in the renormalisation. The remaining Gaussian field with covariance $\tilde{C}=(C^{-1}- s\Delta)^{-1}$ will after fine-tuning of $s$ converge in the scaling limit to the limiting GFF on the right-hand side of \eqref{e:highbeta-convergence}, and thus constitutes at the discrete level  the `correct' (Gaussian) large-scale approximation of the (non-Gaussian) Discrete Gaussian model.  

The analysis to identify the scaling limit is driven by a decomposition of the field $\tilde{C}$ for generic small $s$, into a sum of spatially localised Gaussian fields $\zeta_j$, corresponding to contributions stemming from scales $\approx L^j$ for a parameter $L \gg 1$. This finite-range decomposition is the object of Section~\ref{sec:finite_range_decomposition}.
The main novelty of our construction is that we obtain a finite-range decomposition (with precise control on both gradients and on-diagonal behavior of the covariance) in the presence of the parameters $\gamma$ and $s$.

Starting from this finite-range decomposition,
the renormalisation group analysis then proceeds using the general strategy proposed in \cite{MR2523458}, by which the modes $\zeta_j$ are progressively integrated, starting at small scales. This general method was also used for the lattice Coulomb gas in \cite{MR2917175} in combination with the ideas from \cite{MR1777310,MR1101688}. Sections~\ref{sec:scales} and \ref{sec:norms} contain the framework that will be used in this paper: Section~\ref{sec:scales} sets up the decomposition of the torus into blocks of increasing scales $L^j$ and associated polymers, and Section~\ref{sec:norms} the norms needed for this analysis, with particular care on the dependence on the finite-range step distribution, which improves as the range becomes large,
and that we expect are suitable for studying the critical spread-out problem. These norms are used to control the renormalisation group map $\Phi_j$, which encodes the evolution of both the coupling constants $(s_j,z_j)$ (along with an inessential scalar constant $E_j$) as well as a remainder coordinate $K_j$, which altogether parametrise the quantity of interest -- typically a (generalised) partition function, see~\eqref{eq:general_RG_step2}. The quantity $K_j$ is an infinite-dimensional object (a so-called polymer activity), which is only tractable explicitly to a given order.

The key is to devise a careful inductive choice of $(s_j,z_j, E_j, K_j)$ rendering $\Phi_j$ contractive. This requires identifying the natural contraction mechanisms in the problem, some of which are specific to periodic potentials. We list these mechanisms separately in Section~\ref{sec:inequalities}. They allow to understand at once \textit{all}
the `relevant contributions' to any polymer activity, i.e.,~those contributions which do not automatically contract, and thus ultimately need to be followed carefully through the iteration.
Our organisation, which records these relevant contributions in terms of a certain localisation operator, is inspired by \cite{MR3332939}, which we believe makes Section~\ref{sec:inequalities}  insightful; see in particular Propositions~\ref{prop:Loc-coupling} and~\ref{prop:Loc-contract}, which capture the essence of the matter: upon removal of their localisation, polymer activities contract (see \eqref{e:Loc-contract-full}), and the localisation possesses the required algebraic structure (see \eqref{e:Loc-coupling}).

The actual renormalisation group map $\Phi_j$ is then defined and analysed in Section~\ref{sec:rg_generic_step}. The evolution of the remainder coordinate $K_j$ is the most difficult. It splits into a linear part, and a nonlinear one. Controlling the latter is among the most technically involved parts of the paper. In Section~\ref{sec:stable_manifold_theorem} we finally analyse the stable manifold of the renormalisation group, i.e.,~we
construct initial conditions which allow to analyse the Discrete Gaussian model. Among other things, one issue throughout is to exhibit enough regularity in $s$, which ultimately translates into a  continuity property (in $s$) of the initial condition $\mathfrak{s}_0^c (\beta, s)$ for a stable stiffness flow. This is key to the final \textit{coup de gr\^ace}, which is an application of the intermediate value theorem to resolve the constraint $s = \mathfrak{s}_0^c (\beta, s)$ at high enough $\beta$, i.e.,~matching the initial condition for a stable flow with the stiffness $s$ sacrificed initially in the Discrete Gaussian model. 

This careful choice of $s$ transcribes into the effective temperature $\beta_{\rm eff}$ of our main result, Theorem~\ref{thm:highbeta}, the proof of which is completed in the short
Section~\ref{sec:integration_of_zero_mode}, by connecting all of the previously obtained ingredients.

In the appendices, we include proofs of some results that are essentially known, but for which the proof in our exact
context is hard to locate or difficult to adapt to our setting without going carefully through the arguments.
In particular,
in Appendix~\ref{app:regulator}, we prove the properties of the regulator defined in Section~\ref{sec:norms}.
In Appendix~\ref{app:completeness}, we show that the spaces defined in Section~\ref{sec:norms} are complete.
In Appendix~\ref{app:bump}, we include a decay estimate for the Fourier transform of the standard bump function, used in Section~\ref{sec:finite_range_decomposition}.

\subsection{Notation}
\label{sec:notation}

Throughout the paper, all constants are uniform in $\beta$ and $\rho$ unless explicitly stated.
We use the notation $|a| \leq O(|b|)$ or $a=O(b)$ to denote $|a|\leq C|b|$ for an absolute constant $C>0$
and $a\sim b$ to denote that $\lim a/b =1$ (where the limit is clear from the context).

Usually, the dimension will be assumed to be $d=2$ unless otherwise emphasised.
Let $e_1, \dots, e_d$ be the basis of unit vectors with non-negative components  spanning $\Z^d$ or the local coordinates of $\Lambda$,
and set $\hat{e} = \{\pm e_1, \cdots, \pm e_d \}$.
For a function $f: \Z^d \to \C$ or $f: \Lambda_N \to \C$, we write
$\nabla^{\mu} f(x) = f(x + \mu) - f(x)$ for $\mu \in \hat{e}$.
For any multi-index $\alpha \in \{ \pm 1, \cdots, \pm d\}^n$ with $n =|\alpha| \geq 1$, 
we write $\nabla^{\alpha} f = \nabla^{e_{\alpha_1}}\cdots  \nabla^{e_{\alpha_n}} f$.
The vector of $n$-th order discrete partial derivatives is denoted by 
\begin{equation}
\nabla^n f (x) = (\nabla^{\mu_1} \cdots \nabla^{\mu_n} f (x) : \mu_k \in \hat{e} \text{ for all } k),
\end{equation}
and we write $|\nabla ^nf(x)|$ for the maximum over all of its components.
The symbol $\Delta$ without subscript denotes the \emph{normalised} nearest-neighbour Laplacian,
\begin{equation}
  \Delta f (x) = \sum_{\mu \in \hat{e}} (f(x+ \mu) - f(x))
  = \sum_{\mu \in \hat{e}} \nabla^\mu f(x)
  = {-} \frac12 \sum_{\mu \in \hat{e}} \nabla^\mu \nabla^{-\mu} f(x),
\end{equation}
whereas $\Delta_J$ denotes the \emph{normalised} Laplacian \eqref{eq:Delta_J_definition} with finite-range step distribution $J$.

Throughout,  
$\E^{\zeta}_{\Gamma}$ denotes the expectation with respect to the (centred) Gaussian random variable $\zeta$ with covariance $\Gamma$.
We omit $\zeta$ and $\Gamma$ whenever they are clear from the context.
Typically, we write $\Eplus$ for $\E_{\Gamma_{j+1}}$, 
where $\Gamma_{j+1}= \Gamma_{j+1}(s,0)$ is the covariance introduced below in \eqref{eq:Gammaj}
and $j$ without further specification is allowed to take values $j=1,\dots, N-2$, where $N \geq 1$ refers to the exponent of the underlying torus size $L^N$, for some $L >1$. Unless explicitly stated otherwise, all results tacitly hold for all choices of integers $N \geq 1$ and $L>1$. 
Throughout $c,C,\dots$ refer to generic positive numerical constants that may vary from place to place.

\section{Smoothing of discrete to continuous model}
\label{sec:first_step}

In this section, we prepare for the renormalisation group analysis by performing several
initial manipulations of the Discrete Gaussian model as defined in \eqref{eq:DG_model_1}.
As a first preliminary,
it is convenient for the subsequent analysis to rescale $\sigma$ by a factor $1/\sqrt{\beta}$ 
so that $\Omega^{\Lambda_N}$ in \eqref{e:Omega-def} gets replaced by
\begin{equation} \label{eq:Omega_rescaled}
   \Omega_\beta^{\Lambda_N} = \{\sigma \in \Z_{\beta}^{\Lambda_N}: \sigma_0 = 0 \}, \quad \Z_{\beta}= 2\pi \beta^{-1/2} \Z.
\end{equation}
Thus, from now on $\avg{\cdot}_\beta$ is the Discrete Gaussian model on $\Omega_\beta^{\Lambda_N}$, i.e.,
for any bounded $F:\Omega_\beta^{\Lambda_N} \to \R$,
\begin{equation}
  \avg{F}_\beta = \frac{1}{Z_{\beta}} \sum_{\sigma\in\Omega_\beta^{\Lambda_N}} e^{-\frac12 (\sigma,-\Delta_J\sigma)} \, F(\sigma).
\end{equation}
As with $\avg{\cdot}_\beta$, we will mostly keep the interaction $J$ and the underlying torus $\Lambda= \Lambda_N$ (of side length $L^N$) implicit throughout this section. The corresponding statements are then simply understood to hold for any choice of $J$ satisfying the conditions above \eqref{eq:Delta_J_definition}, all $N \geq 0$ and $L \geq 1$. In fact the choice of side length for $\Lambda$ will play no role in the present section.

\subsection{Mass regularisation}

In the first regularisation step, we replace the Discrete Gaussian model $\avg{\cdot}_\beta$ supported on $\sigma\in \Omega_\beta^{\Lambda_N}$ by a mass-regularised version
supported on all $\sigma \in \Z_{\beta}^{\Lambda_N}$, i.e., without fixing $\sigma_0$ to be $0$, thus restoring translation invariance.
To this end, given $m^2>0$, for any bounded function $F: \Z_{\beta}^{\Lambda_N} \to \R$, let
\begin{equation} \label{eq:DG-withmass}
  \avg{F}_{\beta,m^2} =
  \frac{1}{Z_{\beta,m^2}}
  \sum_{\sigma \in \Z_{\beta}^\Lambda} e^{-\frac12 (\sigma, (-\Delta_J+m^2)\sigma)} F(\sigma),
\end{equation}
where $Z_{\beta,m^2}$ is the corresponding normalisation constant.
The following lemma shows that we recover $\avg{\cdot}_\beta$ in the limit $m^2\downarrow 0$. In the sequel, for $ t\in \Z_{\beta}$ and $\sigma \in  \Z_{\beta}^{\Lambda_N}$ we write $\sigma+t$ for the shifted configuration with entries $(\sigma +t)_x=\sigma_x + t$, $x \in \Lambda_N$.
\begin{lemma} \label{lemma:m2to0}
Let $F: \Z_{\beta}^\Lambda \to \R$ be such that $F(\sigma) = F(\sigma +t)$ for any constant $t\in \Z_{\beta}$,
and assume that $F\vert_{ \Omega_\beta^{\Lambda}}$ is integrable with respect to $\langle \cdot \rangle_{\beta}$. Then $F$ is integrable under $\langle \cdot  \rangle_{\beta, m^2}$ for all $m^2>0$ and
\begin{equation}
\langle F(\sigma) \rangle_{\beta} = \lim_{m^2 \downarrow 0} \langle F(\sigma) \rangle_{\beta, m^2}.
\end{equation}
\end{lemma}

\begin{proof}
 For $F$ having the above properties, writing any element of $\Z_{\beta}^\Lambda$ as $ \sigma+t$ with $t \in \Z_{\beta}$ and $\sigma \in \Omega_{\beta}^\Lambda$, cf.~\eqref{eq:Omega_rescaled}, one has that
\begin{equation}
\label{eq:F-L^1_m}
\begin{split}
\sum_{\sigma \in   \Z_{\beta}^\Lambda} e^{-\frac12 (\sigma, (-\Delta_J+m^2)\sigma)} |F(\sigma)| 
&= \sum_{t \in \Z_{\beta}} \sum_{\sigma \in   \Omega_{\beta}^\Lambda} e^{-\frac12 (\sigma+t, (-\Delta_J+m^2)(\sigma+t))} |F(\sigma+t)|\\
&= \sum_{\sigma \in   \Omega_{\beta}^\Lambda} e^{-\frac12 (\sigma, -\Delta_J\sigma)} |F(\sigma)| \sum_{t \in \Z_{\beta}} e^{-\frac12 m^2(\sigma+t,\sigma+t)},
\end{split}
\end{equation}
where the second line is obtained using that $F(\sigma) = F(\sigma +t)$ and expanding the exponential (note that $\Delta_J f=0$ when $f$ is constant-valued). Since, uniformly in $\sigma \in \Omega_{\beta}^\Lambda$,
\begin{equation}
 \sum_{t \in \Z_{\beta} } e^{-\frac12 m^2(\sigma+t,\sigma+t)} = \sum_{t \in \Z_{\beta}} \prod_{x \in \Lambda }e^{-\frac12 m^2(\sigma_x+t)^2} \leq \sum_{t \in \Z_{\beta}}  e^{-\frac12 (mt)^2}, 
\end{equation}
where the inequality follows by retaining only $x=0$ with $\sigma_x=0$, and combining with the integrability of $F\vert_{ \Omega_\beta^{\Lambda}}$, it follows that the left-hand side of \eqref{eq:F-L^1_m} is finite; hence $F$ is in $L^1(\langle \cdot  \rangle_{\beta, m^2})$. Moreover \eqref{eq:F-L^1_m} continues to hold without absolute values, as follows readily by the Dominated convergence.
Now, as $m^2\downarrow 0$, for any fixed $\sigma \in \Omega_\beta^{\Lambda}$, we claim that
  \begin{equation} \label{e:m2to0-pf2}
    \sum_{t\in \Z_{\beta}}  e^{-\frac12 m^2 (\sigma+t,\sigma+t)}
    \sim \sum_{t\in \Z_{\beta}}  e^{-\frac12 m^2 t^2 |\Lambda|}.
  \end{equation}
  Indeed, since $|(\sigma,1)t| \leq \frac12 \epsilon t^2 + \frac{1}{2\epsilon} (\sigma,1)^2$ for any $\varepsilon > 0$ by Young's inequality,
  the left-hand side is
  \begin{equation}
    e^{-\frac12 m^2 (\sigma,\sigma)} \sum_{t\in \Z_{\beta}}  e^{-\frac12 m^2 t^2} e^{-m^2 (\sigma,1)t}
    \leq
    e^{-\frac12 m^2 (\sigma,\sigma)}
    e^{\frac{1}{2\epsilon}m^2 (\sigma,1)^2}\sum_{t\in \Z_{\beta}}  e^{-\frac12 (1 - \epsilon ) m^2 t^2}.
  \end{equation}
  Therefore, for all $\varepsilon > 0$,
  \begin{equation}
    \limsup_{m^2 \downarrow 0}\frac{\sum_{t\in \Z_{\beta}}  e^{-\frac12 m^2 (\sigma+t,\sigma+t)}}
    {\sum_{t\in \Z_{\beta}}  e^{-\frac12 (1  - \epsilon  ) m^2 t^2}} \leq 1,
  \end{equation}
  and analogously
  \begin{equation}
    \liminf_{m^2 \downarrow 0}\frac{\sum_{t\in \Z_{\beta}}  e^{-\frac12 m^2 (\sigma+t,\sigma+t)}}
    {\sum_{t\in \Z_{\beta}}  e^{-\frac12 (1  + \epsilon  ) m^2 t^2}} \geq 1.
  \end{equation}
  From this, \eqref{e:m2to0-pf2} follows by taking $\epsilon \to 0$.
  By \eqref{eq:F-L^1_m} and the Dominated convergence theorem, thus
  \begin{equation}
    \sum_{\sigma \in \Z_{\beta}^\Lambda} e^{-\frac12 (\sigma, (-\Delta_J+m^2)\sigma)} F(\sigma)
    \sim
    \sum_{t\in \Z_{\beta}}  e^{-\frac12 m^2 t^2} \sum_{\sigma \in\Omega_\beta^{\Lambda}}  e^{-\frac12 (\sigma, -\Delta_J\sigma)}  F(\sigma) 
  \end{equation}
  and the claim follows by taking a ratio of this and the expression with $F$ replaced by $1$.
\end{proof}

\subsection{Smoothing the Discrete Gaussian model}
\label{sec:tildeVdef}

In the next step, we replace the Discrete Gaussian model with mass $m^2 \in (0,1]$ given by \eqref{eq:DG-withmass}
with a smoothed-out version. For this we write
\begin{equation} \label{e:GFdecomp1}
  (-\Delta_J+m^2)^{-1} = \gamma \id + C(m^2)
\end{equation}
where $\gamma>0$ is a positive constant chosen such that $C(m^2)$ is positive definite.
Assuming $m^2 \in (0,1]$, we have $0<-\Delta_J+m^2 \leq 3 \id$ as quadratic forms, and one
can choose any $\gamma \in (0,1/3)$. Note that $C(m^2)$ inherits symmetry from $-\Delta_J+m^2$ by \eqref{e:GFdecomp1}.
 The parameter $\gamma$ will later be fixed in Proposition~\ref{prop:decomp} below. Until then, the dependence of assertions on $\gamma$ 
 will be kept explicit, and all results hold for any choice of $\gamma \in (0,1/3)$. For later reference, note that the covariance $C(m^2)$ appearing in \eqref{e:GFdecomp1} will be denoted by $C^{\Lambda_N} (m^2) \equiv C (m^2)$ from Section~\ref{subsec:FRresults} onwards (see for instance \eqref{e:iidtorus}). We omit the superscript $\Lambda_N$ for the remainder of Section~\ref{sec:first_step} to avoid unnecessary clutter. 
 

The decomposition \eqref{e:GFdecomp1} implies that for any $\sigma \in \R^\Lambda$, one can rewrite
\begin{equation} \label{e:convid}
  e^{-\frac12 (\sigma,(-\Delta_J+m^2)\sigma)}
  = c(\gamma,m^2) \int_{\R^\Lambda} e^{-\frac{1}{2\gamma}\sum_x(\varphi_x-\sigma_x)^2}
  e^{-\frac12 (\varphi, C(m^2)^{-1} \varphi)} \, d\varphi,
\end{equation}
for a suitable constant $c(\gamma,m^2) \in (0,\infty)$. The identity \eqref{e:convid} is of central importance
%
and a special case of
the well-known property that the sum of two independent Gaussian vectors is Gaussian with covariance the sum of the two individual covariances.

Inserting the identity \eqref{e:convid} into the partition function of the (mass regularised) Discrete Gaussian model \eqref{eq:DG-withmass}, one obtains, for all $\beta ,m^2> 0$ (and $\gamma \in (0,\frac13)$),
\begin{equation} \label{e:onestep}
  Z_{\beta,m^2} = \sum_{\sigma\in\Z_{\beta}^\Lambda} e^{-\frac12 (\sigma,(-\Delta_J+m^2)\sigma)}
  = c(\gamma,m^2,\beta ) \int_{\R^\Lambda}   e^{-\frac12 (\varphi, C(m^2)^{-1} \varphi)} e^{\sum_x\tilde U(\varphi_x)}\, d\varphi,
\end{equation}
where for $\varphi \in \R$ and all $\gamma >0$ we define
\begin{equation} \label{e:tildeF}
  \tilde F(\varphi) = c(\gamma, \beta) \sum_{\sigma \in \Z_{\beta}} e^{-\frac{1}{2\gamma}(\sigma-\varphi)^2},
  \qquad
 \tilde U(\varphi) =  \log \tilde F(\varphi).
\end{equation}
Here $c(\gamma, \beta)>0$ is a constant that is chosen for later convenience such that
\begin{equation}\label{e:norm-tildeF}
1=\frac{1}{2\pi}\int_0^{2\pi} \tilde F(\varphi/\sqrt{\beta}) \, d\varphi= \frac{c(\gamma,\beta)}{2\pi} \int_{\R}e^{-\frac{1}{2\gamma\beta}\varphi^2}  d\varphi.
\end{equation}
Both $\tilde F$ and $\tilde U$ are smooth periodic functions of the single real variable
$\varphi \in \R$. 
For later application, we record the following properties of their Fourier representations.

\begin{lemma}
\label{lemma:Fourier_repn_of_V}
 For any $\gamma>0$ and $\beta>0$,
the Fourier representation of $\tilde F$ is given by
\begin{equation} 
  \tilde F(\varphi) = 1+ \sum_{q=1}^\infty 2e^{-\frac{\gamma\beta}{2} q^2} \cos(q\sqrt{\beta}\varphi), \quad \varphi \in \R. \label{eq:fourier_F}
\end{equation}
Moreover, there exists $C\in (0,\infty)$ such that for any $ \gamma \beta \geq  C$,
the function $\varphi \mapsto\tilde U(\varphi)$ has the Fourier representation
\begin{equation} \label{eq:fourier_tildeV}
\tilde U(\varphi) = \sum_{q=1}^\infty \tilde{z}^{(q)} (\beta) \cos(q \sqrt{\beta}\varphi), \quad \varphi \in \R, 
\end{equation}
with coefficients satisfying
\begin{equation}
  |\tilde{z}^{(q)}(\beta)| \leq 16 e^{-\frac{1}{4} \gamma \beta (1 + q)}.
  \label{eq:z0_estimate}
\end{equation}
\end{lemma}

\begin{proof}
  Let $F(\varphi) = \tilde F(\varphi/\sqrt{\beta}) =e^{\tilde U(\varphi/\sqrt{\beta})}$. Then $F$ is $2\pi$ periodic, see \eqref{e:tildeF} and recall \eqref{eq:Omega_rescaled}, and even. Its Fourier coefficients are given by
  \begin{equation}
    \hat F(q) = \frac{1}{2\pi}\int_0^{2\pi} F(\varphi) e^{iq\varphi} \, d\varphi
    = \frac{c(\gamma,\beta)}{2\pi} \int_\R e^{-\frac{1}{2\gamma \beta}\varphi^2} e^{iq\varphi} \, d\varphi
    \stackrel{\eqref{e:norm-tildeF}}{=} e^{-\frac{\gamma\beta}{2}q^2}, \quad q \in \Z.
  \end{equation}
  Thus \eqref{eq:fourier_F} follows. To prove \eqref{eq:fourier_tildeV}, \eqref{eq:z0_estimate}, consider the following norm on $2\pi$-periodic  functions $f$ (for which the norm is finite): for $c = \frac{1}{4}\gamma \beta$,  denoting by $\hat{f}(q) = \frac{1}{2\pi}\int_0^{2\pi} f(\varphi) e^{iq\varphi} \, d\varphi$ the corresponding Fourier coefficients, one sets
\begin{align}\label{e:fouriernorm}
  \norm{f}_{\ell^1 (c)} = \sum_{q\in \mathbb{Z}} e^{c |q|} |\hat{f}(q)| .
\end{align}    
Using the fact that $\widehat{fg}(q)= \sum_{q' \in \mathbb{Z}}\hat{f}(q)\hat{g}(q'-q)$ for periodic $f$ and $g$, one readily deduces that $  \norm{\cdot}_{\ell^1 (c)} $ is submultiplicative, i.e.,~that $\|fg\|_{\ell^1(c)} \leq \|f\|_{\ell^1(c)}\|g\|_{\ell^1(c)}$,
making the space of $2\pi$-periodic functions with finite norm a unital Banach algebra with unit $f\equiv1$.
Moreover, for $\beta \gamma \geq 4$,
\begin{align}
\norm{F-1}_{\ell^1 (c)} = 2 \sum_{q \geq 1} e^{-\frac{\gamma \beta}{2} q^2 + c q} \leq  4 e^{-\frac{1}{4} \gamma \beta},
\end{align}
where the second inequality follows for instance by completing the square, comparing with a Gaussian integral and applying a standard Gaussian tail estimate. Since $\tilde{U} (\varphi/ \sqrt{\beta}) = \log F(\varphi)$, we have
\begin{align}
\norm{\tilde{U}(\cdot/ \sqrt{\beta})}_{\ell^1 (c)} \leq 2 \norm{F-1}_{\ell^1 (c)} \leq 8 e^{-\frac{1}{4} \gamma \beta},
\end{align}
where we have used the estimate $\|\log F\| \leq 2\|F-1\|$ which is valid in any (unital) Banach algebra with norm $\| \cdot \|$ if $\|F-1\|$ is small, as follows e.g.~by bounding the relevant Taylor remainder.
In view of \eqref{e:fouriernorm}, this yields that $|\tilde{z}^{(q)}(\beta)| \leq 16 e^{-\frac{1}{4} \gamma \beta - c |q|}$ for all $q \geq 1$ with $\tilde{z}^{(q)}(\beta)$ as defined by \eqref{eq:fourier_tildeV}.
\end{proof}

\subsection{Temperature renormalisation}

The identity \eqref{e:onestep} for the partition function
and its  extension to the moment generating function in \eqref{e:reformulation} below
reformulates the analysis of the Discrete Gaussian model in terms of
a smooth periodic potential integrated against a Gaussian field. Ideas of this flavour have been used in various contexts in the past.
However, to achieve sufficient precision to control the scaling limit,
it is crucial for our work to allow for the parameter $s\neq 0$ below, which will correspond to the
stiffness renormalisation of the limiting Gaussian free field, or equivalently, the exact identification of the effective temperature $\betaeff$ in \eqref{e:highbeta-convergence}.

To set up this stiffness renormalisation, first
recall that $\Delta_J$ is the normalised Laplacian, see \eqref{eq:Delta_J_definition}, with step distribution $J$,
and $C(m^2)$ is its Green's function minus a diagonal part, see \eqref{e:GFdecomp1}.
For convenience, we will denote by $\Delta$ without subscript the standard \emph{unnormalised} nearest-neighbour Laplacian; the irrelevant omission of the normalisation for $\Delta$ simplifies some formulas later.
For $|s|$ sufficiently small, $C (m^2)^{-1} - s\Delta$ is positive definite, as shown in Proposition~\ref{prop:decomp_compatible} below (see in particular \eqref{e:C(m)-torus-decomp}, where $C\equiv C^{\Lambda_N}$), hence
\begin{equation}
  C(s, m^2) = (C (m^2)^{-1} - s\Delta)^{-1}, \label{eq:C_s_m2}
\end{equation}
is well-defined and positive-definite. We then introduce, for $s_0 \in \R$,
\begin{equation}
  Z_0(\varphi) \equiv Z_0 (\varphi | \Lambda_N ) = e^{U_0(\varphi)} \stackrel{\text{def.}}{=} e^{\frac{s_0}{2}(\varphi,-\Delta \varphi) + \sum_x\tilde U(\varphi_x)} \label{eq:Z_0_definition}
\end{equation}
with $\tilde{U}$ given by \eqref{e:tildeF} and
\begin{equation}\label{e:tildeC}
  \tilde C(s, m^2) = \gamma(1 + s\gamma\Delta) + (1 + s\gamma \Delta) C(s, m^2) (1 + s\gamma \Delta).
\end{equation}
We return to discuss the interpretation of $Z_0$ and $\tilde C(s,m^2)$ below the proof of the following lemma.
This lemma generalises the partition function identity \eqref{e:onestep},
both by allowing a test function and by allowing the parameter $s\neq 0$ that will
later correspond to the stiffness renormalisation.
The right-hand side of \eqref{e:reformulation} will be our starting point for the renormalisation group analysis.
Recall that $\E_C$ denotes the expectation with respect to the (centred) Gaussian measure with covariance $C$.



\begin{lemma}
    \label{lemma:reformulation}
    For all $\beta>0$,
    $\gamma\in (0,\frac13)$, $m^2 \in (0,1]$, 
    and $|s|$ small enough that \eqref{eq:C_s_m2} is well-defined and positive-definite,
  and with $s_0=s$, one has for any $f\in \R^\Lambda$,
  \begin{align} \label{e:reformulation}
  	\avgb{ e^{(f, \sigma)} }_{\beta, m^2} = e^{\frac12 (f, \tilde C(s, m^2) f)} \frac{\E_{C(s,m^2)} [ Z_0(\varphi+Af) ) ]}{\E_{C(s,m^2)} [ Z_0(\varphi ) ]}
  \end{align}
   where $A=(1 {+} s\gamma \Delta)^{-1}\tilde C(s,m^2)$.
\end{lemma}
\begin{proof}
 Completing the square and recalling \eqref{e:tildeF}, one sees that for any $f,\varphi \in \R$,
  \begin{equation}
    \sum_{\sigma\in \Z_{\beta}} e^{f\sigma} e^{-\frac{1}{2\gamma}(\sigma-\varphi)^2}
    =
    e^{\frac{\gamma}{2} f^2} e^{f\varphi} \sum_{\sigma\in\Z_{\beta}} e^{-\frac{1}{2\gamma}(\sigma-\varphi-\gamma f)^2}
   \propto     e^{\frac{\gamma}{2} f^2}  e^{f\varphi} e^{\tilde U(\varphi+\gamma f)}.
  \end{equation}
  Using the convolution identity \eqref{e:convid}, for any $f\in \R^\Lambda$, one therefore obtains that
  \begin{equation}
    \sum_{\sigma \in \Z_{\beta}^\Lambda} e^{-\frac12(\sigma,(-\Delta_J +m^2)\sigma)} e^{(f,\sigma)}
    \propto e^{\frac{\gamma}{2}(f,f)}\E_{C(m^2)}[e^{(f,\varphi)}e^{\tilde U(\varphi+\gamma f)}]
    .
  \end{equation}
  By definition of $C(s,m^2)$, see \eqref{eq:C_s_m2}, the right-hand side is proportional to
  \begin{align} \label{e:trans1}
    &e^{\frac{\gamma}{2}(f,f)}\E_{C({s,} m^2)}[e^{(f,\varphi)}e^{\frac{s}{2}(\varphi,-\Delta\varphi)}e^{\tilde U(\varphi+\gamma f)}]
      \nnb
    &=e^{\frac{\gamma}{2}(f,f) - \frac{s}{2}\gamma^2 (f,-\Delta f)}\E_{C(s, m^2)} [e^{(f+s\gamma \Delta f,\varphi)}e^{\frac{s}{2}(\varphi+\gamma f,-\Delta(\varphi+\gamma f))}e^{\tilde U(\varphi+\gamma f)} ]
      \nnb
    &=e^{\frac{\gamma}{2}(f,g)}\E_{C(s, m^2)} [e^{(g,\varphi)}Z_0(\varphi+\gamma f) ]
  \end{align}
  where in the second line we again completed the square and in the third line we set
  \begin{equation}
    g= f + s\gamma\Delta f
  \end{equation}
  and used that $s=s_0$ along with \eqref{eq:Z_0_definition}.
  Changing variables from $\varphi$ to $\varphi-C(s,m^2)g$, the last line of \eqref{e:trans1} is seen to equal
  \begin{equation}
    e^{\frac{\gamma}{2}(f,g)+\frac{1}{2} (g, C(s,m^2)g)} \E_{C(s, m^2)}[ Z_0(\varphi+\gamma f+C(s,m^2)g) ]
  \end{equation}
and the term in the exponential can be simplified as
  \begin{equation}
    \gamma (f,g) + (g,C(s, m^2)g)
    =
    (f,\tilde C(s, m^2) f),
  \end{equation}
  and the term in the argument of $Z_0$ can be written as
  \begin{equation}
    Af = \gamma f + C(s,m^2)g = \gamma f + C(s,m^2)(1+s\gamma \Delta) f
    = (1 {+} s\gamma \Delta)^{-1}\tilde C(s,m^2),
  \end{equation}
  yielding \eqref{e:reformulation}.
\end{proof}

To conclude this section, we briefly discuss the role of $U_0$ and $\tilde C(s,m^2)$ introduced in \eqref{eq:Z_0_definition}-\eqref{e:tildeC}.
Compared to $\tilde U$ the potential $U_0$ includes an additional Dirichlet energy term  $(\varphi,-\Delta \varphi)=(\nabla\varphi,\nabla\varphi)$
with prefactor $s_0$. 
This parameter $s_0$ is essentially arbitrary for the moment 
and can be chosen independently of $s$ for most part of our analysis. 
However, it can be
compensated 
by the $s$-dependence of $\tilde C(s,m^2)$ on the right-hand side of \eqref{e:reformulation}
by enforcing the restriction
$s=s_0$ as in the assumption of the last lemma.
Thus the parameter $s=s_0$ corresponds to a division of the Gaussian free field into a part that
serves as reference measure, i.e., the Gaussian measure with covariance $\tilde C(s,m^2)$,
and a part that is interpreted as a perturbation of it.
A careful choice of this division will be made at the end of the analysis.
This choice will be such that
\begin{align}
	\frac{\E_{C(s,m^2)} [Z_0 (\varphi + A f)]}{\E_{C(s,m^2)} [Z_0 (\varphi)]} \rightarrow 1
	\label{eq:Z_N_ratio_convergence}
\end{align}
as $N\rightarrow \infty$ and
the covariance $\tilde C(s,m^2)$ is that of a limiting
Gaussian field that approximates the Discrete Gaussian model on large scales
and that converges to the multiple of the Gaussian free field in Theorem~\ref{thm:highbeta}.
Namely, if $f\in C^\infty(\T^2)$ and $f_N$ is as in the statement of Theorem~\ref{thm:highbeta},
then
\begin{equation}\label{e:sl-beta_eff}
  \lim_{N\to\infty} \lim_{m^2\downarrow 0}(f_N,\tilde C(s,m^2)f_N) = \frac{1}{s+v_{J}^2} (f,(-\Delta_{\T^2})^{-1}f)_{\T^2} ,
\end{equation}
see Lemma~\ref{lemma:integral_of_zero_mode},
and $\betaeff(J,\beta)$ defined as $\beta(1+sv_J^{-2})$
will eventually be the effective temperature in Theorem~\ref{thm:highbeta}.
Then the proof of the theorem will be complete in view of Lemma~\ref{lemma:reformulation}.

\section{Finite-range decomposition}
\label{sec:finite_range_decomposition}

The starting point for our renormalisation group analysis is a finite-range decomposition for the covariances \eqref{eq:C_s_m2}
(recalled for convenience in \eqref{eq:decomp2} below),  which we construct in the present section.
We expect that this construction will be useful for the analysis of other models where an initial renormalisation step can be carried out
due to the removal of a diagonal part from the covariance, i.e., an i.i.d.\ contribution of the Gaussian field.
We also refer to Section~\ref{sec:related_problems} for possible applications.
Since it comes at no additional cost, we formulate the decomposition in any dimension $d\geq 2$. The main results of this section are Propositions~\ref{prop:decomp} and~\ref{prop:decomp_compatible} below, which exhibit the desired decomposition, first on $\Z^d$ and then on $\Lambda_N= (\mathbb{Z}/L^N \mathbb{Z})^d$, respectively.

\subsection{Existence of the finite-range decomposition}
\label{subsec:FRresults}

First recall our convenient convention
that $\Delta$ denotes the standard \emph{unnormalised} nearest neighbour lattice Laplacian
while
$\Delta_J$ is the \emph{normalised} Laplacian with finite-range step distribution $J \subset \Z^d\setminus \{0\}$.
As discussed at the beginning of Section~\ref{sec:tildeVdef},
for any $\gamma \in (0,1/3)$ and any finite-range step distribution $J$, we can then decompose (cf.~\eqref{e:GFdecomp1})
\begin{equation}
  (-\Delta_J + m^2)^{-1} = \gamma  + C(m^2), 
  \label{eq:decomp1}
\end{equation}
with $C(m^2)$ a $J$-dependent positive-definite symmetric matrix.
Then recall \eqref{eq:C_s_m2}, i.e.,
\begin{equation}
  C(s, m^2) := (C(m^2)^{-1}-s\Delta)^{-1} = C(m^2)(1-s\Delta C(m^2))^{-1},
  \label{eq:decomp2}
\end{equation}
which makes sense and is positive definite for $|s|$ small (depending on $J$; see Proposition~\ref{prop:decomp} below), as can be seen from the second representation. The main result of this section yields a decomposition
of $C(s,m^2)$ into an integral of covariances with a finite-range property.
Available results on such decompositions apply to $s=0$ or without subtraction of the constant $\gamma$,
i.e., to $(-\Delta+m^2)^{-1}$ instead of $C(s,m^2)$,
but for our purposes it is important to permit both $s\neq 0$ and the subtraction of $\gamma$.

Recall that the step distribution $J  \subset \Z^d \setminus \{0\}$ is assumed to satisfy the conditions above \eqref{eq:Delta_J_definition}.
The estimates for the resulting decomposition depend on the following parameters specific to $J$:
\begin{alignat}{2}
  \rho_J &= \sup\{ |x|_{\infty} : x\in J \}  &\qquad & \text{(range)}, \label{eq:rJ_range}\\
  v_{J}^2 &= \frac{1}{2|J|} \sum_{x\in J} |x_1|^2  &\qquad & \text{(variance)}, \label{eq:vJ2_variance}\\
  \theta_J &= \inf_{p\neq 0}(\lambda_J (p) / \lambda(p)) &\qquad & \text{(spectral lower bound)}. \label{eq:theta_def}
\end{alignat}
In the spectral lower bound,
$\lambda_J(p)$ and $\lambda(p)$ are the Fourier multipliers of $-\Delta_J$ and $-\Delta$,
defined precisely in Section~\ref{sec:decomp-fourier} below.

\begin{example}\label{EX:range-rho}
  For the standard range-$\rho$ step distribution $J_\rho = \{x\in \Z^d \setminus \{0\}: |x|_\infty \leq \rho\}$,
  \begin{equation} \label{eq:Jrho-constants}
    \rho_{J_\rho} = \rho, \qquad v_{J_\rho}^2 \sim \frac{1}{6} \rho^2,
    \qquad
    \theta_{J_\rho} \geq 3^{-d},
  \end{equation}
  see Lemma~\ref{lem:lambda_rho_error} below.
\end{example}

We now state the main results of this section. We refer to Remark~\ref{r:gamma=0} below regarding a version of these findings for the 
(more standard)
choice $\gamma=0$ in \eqref{eq:decomp1}, which implies various known results of this type (notably for the usual Green's function in the nearest-neighbour case). In the following proposition,
we first consider \eqref{eq:decomp1} and \eqref{eq:decomp2} as operators
on $\Z^d$; the inverses are then well defined as bounded operators on $\ell^2(\Z^d)$ if $m^2>0$.
(To emphasise which space an inverse of an operator $A$ is taken in we will sometimes write $A_{\Z^d}^{-1}$ for the inverse of $A$
on $\ell^2(\Z^d)$ and $A_{\Lambda_N}^{-1}$ for its inverse of $A$ acting on $\R^{\Lambda_N}$).
Thus the proposition considers the infinite-volume case of $\mathbb{Z}^d$ rather than the finite torus
relevant for our application to the Discrete Gaussian model.
The torus case is treated thereafter in Proposition~\ref{prop:decomp_compatible}. In fact, Proposition~\ref{prop:decomp_compatible} can be obtained in large part as a corollary of Proposition~\ref{prop:decomp}. Indeed the contributions to the torus decomposition comprising ranges smaller than the torus size are directly inherited from the decomposition on $\Z^d$, cf.~\eqref{eq:decomp3} and~\eqref{e:C(m)-torus-decomp} below. Scales which `feel' the periodic boundary condition however must be treated separately. In what follows, recall the discrete gradient notation from Section~\ref{sec:notation}.

\begin{proposition} \label{prop:decomp}
  Let $d \geq 2$.
  There exist absolute constants $\gamma >0$ and $\epsilon_{s}>0$ (both purely numerical)
  such that for any finite-range step distribution $J \subset \Z^d \setminus \{0\}$ as specified above \eqref{eq:Delta_J_definition}, 
  the following holds.
  For all $|s| \leq \epsilon_s \theta_J$ and $m^2 \in (0, 1]$, one has a decomposition of the form
  \begin{equation}
    C^{\mathbb{Z}^d} (s, m^2) := \big(  C^{\mathbb{Z}^d} (m^2)^{-1}-s\Delta \big)^{-1}_{\mathbb{Z}^d} = \int_{\rho_J}^{\infty} D^{\mathbb{Z}^d}_t(s, m^2) \, dt, 
    \label{eq:decomp3}
  \end{equation}
where the $D^{\mathbb{Z}^d}_t(s, m^2)$ are positive-definite symmetric kernels with range less than $t$, i.e.,
\begin{equation} \label{eq:Dt_range}
  D^{\mathbb{Z}^d}_t(x,y; s,m^2):= ( \delta_x , D^{\mathbb{Z}^d}_t(s, m^2) \delta_y ) = 0 \;\; \text{whenever} \;\; |x-y|_{\infty} \geq t. 
\end{equation}
The left-hand side depends only on $x-y \in \Z^d$ and is invariant  under lattice rotations.
Moreover,
\begin{itemize}
\item[(i)] uniformly in $(s, m^2) \in [-\epsilon_s \theta_J, \epsilon_s \theta_J] \times (0, 1 ]$, all multi-indices $\alpha$ (including $|\alpha|=0$),
  $t \geq \rho_J$,
\begin{equation}
  |\nabla^\alpha D^{\mathbb{Z}^d}_t(0,x ; s, m^2)|
  \leq C_{\alpha}\rho_J^{-2} t \Big( \frac{\rho_J}{v_J t } \Big)^{d+|\alpha|}   \,
  + C_\alpha \theta_J^{-d-|\alpha|} \rho_J^{-2} t \Big( \frac{\rho_J}{t^2 } \Big)^{d+|\alpha|} e^{-c(\theta_J^{1/2} t)^{1/4}}
  \label{eq:decomp3-bd}
\end{equation}
(note: if $\theta_J$ is bounded from below by a positive value,  the second term can be omitted since $v_J \leq \rho_J/2$);
\item[(ii)] for all $|s| \leq \epsilon_s \theta_J$ and all $t$, the map $m^2 \mapsto D^{\mathbb{Z}^d}_t (s, m^2)$ is continuous and has a limit
\begin{equation}
D^{\mathbb{Z}^d}_t (s) \equiv D^{\mathbb{Z}^d}_t (s, 0) = \lim_{m^2 \downarrow \, 0} D^{\mathbb{Z}^d}_t (s, m^2) \, ,
\label{eq:decomp4}
\end{equation}
and the map $s \mapsto D_t^{\Z^d}(s)$ is analytic in $|s|\leq \epsilon\theta_J$;
\item[(iii)]  if $d=2$, then for all $|s| \leq \epsilon_s \theta_J$, 
  \begin{equation}
    D^{\mathbb{Z}^d}_t(0,0 ; s) =
    \frac{1}{2\pi t(v_J^2+s)}
    \Big( 1+ O \Big(\frac{\rho_J}{t} + \frac{\rho_J^4}{v_J^2 t^2} + \frac{\theta_J^{-2} v_J^2}{t^2} \Big) \Big),
    \quad
    \text{as} \;\;\; \frac{\rho_J}{t} + \frac{\rho_J^4}{ v_J^{2}t^2} + \frac{\theta_J^{-2} v_J^2}{t^2} \to 0 .
    \label{eq:decomp5}
  \end{equation}
\end{itemize}
In the above estimates, all constants are independent of $J$ (and $s$).
\end{proposition}

In the particular case of the standard range-$\rho$ step distribution the conclusions simplify as follows.

\begin{remark}
  For $J = J_{\rho}$ (see Example~\ref{EX:range-rho}), 
  the uniform lower bound on $\theta_{J_\rho}$ in \eqref{eq:Jrho-constants} implies that
  that the domain of $s$ can be chosen independently of $\rho$. For such $s$,
using the bounds from \eqref{eq:Jrho-constants}, the estimates in items \emph{(i)} and \emph{(iii)} above become (see also the note below \eqref{eq:decomp3-bd}),
for $t \geq \rho$,
\begin{equation}
  |\nabla^\alpha D^{\mathbb{Z}^d}_t(0,x ; s, m^2)|
  \leq C_{\alpha}  \rho^{-2} t^{1-d-|\alpha|}
\end{equation}
and (in $d=2$)
\begin{equation}
  D^{\mathbb{Z}^d}_t(0,0 ; s)
  =
  \frac{1}{2\pi t (v_{J_\rho}^2+s)}
  \Big( 1+ O \Big(\frac{\rho}{t} \Big)  \Big),
\end{equation}
with all constants independent of $\rho$, $s$ and $m^2$.
\end{remark}

Proposition~\ref{prop:decomp} applies to covariances defined on all of $\Z^d$.
By periodisation, Proposition~\ref{prop:decomp} and its proof also imply an analogous statement for the discrete torus, which we state next.
Since this is the decomposition we will use in the present article, we
consider the torus $\Lambda_N$ of side length $L^N$
(even though an analogous statement holds for any side length).
For $t< { \frac14 } L^N$, the covariances $D_t^{\Z^d}$
from Proposition~\ref{prop:decomp}
are translation invariant and have range less than half the diameter of the torus. They can thus naturally be identified with
covariances on the torus $\Lambda_N$ by projection. More precisely, with $\pi_N: \Z^d \to \Lambda_N$ denoting the canonical projection and for any $f:\Lambda_N \to \R$, $t< { \frac14 }  L^N$, one sets $D_t f (\pi_N(x)):= D_t^{\Z^d}(f\circ \pi_N)(x)$, for $x \in \Z^d$, and readily verifies that this is well-defined, i.e.,~the right-hand side does not depend on the choice of representative $x$ in the equivalence class. 
On the other hand, for $t\geq { \frac14 } L^N$, the periodisation of the covariance $D_t^{\Z^d}$ does depend on the torus.

\begin{proposition} \label{prop:decomp_compatible}
Let $d \geq 2$, $L > 1$, $N\geq 1,$ and $\Lambda_N= (\mathbb{Z}/L^N \mathbb{Z})^d$.
With the same constants $\gamma>0$ and $\epsilon_s > 0$ as in Proposition~\ref{prop:decomp}, $m^2 \in (0, 1]$, the matrix
\begin{align}\label{e:iidtorus}
(-\Delta_J + m^2)^{-1}_{\Lambda_N} - \gamma = C^{\Lambda_N} (m^2 )
\end{align}
is positive definite and for all $|s| \leq \epsilon_s \theta_J$,
\begin{equation}\label{e:C(m)-torus-decomp}
  (C^{\Lambda_N} (m^2) - s\Delta)_{\Lambda_N}^{-1}
  =
  \int_{ \rho_J}^{\frac{1}{4} L^{N-1}} D^{\mathbb{Z}^d}_t (s, m^2) \, dt
  + \int_{\frac14 L^{N-1}}^\infty \tilde D^{\Lambda_N}_t(s,m^2) \, dt + t_N(s,m^2) Q_N
\end{equation}
where the $D^{\Z^d}_t (s, m^2)$ are the same as in \eqref{eq:decomp3} (with the identification discussed above),
for all $t > \frac14 L^{N-1}$, the covariances
$\tilde D^{\Lambda_N}_t(s,m^2)$ satisfy 
translation and lattice rotation invariance, 
the same upper bounds as $D^{\Z^d}_t$ in \eqref{eq:decomp3-bd},
{the same analyticity in $s$, and the same continuity in $m^2$ (including as $m^2\downarrow 0$).}
Finally, $Q_N$ denotes the matrix with all entries equal to $1/|\Lambda_N|=L^{-dN}$
and $t_N(s, m^2) \in (0,m^{-2})$ is a constant satisfying
\begin{equation} \label{eq:t_N_bound}
  |t_N (s,m^2) - m^{-2}| \leq C\rho_J^{-2}L^{2N}. 
\end{equation}
\end{proposition}

\begin{remark}\label{r:gamma=0} Analogues of Propositions~\ref{prop:decomp} and~\ref{prop:decomp_compatible} continue to hold for the choice $\gamma=0$ in \eqref{eq:decomp1}, yielding for $|s| \leq \epsilon_s \theta_J$ and $m^2 \in (0,1]$ the decomposition 
\begin{equation}
\label{e:decompGF}
(-\Delta_J + m^2- s\Delta)_{\mathbb{Z}^d}^{-1} = \int_{0}^{\infty} D^{\mathbb{Z}^d}_t(s, m^2) \, dt, 
\end{equation}
(along with a corresponding analogue on $\Lambda_N$); the properties \eqref{eq:Dt_range}--\eqref{eq:decomp5} remain valid for all $t \geq \rho_J$. Moreover, the range of $D^{\mathbb{Z}^d}_t$ is $0$ for $t \leq \rho_J$, i.e., $D^{\Z^d}_t(0,x)=1_{x=0}D_t^{\Z^d}(0,0)$ and \eqref{eq:decomp3-bd} is complemented by the fact that $D^{\mathbb{Z}^d}_t(0,0)>0$ is constant for all $t \leq \rho_J$.
The decomposition \eqref{e:decompGF} is obtained by Lemma~\ref{lemma:basic_frd} almost directly (essentially boiling down to \cite[Section~3]{MR3969983}, without needing to perform the series expansion of Definition~\ref{def:C_t}).
In particular, for $J$ the usual nearest-neighbour interaction and $s=0$, \eqref{e:decompGF} recovers a well-known decomposition for the Green's function $(-\Delta + m^2)^{-1}$, see e.g.~\cite{MR3129804,MR3969983,MR2070102}.
Compared to these works, our technical difficulty includes an extra series expansion step of Definition~\ref{def:C_t}.

\end{remark}

The rest of this section is devoted to proving Propositions~\ref{prop:decomp} and~\ref{prop:decomp_compatible}.

\subsection{Preliminaries on Fourier transforms}
\label{sec:decomp-fourier}

Before proving Proposition~\ref{prop:decomp}, we collect some preliminaries and conventions
about normalisation of Fourier transforms and of the lattice Laplacian $\Delta$ and its generalised
version $\Delta_J$ with step distribution $J$.
Throughout,  $\Lambda$ is a discrete $d$-dimensional torus
of integer period $R$
with Fourier dual
\begin{align}
\Lambda^* = \{ 2\pi R^{-1} k \, : \, k \in \{  -\lceil (R-2)/2 \rceil, \dots, \lfloor R/2 \rfloor \}^d  \} \subset (-\pi, \pi]^d .
\end{align}
For an integrable function $\hat f: (-\pi,\pi]^d \to \R$, we define
\begin{align}
  f^{\Z^d}(x) &= \int_{(-\pi,\pi]^d} \hat f(p) e^{ip\cdot x} \, \frac{dp}{(2\pi)^d},\label{eq:discr-FT}
  \\
  f^{\Lambda}(x) &= \frac{1}{|\Lambda|} \sum_{p\in\Lambda^*} \hat f(p) e^{ip\cdot x}, 
\end{align}
where $|\Lambda|=R^d$ denotes the number of points in $\Lambda$.
Then by the Poisson summation formula
\begin{equation}
  f^{\Lambda}(x) = \sum_{y\in \Z^d} f^{\Z^d}(x+yR). \label{eq:Poisson_summation_formula}
\end{equation}
This notation also applies for translation invariant covariances, i.e., when a function $f(x,y)$ depends only on $x-y$
we will usually identify it with the function $f(0,x)$.

We write $\lambda = \lambda(p)$ and $\lambda_{J,m^2} = \lambda_{J,m^2} (p)\geq 0$ for the Fourier multipliers of $-\Delta$ and $-\Delta_J + m^2$: 
\begin{equation}
 \label{eq:def_lambda}
 \begin{split}
&\lambda(p)= \sum_{ |e|=1} (1- \cos(p\cdot e)), \\
&\lambda_J (p)  = \frac{1}{|J|} \sum_{x\in J} (1- \cos(p\cdot x)), \qquad \lambda_{J, m^2} (p)=\lambda_J (p) + m^2 
\end{split}
\end{equation}
(recall our convention regarding normalisation of $\Delta$ and $\Delta_J$).
The following standard lemmas provides some comparison estimates
for $\lambda_J (p)$ and $\lambda(p)$, which will be useful in the sequel.

\begin{lemma}\label{lem:lambda_error}
For any 
step distribution $J$ as above \eqref{eq:Delta_J_definition} (with implicit constants independent of $J$),
\begin{align}
\lambda_J (p) = v_J^2 |p|^2 + O(\rho_J^2 v_J^2|p|^4) & \qquad (p\rightarrow 0)  \label{e:lambdarhoasymp} \\
\lambda_J (p) \leq \min \{1 , v_J^2 |p|^2  \} & \qquad (p \in (-\pi , \pi]^{d} ), \label{e:lambdaJ_upper_bound}
\end{align}
with $\rho_J$ and $v_J$ defined by  \eqref{eq:rJ_range} and \eqref{eq:vJ2_variance}.
Moreover, 
$\lambda(p) = |p|^2 +O(|p|^4)$ as $p \rightarrow 0$
and $\lambda(p) \in [ \frac{4}{\pi^2}|p|^2, |p|^2 ]$ for $p \in (-\pi, \pi]^d$,
hence in particular $\theta_J \leq \frac{\pi}{4} v_J^2$.
\end{lemma}

\begin{proof}
Let $v_J^2$ be as defined by \eqref{eq:vJ2_variance}. 
  Then as $p\to 0$, substituting $1-\cos x = \frac{x^2}{2} + O(x^4)$ in \eqref{eq:def_lambda}, one finds that
  \begin{equation} 
  \begin{split}
    \lambda_J(p)
    &= \frac{1}{|J|} \sum_{y \in J} \Big( 1-\cos \big( \sum_{i=1}^d p_i y_i \big) \Big) \\
    &= \frac{1}{2 |J|} \sum_{y\in J} |p|^2 y_1^2  + O\pb{ \frac{1}{|J|}\sum_{y\in J} |y|^4 |p|^4} 
     = v_J^2 |p|^2 + O(\rho_J^2 v_J^2 |p|^4).
   \end{split}
  \end{equation}
  The upper bound in \eqref{e:lambdaJ_upper_bound} follows similarly, using the inequality $1- \cos x \le x^2/2$ valid for all $x \in \mathbb{R}$ instead.

To see the lower bound for $\lambda(p)$, 
consider the function 
$g(q) = 1 - \cos(q) - \frac{2}{\pi^2} q^2$
on $q\in [0, \pi]$. Then $g(0) = g(\pi) =0$ while $g'(q) = \sin(q) - \frac{4}{\pi^2} q$ has only one non-zero root,
hence $g(q)$ does not attain 0 on $(0, \pi)$, i.e.,  $g(q) \geq 0$ on $[0, \pi]$. Therefore
\begin{equation}
  \lambda(p) = 2\sum_{i=1}^d (1-\cos(p_i)) \geq \frac{4}{\pi^2} |p|^2
\end{equation}
which is the claimed lower bound.
\end{proof}

\begin{lemma}\label{lem:lambda_rho_error}
For the step distribution $J = J_{\rho} = \{x\in\Z^d \setminus \{0\}: |x|_\infty \leq \rho\}$,
\begin{equation}
  \rho_{J_{\rho}} = \rho,\qquad
  v_{J_{\rho}}^2 \equiv v_\rho^2 \sim \frac{1}{6} \rho^2,
  \quad \text{as }\rho\to\infty, \label{eq:v_rho^2_asymptotic}
\end{equation}
and with $\lambda_{\rho} \equiv \lambda_{J_\rho}$,
\begin{equation} 
    \label{e:lambda-lambdaJ_ratio}
  \lambda(p) \leq 3^d \lambda_{\rho} (p),
\end{equation}
i.e., $\theta_{J_{\rho}} \equiv \theta_{\rho} = \inf_{p \neq 0} \lambda_{\rho}(p) / \lambda (p) \geq 3^{-d}$.
\end{lemma}

\begin{proof}
Using that $\sum_{j=1}^\rho j^2 \sim \frac13 \rho^3$ as $\rho\to\infty$,
\begin{equation}
  v_\rho^2 = \frac{(2\rho+1)^{d-1}}{(2\rho+1)^d-1}\sum_{j=1}^\rho j^2
  \sim \frac{\rho^2}{6}.
\end{equation}
  To show $\lambda \leq 3^d\lambda_\rho$, first note that since
  $\sum_{a=1}^{\rho} \cos(ax) = \frac{\sin((\rho + 1/2) x)}{2\sin(x/2)} - \frac{1}{2}$,
\begin{align}
  \lambda_{\rho} (p)
  & = \frac{1}{(2\rho +1)^d -1} \sum_{|y|_{\infty} \leq \rho} (1- \prod_{i=1}^d e^{ip_iy_i}) \nnb
  & = \frac{(2\rho+1)^d}{(2\rho +1)^d -1} \Big(1- \prod_{i=1}^d \frac{1}{2\rho+1} \sum_{x=-\rho}^\rho e^{ip_i x} \Big) \nnb
  & = \frac{(2\rho +1)^d}{(2\rho +1)^d -1} \Big( 1 - \prod_{i=1}^d \frac{\sin( (2\rho + 1 ) p_i /2 )}{(2\rho + 1)\sin(p_i / 2)} \Big).
\end{align}
But
\begin{align}
\sup_{\rho \geq 1} \frac{\sin ( (2\rho + 1) p_1  /2 )}{ (2\rho +1)  \sin (p_1 / 2)} = \frac{\sin(3 p_1 / 2)}{3\sin(p_1 / 2)} ,
\end{align}
so $\lambda_{\rho} (p) \geq  (1-3^{-d}) \lambda_{\rho=1} (p)$,
and
\begin{align}
  (3^d-1)
  \lambda_{\rho=1}(p)
  = \sum_{|y|_{\infty} = 1} (1-\cos(\sum_{i=1}^d p_i y_i) )
  \geq \sum_{|y|_{1} = 1} (1-\cos(\sum_{i=1}^d p_i y_i) )
  = \lambda(p)
  .
\end{align}
Since $(1-3^{-d})/(3^d-1)= 3^{-d}$, the claim holds. 
\end{proof}

\subsection{Proof of Proposition~\ref{prop:decomp}: finite-range property}

The starting point for the construction of the finite-range decomposition is the following lemma,
from which one can directly obtain the finite-range decomposition when $s=0$.
The lemma originated in \cite{MR3129804}, but we obtain here a better decay estimate, which
is important for our construction of the finite-range decomposition for $s \neq 0$. 
Also, the lemma specifies the choice of $\gamma$ for \eqref{eq:decomp1}.

\begin{lemma} \label{lemma:basic_frd}
  For $t>0$, there exist polynomials $P_t$ of degree at most $t$ such that for $\lambda \in (0,3]$,
  \begin{equation}
  \label{eq:1/xrewrite}
    \frac{1}{\lambda} = \int_0^\infty t^2 P_t(\lambda) \frac{dt}{t}.
  \end{equation}
  For $\lambda \in (0,3]$ and $t >1$, the polynomials satisfy $P_t(\lambda) \geq 0$ and there is 
an entire function $f$ that is non-negative on the real axis and
  satisfies $\int_0^\infty t^2 f(t) \frac{dt}{t} = 1$
  such that
  \begin{align}
    \label{e:Pt-bd}
    P_t(\lambda) &\leq C e^{-c(\lambda t^2)^{1/4}}
    \\
    \label{e:PtFt-bd}
    |P_t(\lambda) - f(\sqrt{\lambda} t)| &\leq Ct^{-1}e^{-c(\lambda t^2)^{1/4}}.
  \end{align}
  For $t \leq 1$, $P_t(\lambda) = \gamma/t$ for some constant $\gamma>0$.
\end{lemma}

\begin{proof}
  Let $f: \R \to [0,\infty)$ be any non-negative function 
  with the following properties: the Fourier transform of $f$ is smooth, symmetric and has support in $[-1,1]$, and $ \int_0^{\infty} t^2 f(t)\frac{dt}{t}=1$.
  Then by \cite[Lemma~3.3.3]{MR3969983}, \eqref{eq:1/xrewrite} holds for $\lambda \in [0,4]$ with the function $P_t$ given by
  \begin{equation}
  \label{eq:P_t-def}
    P_t(\lambda) = f_t^*(\arccos(1-\frac12\lambda))
  \end{equation}
  where
  \begin{equation}
    f_t^*(x) = \sum_{n\in\Z} f(xt-2\pi nt).
  \end{equation}
 By \cite[Lemma 3.3.5]{MR3969983}, \eqref{eq:P_t-def} defines a polynomial $P_t(\cdot)$ on $(0,3]$, of degree bounded by $t$.
  We will now choose $f$ as follows. Let
  \begin{equation}
    \kappa(s) = e^{-(1-(2s)^2)^{-1}} 1_{|s| < 1/2}
  \end{equation}
  be the standard bump function with support $[-1/2,1/2]$.
  By Proposition~\ref{prop:kappa_hat_decay}, $|\hat\kappa(x)|  \leq Ce^{-|x|^{1/2}}$ for all $x \in \R$.
  We set $\hat f(s) = c (\kappa*\kappa)(s)$,
  with $c>0$ chosen as to ensure the normalisation $\int_0^\infty t^2 f(t) \,  \frac{dt}{t} =1$.
  Then $f$ has all the required properties. In particular, since
  its Fourier transform has compact support, it extends to an entire function, as easily seen 
  by expanding the exponential in the Fourier integral, yielding an absolutely convergent power series.
  Also, $f(x) = c\hat\kappa(x)^2  \leq C' e^{-2 |x|^{1/2}}$ for $x\in\R$.
  For $t\geq 1$,
  \begin{equation}
  \label{eq:f_t^*-bd}
    f_t^*(x) \leq C' e^{-(2 t|x|)^{1/2}} \sum_{n\geq 0}e^{-\sqrt{4 \pi n}} \leq C'' e^{-(2 t|x|)^{1/2}}.
  \end{equation}
  Since $\arccos(1-\frac12 \lambda) \geq \sqrt{\lambda}$, the estimate \eqref{e:Pt-bd} follows immediately from \eqref{eq:P_t-def} and \eqref{eq:f_t^*-bd}.
  The bound \eqref{e:PtFt-bd} follows similarly using $|f'(x)| \leq C''' e^{- |x|^{1/2}}$ (which follows from the explicit form of $f$ and that $\kappa$ has compact support) and using that $\arccos(1-\frac12\lambda)-\sqrt{\lambda} = O(\lambda)$; see \cite[Proposition~3.1]{MR3129804} for a similar argument.
  The constant $\gamma$ is given by $\hat f(0)/2\pi$, see
  \cite[Lemma~3.3.6]{MR3969983}.
\end{proof}

By applying the previous lemma, we first construct a finite-range decomposition for $s=0$.
To this end, insert $\lambda_{J,m^2}$ into \eqref{eq:1/xrewrite} for $m^2 \leq 1$ (so that $\lambda_{J,m^2} \leq 3$ and Lemma~\ref{lemma:basic_frd} is in force). Since $\lambda_{J,m^2}$ has range $\rho_J$,
in the sense that it is the Fourier multiplier of an operator with range $\rho_J$,
and since $P_t$ is a polynomial of degree at most $t$,
it follows that 
$P_t(\lambda_{J,m^2})$ has range $\rho_Jt$.
We therefore set in Fourier space
\begin{equation}
\label{eq:defC_t}
\hat{D}_t ( p ; m^2)
=  \rho_J^{-2} t P_{\rho_J^{-1}t}(\lambda_{J, m^2} (p )), \quad p \in (-\pi,\pi]^d.
\end{equation}
By \eqref{eq:1/xrewrite} and the explicit form of $P_t$ for $t \leq 1$, it follows with $\hat{D}_t (m^2) \equiv\hat{D}_t (\cdot ;  m^2)$ that
\begin{equation}
\label{eq:def_C}
  \frac{1}{\lambda_{J, m^2}} = \int_0^\infty \hat{D}_t (m^2) \, dt = \gamma + \int_{\rho_J}^\infty \hat{D}_t (m^2) \, dt = \gamma + \hat{C} (m^2) , 
\end{equation}
with the last equality defining $\hat{C} (m^2)= \hat{C} (p; m^2)$, $p \in (-\pi,\pi]^{d}$,
and we used that 
\begin{equation}
  \int_0^{ \rho_J} \rho_J^{-2} t P_{\rho_J^{-1}t} \, dt
  =
  \int_0^{1} t P_{t} \, dt = \gamma.
\end{equation}
 By \eqref{eq:def_C}, the function $\hat{C} (m^2)$ thus defined in terms of $\hat{D}_t (m^2)$ is indeed the Fourier transform of $C(m^2)$ appearing in \eqref{eq:decomp1}. With a view towards our aim \eqref{eq:decomp3}, we expand for $|s|< \theta_J = \inf_{m^2} \inf_{\lambda \neq 0} (\lambda_{J, m^2}/\lambda)$ 
and $m^2 \in (0,1]$,
\begin{align}
\label{eq:Cs1}
  \hat{C} (s, m^2)
  \stackrel{\text{def.}}{=} ( \hat{C} (m^2)^{-1} +s\lambda)^{-1}
  &= \hat{C}(m^2)(1+s\lambda \hat{C}(m^2))^{-1}
    \nnb
  &  = \sum_{l=0}^\infty s^{2l} \lambda^{2l} \hat{C} (m^2)^{2l+1}(1 -s\lambda \hat{C}(m^2)).
\end{align}
The expansion is absolutely convergent since $|s|\lambda \hat C(m^2) \leq |s| \lambda/\lambda_{J,m^2}\leq |s|/\theta_J < 1$.
Moreover, this condition implies that
the following integrand is positive: 
\begin{equation}
\label{eq:Cs2}
  1 -s\lambda \hat{C} ( m^2)
  = \frac{\lambda_{J, m^2}}{\lambda_{J, m^2}}-s\lambda \hat{C} (m^2)
  = \int_0^\infty (\lambda_{J,m^2} -s\lambda 1_{t> \rho_J}) \hat{D}_t ( m^2) \, dt.
\end{equation}
This motivates the following definition of the finite-range decomposition for $|s|<\theta_J$.

\begin{definition}
\label{def:C_t}
For $m^2 \in (0,1]$, all $|s| < \theta_J$ and $t>0$, let
\begin{multline} 
  \hat D_t(s, m^2) 
  =  \\
   \frac{1}{4\lambda}
  \sum_{l=0}^\infty s^{2l} \int_{[0,\infty)\times [\rho_J,\infty)^{2l+1}: \sum t_i=(t-\rho_J)/4} (\lambda_{J,m^2} -s \lambda 1_{t_0 > \rho_J}) \hat D_{t_0} (m^2)  \prod_{i=1}^{2l+1}  \lambda \hat D_{t_i} (m^2) \, dt_i \, dt_0
  . 
  \label{eq:C_t_s_m2_definition}
\end{multline}
\end{definition}

In this definition, the integral $\int_{\sum t_i = T} \prod_{i=0}^{2l+1} dt_i$  over the simplex is the push-forward of the Lebesgue measure on $\mathbb{R}^{2l+1}$ along the map $(t_1, \dots, t_{2l+1}) \mapsto (T-\sum_{k=1}^{2l+1} t_i, t_1, \dots, t_{2l+1})$,
i.e.,
\begin{equation}
  \int_{[0,\infty)\times [\rho_J,\infty)^{2l+1}: \sum t_i=T} f(t_0,\dots, t_{2l+1})  \prod_{i=0}^{2l+1} dt_i = \int_{ [\rho_J,\infty)^{2l+1}: \sum t_i\leq T} f(T- \sum t_i, t_1,\dots, t_{2l+1}) \prod_{i=1}^{2l+1} dt_i
\end{equation}
for $T > (2l+1)\rho_J$, and the left-hand side is interpreted as $0$ when $T\leq (2l+1)\rho_J$.
The same remark applies to various similar quantities below.
In particular, $\hat D_t(s,m^2)=0$ for $t\leq 5\rho_J$, and if $\hat D_t(s,m^2)$ is nonzero, then $T=(t-\rho_J)/4$ in \eqref{eq:C_t_s_m2_definition}
satisfies $T \in [\frac15 t, \frac14 t]$. 

\begin{proof}[Proof of Proposition~\ref{prop:decomp}: finite-range property]
  We will show that the covariances $D_t^{\Z^d}(s,m^2)$ 
  with Fourier transforms given by \eqref{eq:C_t_s_m2_definition} define the desired decomposition \eqref{eq:decomp3}.
  First, it is clear
  from \eqref{eq:defC_t} and Lemma~\ref{lemma:basic_frd} that $D_t^{\mathbb{Z}^d} (s,m^2)$ is positive definite. 
  That the decomposition \eqref{eq:decomp3} holds 
  follows by substituting \eqref{eq:def_C} and \eqref{eq:Cs2} into \eqref{eq:Cs1}
  and using the change of variables
  \begin{equation}
    \int_{[0,\infty)^{2l+2}}
    dt_0 \cdots \, dt_{2l+1} \, f(t_0, \dots, t_{2l+1})
    =
    \int_{0}^\infty dT \int_{[0,\infty)^{2l+2}: \sum t_i=T}
    dt_0 \cdots \, dt_{2l+1} \, f(t_0, \dots, t_{2l+1}),
  \end{equation}
  with $T=(t-\rho_J)/4$.

Next we verify the finite-range property.
Since $\lambda$ has range $1$ and $D_{t_i}(m^2)$ has range $t_i$,
we see that $\lambda D_{t_i} (m^2)$ has range at most $1+t_i \leq 2t_i$ for $t_i \geq \rho_J \geq 1$
and $\lambda_{J, m^2} D_{t_0} (m^2)$ has range $\rho_J + t_0 \leq \rho_J+ 2t_0$.
Since $\sum t_i = \frac14 (t-\rho_J)$, from the definition \eqref{eq:C_t_s_m2_definition},
it follows that the range of $D_t(s,m^2)$ is at most
$\rho_J+\frac12 (t-\rho_J)=\frac12 (t+\rho_J) \leq t$ for $t >  \rho_J$.
On the other hand, $D_t(s,m^2)=0$ for $t \leq  \rho_J$.
\end{proof}

We now proceed to prove the estimates asserted in items (i)-(iii) of Proposition~\ref{prop:decomp} for the above finite-range decomposition.

\subsection{Proof of Proposition~\ref{prop:decomp}: (i) and (ii)}

To prove the estimates (i) and (ii), 
we begin with the following lemma which we will use repeatedly.
In the lemma, we use conventions $\norm{g}_{\infty} = \sup_x |g(x)|$ and $\norm{g}_{1} = \int |g(x)| dx$.

\begin{lemma} \label{lem:fbd}
Let $g: [0,\infty)\to \R$ be submultiplicative, i.e., $g(x)g(y) \leq g(x+y)$,
and satisfy $\tilde C = \max\{\|xg\|_\infty,\|xg\|_1\} < \infty$. 
Then for all integers $k \geq1$,
  \begin{equation}
    \int_{[0,\infty)^k: \sum_{i=1}^k t_i=t} \prod_{i=1}^k \lambda g(\sqrt{\lambda}t_i)t_i dt_i
    \leq
    \sqrt{\lambda}\min\Big\{ \tilde C^k, \frac{( \sqrt{\lambda}t)^{2k-1}}{(2k-1)!} g(\sqrt{\lambda}t) \Big\}. \label{eq:fbd}
  \end{equation}
 In particular, the estimate holds for $g(x) = e^{-c\sqrt{x}}$ for any $c > 0$.
\end{lemma}

\begin{proof}
  We bound the left-hand side in two ways. First, the left-hand side equals
\begin{equation}
  \sqrt{\lambda} \int_{[0,\infty)^k: \sum_{i=1}^k u_i=\sqrt{\lambda}t} \prod_{i=1}^k g (u_i)u_i du_i
  \leq
  \sqrt{\lambda} \|g u\|_\infty \|g u\|_1^{k-1} 
  \leq \sqrt{\lambda} \tilde C^k.
\end{equation}
On the other hand, since $g(x)g(y) \leq g(x+y)$, we can also bound it by
\begin{align}
  \int_{[0,\infty)^k: \sum_{i=1}^k t_i=t} \prod_{i=1}^k \lambda g(\sqrt{\lambda}t_i)t_i dt_i
  &\leq
  \lambda^k g (\sqrt{\lambda}t) \int_{[0,\infty)^k: \sum_{i=1}^k t_i=t} \prod_{i=1}^k t_i dt_i
  \nnb
  &= \sqrt{\lambda} \frac{(\sqrt{\lambda}t)^{2k-1}}{ (2k-1)!} g(\sqrt{\lambda}t)
\end{align}
where we used
\begin{equation}
h_k (t) :=  \int_{[0,\infty)^k: \sum_{i=1}^k t_i=t} \prod_{i=1}^k t_i dt_i
  = t^{2k-1} \int_{[0,\infty)^k: \sum_{i=1}^k u_i=1} \prod_{i=1}^k u_i du_i 
  = \frac{t^{2k-1}}{ (2k-1)!}.
\end{equation}
The last equality can be seen by induction: $h_2(1)=1/6$ and
\begin{align}
h_{k}(1) = \int_{0}^1 s h_{k-1}(1-s) ds
= \int_0^{1} s(1-s)^{2k-3} h_{k-1}(1) ds
= \frac{h_{k-1}(1)}{(2k-2)(2k-1)}
\end{align}
advances the induction.
\end{proof}

\begin{lemma} \label{lem:fbd-sum}
  Let $g(x)=e^{-c\sqrt{x}}$. Then there are $\epsilon_s>0$ and constants $C,\tilde c$ such that for $|s|\leq \epsilon_s$,
  \begin{equation} \label{eq:fbd-sum}
    \frac{1}{\lambda} \sum_{l=0}^\infty s^{2l} \int_{[0,\infty)^{2l+2}: \sum t_i=t} \prod_{i=0}^{2l+1}  \lambda t_i g(\sqrt{\lambda}t_i) \, dt_i
    \leq
    C t e^{-\tilde c(\sqrt{\lambda}t)^{1/4}}
    .
  \end{equation}
\end{lemma}

\begin{proof}
  By Lemma~\ref{lem:fbd}, the left-hand side is bounded by
  \begin{align}
    \frac{1}{\sqrt{\lambda}} \sum_{l=0}^\infty s^{2l}\min \ha{ \tilde C^{2l}, \frac{(\sqrt{\lambda}t)^{4l+3}}{(4l+3)!} g(\sqrt{\lambda}t)  }
\end{align}
where $\tilde C=\max\{ \|xg\|_1 ,\|xg\|_\infty\}$. 
We set $\epsilon_s= \frac{1}{4} \max\{ 1, \tilde C \}^{-1}$ so that $\tilde C |s| \leq 1/4$ whenever $|s| \leq \epsilon_s$.
For $\lambda t^2 \leq 2$ by using the second term in the minimum, this immediately gives the desired estimate since
\begin{equation}
  t (\sqrt{\lambda}t)^{2}g(\sqrt{\lambda}t) \sum_{l=0}^\infty \frac{(\lambda t^2s)^{2l}}{(4l+3)!} 
  \leq t (\sqrt{\lambda}t)^{2}g(\sqrt{\lambda}t)\sum_{l=0}^\infty 2^{-l} 
  \leq 4 t g(\sqrt{\lambda}t).
\end{equation}
Thus assume $\lambda t^2 \geq 2$.
By switching between the two terms in the minimum at $l = l_0$, 
the left-hand side of the claim
is bounded by the sum of the following two contributions:
\begin{align}
  \frac{1}{\sqrt{\lambda}}\sum_{l=l_0+1}^\infty (\tilde C s)^{2l}
  &\leq \frac{1}{\sqrt{\lambda}}\sum_{l=l_0+1}^\infty 16^{-l}
  \leq \frac{1}{\sqrt{\lambda}} 16^{-l_0}
\intertext{and}
  \frac{g(\sqrt{\lambda}t)}{\sqrt{\lambda}} \sum_{l=0}^{l_0} \frac{(\sqrt{s\lambda} t)^{4l}(\sqrt{\lambda} t)^{3}}{(4l+3)!} 
  &\leq
    \lambda t^3 g(\sqrt{\lambda}t) \sum_{l=0}^{l_0} (\sqrt{\lambda} t)^{4l}
        \leq
    t g(\sqrt{\lambda}t) (\sqrt{\lambda} t)^{4l_0+2}
      .
\end{align}
Choosing $l_0= (c/16) (\sqrt{\lambda}t)^{1/2}(\log(\sqrt{\lambda}t))^{-1}$ gives the upper bound
\begin{align}
  t\pa{
    \frac{e^{-\log(16) l_0}}{\sqrt{\lambda t^2}} +
    g(\sqrt{\lambda}t) e^{8\log(\sqrt{\lambda} t)l_0}
  }
  &\leq
  t\pa{
    e^{-\log(16) l_0} +
    e^{-c(\sqrt{\lambda}t)^{1/2}} e^{8\log(\sqrt{\lambda} t)l_0}
    }
  \nnb
  &\leq
  t\pa{
    e^{-\log(16) (c/16) (\sqrt{\lambda}t)^{1/2}(\log(\sqrt{\lambda} t))^{-1}} +
    e^{-\frac{c}{2}(\sqrt{\lambda}t)^{1/2}}
    }
  \nnb
  &\leq
  2t  e^{-\tilde c (\sqrt{\lambda}t)^{1/2}(\log(\sqrt{\lambda} t))^{-1}} 
\end{align}
which is less than the claimed bound.
\end{proof}

\begin{proposition} \label{prop:Ct_elementary_bound}
  There are constants $\tilde C,\tilde c, \epsilon_s >0$ independent of $J$, $m^2$, $s$ 
  such that for $|s| \leq \epsilon_s \theta_J$ and $m^2 \in (0,1]$,
  \begin{equation}
  \label{eq:Ct_elementary_bound}
  0 \leq \hat D_t( p \,  ; s, m^2)
  \leq \tilde C \rho_J^{-2} t \exp \Big( -\tilde c \Big( \rho_J^{-1} \sqrt{\lambda_{J,m^2} (p)} \, t \Big)^{1/4} \Big).
\end{equation}
\end{proposition}

\begin{proof}
  Let $g(x)=e^{-c\sqrt{x}}$ so that $P_t(\lambda) \leq Cg(\sqrt{\lambda}t)$ by \eqref{e:Pt-bd}.
Then by the definition \eqref{eq:C_t_s_m2_definition}, 
\begin{align}
\hat D_t (\cdot ; s, m^2) &\leq \frac{1}{4\lambda_{J,m^2}} \sum_{l=0}^{\infty} s^{2l} \sup_{\lambda \neq 0} \big( \frac{\lambda}{\lambda_{J,m^2}} \big)^{2l} \int_{[0,\infty)^{2l+2}: \sum t_i=T} \prod_{i=0}^{2l+1} \rho_J^{-2}  \lambda_{J,m^2} \, t_i P_{\rho_J^{-1} t_i}(\lambda_{J,m^2}) \, dt_i \nnb
  &\leq \frac{1}{4\rho_J \lambda_{J,m^2}} \sum_{l=0}^{\infty} s^{2l} \, \theta_J^{-2l} \int_{[0,\infty)^{2l+2}: \sum t_i=\rho_J^{-1} T} \prod_{i=0}^{2l+1} \lambda_{J,m^2} \, t_i P_{t_i}(\lambda_{J,m^2}) \, dt_i
    \nnb
  &\leq \frac{C^2}{4\rho_J \lambda_{J,m^2}} \sum_{l=0}^{\infty} (Cs/\theta_J)^{2l} \int_{[0,\infty)^{2l+2}: \sum t_i=\rho_J^{-1} T} \prod_{i=0}^{2l+1} \lambda_{J,m^2} \, t_i g(\sqrt{\lambda_{J,m^2}}t_i) \, dt_i
    ,
\end{align}
where $T=(t-\rho_J)/4 \in [\frac15 t, \frac14 t]$.
Thus the claim follows from  Lemma~\ref{lem:fbd-sum}
with $s$ replaced by $Cs/\theta_J$, with $t$ replaced by $T/\rho_J$,
and with $\lambda$ replaced by $\lambda_{J,m^2}$.
\end{proof}

\begin{proof}[Proof of Proposition~\ref{prop:decomp} (i) and (ii)]
We will show \eqref{eq:decomp3-bd}, i.e.,
\begin{equation}
  |\nabla^\alpha D^{\mathbb{Z}^d}_t(0,x ; s, m^2)|
  \leq C_{\alpha}\rho_J^{-2} t \Big( \frac{\rho_J}{v_J t } \Big)^{d+|\alpha|}   \,
  { + C_\alpha \theta_J^{-d-|\alpha|} \rho_J^{-2} t \Big( \frac{\rho_J}{t^2 } \Big)^{d+|\alpha|} e^{-c(\theta_J^{ 1/2} t)^{1/4}}}.
\end{equation}
Clearly, \eqref{eq:Ct_elementary_bound} implies that $ |\nabla^\alpha D_t(0,x ; s, m^2)|$ is bounded,
uniformly in $s$ and $m^2$, by
\begin{equation} \label{eq:decomp3-pf1}
    \tilde C \int_{[-\pi,\pi]^d}  \lambda^{|\alpha|/2} \rho_J^{-2} te^{-\tilde c(t\rho_J^{-1} \sqrt{\lambda_{J}(p)})^{1/4}} \frac{dp}{(2\pi)^d}.
\end{equation}
To apply the lower bound on $\lambda_{J,m^2}$ from Lemma~\ref{lem:lambda_error},
i.e., $\lambda_{J} = v_J^2|p|^2(1+O(\rho_J^2|p|^2))$,
we will split the above integral into integrals over $|p| \geq 1/\rho_J$ and $|p| \leq 1/\rho_J$.
The latter integral is bounded by (with other constants $C,c$)
\begin{equation}
C_\alpha \int_{[-\rho_J^{-1},\rho_J^{-1}]^d}  |p|^{|\alpha|} \rho_J^{-2} te^{-c(\rho_J^{-1}{v_J} t |p|)^{1/4}} dp,
\end{equation}
which yields the main term in \eqref{eq:decomp3},
as can be seen by substituting $p \to \rho_J v_J^{-1} t^{-1} p$. 
For the integral over $|p|\geq 1/\rho_J$ we use $\lambda_J (p) \geq \theta_J\lambda (p) \geq \frac{4}{\pi^2}\theta_J |p|^2$ on $[-\pi,\pi]^d$ to obtain the bound (again with possibly different constants)
\begin{equation}
  C_\alpha \int_{[-\pi,\pi]^d\setminus[-\rho_J^{-1},\rho_J^{-1}]^d}  |p|^{|\alpha|} \rho_J^{-2} te^{-c(\rho_J^{-1} \theta_J^{1/2} t |p|)^{1/4}} dp,
\end{equation}
which by substituting $p\to \rho_Jp$ is seen to be bounded by
\begin{equation}
  C_\alpha \rho_J^{d-2+|\alpha|}t \int_{\R^d\setminus[-1,1]^d}  |p|^{|\alpha|} e^{-c(t\theta_J^{ 1/2} |p|)^{1/4}} dp
  \leq C_\alpha \theta_{J}^{-(|\alpha|+d)/2} \rho_J^{d-2+|\alpha|} t^{1-|\alpha|-d} e^{- c'(\theta_J^{ 1/2} t)^{1/4}}  .
\end{equation}
Also using $e^{-c(t \theta_J^{ 1/2} )^{1/4}} \leq C_n (t\theta_J^{ 1/2} )^{-n} e^{-\frac{1}{2} c(t \theta_J^{ 1/2} )^{1/4}}$ with $n = d + |\alpha|$,
\begin{equation}
  C_\alpha \rho_J^{d-2+|\alpha|}t \int_{\R^d\setminus[-1,1]^d}  |p|^{|\alpha|} e^{-c(t\theta_J^{ 1/2} |p|)^{1/4}} dp
  \leq C'_\alpha \theta_J^{-(d+|\alpha|)} \rho_J^{-2} t \Big( \frac{\rho_J}{t^2} \Big)^{d+|\alpha|} e^{- \frac{1}{2} c'(\theta_J^{ 1/2} t)^{1/4}}
  .
\end{equation}
Now using this bound and assuming $\theta_J^{-1}$ bounded, $t\geq \rho_J$, we directly have the required bound.

Since all estimates above are uniform in $m^2$, 
the continuity claim of Proposition~\ref{prop:decomp}~(ii) is immediate.

\end{proof}


\subsection{Proof of Proposition~\ref{prop:decomp}: (iii)}

Next we collect the last piece of our proof of Proposition~\ref{prop:decomp}, which are the
asymptotics of the covariances in two dimensions.

\begin{proposition} \label{prop:C_t(0)}
  Let $d=2$. Then for $|s| \leq \epsilon_s\theta_J$, 
  as $\rho_J/t +( \rho_J^4 v_J^{-2} + \theta_J^{-2} v_J^2 )/t^2 \to 0$,
  \begin{equation}
    D^{\mathbb{Z}^2}_t (0,0 \, ; s) = \int_{[-\pi,\pi]^2} \hat{D}_t (p \, ; s) \frac{dp}{(2\pi)^2}
    = \frac{1}{2\pi t(v_J^2+s)} \Big( 1+ O \Big(\frac{\rho_J}{t} + \frac{\rho_J^4 v_J^{-2}}{t^2} + \frac{\theta_J^{-2} v_J^2}{t^2} 
    \Big)  \Big).
  \end{equation}
\end{proposition}

\begin{proof}
  To estimate the integral over \eqref{eq:C_t_s_m2_definition}, we will approximate 
  \begin{equation} \label{eq:D_l_before_approx}
    \frac{1}{4\lambda} \int_{[0,\infty)\times [\rho_J,\infty)^{2l+1}: \sum t_i=(t-\rho_J)/4} (\lambda_J-s \lambda 1_{t_0 > \rho_J}) \hat{D}_{t_0} \prod_{i=1}^{2l+1}  \lambda \hat{D}_{t_i} \, dt_i \, dt_0
  \end{equation}
  as follows:
  replace $P_t(\lambda)$ by $f(\sqrt{\lambda}t)$
  using \eqref{e:PtFt-bd};
  replace $\lambda (p)$ by $|p|^2$ and $\lambda_J (p)$ by $v_J^2 |p|^2$
  using \eqref{e:lambdarhoasymp}; 
  remove the constraints $t_i > \rho_J$ from the integration domain and similarly $t_0 \geq \rho_J$ from the integrand;
  and replace $(t-\rho_J)/4$ by $t/4$.
  After these approximations (which we will justify afterwards, in inverse order), we are left with
  \begin{align} 
    &\frac{1}{4|p|^2}  (v_J^2-s) \int_{[0,\infty)^{2l+2}: \sum t_i=t/4} \prod_{i=0}^{2l+1} |p|^2  \rho_J^{-2} t_i f \big( \rho_J^{-1} t_i v_J |p| \big) \, dt_i 
    \nnb
    &= \frac{1}{4|p|^2}(v_J^2-s) \int_{[0,\infty)^{2l+2}: \sum t_i=t/4} \prod_{i=0}^{2l+1}  |p|^2 \rho_J^{-2} t_i  f\big( \rho_J^{-1} t_i v_J |p| \big) \, dt_i 
    \nnb
    &= \frac{1}{|p|^2} (v_J^2-s) t^{-1} \int_{[0,\infty)^{2l+2}: \sum u_i=1} \prod_{i=0}^{2l+1} \frac{1}{4^2}\rho_J^{-2} t^2  |p|^2 u_i f \big(\frac14  u_i \rho_J^{-1} t \, v_J |p| \big) \, du_i 
    \nnb
    &= \frac{1}{|p|^2} (v_J^2-s) v_J^{-4l-4} \, t^{-1} \tilde f_{2l} \big( \frac14 \rho_J^{-1} t \, v_J |p| \big) \, , \label{eq:D_l_after_approx}
  \end{align}
  with
  \begin{equation}
    \tilde f_{2l}(y) = \int_{[0,\infty)^{2l+2}: \sum u_i=1} \prod_{i=0}^{2l+1}  f (y u_i) y^2 u_i \, du_i , \qquad (y \in [0,\infty)).
  \end{equation}
  Note that $\tilde f_{2l}(0)=0$, that $\tilde f_{2l}$ decays rapidly, and that
  for all $l \in \N$ and $t >0$, 
  \begin{align}
    \int_{\R^2} \frac{dp}{|p|^2} \tilde f_{2l}(t|p|)
    &=
    \int_{\R^2} \frac{dp}{|p|^2} \tilde f_{2l}(|p|)
    \nnb
    &= 2\pi \int_{0}^\infty \frac{dy}{y} \tilde f_{2l}(y)
      \nnb
    &=2\pi \int_{0}^\infty \frac{dy}{y} \int_{[0,\infty)^{2l+2}: \sum u_i=1} \prod_{i=0}^{2l+1}  f (y u_i) y^2 u_i \, du_i
      \nnb
    &=2\pi \int_{[0,\infty)^{2l+2}} \prod_{i=0}^{2l+1}  f (u_i) u_i \, du_i 
      =2\pi \pa{\int_0^\infty f (u) u \, du }^{2l+2} = 2\pi
      \label{eq:tildef_integral}
  \end{align}
  where the last equality follows from Lemma~\ref{lemma:basic_frd}.
  By definition, $D_t^{\Z^2}(0,0;s)$ is the integral of \eqref{eq:C_t_s_m2_definition}
  over $p \in [-\pi,\pi]^2$ with respect to $dp/(2\pi)^2$,
  and \eqref{eq:C_t_s_m2_definition} is the sum over \eqref{eq:D_l_before_approx} multiplied by $s^{2l}$.
  Using the above approximation \eqref{eq:D_l_after_approx} for \eqref{eq:D_l_before_approx} and then replacing the integration domain $[-\pi,\pi]^2$ by $\R^2$,
  we obtain the main contribution to $D_t^{\Z^d}(0,0;s)$ as
  \begin{equation}
     \frac{1}{2\pi t}  (v_J^2-s) v_J^{-4} \sum_{l=0}^\infty (v_J^{-2}  s)^{2l}
    =\frac{1}{2\pi t} (v_J^2-s) v_J^{-4} (1-v_J^{-4}s^2)^{-1}
    =\frac{1}{2\pi t} (v_J^2+s)^{-1}.
  \end{equation}

  In the remainder of the proof, we show that the approximations we made above are smaller than the claimed error term.

\smallskip\noindent
\emph{Error from replacing $(t-\rho_J)/4$ by $t/4$:}
the same computation with $t/4$ instead of $(t-\rho_J)/4$ gives
\begin{equation}
  \frac{1}{2\pi (t-\rho_J)} (v_J^2+s)^{-1}
  =
  \frac{1}{2\pi t} (v_J^2+s)^{-1} (1+\frac{\rho_J}{t-\rho_J})
  =
  \frac{1}{2\pi t} (v_J^2+s)^{-1} \Big( \big(1+O(\frac{\rho_J}{t} \big) \Big),
\end{equation}
so the error is smaller than claimed.

\smallskip\noindent
\emph{Error from extending the integral from $p \in [-\pi,\pi]^2$ to $p \in \R^2$:}
By changing to polar coordinates, this error is of order 
\begin{align}
  &  \frac{1}{t} (v^2_J - s) v^{-4}_J  \sum_{l=0}^{\infty} (v_J^{-2} s)^{2l}  \int_{\R^2 \setminus [-\pi,\pi]^2} \frac{dp}{|p|^2} \tilde f_{2l}(\frac14 \rho_J^{-1} t v_J |p|) \nnb
  &    \leq
    \frac{2\pi }{t} (v^2_J - s) v^{-4}_J 
                                                                                                                                                                                 \int_{\rho_J^{-1} t v_J \pi / 4}^{\infty} \frac{d y}{y} \sum_{l=0}^{\infty} (v_J^{-2}s)^{2l}\int_{[0, \infty)^{2l+2} : \sum u_i = 1} \prod_{i=0}^{2l+1} r^2 u_i f(r u_i) du_i
\end{align}
By Lemma~\ref{lem:fbd-sum}, the right-hand side is bounded, up to some absolute multiplicative factor, by
\begin{equation}
  \frac{1}{t} (v^2_J - s) v^{-4}_J \int_{\rho_J^{-1} t  v_J \pi /4}^{\infty} d y\, e^{-c\, y^{1/4}} 
  =  O\Big( \frac{e^{-c' (\rho_J^{-1} t v_J)^{1/4}}}{v_J^2 t} \Big)
  = \frac{1}{2\pi t v_J^2} O\Big(\frac{\rho_J v_J^{-1}}{t}\Big)
  = \frac{1}{2\pi t v_J^2} O\Big(\frac{\rho_J}{t}\Big)
\end{equation}
where we used  that $v_J \geq 1/4$ for all $J$.

\smallskip\noindent
\emph{Error from removing the restriction on $t_0 \geq \rho_J$ from the integration:}
The error is bounded by
\begin{align}
&  \int_{\R^2} dp \, \frac{1}{\lambda} \sum_{l=0}^{\infty} s^{2l+1}	\int_{[0, \infty)^{2l+2} : \sum t_i = t/4} 1_{t_0 \leq  \rho_J} \prod_{i=0}^{2l+1}  |p|^2  \rho_J^{-2} t_i  f(  v_J \rho_J^{-1} |p| t_i ) dt_i  \nnb
&	\leq \int_{\R^2} dp \, \frac{v_J^{-2}}{\lambda} \sum_{l=0}^{\infty} (v_J^{-2} s )^{2l+1} \int_0^{\rho_J} \Big( \int_{[0,\infty)^{2l+1}: \sum t_i = t/4-t_0} \prod_{i=0}^{2l+1} v_J^2 \rho_J^{-2} |p|^2  t_i f( v_J \rho_J^{-1} |p| t_i) \, dt_i \Big) dt_0  \nnb
& \leq C  |s|  \int_{\R^2} dp \,  \rho_{J}^{-4} |p|^2 \int_0^{\rho_J} t_0 (t/4-t_0) e^{-c' (v_J \rho_J^{-1} |p| (t/4 - t_0))^{1/4}} f( v_J \rho_J^{-1} |p| t_0)  dt_0 \label{eq:t_0_restriction_removed}
\end{align}
where the final inequality follows from Lemma~\ref{lem:fbd-sum} and the fact that $f(x) \leq C e^{-{ c}|x|^{1/2}}$
which follows from \eqref{e:Pt-bd}--\eqref{e:PtFt-bd}.
But since $e^{-c' (v_J \rho_J^{-1} |p| (t/4 - t_0))^{1/4}} f(v_J \rho_J^{-1} |p| t_0) \leq C' e^{-c'' (v_J \rho_J |p| t)^{1/4}}$ for some $c'', C' >0$ and $t/4 -t_0 \geq t/20$ because $t \geq 5\rho_J \geq 5 t_0$, the last integral is bounded by
\begin{equation}
C |s| \int_{\R^2} \rho_J^{-2} |p|^2 t e^{-c'' (v_J \rho_J^{-1} |p| t)^{1/4}  } dp  \leq \frac{C |s| \rho_J^2}{v_J^4 t^3} \leq \frac{1}{2\pi v_J^2 t} O\Big( \frac{\rho_J^2 }{t^2} \Big) \leq  \frac{1}{2\pi v_J^2 t} O\Big( \frac{\rho_J}{t} \Big) ,
\end{equation} 
where the second inequality holds because $|s| \leq \epsilon_s \theta_J \leq c \epsilon_s v_J^2 $ and the final inequality because $t \geq 5\rho_J$.

\smallskip\noindent
\emph{Error from removing the restriction on $t_j \geq \rho_J$ from the integration:}
Similarly as above, the error for removing the restriction on $t_j$ ($j\geq1$) in the integral $\int_{t_j \geq \rho_J,  \sum_i t_i = t/4}$ is bounded by
\begin{align}
& \int_{[0, \infty)^{2l+2} : \sum t_i = t/4} 1_{t_j \leq  \rho_J} v_J^2 \rho_J^{-2} |p|^2 t_0 f(v_J \rho_J^{-1} |p| t_0) \prod_{i=1}^{2l+1}  |p|^2 \rho_J^{-2} t_i  f(  v_J \rho_J^{-1} |p| t_i ) dt_i dt_0  \nnb
& \leq  (C v_J^{-2})^{2l+1} \int_{[0, \infty)^{2l+2} : \sum t_i = t/4} 1_{t_j \leq  \rho_J} \prod_{i=0}^{2l+1}  v_J^2 |p|^2  \rho_J^{-2} t_i e^{-c( v_J \rho_J^{-1} |p| t_i )^{1/2}}  dt_i .
\end{align}
But since the last expression is symmetric in $j$, we can just replace $1_{t_j \leq \rho_J}$ by $1_{t_0 \leq \rho_J}$, so summing the errors over $j \in \{1, \cdots, 2l+1\}$ and applying $\int_{\R^2} dp \, \lambda^{-1} \sum_{l=0}^{\infty} s^{2l+1}$ gives the bound
\begin{align}
& \int_{\R^2} dp \, \frac{1}{|p|^2} \sum_{l=0}^{\infty} ( 2l+1 )(C v_J^{-2} s)^{2l+1}	\int_{[0, \infty)^{2l+2} : \sum t_i = t/4} 1_{t_0 \leq  \rho_J} \prod_{i=0}^{2l+1} v_J^2 |p|^2  \rho_J^{-2} t_i e^{-c( v_J \rho_J^{-1} |p| t_i )^{1/2}}  dt_i \nnb
& \leq C \int_{\R^2} dp \, v_J^4 \rho_J^{-4} \int_0^{\rho_J} t_0 (t/4 - t_0) e^{-c' (v_J \rho_J^{-1} |p| (t/4 - t_0) )^{1/4} }  f(v_J \rho_J^{-1} |p| t_0 ) dt_0.
\end{align}
Comparing this with \eqref{eq:t_0_restriction_removed}, this integral is bounded by $\frac{1}{2\pi v_J^2 t} O( \frac{\rho_J^2}{t^2} ) = \frac{1}{2\pi v_J^2 t} O( \frac{\rho_J}{t} )$ because $t \geq 5\rho_J$.

\smallskip\noindent
\emph{Error from replacement of $\lambda_J(p)$ by $v_J^2|p|^2$ and $\lambda(p)$ by $|p|^2$:}
As in the argument following \eqref{eq:decomp3-pf1},
the contribution from $|p| \geq \rho_J^{-1}$ is bounded by
\begin{equation}
  \int_{|p| \geq \rho_J^{-1}} dp \, \hat{D}_t (p ; s,m^2)
  \leq C_0 \theta_J^{-2} t^{-3} e^{-c(\theta_J^{1/2} t)^{1/4}}
  \leq
  \frac{1}{2\pi v_J^2 t} O( \frac{\theta_J^{-2} v_J^2}{t^2} ).
\end{equation}
It remains to consider the contribution coming from $|p| \leq \rho_J^{-1}$.
But then by Lemma~\ref{lem:lambda_error},
\begin{align}
& 0 \leq |p|^2 - \lambda(p) \leq O( |p|^4  ) \leq O( |p|^2) \\
& 0\leq v_J^2|p|^2 - \lambda_J(p) \leq O(\rho_J^2 v_J^2 |p|^4) \leq O(v_J^2 |p|^2) .
\end{align}
With $\hat \kappa$ as in {in the proof of} Lemma~\ref{lemma:basic_frd}, we have $f = c \hat{\kappa}^2$, so
    \begin{align}
      |f(\rho_J^{-1}t\sqrt{\lambda_J (p)})-f(v_J \rho_J^{-1}t|p|)| 
&      \leq C \rho_J^{-1} t ( v_{J} |p| - \sqrt{\lambda_J (p)} ) \max \{ \hat{\kappa} ( \rho_J^{-1} t \lambda_J (p)  ), \hat{\kappa} ( v_J \rho_J^{-1} t |p|  ) \} \nnb
&      \leq C \rho_J^{-1} t  \min \Big\{ 1,  \frac{\rho_J^2v_J^2}{2v_J} |p|^3 \Big\} e^{-\frac{1}{2} (\sqrt{\lambda_J (p)} \rho_J^{-1} t  )^{1/2}  } \nnb
&      \leq C  \rho_J^4 v_J^{-2} t^{-2}  e^{-c (v_J \rho_J^{-1} t |p| )^{1/2}  }
    \end{align}
where the first inequality holds because $\norm{\hat{\kappa}'}_{\infty} < \infty$ and the second because $\kappa(x)$ is decreasing in $|x|$ and $|\hat{\kappa} (x)| \leq C e^{-\frac{1}{2} |x|^{1/2}}$. 
Thus the error from this approximation is, up to an absolute multiplicative factor, bounded by
\begin{equation}
\frac{\rho_J^4 v_J^{-2}}{t^2} \int_{|p| \leq \rho_J^{-1} } \frac{dp}{|p|^{2}} \sum_{l=0}^{\infty} (2l+2) (C' s)^{2l+1}  \int_{[0, \infty)^{2l+2} : \sum t_i = t/4 }  (v_J^2 +s)  \prod_{i=0}^{2l+1} |p|^2 \rho_J^{-2} t_i e^{-c (v_J \rho_J^{-1} |p| t_i)^{1/2}} .
\end{equation} 
Since $|s| \leq \epsilon_s \theta_J \leq O(v_J^2)$, this error is again of order $\frac{1}{2\pi v_J^2 t} O( \frac{\rho_J^4 v_J^{-2}}{t^2} )$, comparing this expression with \eqref{eq:D_l_after_approx}.

\smallskip\noindent
\emph{Error from replacement $P_t(x)$ by $f(\sqrt{x}t)$:}
We consider the difference between
\begin{equation}
  (1-s)\int_{[0,\infty)^{2l+2}: \sum t_i=t/4}  \lambda_J \rho_J^{-2} t_0 P_{\rho_J^{-1} t_0} (\lambda_J)  \prod_{i=1}^{2l+1} \lambda \rho_J^{-2}  t_i P_{\rho_J^{-1} t_i}(\lambda_J) \, dt_i
\end{equation}
and
\begin{equation}
  (1-s)\int_{[0,\infty)^{2l+2}: \sum t_i=t/4} \lambda_J \rho_J^{-2} t_0 f(\sqrt{\lambda_J} \, \rho_J^{-1} t_0) \prod_{i=1}^{2l+1}  \lambda \rho_J^{-2} t_i f(\sqrt{\lambda_J} \, \rho_J^{-1} t_i ) \, dt_i.
\end{equation}
By     \eqref{e:PtFt-bd},
one has $P_{\rho_J^{-1} t}(\lambda_J)-f(\sqrt{\lambda_J} \, \rho_J^{-1}t) = (\rho_J/t) g(\sqrt{\lambda_J}t)$
with $g(x)=Ce^{-c\sqrt{x}}$.
This is essentially the same bound as $P_t(\lambda_J)$ or $f(\sqrt{\lambda_J}t)$
except for an additional factor $\rho_J/t$.
Therefore, again using Lemma~\ref{lem:fbd}, the difference between the above two displays is bounded by
\begin{equation}
  O\pa{  \frac{\rho_J}{t} C^l v_{J}^{-4l+2} \frac{1}{t} \tilde{g}_{2l} (\sqrt{\lambda_J} \, \rho_J^{-1} t)  }
  ,
\end{equation}
when $\tilde g_{2l}$ is defined analogously to $\tilde f_{2l}$.
As in \eqref{eq:tildef_integral}, the integral of $\tilde g_{2l}(t|p|)$ over $dp/|p|^2$ is bounded by $2\pi C^{2l}$ with $C \geq \int_0^\infty g(u)u\, du$, for all $t>0$.
Hence possibly decreasing $|s|$ relative to $C$ we obtain the claimed relative error $O(\rho_J/t)$.

    Summing up the bounds gives the claimed error.
\end{proof}

\subsection{Proof of Proposition~\ref{prop:decomp_compatible}}

Having proved the estimates for the full plane covariance decomposition, the torus analogue is not difficult to prove.

\begin{proof}[Proof of Proposition~\ref{prop:decomp_compatible}]
By definition, 
  \begin{align}
    & D_t^{\Z^d}(0,x; s,m^2) = \int_{[-\pi,\pi]^d} e^{ip\cdot x} \hat D_t(p \, ; s,m^2) \, \frac{dp}{(2\pi)^d} 
\end{align}    
    and we define 
    \begin{equation}    
    D_t^{\Lambda_N}(0,x; s,m^2) = \frac{1}{|\Lambda_N|}\sum_{p\in\Lambda_N^*} e^{ip\cdot x} \hat D_t(p \, ; s,m^2), \label{eq:D_t^Lambd_N_definition}
  \end{equation}
  where $\Lambda_N^* \subset (-\pi,\pi ]^d$ is the dual torus.
  For $t<L^N /2$, the finite-range property and Poisson summation \eqref{eq:Poisson_summation_formula} imply that
  \begin{equation}
    D_t^{\Z^d}(0,x; s,m^2) = D_t^{\Lambda_N}(0,x; s,m^2).
  \end{equation}
  So we are only left to prove \eqref{eq:t_N_bound} and the bound on $\tilde{D}_t^{\Lambda_N}$.
  Let $t_N = \int_{\frac14 L^{N-1}}^{\infty} \hat{D}_t (0 ; s, m^2)\, dt$ and
  \begin{equation}    
     \tilde D_t^{\Lambda_N}(0,x; s,m^2) = \frac{1}{|\Lambda_N|}\sum_{p\in\Lambda_N^* \setminus \{0\}} e^{ip\cdot x} \hat D_t(p \, ; s,m^2). \label{eq:tildeD_t^Lambd_N_definition}
  \end{equation}
To see the bound for $t_N$, just notice that
\begin{align}
t_N = \int_0^{\infty} \hat D_t (0 \, ; s, m^2) dt - \int_0^{\frac14 L^{N-1}} \hat D_t (0 \, ; s, m^2) dt
\leq m^{-2} - C \rho_J^{-2} L^{2N-2}
\end{align}
by \eqref{eq:def_C} and Proposition~\ref{prop:Ct_elementary_bound}.
  The proof of the bound on $\tilde{D}_t^{\Lambda_N}$ is analogous to the argument below \eqref{eq:decomp3-pf1}
  using that all $p$ that contribute satisfy $|p|>2\pi L^{-N}$ and that $t \geq \frac12 L^{N}$.
  Indeed,
  \begin{equation}
    \frac{1}{|\Lambda_N|} \sum_{p \neq 0} \lambda^{|\alpha|/2} \hat D_t(p;s,m^2)
    \leq
    \frac{1}{|\Lambda_N|} \sum_{p \neq 0}  \lambda^{|\alpha|/2} \rho_J^{-2}te^{-\tilde c(t\rho_J^{-1}\sqrt{\lambda_{J,m^2} (p)}))^{1/4}}.
  \end{equation}
  The contribution from $2\pi L^{-N} < |p| \leq \rho_J^{-1}$ is
  \begin{align}
    & \rho_J^{-2} t (\rho_J^{-1}v_J t)^{-|\alpha|} \frac{1}{|\Lambda_N|} \sum_{0<|p|\leq \rho_J^{-1}}  (\rho_J^{-1} v_J t|p|)^{|\alpha|} e^{-c(\rho_J^{-1}{v_J} t |p|)^{1/4}}
    \nnb
    & \qquad \qquad \qquad \qquad \leq C_\alpha \rho_J^{-2} t  (\rho_J^{-1}v_J t)^{-|\alpha|}
      \frac{1}{|\Lambda_N|} \sum_{0<|p|\leq \rho_J^{-1}} e^{-\frac12 c(\rho_J^{-1}{v_J} t |p|)^{1/4}},
  \end{align}
but since $r \mapsto e^{-c r^{1/4}}$ is decreasing for $r \geq 0$ and $2\pi L^{-N} <  \rho_{J}^{-1}$, we have the domination
\begin{align}
\frac{1}{|\Lambda_N|} \sum_{0<|p|\leq \rho_J^{-1}} e^{-\frac12 c(\rho_J^{-1}{v_J} t |p|)^{1/4}} \leq 2d \int_{|p| \leq 2 \rho_J^{-1} } e^{-\frac12 c(\rho_J^{-1}{v_J} |p| t)^{1/4}} \, dp.
\end{align}
Hence the contribution from $|p| \leq \rho_J^{-1}$ is bounded by $C''_{\alpha} \rho_J^{-2} t ( \frac{\rho_J}{v_J t}  )^{|\alpha| +d}$.
Finally, the contribution from $|p| \geq \rho_J^{-1}$ is bounded by
\begin{equation}
\frac{C}{|\Lambda_N|} \sum_{|p| \geq \rho_J^{-1}} |p|^{|\alpha|} \rho_J^{-2} t e^{-c (\rho_J^{-1} \theta_J^{1/2} |p| t )^{1/4} } 
\leq C_{\alpha} \rho_J^{-2} t (\rho_J^{-1} \theta_J^{-1/2} t)^{-|\alpha|} \sum_{|p|\geq \rho_J^{-1}} e^{-\frac{c}{2} (\rho_J^{-1} \theta_J^{1/2} |p| t )^{1/4} }
\end{equation}
but again by the same domination, the estimate for $|\nabla^{\alpha} \tilde{D}_t^{\Lambda_N}|$ is the same as that for $|\nabla^{\alpha} D_t^{\Z^d}|$.
The claim that $D_t^{\Lambda_N} (s,m^2)$ is continuous in $m^2$ and attains a limit as $m^2 \downarrow 0$ is deduced from the fact that the partial absolute sums $\sum_{|p| \leq R} |\hat{D}_t (p ; s,m^2)|$ have a bound uniform in $R$ and $m^2$.
\end{proof}

\section{Scales and polymers}
\label{sec:scales}

In this and the remaining sections,
$\Lambda_N$ always denotes a discrete torus of side length $L^N$,
for integers $L > 1$, $N\geq 1$. Later we will further assume that $L =\ell^M$ for integers $\ell>1$, $M\geq 1$.

\subsection{Blocks and polymers}
\label{sec:polymersdef0}

We follow the set-up for the renormalisation group coordinates of \cite{MR2523458}.
Thus for any scale $j=0,1,\dots,N$, we call \textit{$j$-block} any set of the form $\pi_N (B)$, where~$\pi_N : \Z^d \to \Lambda_N$ is the canonical projection and $B =x+([0,L^j)\cap \mathbb{Z})^d$ for some $x\in L^{j}\mathbb{Z}^d$.
The set of $j$-blocks is denoted by $\cB_j\equiv \cB_j(\Lambda_N)$. It induces a partition of $\Lambda_N$ into $j$-blocks.
A \textit{$j$-polymer} is any set $X$ which is obtained as the union of $j$-blocks,
and we then denote by $\mathcal{B}_j (X) \subset \mathcal{B}_j$ the set of $j$-blocks contained in $X$.
The set of $j$-polymers is denoted by $\cP_j \equiv \cP_j (\Lambda_N)$. Note that the family $\cP_j$ is decreasing in $j$. For $X \in \cP_j$, its closure $\overline{X} \in \cP_{j+1}$ is the union of all $(j+1)$-blocks which intersect $X$,
i.e., $\overline{X}$ is the `smallest' $Y \in \cP_{j+1}$ such that $X \subset Y$.

Next, a \textit{connected polymer} is a polymer $X \neq \emptyset$ which forms a connected set in $\ell^{\infty}$-sense. 
Two connected polymers $X_1, X_2$ are called \textit{connected} if $X_1 \cup X_2$ is a connected polymer; this is denoted by $X_1 \sim X_2$ and we write $X_1 \not\sim X_2$ if $X_1$ and $X_2$ are not connected.
The set of connected $j$-polymers is denoted by $\cP_j^c \equiv \cP_j^c (\Lambda_N)$.
It is worth highlighting that $\emptyset \notin \cP_j^c$ by this definition.
For $X \in \mathcal{P}_j$ we write $\operatorname{Comp}_j(X) \subset  \mathcal{P}_j^c$ for the set of constituting connected polymers, i.e.,
each $Y \in \operatorname{Comp}_j(X)$ is a maximal connected polymer in $X$ and the union over all such $Y$ is $X$.
Denoting $|X|_j=|\cB_j(X)|$, the number of
$j$-blocks contained in $X$, 
a connected polymer $X \in\cP_j^c$ is called a \emph{small set} if $|X|_j \leq 2^d$,
and denote $X \in \cS_j$.
Finally, for any $X \in \cP_j$, we define its \textit{small-set neighbourhood} as $X^* = \bigcup S$ where the union ranges over all $S \in \cS_j$ such that $S \cap X \neq \emptyset$.

For later reference, we note that the combinatorial results of Lemmas 6.15--6.19 from \cite[Section~6.4]{MR2523458} all hold in the present set-up.

\subsection{Massless finite-range decomposition}
\label{subsec:limit_m_to_0}

As in Section~\ref{sec:finite_range_decomposition} (cf.~also above \eqref{eq:Delta_J_definition}),
let $J \subset \Z^2 \setminus \{0\}$ be a finite-range step distribution that is invariant under lattice symmetries,
and recall the finite-range decomposition of the associated covariance matrix $C (s,m^2)$ from 
Propositions~\ref{prop:decomp} and~\ref{prop:decomp_compatible}.
To simplify the conditions,
we will from now on always assume that $d=2$ and that there is a constant $C>0$ such that
the parameters from \eqref{eq:rJ_range}--\eqref{eq:theta_def} satisfy
\begin{equation} \label{eq:frd_ulbds}
|s| \leq \epsilon_s \theta_J,
  \qquad
  \theta_J \geq C^{-1},
  \qquad
	C^{-1} \rho_J \leq v_J \leq \rho_J / 2
	.
\end{equation}
All constants in the sequel are permitted to depend on this constant $C$ but will be otherwise independent of $J$.
In particular, this assumption holds for any fixed $J$ as in the statement of Theorem~\ref{thm:highbeta}, and it also holds uniformly in $\rho$ for the standard range-$\rho$ distribution $J_\rho$ discussed above Remark~\ref{rk:highbeta-rho}, see Lemma~\ref{lem:lambda_rho_error}.
Since $D_t^{\Z^d}$ is independent of $\Lambda_N$
for scales $< \frac14 L^{N-1}$, setting $D_t^{\Z^d}=0$ for $t< \rho_J$ (cf.~\eqref{eq:decomp3}), we define for $j\geq 0$,
\begin{align}\label{eq:Gammaj}
  \Gamma_{j+1} (s,m^2)
  &= \int_{ \frac14 L^{j}}^{\frac14 L^{j+1}}
  D_t^{ \Z^d} (s, m^2 ) \, dt,
  \\
  \Gamma_N^{\Lambda_N}(s,m^2)
  &= \int_{ \frac14 L^{N-1}}^{\infty}
    \tilde D_t^{\Lambda_N} (s, m^2 ) \, dt,
\end{align}
and set $\Gamma_{j, j'} = \sum_{k=j+1}^{j'} \Gamma_k$ so that, in view of \eqref{eq:decomp2}, \eqref{e:C(m)-torus-decomp}, 
we obtain
\begin{align}
  C^{\Lambda_N} (s,m^2)
  &= \Gamma_1(s,m^2)+\cdots+\Gamma_{N-1}(s,m^2)
    +
    \Gamma^{\Lambda_N}_{N}(s,m^2) + t_N(s,m^2)Q_N
    \nnb
  &= \Gamma_{0,N-1}(s,m^2)
    + \Gamma^{\Lambda_N}_{N}(s,m^2) + t_N(s,m^2)Q_N.
  \label{eq:frd_of_C^Lambda_N}
\end{align}
In particular, the matrices $\Gamma_{j}$ have range $\frac14 L^{j}$ by \eqref{eq:Dt_range} and satisfy the following bounds,
which are straightforward consequences of Propositions~\ref{prop:decomp} and~\ref{prop:decomp_compatible}.

\begin{corollary} \label{cor:Gammaj}
  Assume \eqref{eq:frd_ulbds} (and recall $d=2$).
  Then $\Gamma_{j+1}$ is analytic in $|s|<\epsilon_s\theta_J$,
\begin{equation} \label{eq:Gammaj_bd}
| \nabla^{\alpha} \Gamma_{j+1} (0, x ; s) | \leq
\begin{array}{ll}
\begin{cases}
C_{\alpha} \rho_J^{-2} L^{-j |\alpha|} & \quad \textnormal{if } |\alpha| \geq 1 \\
C_{0} \rho_J^{-2} \log L & \quad \textnormal{if } \alpha = 0
\end{cases}
\end{array}
\end{equation}
and 
\begin{equation} \label{eq:Gammaj0_asymp}
  \Gamma_{j+1} (0, 0 ; s) = 
  \frac{\log L}{2\pi (v_J^2+s)}
  + O(\rho_J^{-1} L^{-j}),
\end{equation}
and the estimates \eqref{eq:Gammaj_bd} also hold for $\Gamma^{\Lambda_N}_N$ and
we have $t_N (s,m^2) = m^{-2} + O(\rho_J^{-2}L^{2N})$.
\end{corollary}

We are ultimately interested in taking $m^2 \downarrow 0$.
While the zero mode $t_N(s,m^2)$ diverges as $m^{-2}$ as $m^2\downarrow 0$ like the torus Green function,
the covariances $\Gamma_j$ and their discrete derivatives are continuous as $m^2\downarrow 0$,
and this allows to directly set $m^2=0$ in these. This is made precise by the following lemma.
To simplify notation, we will abbreviate  from now on
$\Gamma_j = \Gamma_j(s)= \Gamma_j(s,0)$
{and $\Gamma_{N}^{\Lambda_N}=\Gamma_{N}^{\Lambda_N}(s)=\Gamma_N^{\Lambda_N}(s,0)$,}
i.e., $m^2$ is set to $0$ and the dependence on $s$ is often made implicit.

\begin{lemma} \label{lemma:m2to0_with_frd}
   Let $s$ be as in  \eqref{eq:frd_ulbds},
   let $\kappa < \theta_J+s$, and let $F: \R^\Lambda \to \R$ be a smooth function satisfying $|F(\varphi)| \leq e^{\kappa (\nabla \varphi, \nabla \varphi)}$.
   Then as $m^2\downarrow 0$,
  \begin{equation}
    \E_{C^{\Lambda_N}(s,m^2)} F
    \sim
    \E_{t_N(s,m^2)Q_N}^{\varphi'}
    \E^{\zeta}_{\Gamma_{0, N-1}(s)+\Gamma^{\Lambda_N}_{N}(s)} F(\varphi' +\zeta),
  \end{equation}  
  where on the right-hand side $\varphi'$ is (centered) Gaussian with covariance
  $t_N(s,m^2) Q_N$  and $\zeta$ is (centered) Gaussian with covariance $\Gamma_{0, N-1}(s) + \Gamma^{\Lambda_N}_{N}(s)$, and we recall that $a\sim b$ means $\lim a/b=1$.
\end{lemma}

\begin{proof}
  Provided sufficient integrability holds, by \eqref{eq:frd_of_C^Lambda_N}
  and the fact that the sum of independent Gaussian vectors is Gaussian
  with covariance the sum of the covariances, we have the identity
\begin{equation}
  \E_{C^{\Lambda_N} (s,m^2)} F
  =
  \E_{t_N(s,m^2)Q_N}^{\varphi'}
  \E_{\Gamma_1(s,m^2)+\cdots+\Gamma_{N-1}(s,m^2)+\Gamma^{\Lambda_N}_{N}(s,m^2)}^\zeta F(\varphi'+\zeta) 
\end{equation}
and thus
\begin{equation}
  1\sim\frac{\E_{t_N(s,m^2)Q_N}^{\varphi'}
    \E_{\Gamma_1(s,0)+\cdots+\Gamma_{N-1}(s,0)+\Gamma^{\Lambda_N}_{N}(s,0)}^\zeta F(\varphi'+\zeta)}{\E_{C^{\Lambda_N} (s,m^2)} F}
  \qquad (m^2\downarrow 0),
\end{equation}
where we used that $\Gamma_j(s,m^2)$ and $\Gamma_N^{\Lambda_N}(s,m^2)$ are continuous as $m^2\downarrow 0$
which implies that the inner Gaussian expectation in the numerator is continuous as $m^2 \downarrow 0$
if $F$ is integrable uniformly in $m^2$.

To see the integrability of the function $\varphi \mapsto e^{\kappa (\nabla \varphi, \nabla \varphi) } = e^{\kappa (\varphi, -\Delta \varphi)}$, it is enough to check that $\kappa (-\Delta) < C^{\Lambda_N} (s,m^2)^{-1}$ for each $m^2 >0$ and sufficiently small $\kappa$. But by definition  $C^{\Lambda_N} (s,m^2)^{-1} \geq -\Delta_J - s\Delta \geq - (\theta_J + s) \Delta$, so this holds as long as $\kappa < \theta_J + s$.
\end{proof}

Given a function $Z_0 ( \, \cdot \, | \Lambda_N)  : \R^{\Lambda_N} \to \R$,
which in our application will be taken to be \eqref{eq:Z_0_definition},
functions $Z_j ( \, \cdot \, | \Lambda_N) : \R^{\Lambda_N} \rightarrow \R$ are defined inductively by
\begin{equation}\label{eq:Z_j_recursion}
  Z_{j+1} (\varphi | \Lambda_N  ) = 
  \E_{\Gamma_{j+1}} [Z_j (\varphi + \zeta | \Lambda_N)], 
\end{equation}
where the expectation is taken over $\zeta \sim \mathcal{N}(0, \Gamma_{j+1})$; here we emphasise again that $\Gamma_{j+1}=\Gamma_{j+1}(s,0)$,
and we assume that $Z_0$ is such that the integrals exist.
Then by the previous lemma, and again using that the sum of independent Gaussian vectors is Gaussian with covariance the sum of the covariances,
\begin{equation}
  \mathbb{E}_{C^{\Lambda_N} (s,m^2)} Z_0 (\varphi + \zeta | \Lambda_N)
  \sim
  \E_{\Gamma^{\Lambda_N}_{N}(s)+t_N(s,m^2)Q_N}
  Z_{N-1} (\varphi + \zeta | \Lambda_N),
  \qquad (m^2\downarrow 0),
\end{equation}
where as usual both expectations are over $\zeta$.
In our setting, we will see in Section~\ref{sec:integration_of_zero_mode} that the integral 
over $\zeta \sim \cN(0,\Gamma_N^{\Lambda_N}(s)+t_N(s,m^2)Q_N)$ on the right-hand side is negligible as $N \rightarrow \infty$
for the purpose of Theorem~\ref{thm:highbeta}.
Therefore we can and will focus on the massless covariances $\Gamma_j (s)$
in Sections~\ref{sec:norms}--\ref{sec:stable_manifold_theorem}.

We conclude this short section with the following factorisation property implied by the finite
range property of the covariances.

\begin{lemma}
  Let $X, Y \subset\Lambda_N$ with $\min\{|x-y|_\infty: x\in X, y\in Y\} \geq {\frac14} L^{j+1}$.
  Then for all functions
  $F(X) : \R^{X} \to \R$ and $F(Y):\R^{Y} \to \R$ such that the following integrals exist,
  \begin{equation} \label{e:E-factor}
    \E_{\Gamma_{j+1}} \pB{F(X,\varphi+\zeta) F(Y,\varphi+\zeta)}
    =
    \E_{\Gamma_{j+1}} (F(X,\varphi+\zeta)) \E_{\Gamma_{j+1}}(F(Y,\varphi+\zeta)).
  \end{equation}
  In particular, assuming $L \geq 2^{d+2}$,
  this applies if $X$ and $Y$ are scale-$(j+1)$ polymers that do not touch,
  i.e., $X$ and $Y$   are distinct elements of $\operatorname{Comp}_{j+1}(X\cup Y)$, and
  $F(X) : \R^{X^*} \to \R$ and $F(Y) : \R^{Y^*} \to \R$ where $X^*$ and $Y^*$ denote the small set neighbourhoods of $X$ and $Y$ at scale $j$.
\end{lemma}

\begin{proof}
  \eqref{e:E-factor} is immediate from the finite-range property of the covariance $\Gamma_{j+1}$ (recall that $\Gamma_{j+1}$ has range at most $\frac14L^{j+1}$, cf.~\eqref{eq:Gammaj} and \eqref{eq:Dt_range}) and the fact that two jointly Gaussian random variables are independent if their covariance vanishes.

  The claim below \eqref{e:E-factor} then follows from the fact that if $X$ and $Y$
  are scale-$(j+1)$ polymers that do not touch, their $\ell^{\infty}$-distance is at least $L^{j+1}$
  and their scale-$j$ small set neighbourhoods $X^*$ and $Y^*$ then still have distance
  at least $L^{j+1}-2^{d+1}L^j = L^{j+1}(1-2^{d+1}L^{-1}) \geq \frac12 L^{j+1}\geq \frac14 L^{j+1}$. 
\end{proof}

\subsection{Scale subdecomposition}
\label{sec:subscale}

In some places, it is necessary to subdecompose each $\Gamma_{j+1}$ further to obtain better integrability and related
better contractivity of the renormalisation group map.
(For example, in the proof of Proposition~\ref{prop:E_G_j} below this subdecomposition allows to choose
$\kappa_L$ of order $1/(\log L)$.
Since $1/\kappa_L$ appears in various error terms,
this integrability is especially important to get to the critical temperature
or close to it, cf.~Remark~\ref{rk:highbeta-rho}.)
More precisely, we subdecompose each scale $j$ further
into fractional scales $j+s$ with $s\in \{0,1/M, \dots, 1-1/M, 1\}$ where $M$ is an integer such that $L=\ell^M$
for an integer $\ell$.
Corresponding to the fractional scales, we define covariances analogously to \eqref{eq:Gammaj}, i.e., 
\begin{equation}\label{eq:gamma_js}
\Gamma_{j+s, j+s'}
= \int_{\frac14 L^{j+s}}^{\frac14 L^{j+s'}} {D}_t 
\, dt, 
\qquad s,s'\in \{0, M^{-1}, 2M^{-1}, \dots, 1 \}
\end{equation}
for $j< N-1$, and for $j = N-1$,
\begin{equation}
\Gamma^{\Lambda_N }_{j+s, j+s'}
= 
\begin{cases}
\int_{\frac14 L^{j+s}}^{\frac14 L^{j+s'}} {\tilde D}_t^{\Lambda_N} \, dt &  \text{if} \;\; s' < 1 \\
\int_{\frac14 L^{j+s}}^{\infty} {\tilde D}_t^{\Lambda_N} \, dt  & \text{if} \;\; s'=1 . 
\end{cases}
\label{eq:gamma^b_js}
\end{equation}

In particular $\Gamma_{j, j+1}=\Gamma_{j+1}$.
These covariances admit estimates that are analogous to those for $\Gamma_{j+1}$
in Corollary~\ref{cor:Gammaj}
and they are again corollaries of Proposition~\ref{prop:decomp} and Proposition~\ref{prop:decomp_compatible}. 

\begin{lemma} \label{lemma:fine_Gamma_estimate}
    Let $d=2$ and assume \eqref{eq:frd_ulbds}.
  Then for $s,s'\in \{0, M^{-1}, 2M^{-1}, \dots, 1 \}$ and $s' - s = M^{-1}$,
\begin{align}
|\nabla^{\alpha} \Gamma_{j+s, j+s'} | \leq
\begin{array}{ll}
\begin{cases}
C_{\alpha} \rho_J^{-2} L^{- (j+s) |\alpha| } &  \text{if} \;\; |\alpha| \geq 1 \\
C_0 \rho_J^{-2} \log \ell & \text{if} \;\; \alpha = 0 .
\end{cases}
\end{array}
\end{align}
and the estimates also hold
for $\Gamma^{\Lambda_N}_{N-1+s,N-1+s'}$.
\end{lemma}

Finally, for each fractional scale $j+s$, 
we also introduce the corresponding division of the torus into blocks and polymers,
exactly as in Section~\ref{sec:polymersdef0}.
Thus
$\cP_{j+s}$ is the set of polymers composed of blocks in $\cB_{j+s}$ of (integer) side lengths $L^{j+s} = L^j\ell^k$ if $s=k/M$.
Given $X \in \mathcal{P}_{j+s}$, we define $X_{s'}$ ($s\leq s'$) to be the smallest $j+s'$-polymer that contains $X$, i.e.,~$X_{s'}$ consists of all blocks of side length $L^{j+s'}$ that intersect $X$.
In particular, $(X_s)_{s'}= X_{s \vee s'}$ and $\overline{X}=X_1$.

\section{Norms and regulators}
\label{sec:norms}

In this section, we define the norms in terms of which we will eventually measure contractivity and regularity properties of the renormalisation group map, cf.~Theorem~\ref{thm:local_part_of_K_j+1} below.
Our choice of norms is almost the same as that in \cite[Section~5.1]{MR2917175},
 which is closely related to those of \cite{MR2523458,MR1777310}.
Compared to these references, we simplify the construction somewhat and
make the estimates explicit to obtain uniform control
in the range of the step distribution.
Most proofs are given in Appendix~\ref{app:regulator}.

Henceforth, we assume that $d = 2$, and recall that $\Lambda_N$ denotes the discrete $d$-dimensional torus of side length $L^N$, for integers $N,L \geq 1$. Unless explicitly stated otherwise, all results in this section (implicitly) hold for any choice of $N$ and $L$. In the sequel, we make frequent use of the notation and set-up introduced in Sections~\ref{sec:polymersdef0}-\ref{subsec:limit_m_to_0} and write $\E=\E_{\Gamma_{j+1}}$.

\subsection{Norms on polymer activities}
\label{sec:polymersdef}

In Section~\ref{sec:rg_generic_step} below, we will define a renormalisation group map
that parametrises the successive integration (cf.~\eqref{eq:Z_j_recursion} and \eqref{eq:Z_0_definition})
\begin{equation}
  Z_{j+1}(\varphi | \Lambda ) = \Eplus [Z_j(\varphi+\zeta | \Lambda )]
\end{equation}
as
\begin{equation}
  Z_j (\varphi | \Lambda ) = e^{-E_{j+1} |\Lambda|}  \sum_{X\in \mathcal{P}_j (\Lambda)} e^{U_j (\Lambda \backslash X,\varphi)} K_j (X,\varphi), 
\end{equation}
where we recall the definition of polymers from Section~\ref{sec:polymersdef0},
that $\zeta$ is a centred Gaussian random variable with covariance $\Gamma_{j+1}$ and we use $\Eplus$ to denote expectation with respect to $\zeta$.
This notation will be fixed in the rest of the paper whenever the scale $j$ is clear from the context.
In this representation, $E_j$ is going to be a suitable scalar (parametrising the free energy), $U_j$ an explicit leading part (parametrising the effective potential in the Wilsonian picture of the renormalisation group),
whereas $K_j$ will be a so-called remainder coordinate whose main feature is the following \emph{component factorisation property}:
\begin{equation}\label{eq:factorization}
  K_j(X) = \prod_{Y \in \operatorname{Comp}_j(X)} K_j(Y)
\end{equation}
with the convention that the product over the empty set equals $1$.
In particular, the tuple $(K_j(X))_{X \in \cP_j}$ is determined by $(K_j(X))_{X \in \cP_j^c}$. The latter
is an example of a \emph{polymer activity}. We formalise the space of polymer activities as follows,
and then define norms on polymer activites in the remainder of this section.

\begin{definition} \label{def:polymeractivity}
For $X \in \cP_j$, we write $\mathcal{N}_j(X)$ for the space $C^\infty(\R^{X^*})$.
For $F \in \cN_j(X)$ and $\varphi \in \R^\Lambda$ we make the identification $F(\varphi)  = F(\varphi|_{X^*})$.
In particular, $F(\varphi)$ only depends on $\varphi|_{X^*}$
and we have the natural inclusions $\cN_j(X) \subset \cN_j(Y)$ if $X \subset Y$.

A scale-$j$ \emph{polymer activity}
is a tuple $K = (K(X))_{X\in \cP_j^c}$, where for each connected polymer $X \in \cP_j^c$,
the corresponding component is a function $K(X) \in \cN_j(X)$.
Any polymer activity $K$ is 
identified with its extension $(K(X))_{X\in \cP_j}$
to all (not necessarily connected) polymers by means of the component factorisation property
\eqref{eq:factorization} with $K_j \equiv K$.
We denote the space of scale-$j$ polymer activities  by $\cN_j$.
\end{definition}

Note that \eqref{eq:factorization} implies that $K(\emptyset)=1$ for any polymer activity $K$ according to Definition~\ref{def:polymeractivity},
and that, with the restriction to connected polymers, the scale-$j$ polymer activities form a linear space with $0$ element given by $K(X)=1_{X=\emptyset}$.
To define norms on polymer activities, we first define norms of lattice functions which will
enter the definition of norms on polymer activities.
Firstly, recall the definition of the set of standard basis $\hat{e} = \{ \pm e_1, \cdots, \pm e_n \}$, derivatives $\nabla^n f$ for $f : \Lambda_N \rightarrow \C$ and the Laplacian $\Delta f$ from Section~\ref{sec:notation}.
For functions $f, g : \Lambda_N \rightarrow \C$, 
we also define the inner products
\begin{equation} \label{e:innerprod}
  (f,g)_X = \sum_{x\in X} f(x) g(x),
    \quad (\nabla^n f, \nabla^n g)_X =2^{-n} \sum_{(\mu) = (\mu_1, \cdots, \mu_n) \in \hat{e}^n} \sum_{x\in X} \nabla^{(\mu)} f(x) \nabla^{(\mu)} g(x) 
\end{equation}
where $\nabla^{(\mu)} f(x) = \nabla^{\mu_1} \cdots \nabla^{\mu_n} f (x)$ and
\begin{align}
|\nabla f|^2_X = (\nabla f, \nabla f)_X \label{eq:nablaf^2_definition}
\end{align}
so that $(f, -\Delta f)_{\Lambda_N} = |\nabla f|^2_{\Lambda_N}$ by summation by parts.
(The factors $2^{-n}$ in \eqref{e:innerprod} are natural because each coordinate direction appears with positive and negative sign in the sum.)
At scale $j$, it is further natural to consider the rescaled derivatives 
\begin{align}
\nabla^n_j f = L^{jn} \nabla^n f .
\end{align}

\begin{definition} \label{def:L-p-norms}
  Let $n\in \mathbb{N}$, $X\in \mathcal{P}_j$ and $f: \{ x : d_1 (x, X) \leq n \} \rightarrow \mathbb{C}$ where $d_1$ is the graph distance on $\Lambda_N$. With $(\mu)$ ranging over $\{ \pm e_1, \pm e_2 \}^n$ in the sequel, define for $p\in [1, \infty)$
\begin{align}
&\norm{\nabla^n_j f}_{L^{\infty}(X)} = \max_{(\mu)} \max_{x\in X} |\nabla^{(\mu)}_j f (x)| \label{eq:norminfty}\\
&\norm{\nabla^n_j f}^p_{L^{p}_j (X)} = L^{-2j} \sum_{x\in X} \sum_{(\mu)} 2^{-n}
 |\nabla_j^{(\mu)} f (x)|^p  \label{eq:normL2}\\
&\norm{\nabla^n_j f}^p_{L^{p}_j (\partial X)} = L^{-j} \sum_{x\in \partial X} \sum_{(\mu)} 2^{-n} 
|\nabla^{(\mu)}_j f (x)|^p \label{eq:normL2boundary}\\
& \norm{f}_{C^2_j (X)} = \max_{n=0,1,2} \norm{\nabla^n_j f}_{L^{\infty}(X)}. \label{eq:normC2}
\end{align}
(In \eqref{eq:normL2boundary} and elsewhere, $\partial U$ refers to the inner vertex boundary of $U\subset \Lambda_N$ with respect to the graph distance $d_1$).
\end{definition}

These norms on lattice functions provide the basis for the norms on polymer activities that we use and which we introduce next.
This definition is slightly involved, and we therefore briefly highlight its main features before stating the full definition (see Definition~\ref{def:NORM}).
The norm is scale-dependent and measures smoothness of polymer activities with respect to typical fields
at scale $j$,
which are lattice functions $\varphi$ with bounded $C_j^2$ norm.
The norm needs to permit growth when $\nabla \varphi$ is large and give small weight to large sizes of polymers $X$.
These two aspects are accounted for by two weights often called regulators:
the (large-field) regulator $G_j$ for growth in $\nabla\varphi$ (see Definition~\ref{def:G_j} and \eqref{eq:NORm})
and the parameter $A>1$ (the large-set regulator)  for decay in the size of the polymer (see \eqref{eq:NORM}).

We start by measuring the size of a polymer activity for fixed $\varphi$ and $X \subset \Lambda_N$.
For all $n \in \mathbb{N}$, given 
$K(X, \cdot) \in \mathcal{N}_j(X)$, its $n$-th order derivative $D^n K$ along the directions
$f_1,\dots f_n \in \R^{X^*}$ is given by
\begin{equation}\label{eq:def-D}
 D^n K (X,\varphi) (f_1, \cdots, f_n) =  \sum_{x_1, \cdots, x_n \in X^*} \frac{\partial^n K(X,\varphi)}{\partial \varphi (x_1) \cdots \partial \varphi(x_n)} f_1(x_1) \cdots f_n(x_n),
\end{equation}
with the convention $D^0K =K$.
For $X \in  \cP^c_j$, $K( X) \in \mathcal{N}_j( X) $ and $\varphi \in \R^{\Lambda_N}$, then set
\begin{equation}
 \norm{D^n K (X,\varphi)}_{n, T_j (X,\varphi)} = \sup\big\{ |D^n K (X,\varphi) (f_1, \cdots, f_n)| \, : \, \norm{f_k}_{C^2_j(X^*)} \leq 1, \,  k =1,\dots, n \big\}, \label{eq:nTj_norm_definition}
\end{equation}
with the convention $\norm{D^0 K (X,\varphi)}_{0, T_j (X,\varphi)}=|K (X,\varphi)|$.
Then, for a parameter $h >0$,  define
\begin{equation}
\label{eq:seminorm}
\norm{K(X,\varphi )}_{h, T_j (X,\varphi )} = \sum_{n=0}^{\infty} \frac{h^n}{n!} \norm{D^n K (X,\varphi)}_{n, T_j (X,\varphi)}.
\end{equation}
Note that \eqref{eq:def-D} only depends on the $f_k$ in $X^*$,
but that the norms $\|f_k\|_{C_j^2(X^*)}$ in \eqref{eq:nTj_norm_definition} actually depend on $f_k$
in a neighbourhood of $X^*$. The supremum  in \eqref{eq:nTj_norm_definition}
is thus over all $f_k \in \R^{\Lambda}$ or equivalently
over all extensions of $f_k\in \R^{X^*}$ to a suitable neighbourhood of $X^*$.
More generally, the above definitions of $\norm{D^nK(X,\varphi)}_{n,T_j(X,\varphi)}$ and
of $\norm{K(X,\varphi)}_{h,T_j(X,\varphi)}$  continue to make sense when $K(X) \in \cN_k(X)$ and $X\in \cP_k$ with $k \leq j$.

The $\norm{\cdot}_{h, T_j(X,\varphi)}$-norm measures the size of $K$ in a manner depending on $\varphi$ and $X$.
The norms on functions of $(X,\varphi)$ are defined by weighted supremum norms.
The \emph{large-field regulator} which is the $\varphi$-dependent weight is defined next.

\begin{definition}
\label{def:G_j}
Given $\cwone c_2,\kappa_L>0$,
define the large-field regulator for $X \in \cP_j$ and $\varphi \in \R^{\Lambda_N}$ by
\begin{equation}
G_j (X,\varphi) = \exp \Big\{\cwone \kappa_L \norm{\nabla_j \varphi}_{L^2_j (X)}^2 + c_2 \kappa_L \norm{\nabla_j \varphi}_{L^2_j (\partial X)}^2 + \cwone \kappa_L W_j (X, \nabla^2_j \varphi)^2 \Big\} \label{eq:def_large_field_regulator}
\end{equation}
where 
\begin{equation}\label{eq:W_j}
W_j (X, \nabla^2_j \varphi)^2 = \sum_{B\in \mathcal{B}_j (X)} \norm{\nabla^2_j \varphi}_{L^{\infty}(B^*)}^2 .
\end{equation}
\end{definition}

The particular form of the regulator is motivated by its properties stated in Section~\ref{sec:norm_properties} below.
Finally, the definition of the norms on polymer activities is given by the following definition.
\begin{definition}
\label{def:NORM}
Given $G_j (X,\varphi)$ as in \eqref{eq:def_large_field_regulator} with $\cwone c_2,\kappa_L>0$,  $h >0$,  and $A > 1$,
for any scale-$j$ polymer activity $K$, define
\begin{align}
& \norm{D^n K(X)}_{n, T_j (X)} = \sup_{\varphi\in \mathbb{R}^{X^*}} (G_j (X, \varphi))^{-1} \norm{D^n K(X, \varphi)}_{n, T_j (X, \varphi)} \\
& \norm{K (X)}_{h, T_j (X)} = \sup_{\varphi\in \mathbb{R}^{X^*} }  (G_j (X, \varphi ))^{-1} \norm{K(X, \varphi)}_{h, T_j (X, \varphi)} \label{eq:NORm}\\
& \norm{K}_{h, T_j} = \sup_{X\in\cP_j^c (\Lambda_N)} A^{|X|_j}\norm{K (X)}_{h, T_j (X)}. \label{eq:NORM}
\end{align}
We will sometimes abbreviate $\norm{K}_j = \norm{K}_{h, T_j}$.
\end{definition}

This norm and the associated spaces (of polymer activities of finite norm) implicitly depend on the choice of $\Lambda_N$.
However, the definitions are essentially local and it is thus possible to define
an infinite-volume analogue of the norm, see Section~\ref{sec:infvol}.
                                                                                                                                                            The space of polymer activities in $\cN_j(X)$ with finite $\norm{\cdot}_{h, T_j (X)}$ norm is complete, see Appendix~\ref{app:completeness}, and
                                                                                                                                                            as a consequence the space of polymer activities in $\cN_j$ with finite $\norm{\cdot}_{h,T_j}$ norm is also complete.

For the reader's orientation, we now give an overview of how the parameters $h,A,\cwone c_2,\kappa_L$ will eventually be chosen;
see also Definition~\ref{def:K_space}.
The constants $\cwone c_2,\kappa_L$ will be fixed below Proposition~\ref{prop:E_G_j} (see Remark~\ref{R:choice-parameters})
as 
$c_2>0$ sufficiently small (independent of $L$), and $\kappa_L$ of order $(\log L)^{-1}$.
The large set weight $A$ will be chosen large enough as a function of $L$ in Theorem~\ref{thm:local_part_of_K_j+1} (essentially in such a way that the conclusions of Proposition~\ref{prop:largeset_contraction-v2} below hold for a suitably large value of $p$).
This leaves $h$, which will be picked large enough (larger than $1$) so that the conclusions of Lemma~\ref{lemma:contraction_of_charge_q_term} hold.

\subsection{Main properties of the norm and the regulator}
\label{sec:norm_properties}

The most fundamental properties of the seminorm $\norm{\cdot}_{h, T_j (X, \varphi)}$
are its submultiplicativity property,
and its monotonicity in the base polymer $X$ and in the scale $j$.
These properties will be used heavily.

\begin{lemma}
  Suppose $X, Y \in \cP_j$ with $Y\subset X$, and let $F(Y) \in \cN_j(Y)$
  (here recall the inclusion $\cN_j(Y) \subset \cN_j(X)$ from Definition~\ref{def:polymeractivity}).
  Then for each $n\geq 0$,  (i) $\norm{D^n F(Y, \varphi)}_{n, T_{j+1} (Y, \varphi)} \leq \norm{D^n F(Y, \varphi)}_{n, T_{j} (Y, \varphi)}$, and (ii) $\norm{D^n F(Y, \varphi)}_{n, T_j (Y,  \varphi)} \leq \norm{D^n F(Y, \varphi)}_{n, T_j (X, \varphi)}$. 
  Hence,
  \begin{equation}\label{e:norm-monot}
    \norm{F(Y, \varphi)}_{h, T_{j+1} (Y, \varphi)} \leq \norm{F(Y, \varphi)}_{h, T_j (Y, \varphi)},
    \quad
  \norm{F(Y, \varphi)}_{h, T_{j} (Y, \varphi)} \leq \norm{F(Y, \varphi)}_{h, T_j (X, \varphi)}.
  \end{equation}
  Moreover, for $Y_1,Y_2,X \in \cP_j$ with $Y_1,Y_2 \subset X$ (with $Y_1$ and $Y_2$ not necessarily disjoint),
  and $F(Y_i) \in \cN_j(Y_i)$, the following submultiplicativity property holds:
  \begin{equation}
    \label{eq:prodprop}
    \norm{F_1(Y_1,\varphi)F_2(Y_2,\varphi)}_{h, T_j (X,\varphi)}
    \leq
    \norm{F_1(Y_1,\varphi)}_{h, T_j (Y_1, \varphi )}\norm{F_2(Y_2,\varphi)}_{h, T_j (Y_2, \varphi)}.
  \end{equation}
\end{lemma}

\begin{proof}
To see (i), notice that for any $f\in \R^{\Lambda}$, we have $\norm{f}_{C_{j}^2 (Y^*)} \leq \norm{f}_{C_{j+1}^2 (Y^*)}$. Hence $\{ f \in \R^{\Lambda} : \norm{f}_{C^2_{j+1} (Y^*)} \leq 1 \}  \subset \{ f \in \R^{\Lambda} : \norm{f}_{C^2_{j} (Y^*)} \leq 1 \}$ and (i) follows readily in view of \eqref{eq:nTj_norm_definition}.
For (ii), we have for any $f: \Lambda \rightarrow \R$ that 
$\norm{f}_{C_j^2 (Y^*)} \leq \norm{f}_{C_j^2 (X^*)}$, and the result follows similarly. On account of \eqref{eq:seminorm}, the inequalities in \eqref{e:norm-monot} are immediate consequences of (i) and~(ii), respectively.
For the submultiplicativity property, see \cite[Lemma~6.7]{MR2523458}.
\end{proof}

The second~key property of the norm is that it enforces analyticity of polymer activities in a strip.
For open $U\subset \C^{\Lambda}$, the function $F:  U \to \C$ is called complex analytic in $U$ if it admits a local representation as a convergent power series  around any point in $U$.

\begin{proposition} \label{prop:complex_extension_of_polymer_activity}
Let $h>0$, let $\norm{F}_{h, T_j} < + \infty$ and $X\in \mathcal{P}_j^c$. Then $F (X) \in \mathcal{N}_j(X)$ can be extended to a complex analytic function on the domain $S_{h} := \{ \varphi + i\psi \in \C^{\Lambda} : \varphi(x), \psi(x) \in \R, \, \norm{\psi}_{C_j^2 (X^*)} < h   \}$.
\end{proposition}

\begin{proof}

Let $D_{h} (0) = \{ \psi \in \mathbb{C}^{\Lambda} : \norm{\psi}_{C_j^2 (X^*)} <  h \}$. Note that $D_{h} (0) \subset \C^{\Lambda}$ is open because $\norm{\psi}_{L^{\infty}(X^*)} < \frac{1}{4} L^{-2j}{h}$ implies $\psi \in D_{h} (0)$. For $\varphi \in \R^{X^*}$ and $\psi \in D_{h} (0)$, let
\begin{equation}
F_{[\varphi]}(X,\varphi + \psi) = \sum_{n=0}^{\infty} \frac{1}{n!} D^n F(X,\varphi) (\psi^{\otimes n}) . \label{eq:complex_extension_via_Taylor}
\end{equation}
Since $\norm{F(X, \varphi)}_{{h}, T_j (X,\varphi)} \leq \norm{F}_{{h}, T_j (X)} G_j (X,\varphi) < + \infty$ and since $\norm{\psi}_{C_j^2 (X^*)} < {h}$,
the series \eqref{eq:complex_extension_via_Taylor} converges absolutely. 
Thus
$\widetilde{F}( X, \cdot ): S_h \to \C$ given by 
\begin{align}
\widetilde{F}( X,z) \stackrel{\text{def.}}{=} F_{[\varphi]}(X, \varphi+\psi), \quad \text{ for any $\varphi \in \R^{\Lambda} $ and $\psi \in D_{h} (0)$ s.t.~$z=\varphi+\psi$ }\label{eq:complex_extension_via_Taylor'}
\end{align}
is well-defined (that does not depend on $\varphi$) and extends $F$. Moreover, in view of \eqref{eq:complex_extension_via_Taylor}, $\widetilde{F}(X,\cdot)$ is (plainly) given by a convergent power series in a neighbourhood of $\varphi$, for any $\varphi \in \R^{\Lambda}$. 

Now consider an arbitrary point $\zeta \in S_h$. It remains to argue that $\widetilde{F}(X, \cdot)$ defined by \eqref{eq:complex_extension_via_Taylor'} can be represented as convergent power series around $\zeta$. Write $\zeta= \varphi+\psi$ where $\varphi = \text{Re}(\zeta)$ componentwise. Now observe that for $\delta \zeta \in \C^{\Lambda}$ small enough (such that $\psi + \delta \zeta \in D_h(0)$), one has
\begin{equation}
\label{eq:reexpand}
\widetilde{F}( X, \zeta+ \delta \zeta) \stackrel{\eqref{eq:complex_extension_via_Taylor'}}{=} \sum_{n \geq 0} \frac{1}{n!} D^n F(X, \varphi) ((\psi+ \delta \zeta)^{\otimes n}) = \sum_{k \geq 0} \frac{1}{k!} A_k  (\delta \zeta^{\otimes k}),
\end{equation}
where
\begin{align}
	A_k (f_1, \cdots, f_k) = \sum_{l=0}^{\infty} \frac{1}{l !}
	D^{k+l} F(X, \varphi) (\psi^{\otimes l}, f_1, \cdots, f_k)
\end{align}
and the right-hand side of \eqref{eq:reexpand} is obtained by expanding $(\psi+ \delta z)^{\otimes n}$, using multilinearity and re-arranging terms according to the number $k$ that $\delta z$ appears. Now use $\norm{F}_{h, T_j} < + \infty$ once again to show the series in \eqref{eq:reexpand} converges. All in all, it follows that $\widetilde{F}$ is complex-analytic, as desired.
\end{proof}

The next property of the regulator is the following basic inequality that allows to absorb polynomial error bounds in the fields
when changing from one scale to the next.

\begin{lemma} \label{lemma:bound_of_Gj_by_Gjplus1} For all $\cwone c_2, \kappa_L >0$, $X\in \mathcal{S}_j$ for some $0 \leq j< N$, all $x_0 \in X$ and $\varphi \in \R^{\Lambda_N}$, defining $\delta \varphi(x)= \varphi(x)-\varphi(x_0)$, one has
\begin{equation} \label{eq:deltaphi_C2}
(\cwone \kappa_L)^{k/2} \norm{\delta \varphi}^{k}_{C^2_j (X^*)} \leq C(k) G_j (X, \varphi), \quad k \in \mathbb{N}.
\end{equation}
\end{lemma}

\begin{proof}
  See Appendix~\ref{sec:A.5}.
\end{proof}

The next property of the regularity involves what is called the strong regulator in \cite{MR2523458}.
(The term strong regulator refers to the left-hand side of \eqref{eq:strong_regulator1} below.)
To this end, define
\begin{equation} \label{eq:wj_def}
  w_j (X,\varphi )^2 = \sum_{B\in \mathcal{B}_j (X)} \max_{n=1,2} \norm{\nabla_j^n \varphi}^2_{L^{\infty}(B^*)}, \quad X \in \mathcal{P}_j
  .
\end{equation}

\begin{lemma} \label{lemma:strong_regulator}
  For all $\cwone c_2, \kappa_L>0$, the following hold. If $c_w > 0$ is sufficiently small,
  \begin{equation}
    e^{c_w \kappa_L w_j ( X, \varphi)^2} \leq G_j (X, \varphi), \quad X \in \mathcal{P}_j. \label{eq:strong_regulator1}
  \end{equation}
  Moreover, for all $X,Y \in \mathcal{P}_j$ satisfying $X \cap Y = \emptyset$ and all $c_2,c_w$ sufficiently small,
  \begin{equation}
    e^{c_w \kappa_L w_j (X, \varphi)^2} G_j (Y, \varphi) \leq G_j (X\cup Y, \varphi) \label{eq:strong_regulator2}.
  \end{equation}
\end{lemma}

\begin{proof}
  See Appendix~\ref{sec:A.5}.
\end{proof}

Finally, we will need the following stability result for the Gaussian expectation with respect to the covariances $\Gamma_j$.
Thus recall the finite-range decomposition from Section~\ref{sec:scales}, in particular the definition of $\Gamma_j$ from~\eqref{eq:Gammaj}. For a scale-$j$ polymer activity $F$, a common strategy to bound $\Eplus[F(X, \varphi' + \zeta)]$, for fixed $\varphi'$ and $X \in \cP^c_j$, will be to first bound $\norm{F(X, \varphi' + \zeta)}_{h, T_j (X,\varphi')} \leq \norm{F(X)}_{h, T_j(X)} G_j (X, \varphi' + \zeta)$, which follows immediately from \eqref{eq:NORm}, so that the fluctuation integral acts effectively on the large field regulator $G_j$ only.
In this regard, the following Proposition~\ref{prop:E_G_j} yields that the form of the large field regulator is stable under the fluctuation integral up to a factor $2^{|X|_{j}}$, where $\overline{X}$ denotes the closure of $X$, cf.~Section~\ref{sec:polymersdef0}.
The proposition is a version of \cite[Lemma~10]{MR2917175}.

\begin{proposition} \label{prop:E_G_j}
Assume \eqref{eq:frd_ulbds} and that $L=\ell^M$ for integers $\ell, M\geq 1$.
For $c_2 >0$ sufficiently small,
there exists 
an integer $\ell = C(\cwone c_2)$ and a constant $c=c(\cwone c_2)>0$
such that with $c_{\kappa} = c(\cwone c_2)\ell^{-2} \in (0,1)$, the following holds: for all $0\leq j < N-1$
and all $\kappa_L \leq c_{\kappa}{\rho_J^{2}} (\log L)^{-1}$, 
\begin{equation}
\Eplus [G_j (X, \varphi'+ \zeta )] 
\leq 2^{|X|_{j}} G_{j+1} (\overline{X}, \varphi'),
\qquad X \in \mathcal{P}_j^c, \; \varphi' \in \R^{\Lambda_N} \label{eq:E_G_j}
  .
\end{equation}
For the last scale $j=N-1$ an analogous bound holds with $\Eplus$ replaced by $\E_{\Gamma_{N}^{\Lambda_N}}$, i.e., $\zeta \sim \cN(0, \Gamma_N^{\Lambda_N})$.
\end{proposition}

\begin{proof}
  See Appendix~\ref{sec:A4}.
\end{proof}

The following corollary is a simple consequence of Proposition~\ref{prop:E_G_j}.

\begin{lemma} \label{lemma:linearity_of_expectation}
  Suppose the assumptions of Proposition~\ref{prop:E_G_j} hold.
  Then for any $F \in \mathcal{N}_j$ with $\norm{F}_{h, T_j}< \infty$,
  all $X \in \cP_j^c$, and all $\varphi' \in \R^{\Lambda_N}$, 
\begin{align}
  & \norm{\Eplus [F(X,\varphi' + \zeta)]}_{h, T_{j} (X,\varphi' )}
  \leq (A/2)^{-|X|_j} \norm{F}_{h, T_j} \, G_{j+1} (\overline{X},\varphi'),
    \label{eq:linearity_of_expectation1} \\
  \intertext{and, slightly more generally, for all $n\geq 0$,}
   &{\norm{D^n_{\varphi'}\Eplus [F(X,\varphi' + \zeta)]}_{n, T_{j} (X,\varphi' )} \leq 2^{|X|_j} \norm{D^nF(X)}_{n, T_j(X)} \, G_{j+1} (\overline{X},\varphi').   \label{eq:linearity_of_expectation1.1}}
\end{align}
For $j = N-1$, analogous bounds hold with $\Eplus$ replaced by $\E_{\Gamma_N^{\Lambda_N}}$.
\end{lemma}

\begin{proof} The derivative $D_{\varphi'}$ can be exchanged with the expectation $\Eplus $, hence for all functions $f_k$ with $\norm{f_k}_{C_{j}^2(X^*)} \leq 1$, $k=1, \ldots, n$, by \eqref{eq:nTj_norm_definition},
\begin{align}
\label{eq:D_neasybd}
|D^n_{\varphi'}\Eplus[F(X,\varphi' + \zeta)](f_1, \cdots, f_n)| &\leq \Eplus[ |D^n_{\varphi'} F(X,\varphi' + \zeta)(f_1,\cdots, f_n) | ] \nnb 
&\leq \Eplus[ \norm{D^n_{\varphi'} F(X,\varphi'+\zeta)}_{n, T_{j} (X,\varphi'+\zeta)} ]
\end{align}
and so, taking suprema over the $f_k$'s, multiplying by $h^n/n!$ and summing over $n \geq 0$, recalling \eqref{eq:NORm}, \eqref{eq:NORM} 
and applying the bound \eqref{eq:E_G_j} from Proposition~\ref{prop:E_G_j}, one obtains that
\begin{align}
\norm{\Eplus [F(X,\varphi' + \zeta )]}_{h, T_{j}(X,\varphi')}
&\leq \Eplus[ \norm{F(X,\varphi' + \zeta )}_{h, T_j (X, \varphi'+\zeta )}  ]  \nnb 
&\leq  A^{-|X|_j} 
\Eplus[ \norm{F}_{h, T_j} \, G_{j}(X,\varphi' + \zeta)  ]  
\leq (2/A)^{|X|_j}
 \norm{F}_{h, T_j} \, G_{j+1}(\overline{X},\varphi'), 
\end{align}
giving \eqref{eq:linearity_of_expectation1}. The bound \eqref{eq:linearity_of_expectation1.1} is obtained similarly.
\end{proof}

We conclude this section by fixing the parameters $\cwone c_2, \kappa_L$ appearing in \eqref{eq:def_large_field_regulator}.
\begin{remark}\label{R:choice-parameters}
We choose $c_2>0$ small enough such that both i) the estimate
 \eqref{eq:strong_regulator2} in Lemma~\ref{lemma:strong_regulator} holds whenever $c_w$ is sufficiently small and ii) Proposition~\ref{prop:E_G_j} is in force. Having fixed $c_2$, we choose $\ell =C\rho_J$ according to Proposition~\ref{prop:E_G_j} and set $\kappa_L =c_{\kappa}{\rho_J^{2}} (\log L)^{-1}$ with $c_{\kappa}=c\ell^{-2}$, so that the conclusions of Proposition~\ref{prop:E_G_j} (i.e.,~\eqref{eq:E_G_j}) hold. 
We can thus freely apply the bounds derived in Lemmas~\ref{lemma:bound_of_Gj_by_Gjplus1} and \ref{lemma:strong_regulator} (in the latter case whenever $c_w$ is sufficiently small) and Proposition~\ref{prop:E_G_j} in the sequel.
Throughout the rest of this article, we always implicitly assume that the base scale $L$ is of the form $L=\ell^M$ with $\ell$ as fixed above. Unless stated otherwise, all statements hold uniformly in  $M \geq 1$, and when we write $L \geq C$ in the sequel, we tacitly view this as a condition on $M$ being sufficiently large.

\end{remark}
%

\subsection{Subdecomposition of the regulator}
\label{subsec:subdecomp_of_the_reg}

The final property of the regulator is a technical property involving
the scale subdecomposition from Section~\ref{sec:subscale} and that is needed to obtain sharp integrability estimates.
It is used as an ingredient of the proof of Proposition~\ref{prop:E_G_j} above
and also in the justification of complex translations in the proof of Lemma~\ref{lemma:charged_term_exponential_contraction} below.

Throughout this section, assume $L=\ell^M$ with integers $\ell$ and $M$.
For a parameter $c_4 >0$ and $X \in \mathcal{P}_{j+s}$ (recall the notion of fractional scales from Section~\ref{sec:subscale}), let
\begin{equation}
g_{j+s} (X, \xi) = \exp\Big( c_4 \kappa_L \sum_{a=0,1,2} W_{j+s} ( X, \nabla^a_{j+s} \xi)^2 \Big),
\end{equation}
with $W_{j+s}$ defined analogously to \eqref{eq:W_j} by
\begin{equation}
W_{j+s} (X, \nabla^a_j \varphi)^2 = \sum_{B\in \mathcal{B}_{j+s} (X)} \norm{\nabla^a_{j+s} \varphi}_{L^{\infty}(B^*)}^2 .
\end{equation}
Then, with hopefully obvious notation, define $G_{j+s}(X, \cdot)$ for $X \in \mathcal{P}_{j+s} $ as in \eqref{eq:def_large_field_regulator} but with $j+s$ in place of $j$ everywhere.
The following Lemmas~\ref{lemma:g_j+s_bound_by_quadratic} and~\ref{lemma:G_change_of_scale} can be extracted from \cite[Lemma~19]{MR2917175}
and its proof.
For completeness, we have again included proofs in Appendix~\ref{app:g_j+s_bound_by_quadratic} and \ref{app:lem_G_change_of_scale}.

\begin{lemma} \label{lemma:g_j+s_bound_by_quadratic}
There exists $C>0$ such that
for any $X\in \cP_{j+s}$ and $\zeta \in \R^{\Lambda_N}$, 
\begin{equation}  \label{eq:g_j+s_bound}
  g_{j+s} (X, \zeta) \leq
  \exp (\frac{1}{2} Q_{j+s} (X,\zeta)) 
  := \exp\bigg( C c_4 \kappa_L \sum_{a=0}^4 
  \sum_{(\mu)}   \norm{\nabla_{j+s}^{(\mu)} \zeta}^2_{L^2_{j+s} (X^*)} \bigg),
\end{equation}
where the sum ranges over multiindices $(\mu) = (\mu_1, \dots, \mu_a) \in \{ \pm e_1, \pm e_2 \}^{a}$.
Moreover, for any $c_4>0$, any integer $\ell$, there is $c_\kappa = c_\kappa(c_4,\ell)>0$ such
  that if $\kappa_L=c_\kappa \rho_J^2 (\log L)^{-1}$ then
\begin{equation} \label{eq:g_j+s_expectation}
  \E_{\Gamma_{j+s,j+s'}}(e^{Q_{j+s}(X,\zeta)}) \leq 2^{M^{-1}|X|_{j+s}}.
\end{equation}
\end{lemma}

\begin{lemma} \label{lemma:G_change_of_scale}
  For any $c_2>0$ small enough,
  there exist $c_4=c_4(\cwone c_2)>0$ and an integer $\ell_0=\ell_0(\cwone c_2)>1$ (both large), such that for all $\ell \geq \ell_0$, $M \geq 1$, $0\leq j < N$, $s \in \{0,\frac1M,\dots,1-\frac1M\}$ and $\kappa_L > 0$,
  for $X\in \cP_{j+s}^c$, $\varphi, \xi \in \R^{\Lambda_N}$,
\begin{equation}\label{eq:G_change_of_scale}
  G_{j+s} (X, \varphi + \xi ) \leq g_{j+s}( X_{s+ M^{-1}}, \xi ) G_{j+s+M^{-1}} ( X_{s+ M^{-1}}, \varphi).
\end{equation}
\end{lemma}

\subsection{Continuity of the expectation}
\label{sec:continuity}

The next property shows that the expectation is continuous with respect to the parameter $s$
of the covariances.

\begin{lemma} \label{lemma:stability_of_expectation_singleX}
  For any $X \in \cP_j^c$ and $F(X)$ with $\|F(X)\|_{h,T_j(X)}<\infty$,
  for $|s|, |s'| < \theta_J \epsilon_s$,
\begin{equation}
  \lim_{s' \rightarrow s}
  \norm{ \E_{\Gamma_{j+1} (s')} [F(X,  \cdot + \zeta )] - \E_{\Gamma_{j+1} (s)} [F( X, \cdot+ \zeta)] }_{h, T_{j+1} (\bar{X})} = 0. \label{eq:stability_of_expectation_singleX}
\end{equation}
More precisely, for any $C>0$, the convergence is uniform over all $F$ with $\|F(X)\|_{h,T_j(X)} \leq C$.The same conclusion holds when $X\in \cP_{j+1}^c$ and we assume
\begin{equation} \label{eq:stability_of_expectation_scalej+1}
\sup_{\varphi} G_j (X, \varphi)^{-1} \norm{F(X, \varphi)}_{h, T_{j+1} (X, \varphi)} < \infty,
\end{equation}
i.e.,~the convergence is uniform in $F$ for which the left-hand side of \eqref{eq:stability_of_expectation_scalej+1} is bounded by a given~$C>0$.
\end{lemma}

\begin{proof}
  We start from the following elementary identity for the derivative of a Gaussian integral with respect to its covariance: abbreviating $\Gamma_{j+1}(x,y)\equiv \Gamma_{j+1}(x,y; s, m^2)$, considering first the centered Gaussian vector on $\Lambda_N$ with covariance $\Gamma_{j+1,\varepsilon}= \Gamma_{j+1}+ \varepsilon \text{Id}$ with density $f_{\varepsilon}$, computing the derivatives $\partial f_\varepsilon/\partial \Gamma_{j+1,\varepsilon}(x,y) $ and letting $\varepsilon \to 0$ using the Dominated convergence, one finds that
  \begin{equation}
    \ddp{}{s}\E_{\Gamma_{j+1} (s)} [F(X,  \varphi + \zeta )]
    =
    \frac12
    \sum_{x,y \in \Lambda_N}\ddp{\Gamma_{j+1}(x,y)}{s} 
    \E_{\Gamma_{j+1} (s)} [\ddp{^2F(X,  \varphi + \zeta )}{\varphi(x)\varphi(y)}].
  \end{equation}
  Let $f^z(x)=\Gamma_{j+1}(z,x)$ and $g^z(x) = \delta(z,x)$.
  It then follows with the notation from \eqref{eq:def-D} that
  \begin{equation} 
    \E_{\Gamma_{j+1} (s)} [F(X,  \varphi + \zeta )]-\E_{\Gamma_{j+1} (s')} [F(X,  \varphi + \zeta )]
    =
    \frac12
      \sum_{z\in X^*}
            \int_s^{s'} ds''\,
      \E_{\Gamma_{j+1}(s'')} [D^2F(X,  \varphi + \zeta; f^z,g^z)].
  \end{equation}
  By taking the $\|\cdot\|_{h,T_{j+1}(\bar X)}$ norm
  of this and using Proposition~\ref{prop:E_G_j},
  it follows  that the left-hand side of \eqref{eq:stability_of_expectation_singleX} is bounded by
  \begin{equation} \label{eq:stability_of_expectiation_local_conclusion}
    |s-s'| \pa{ 2^{|X|_j} h^{-2} \sum_{z\in X^*}\|f^z\|_{C_j^2(X^*)} \|g^z\|_{C_j^2(X^*)} \|F(X)\|_{h,T_j(X)}} 
    .
  \end{equation}
  Since $X^*$ is finite and $\|f^z\|_{C_j^2(X^*)}$ 
   and $\|g^z\|_{C_j^2(X^*)}$ are bounded uniformly in $|s|\leq \epsilon_s\theta_J$
  (their dependence on $j$ and $X$ is not relevant),
  the claim follows.
  The case $X\in \cP_{j+1}^c$ assuming \eqref{eq:stability_of_expectation_scalej+1} 
  follows using the same proof, and we now obtain
  \begin{equation}
    |s-s'| \pa{ 2^{L^2 |X|_{j+1}} h^{-2} \sum_{z\in X^*}\|f^z\|_{C_{j+1}^2(X^*)} \|g^z\|_{C_{j+1}^2(X^*)} \sup_\varphi \frac{ \|F(X, \varphi)\|_{h,T_{j+1}(X,\varphi)} }{G_j(X,\varphi)}} 
  \end{equation}
instead of \eqref{eq:stability_of_expectiation_local_conclusion}.
\end{proof}

\section{Contraction mechanisms} \label{sec:inequalities}

The estimates derived in this section exhibit the contraction mechanisms that will be used to identify
contracting (also called irrelevant) terms along the renormalisation flow. There are essentially three sources of contraction in our set-up, one stemming from periodicity of polymer activities in $\varphi$ (which is inherited from the original potential), one from terms involving only gradients of $\varphi$ (or higher-order derivatives), and one coming from large polymers $X \notin \mathcal{S}_j$.

The main results of this section are Propositions~\ref{prop:Loc-coupling} and \ref{prop:Loc-contract}, which concern small polymers.
Most of the remainder of this section consists of supporting arguments that are used only for the proof of these
propositions and will not be applied directly in the rest of this paper.
Finally, in Section~\ref{sec:reblocking}, we show that large polymers contract.

\subsection{Periodicity, charge decomposition, and lattice symmetries} \label{sec_chargedecomp}

For a field $\varphi=(\varphi_x)$ and scalar $t \in \R$ we often write $\varphi+t= (\varphi_x + t)$ in the sequel.
Our starting point is the following \emph{charge decomposition}
of a globally periodic field functional, introduced in \cite{MR1101688}.

\begin{definition}\label{def:chargedecomp}
Let $F(X,\varphi)$ be a polymer activity
  such that $t\in \R \mapsto F(X,\varphi+t)$ is $2\pi/\sqrt{\beta}$ periodic, for some $\beta >0$. 
Its Fourier expansion in the constant part is denoted by
\begin{equation}\label{eq:charge_qdefin}
F(X,\varphi+ t)= \sum_{q\in \mathbb{Z}} e^{i\sqrt{\beta} q t} \hat{F}_q(X, \varphi ), \quad t \in \R
\end{equation}
where
\begin{equation}
\label{eq:Fhat}
\hat{F}_q( X, \varphi) = \frac{\sqrt{\beta}}{2\pi} \int_0^{\frac{2\pi}{\sqrt{\beta}}} ds \; e^{-i\sqrt{\beta} q s} F(X, \varphi + s), \quad q \in \Z.
\end{equation}
The polymer activity $\hat{F}_q$ is called the charge-$q$ part of $F$ (and the neutral part when $q=0$).
Moreover, a polymer activity $F$ is said to have charge $q$ (be neutral) if $\hat{F}_{q'}=0$ except when $q'=q$ ($q'=0$), i.e.,~if $F=\hat{F}_{q}$. 
\end{definition}

We simply refer to a $2\pi/\sqrt{\beta}$-periodic polymer activity for some $\beta > 0$ as \textit{periodic} in the sequel. In doing so, we always assume that statements hold for any value of $\beta$, unless explicitly stated otherwise.

Notice that the smoothness assumption on $F$ guarantees the existence and absolute convergence of the Fourier series \eqref{eq:charge_qdefin}.
Moreover, $F$ having charge $q$ is equivalent to the condition that
\begin{equation}\label{eq:purecharge}
F(X, \varphi+ t)= e^{i\sqrt{\beta} q t} F(X, \varphi), \quad \text{for all }t \in \R
\end{equation}
(the direct implication follows plainly from \eqref{eq:Fhat} and the converse by comparing \eqref{eq:purecharge} and \eqref{eq:charge_qdefin}). 

For later use, we
record the following instance of the above set-up.
For any polymer activity $F(X,\varphi)$ as appearing in Definition~\ref{def:chargedecomp}, fixing a point $x_0 \in X$ and denoting $\delta \varphi (x) = \varphi(x) - \varphi(x_0)$, using that $F(X, \varphi)=F(X, \varphi(x_0) + \delta \varphi)$, one sees that
\begin{equation} \label{e:chargedecompx0}
  F(X, \varphi) = \sum_{q\in \mathbb{Z}} e^{i\sqrt{\beta} q \varphi (x_0)} \hat{F}_q( X,  \delta \varphi ).
\end{equation}
The following elementary lemma states that the charge-$q$ part $\hat F_q$ of a polymer activity is bounded in terms of
the norm of the polymer activity (defined in Definition~\ref{def:NORM}),
and also gives the norm of the $F$-independent exponential factor in \eqref{e:chargedecompx0}.

\begin{lemma}
  Let $F$ be a periodic polymer activity. For all $\varphi \in \R^{\Lambda_N}$ and $X \in \mathcal{P}_j^c$,
\begin{equation}
\norm{\hat{F}_q(X, \varphi)}_{h, T_{j} (X, \varphi)} \leq \norm{F (X)}_{h, T_j( X)} G_j (X, \varphi) \label{eq:linearity_of_expectation2}
\end{equation}
and
\begin{equation} \label{e:normbd-eiphi}
\norm{e^{i\sqrt{\beta} q \varphi (x_0)}}_{h, T_{j} (X, \varphi)} = e^{\sqrt{\beta} |q| h}, \quad x_0 \in X.
\end{equation}
\end{lemma}

\begin{proof}
  The inequality \eqref{eq:linearity_of_expectation2} is obtained by starting from \eqref{eq:Fhat} and
  then using the definition of the norm: for $(f_k)_{k=1}^n$ with $\norm{f_k}_{C_j^2(X^*)} \leq 1$ for each $k$,
\begin{align}
\big| D^n \hat{F}_q (X, \varphi) (f_1, \cdots, f_n) \big| &\leq \int_0^{1} ds \, |D^n F(X, \varphi + 2\pi \beta^{-1/2} s) (f_1, \cdots, f_n)| \nnb
&\leq \int_0^1 ds  \norm{D^n F (X, \varphi + 2\pi \beta^{-1/2} s)}_{h, T_j (X, \varphi)}
\end{align}  
hence
\begin{align}
\norm{\hat{F}_q (X, \varphi)}_{h, T_j (X, \varphi) } \leq \int_{0}^1 ds \norm{F (X)}_{h, T_j (X)} G_j (X, \varphi + \frac{2\pi}{\sqrt{\beta}} s) \stackrel{\eqref{eq:def_large_field_regulator}}{=} \norm{F (X)}_{h, T_j (X)}  G_j (X, \varphi).
\end{align}

The identity \eqref{e:normbd-eiphi} also follows easily from the definition of the norm since
\begin{equation}
  D^n_{\varphi} e^{i\sqrt{\beta} q \varphi(x_0)} (f_1, \cdots, f_n) = (i\sqrt{\beta} q)^n e^{i\sqrt{\beta} q \varphi(x_0)} \prod_{k=1}^n f_k (x_0),
\end{equation}
which gives the claimed bound when substituted in the definition of the norm.
It can conceptually be understood from the fact that
the right-hand side is the supremum of $|e^{i\sqrt{\beta} q \varphi}|$ for $\varphi$ in a
strip of width $h$ around the real axis.
\end{proof}

The localisation operators 
which will be used to extract the relevant and marginal part from the remainder coordinates
rely on the charge decomposition as well as 
on lattice symmetries, so we define these first.

\begin{definition} \label{def:latticesym}
A scale-$j$ polymer activity $F=(F(X))_{X\in \cP_j^c}$ is
\emph{invariant under lattice symmetries} if for every graph automorphism $A$ of the torus $\Lambda_N$
that maps any block in $\cB_j$ to a block in $\cB_j$ one has $F(AX,A\varphi)=F(X,\varphi)$
where $(A\varphi)(x) =\varphi(A^{-1}x)$.
$F$ is \emph{even} if $F(X, \varphi) = F(X, -\varphi)$ for every $(X, \varphi)$.
\end{definition}

\subsection{Localisation operator}\label{sec:Loc}

The main result of Section~\ref{sec:inequalities} are the following
localisation operators $\Loc_{X,B}$
which will be used to extract the relevant and marginal part from the remainder coordinates.
Our notation $\Loc_{X,B}$ is inspired by that of \cite{MR3332939}, but
compared to this reference,
the contraction mechanisms in this section rely on oscillations under the Gaussian expectation for the charged terms in addition.
These operators are defined explicitly in the next definition,
 but the explicit definition does not play a direct role in the remainder of the paper:
all that we will require in the following sections are its main 
properties which are stated in Propositions~\ref{prop:Loc-coupling} and \ref{prop:Loc-contract} below.

The definition of $\Loc$ has two motivations, 
one analytic and one algebraic. 
In analytic considerations, 
the intuition 
(and is substantiated by its properties stated in the next propositions) is related to which terms of a given periodic polymer $F$ are \emph{relevant} or \emph{marginal}: all the higher order Fourier coefficients $\hat{F}_q$, $q \geq 1$, contract at large $\beta$ (i.e., they become irrelevant along the renormalisation group flow), cf.~Lemma~\ref{lemma:contraction_of_charge_q_term} below, and so does the neutral part $\hat{F}_0$ after removal of its Taylor expansion in $\nabla \varphi$ up to terms of second order, cf.~Lemma~\ref{lemma:gaussian_contraction} below.
The combination of these mechanisms culminates in Proposition~\ref{prop:Loc-contract}.
Moreover, it is sufficient to exhibit these cancelation for small polymers $X \in \cS_j$. Large polymers will turn out to contract automatically (due to their size), as explained further in Section~\ref{sec:reblocking}, see in particular Proposition~\ref{prop:largeset_contraction-v2}.
In algebraic considerations, 
we use relation such as Proposition~\ref{prop:Loc-coupling} to define coupling constants ($(\bar{E}, \bar{s})$ in the proposition).
Thus we define $\Loc$ as a modified Taylor expansion with symmetry \eqref{e:Loc-coupling}.

\begin{definition} \label{def:Loc}
Let $F$ be a periodic scale-$j$ polymer activity, and let $\hat F_0$ be its neutral part.
For $X \in \cS_j$, $B \in \cB_j(X)$, 
define
\begin{align} \label{e:Loc-def}
  &\Loc_{X,B} F(X) \equiv \Loc_{X,B} F(X,\varphi)
  = \frac{1}{|X|_j} \hat{F}_0(X,0)
  \\\nonumber
  &
    +
  \frac{1}{8 |X||B|}
  \sum_{x_0,y_0\in B}\sum_{x_1,x_2 \in X^*}  \partial_{\varphi(x_1)}\partial_{\varphi(x_2)} \hat{F}_0(X,0)
  \sum_{\mu, \nu \in \hat{e}} (1 + \delta_{\mu, \nu} - \delta_{\mu, -\nu})  \nabla^{\mu} \varphi(y_0) (\delta x_1)^{\mu} \, \nabla^{\nu} \varphi(y_0) (\delta x_2)^{\nu}
\end{align}
where $\delta x_i = x_i - x_0$ for $i\in\{1,2\}$,
$\delta_{\mu,\nu} = 1$ if $\mu = \nu$ and is $0$ otherwise, 
and $y^{\mu}$ is the $\mu$-component of $y$ with the convention $y^{-\mu} = -y^{\mu}$. 
For $X \in \cS_j$, also define
\begin{equation} \label{eq:Loc_decomp}
  \Loc_XF(X) = \sum_{B \in \cB_j(X)} \Loc_{X,B} F(X).
\end{equation}
\end{definition}

Following our convention, recall that $j$ is tacitly allowed to take values $j=1,\dots, N-2$ for a given torus of side length $L^N$ and the following statements hold uniformly in $N$ (and $L$ unless stated otherwise).

\begin{proposition} \label{prop:Loc-coupling}
  Let $F$ be a periodic scale-$j$ polymer activity that is even and invariant under lattice symmetries.
  Then there are scalars $\bar{E}=\bar{E}(F)$, $\bar{s}=\bar{s}(F)$
  satisfying (with purely geometric implicit constants)
  \begin{equation}\label{e:Loc-coupling-bds}
    \bar{E} = O(A^{-1}{L^{-2j}} \norm{F}_{h,T_j} ), \quad \bar{s}=O(A^{-1} h^{-2} \norm{F}_{h,T_j})
  \end{equation}
  such that for any $B \in \cB_j$,
  \begin{equation} \label{e:Loc-coupling}
    \sum_{X\in \cS_j: X \supset B} \Loc_{X,B} \Eplus F(X, \varphi' + \zeta)
    = \bar{E} |B| +
    \frac{1}{2} \bar{s} | \nabla \varphi' |^2_B \, ;
  \end{equation}
  here and in the sequel, $ \Loc_{X,B} \Eplus F(X, \varphi' + \zeta)$ refers to the localisation operator applied to the polymer activity $\Eplus F(X, \cdot + \zeta)$ and evaluated at $\varphi'$.
  Moreover, whenever $\norm{F}_{h,T_j}< \infty$, both $\bar{E}=\bar{E}(F)$ and $\bar{s}=\bar{s}(F)$ are continuous functions of the implicit parameter $s \in [-\epsilon_s \theta_J, \epsilon_s \theta_J]$ (inherent to $\Eplus$).
\end{proposition}

\begin{proposition} \label{prop:Loc-contract}
  There exists a constant $c_h>0$ such that the following holds for all $X \in \cS_j$ and periodic scale-$j$ polymer activities $F$ such that $F(X, \varphi) = F(X, -\varphi)$ and $X \in \cS_j$.
  Let $r\in (0,1]$ and assume that  $h \geq \max \{r c_h  \rho_J^{-2} \sqrt{\beta},  \rho_J^{-1}\}$,
  that $\kappa_L = c_\kappa \rho_J^{2}(\log L)^{-1}$ as in Proposition~\ref{prop:E_G_j},
  and that $L \geq C$ and $A\geq 1$. Then for all $\varphi' \in \R^{\Lambda_N}$,
  \begin{equation} \label{e:Loc-contract-full}
    \|\Loc_X \Eplus F(X, \varphi'+\zeta) - \Eplus F(X, \varphi'+\zeta)\|_{h,T_{j+1}(\bar X, \varphi')}
    \leq  
    \alphaLoc A^{-|X|_j}
    \|F\|_{h,T_j} G_{j+1}(\bar X,\varphi')
    ,
  \end{equation}
  where
  \begin{equation} \label{e:Loc-contract-kappa}
    \alphaLoc
    = CL^{-3}(\log L)^{3/2}+ C\min\ha{1,\sum_{q\geq 1}
      e^{\sqrt{\beta}qh}
      e^{-(q-1/2)r\beta\Gamma_{j+1}(0)}}.
  \end{equation}
  Also, $\Loc_{X,B}$ is bounded in the sense (note the $T_j$ instead of $T_{j+1}$ norm on the left-hand side)
  \begin{equation} \label{e:Loc-bounded}
    \norm{\Loc_{X, B} \Eplus F (X, \varphi'+\zeta)}_{h, T_j (X, \varphi')} \leq C (\log L) \norm{F(X)}_{h, T_j (X)} e^{c_w \kappa_L w_j (B, \varphi')^2}
    ,
  \end{equation}
  and $\Loc_{X,B}\Eplus F(X,\cdot+\zeta)$ is continuous in the implicit parameter $s \in [-\epsilon_s\theta_J,\epsilon_s\theta_J]$ (inherent to $\Eplus$) with respect to the same norms.
\end{proposition}

In our application (carried out precisely in Section~\ref{sec:stable_manifold_proof}), 
we will choose $h \leq \max\{ c,  r c_h \rho_J^{-2} \sqrt{\beta} \}$. The 
expression for $\alphaLoc$ can then be simplified as follows: since with this choice of $h$,
\begin{equation}
e^{\sqrt{\beta}h} \leq C e^{r c_h \rho_J^{-2} \beta} ,
\end{equation}
the minimum in \eqref{e:Loc-contract-kappa} is bounded by
\begin{align}
  e^{-\frac12 r\beta\Gamma_{j+1}(0)} e^{\sqrt{\beta} h}
  \sum_{q \geq 0} e^{\sqrt{\beta} qh}  e^{-qr\beta\Gamma_{j+1}(0)} 
  \leq
\big( C e^{-\frac{1}{2} \Gamma_{j+1}(0)}  \big)^{r\beta} \sum_{q \geq 0} \big( C e^{-\Gamma_{j+1}(0)}  \big)^{q r\beta} .
\label{eq:alphaLoc_series_part_bound}
\end{align}
By Corollary~\ref{cor:Gammaj}, the covariances satisfy
$\Gamma_j(0) \sim (4/\betafree + O(s/\betafree^2)) \log L$
with $\betafree = 8\pi (v_J^2+s)$.
For any $\theta \in (0,1/2]$ and $r\beta>\betafree (1+2\theta)$,
it follows that if $L$ is sufficiently large depending on $C$, $\theta$, and $v_J$
(to ensure that $C \leq e^{\frac{1}{4} \theta\Gamma_{j+1} (0)}$),
and $s$ is sufficiently small,
\begin{equation}
( C e^{-\frac12 \Gamma_{j+1}(0)} )^{r\beta}
 \leq  L^{-2 r\beta (1 -  \theta/2) \big( (1/\betafree + O(s/\betafree^2) \big)} 
 \leq   L^{- 2 (1+ 2\theta)(1-\theta/2) } \leq L^{-2 (1+ \theta)} ,
\end{equation}
and hence \eqref{eq:alphaLoc_series_part_bound} is bounded by $C L^{-2 - 2\theta}$.
In particular,
\begin{equation} \label{e:Loc-contract-kappa2}
  \alphaLoc \leq C(L^{-3}(\log L)^{3/2}+L^{-2-2\theta})\leq L^{-2-\theta}.
\end{equation}
The contractivity of the renormalisation group map will later be ensured by $CL^2 \alphaLoc < 1$.

Much of the remainder of this section is concerned with the proof of these propositions.
Proposition~\ref{prop:Loc-coupling} is a relatively straightforward consequence of the definitions.
Proposition~\ref{prop:Loc-contract} is more involved and combines different contraction mechanisms
for neutral and charged terms. We thus discuss these mechanisms separately.

\subsection{Proof of Proposition~\ref{prop:Loc-coupling}}

\begin{proof}[Proof of Proposition~\ref{prop:Loc-coupling}]
    By Definition~\ref{def:Loc}, the left-hand side of \eqref{e:Loc-coupling} equals
  \begin{align} \label{e:Loc-def2}
    &\sum_{X\in \cS_j: X\supset B} \frac{1}{|X|_j} \Eplus \hat{F}_0(X,\zeta) 
  \\
  +
  &\sum_{X\in \cS_j: X\supset B}  \frac{1}{|X||B|} \sum_{x_0,y_0\in B}\sum_{x_1,x_2 \in X^*} \frac12 \partial_{\varphi(x_1)}\partial_{\varphi(x_2)} \Eplus \hat{F}_0(X,\zeta) \avg{\nabla \varphi'(y_0), x_1-x_0, \nabla \varphi'(y_0), x_2-x_0},   \nonumber
\end{align}
where
\begin{equation} \label{eq:doublebracket-def}
  \avg{\nabla \varphi'(y_0), y_1, \nabla \varphi'(y_0), y_2}
  = \frac{1}{4} \sum_{\mu, \nu \in \hat{e}} (1 + \delta_{\mu, \nu} - \delta_{\mu, -\nu})  \nabla^{\mu} \varphi'(y_0) y_1^{\mu} \, \nabla^{\nu} \varphi'(y_0) y_2^{\nu} .
\end{equation}
As we now explain, by invariance under lattice rotations, only the diagonal terms in the inner product contribute
and we see that this expression equals the right-hand side of \eqref{e:Loc-coupling} with
  \begin{align}
    \bar{E} &= \sum_{X\in \cS_j: X \supset B} \frac{1}{|X|} \Eplus \hat{F}_0(X,\zeta)
    \\
    \bar{s} &= \sum_{X \in \mathcal{S}_j: X \supset B}\frac{1}{|X| |B|} \sum_{x_0 \in { B}, x_1, x_2 \in X^*} \partial_{\varphi (x_1)} \partial_{\varphi (x_2)} \Eplus \hat{F}_0 (X, \zeta) ( x_1 - x_0, x_2 - x_0 ) ,
\end{align}
where $(\cdot,\cdot)$ is the standard $\ell^2$ inner product on $\Z^2$---although the points lie in $\Lambda_N$, since they live in a small polymer $X$, we can define subtraction and inner products as if thy live in $\Z^2$.
To see this in detail, expand the second term of \eqref{e:Loc-def2} using the definition \eqref{eq:doublebracket-def}
and let $I_{\mu \nu}$ be
the (scaled) coefficient of $\nabla^{\mu} \varphi' (y_0) \nabla^{\nu} \varphi' (y_0)$ written explicitly as
\begin{align}
I_{\mu \nu} = i_{\mu \nu} \sum_{X \in \cS_j : X\supset B} \frac{1}{|X|} \sum_{x_0 \in B} \sum_{x_1, x_2 \in X^*} \partial_{\varphi(x_1)}\partial_{\varphi(x_2)} \Eplus \hat{F}_0 (X, \zeta) (x_1 - x_0)^{\mu} (x_2 - x_0)^{\nu}
\end{align}
where $i_{\mu \mu}=2$, $i_{(-\mu) \mu} = 0$ and $i_{\mu \nu} = 1$ if $\mu \perp \nu$. But by rotational invariance, we have $I_{\mu \mu } = I_{\nu \nu}$ for any $\mu, \nu \in \hat{e}$, so
\begin{align}
I_{\mu \mu} = \frac{1}{4} \sum_{\nu \in \hat{e}} I_{\nu \nu} = \sum_{X \in \cS_j : X\supset B} \frac{1}{|X|} \sum_{x_0 \in B} \sum_{x_1, x_2 \in X^*} \partial_{\varphi(x_1)}\partial_{\varphi(x_2)} \Eplus \hat{F}_0 (X, \zeta) (x_1 - x_0, x_2 - x_0)
\end{align}
Therefore summing over $\mu = \pm \nu$ and $y_0 \in B$ simply gives $\frac{1}{2}\bar{s}|\nabla \varphi'|^2_B$. Now for the case $\mu \perp \nu$, it is direct from the expression that $I_{(-\mu) \nu} = - I_{\mu \nu}$ and $I_{\mu \nu} = I_{\nu \mu}$.
But since $\mu \perp \nu$, by rotation invariance, $I_{\nu (-\mu)} = I_{\mu \nu}$ and it follows that $I_{\mu \nu} =0$.

To bound $\bar{s}$,
let $f_{\nu}^{x_0} (x_1) = (x_1 - x_0)^{\nu}$ for $x_1 \in X^*$ and a fixed $x_0 \in X$. Then
$\norm{f^{x_0}_{\nu}}_{C_j^2 (X^*)} \leq C L^j$ and with \eqref{eq:linearity_of_expectation1} it follows that
\begin{align} \label{eq:s-formula}
|\bar{s}| &= \Big| \sum_{X\in \mathcal{S}_j: X \supset B} \frac{1}{|X| |B|} \sum_{x_0 \in { B}, \, \nu = 1,2} D^2 \Eplus \hat{F}_0 (X, \zeta) (f_{\nu}^{x_0}, f_{\nu}^{x_0})   ] \Big| \nnb
& \leq C h^{-2} \sum_{X\in \mathcal{S}_j: X \supset B} \frac{1}{|X| |B|} \sum_{x_0 \in { B}, \, \nu=1,2} L^{2j} \norm{\Eplus F(X, \zeta)}_{h,T_j(X, 0)} \nnb
  & \leq 2^{4} C h^{-2} A^{-1} \norm{F}_{h,T_j} \sum_{X\in \cS_j: X \supset B} \frac{1}{|X|_j} 
   \leq C' h^{-2} A^{-1} \norm{F}_{h,T_j}
    .
\end{align}
The bound for $E$ is proved similarly:
\begin{equation}\label{eq:E-formula}
|\bar{E}| = \Big| \sum_{X\in \mathcal{S}_j: X \supset B} \frac{1}{|X|} \Eplus  \hat{F}_{0} (X, \zeta)  \Big| 
\leq \sum_{X\in \mathcal{S}_j: X \supset B} \frac{1}{|X|}  (A/2)^{-|X|_j} \norm{F}_{h,T_j}
\leq C L^{-2j} A^{-1} \norm{F}_{h,T_j}
.
\end{equation}
The asserted continuity in the implicit parameter $s$ follows from the expressions in the first line of \eqref{eq:s-formula} and \eqref{eq:E-formula} in combination with Lemma~\ref{lemma:stability_of_expectation_singleX}. This completes the proof.
\end{proof}

\subsection{Proof of Proposition~\ref{prop:Loc-contract}: preliminaries}

As a preliminary to the proof of Proposition~\ref{prop:Loc-contract},
we state how the norm of a polymer activity changes when measured in terms of $T_{j+1}$
compared to the $T_j$-norm. We will use the following elementary  inequality.

\begin{lemma} \label{lemma:estimate_of_deltaf}
  Let $X\in \mathcal{S}_j$. Fix $x_0 \in X$ and for $f : X^* \rightarrow \mathbb{C}$, define $\delta f(x) = f(x) - f(x_0)$. Then
\begin{equation}
\norm{\delta f}_{C_{j}^2 (X^*)} \leq  C_g L^{-1} \max_{m=1,2} \norm{\nabla^m_{j+1} f}_{L^{\infty} (X^*)}.
\end{equation}
for some geometric constant $C_g >0$.
\end{lemma}

\begin{proof}
  Since $X \in \cS_j$, its small set neighbourhood $X^*$ contains at most $4 b$ blocks,
  where $b=|B^*|_j$ for any $B \in \cB_j$.
  Thus the $\ell^\infty$-diameter of $X^*$ is at most $C_g L^j$
  and thus
\begin{equation} \label{eq:estimate_of_deltaf-pf}
	\norm{\delta f}_{L^{\infty}(X^*)} = \max_{x\in X^*} |f(x) - f(x_0)| \leq C_g L^{j} \max_{x\in X^*, \, \mu \in \hat{e}} |\nabla^{\mu} f (x)| \leq C_g L^{-1} \norm{\nabla_{j+1} f}_{L^{\infty}(X^*)}.
\end{equation}
Also, for $m\geq 1$, $\nabla^m \delta f = \nabla^m f$, and the result follows.
\end{proof}

This lemma has the following important consequence for neutral polymer activities.

\begin{lemma} \label{lemma:neutral_term_under_rescaling}
  Let $F$ be a neutral scale-$j$ polymer activity.
  Then for $X\in \mathcal{S}_{j}$, $\varphi'  \in \R^{\Lambda_N}$ and $n \geq 0$,
\begin{equation}\label{eq:neutral_contract1}
  \norm{D^n F (X, \varphi)}_{n, T_{j+1} (X, \varphi)} \leq (C_g^{ -1} L)^{-n} \norm{D^n F(X, \varphi)}_{n, T_j (X, \varphi)}
  .
\end{equation}
In particular, 
\begin{equation}
\norm{F(X, \varphi)}_{h, T_{j+1} (X,\varphi)} \leq \norm{F(X, \varphi)}_{C_g L^{-1}h, T_j (X, \varphi)}.
\end{equation}
\end{lemma}

\begin{proof}
 Since $F$ is neutral, i.e.,~ has charge $q=0$, cf.~\eqref{eq:purecharge}, for $f_1$ constant-valued one has
\begin{align}
D F (X, \varphi) (f_1) 
= f_1 \sum_{x_0 \in X} \frac{\partial}{\partial \varphi_{x_0}}F (X, \varphi) 
= f_1 \frac{d}{dc} F(X, \varphi + c) \big|_{c=0} 
= 0
\end{align}
and the same reasoning implies that $D^n F(X, \varphi) (f_1, \cdots, f_n) =0$ whenever any of $f_1, \dots, f_n$ is constant-valued.
Therefore, having fixed $x_0 \in X$, for any $f_i \in \R^{\Lambda_N}$, by multilinearity,
\begin{align*}
D^n F(X, \varphi) (f_1, \dots, f_n) = D^n F(X, \varphi) (\delta f_1, \dots, \delta f_n),
\end{align*}
where $(\delta f_k) (x) = f_k (x) - f_k(x_0)$. Therefore if $\norm{f_k}_{C_{j+1}^2 (X^*)} \leq 1$ for $k=1, \ldots, n$,
\begin{align}
|D^n F(X, \varphi)(f_1, \dots, f_n)| \leq (C_g^{ -1} L)^{-n} \norm{D^n F(X, \varphi)}_{n, T_{j}(X, \varphi)}
\end{align}
by Lemma~\ref{lemma:estimate_of_deltaf}. In view of \eqref{eq:seminorm}, the claim follows.
\end{proof}

The following similar but weaker bound holds for charged polymer activities.

\begin{lemma} \label{lemma:charged_term_under_rescaling}
Let $F$ be a scale-$j$ polymer activity of charge $q$ that is supported on $X\in \cS_j$. Then
\begin{equation}
\norm{F (X, \varphi)}_{h, T_{j+1} ( X, \varphi)} \leq e^{\sqrt{\beta} |q| h} \norm{F (X, \varphi)}_{C_g L^{-1}h, T_j (X, \varphi)}.
\end{equation}
\end{lemma}
\begin{proof}
  One may decompose $F(X, \varphi) = e^{i\sqrt{\beta} q \varphi (x_0)} F(X, \delta \varphi)$ where $\delta \varphi(x) = \varphi (x) - \varphi(x_0)$.
  Define $\bar{F}(X, \varphi) := F(X, \delta \varphi)$, then $\bar{F}$ is now neutral.
  The estimate of Lemma~\ref{lemma:neutral_term_under_rescaling} applies to $\bar{F}$, giving
  \begin{align}
    \norm{\bar{F} (X, \varphi)}_{h, T_{j+1} (X, \varphi)} \leq \norm{F}_{C_g L^{-1} h, T_j (X, \varphi)}.
  \end{align}
  The conclusion now follows from \eqref{e:normbd-eiphi} and the submultiplicativity property of the norm \eqref{eq:prodprop}.
\end{proof}

\subsection{Proof of Proposition~\ref{prop:Loc-contract}: charged part}
\label{sec:pf-Loc-charged}

We will prove Proposition~\ref{prop:Loc-contract} by decomposing $F$ into its neutral and charged part and considering both contributions separately, starting with the latter. The estimate \eqref{e:Loc-contract-full} for charged $F$ relies crucially on the expectation
of the charged components on the left-hand side of \eqref{e:Loc-contract-full}. 
The contraction mechanism for charged polymer activities is a generalisation of the elementary identity
\begin{equation}
\Eplus [ e^{i \sqrt{\beta} q \zeta_{x_0}} ] = e^{-\frac{1}{2} \beta q^2 \Gamma_{j+1} (0)}, \label{eq:integral_of_charge_q}
\end{equation}
valid for all integers $q$ and $\beta >0$, where here and in the sequel, $\Gamma_{j+1}(x)= (\delta_0, \Gamma_{j+1} \delta_x)$.
The generalisation uses the analyticity of polymer activities
 with finite $\|\cdot\|_{h,T_j}$-norm,
see Proposition~\ref{prop:complex_extension_of_polymer_activity},
which justifies the following complex translation.

\begin{lemma} \label{lemma:complex_shift_of_charged_activity}
  Let $h>0$, and let $F$ be a charge-$q$ polymer activity with $\|F(X)\|_{h,T_j(X)}<\infty$, $q\in \mathbb{Z}$. Then for any constant $c\in \R$ with $|c| < h$,
  \begin{align}
    F (X, \varphi + ic) = e^{-\sqrt{\beta} q c} F (X, \varphi) . \label{eq:complex_shift_of_charged_activity}
\end{align}
\end{lemma}
\begin{proof}
  First recall that by Proposition~\ref{prop:complex_extension_of_polymer_activity},
  $F(\varphi + z)$ is well-defined and complex analytic for $z\in \{ w\in \C : |w| < h \}$.
  Hence $f: \{ z \in \C : |z| < h \} \rightarrow \C$, $z\mapsto F(X, \varphi + z) - e^{iq\sqrt{\beta} z} F(X, \varphi)$
  is a complex analytic function that takes value $0$ on the real line by \eqref{eq:purecharge}. Therefore $f\equiv 0$.
\end{proof}

Before we jump into the main result, we first discuss a technical point,
which defines the constant $c_h$ appearing in the statement of Proposition~\ref{prop:Loc-contract}.
Ultimately, we are interested in the covariance $\Gamma_{j+1}$, but to obtain
the optimal estimates we must work with its subdecomposition into fractional scales
introduced in Section~\ref{sec:subscale}.
Thus for $\ell, M$ as in Section~\ref{subsec:subdecomp_of_the_reg}, let $I = \{ 0, M^{-1}, 2M^{-1}, \dots , 1-M^{-1} \}$ be the set of fractional scales.
Then for $s= k M^{-1} \in I$ and $s' = s+ M^{-1}$ (cf.~Remark~\ref{R:choice-parameters} regarding $M$), set
\begin{equation}\label{e-xi}
  \xi_s (x) = \sqrt{\beta} ( \Gamma_{j+s, j+s'} (x-x_0) - \Gamma_{j+s, j+s'} (0)),
  \qquad
  \xi_{< s} = \sum_{t\in I, t<s} \xi_t.
\end{equation}

\begin{lemma}[Choice of $c_h$] \label{lemma:choice_of_c_h}
  There exists $c_h >0$ such that
  for any $X \in \cS_j$, $s\in I$  and $\beta >0$,
\begin{equation} \label{eq:choice_of_c_h}
  \norm{\xi_s }_{C_j^2 (X^*)} \vee  \norm{\xi_{<s}}_{C_j^2 (X^*)} < \frac{1}{2} c_h \rho_J^{-2} \sqrt{\beta}
  .
\end{equation}
\end{lemma}
\begin{proof}
By  Lemma~\ref{lemma:fine_Gamma_estimate},
\begin{align}
\norm{ \sum_{s\in J} \nabla^{\alpha}_j \xi_s }_{L^{\infty}(X^*)} \leq C_{\alpha} \rho_J^{-2} \sqrt{\beta}
\end{align}
for any $J \subset I$ and $|\alpha| \in \{1,2\}$. Also for $\alpha = 0$ and $X \in \cS_j$,
with the same constant $C_g$ as in \eqref{eq:estimate_of_deltaf-pf},
\begin{equation}
  \sup_{x\in X^*} |\sum_{s\in J} \xi_s (x) |
  \leq C_g L^j  \norm{\sum_{s\in J} \nabla \xi_s }_{L^{\infty}(X^*)} \leq C_g \norm{\sum_{s\in J} \nabla_j \xi_s }_{L^{\infty}(X^*)}
  \leq C_g C_\alpha \sqrt{\beta} \rho_J^{-2}.
\end{equation}
Combining both inequalities gives the claim with $c_h=3 C_\alpha (C_g \vee 1)$.
\end{proof}

Henceforth, we fix $c_h$ so that the conclusions of Lemma~\ref{lemma:choice_of_c_h} hold. The formula \eqref{eq:integral_of_charge_q} can now be generalised to the following identity.

\begin{lemma} \label{lemma:charged_term_exponential_contraction}
  Let $r \in (0,1]$, $h \geq r c_h  \rho_J^{-2} \sqrt{\beta}$, and
  let $F$ be a charge-$q$ 
  polymer activity with
  $\|F(X)\|_{h,T_j(X)}<\infty$.
  Then for $X \in \cS_j$, $q\in \Z$, $x_0 \in X$ and $\xi (x) = \sqrt{\beta}  (\Gamma_{j+1} (x-x_0)- \Gamma_{j+1} (0))$,
  for all $\varphi' \in \R^{\Lambda_N}$,
\begin{equation}
  \Eplus [ F(X, \varphi' +\zeta)] = e^{-\frac{1}{2} \beta \Gamma_{j+1} (0) (2 r  |q  |- r^2) } \Eplus \big[ e^{-i\sqrt{\beta}r \sigma_q   \zeta ( x_0 ) } F ( X, \varphi' + \zeta + i r \sigma_q   \xi  ) \big],
  \label{eq:exponential_contraction}
\end{equation}
where $ \sigma_q= \textnormal{sign}(q)$.
\end{lemma}

\begin{proof}
Recall that $\zeta \sim \cN (0, \Gamma_{j+1})$ under $\Eplus$.
We will need to work with the subdecomposition of the covariance $\Gamma_{j+1}$ discussed above the lemma;
see the discussion below \eqref{eq:g-o1} for the reason.  
Since $\xi = \sum_{s\in I} \xi_s$, it is sufficient to show the lemma for $\norm{F}_{h, T_{j+s}} < + \infty$ and $\Gamma_{j+1}$ replaced by $\Gamma_{j+s, j+s'}$ where $s = s' - M^{-1} \in I$ and $\zeta$ replaced by $\zeta_s + i r \xi_{<s}$ where $\zeta_s \sim \cN (0, \Gamma_{j+s, j+s'})$.
It is convenient to work with invertible covariance matrices,
so we will work with $C = \Gamma_{j+s, j+s'} + \delta$ for $\delta >0$
so that $C$ is strictly positive definite,
and then take the limit $\delta \downarrow 0$ to conclude.
All in all, it thus suffices to show that
\begin{equation}
\begin{split}\label{e:shift-red-step}
& \int e^{-\frac{1}{2}(\zeta_s, C^{-1} \zeta_s)} F (X, \varphi' + \zeta_s + ir \sigma_q \xi_{<s}) \, d\zeta_s \\
&\qquad = e^{-\frac{1}{2}\beta C(0) (2r|q| -r^2)} \int e^{-\frac{1}{2}(\zeta_s, C^{-1} \zeta_s) - i \sqrt{\beta} r \sigma_q \zeta_{s}(x_0)} F (X, \varphi' + \zeta_s +  i r \sigma_q  (\xi_{<s} +\xi_s)) \, d\zeta_s,
\end{split}
\end{equation}
from which \eqref{eq:exponential_contraction} readily follows by integrating successively over $\zeta_s$, $s \in I$ and letting $\delta \downarrow 0$. Here, with a slight abuse of notation, we define $\xi_s (x) = \sqrt{\beta} ( C (x -x_0) - C (0))$, from which $\xi_s$ as introduced in \eqref{e-xi} is obtained in the limit $\delta \to 0$.
Then the bound of Lemma~\ref{lemma:choice_of_c_h} holds the same for this modified $\xi_s$ when $\delta$ is sufficiently small, which we henceforth tacitly assume.

We now show \eqref{e:shift-red-step}.
Let $X' = \{x\in \Lambda : d_1 (x, X^*) \leq 2 \}$ so that $\norm{\psi}_{C_j^2 (X^*)}$ only depends on $\psi |_{X'}$.
Performing a change of variable from $\zeta_s$ to $\zeta_s- ir \sigma_q \xi_s$,
the integral on the right-hand side of \eqref{e:shift-red-step} can be recast as
\begin{equation}
  \int_{\mathbb{R}^{X'}} e^{-\frac{1}{2} (\zeta_s, C^{-1} {\zeta}_s) - i\sqrt{\beta}  r \sigma_q \zeta_s (x_0)} F (X, \varphi' + \zeta_s + ir \sigma_q  \xi_{<s} + ir \sigma_q \xi_s) \, d\zeta_s
  = e^{-\frac{1}{2} \beta r^2 C(0)} R_{ r \sigma_q \xi_s} (r)
\end{equation}
where
\begin{equation}\label{e:R_psi}
R_{\psi} (r') = \int_{\mathbb{R}^{X'} + i \psi} e^{-\frac{1}{2} (\zeta + i r' z, C^{-1} (\zeta + ir' z) )} F(X, \varphi' + \zeta + ir {\sigma_q} \xi_{<s} ) \, d\zeta.
\end{equation}
and $z=\sigma_q\sqrt{\beta} C(0)$.
To show that $e^{-\beta C(0)r|q|}R_{r\sigma_q \xi_s}(r)$ equals the left-hand side of \eqref{e:shift-red-step},
we will apply Cauchy's formula {to first show $R_{r\sigma_q \xi_s}(r)=R_0(r)$. Indeed,}
by Proposition~\ref{prop:complex_extension_of_polymer_activity}, $F(X)$ is complex analytic on $S_{{h}}$.
By Lemma~\ref{lemma:choice_of_c_h} and assumption on $\delta$, the condition $h \geq c_h r \rho^{-2} \sqrt{\beta}$ guarantees $\norm{r\sigma_q \xi_s}_{C_{j}^2 (X^*)}$, $\norm{r \sigma_q\xi_{<s}}_{C_{j}^2 (X^*)}$
are strictly less than $ {h} /2$, and thus in particular $\norm{r  \sigma_q( \xi_s + \xi_{<s})}_{C_{j}^2 (X^*)}< h$.
We claim that $R_{\psi} \equiv R_0$ for any $\norm{\psi}_{C^2_{j} (X^*)} <  {h}/2$. To see this, consider $\zeta$ as a vector in the space $\C^{X'}$ and make the orthogonal (isometric) change of coordinates $(\delta_x : x \in X')$ to $(e_y : y \in E)$ (so that $|E| = |X'|$) with $\psi = c e_{y_0}$ for some $y_0 \in E$, $c \in \R$. 
Then showing $R_{\psi} (r) = R_0 (r)$ is equivalent to showing that
\begin{multline}
  \int_{\R + ic} e^{ - \frac{1}{2} (\zeta + i r z, C^{-1} (\zeta + ir z) )} F (X, \varphi' + \zeta + i r \sigma_q\xi_{<s}) \, d\zeta (y_0)
  \\
  = \int_{\R} e^{- \frac{1}{2} (\zeta + i r z, C^{-1} (\zeta + ir z) )} F (X, \varphi' + \zeta + i r \sigma_q \xi_{<s}) \, d\zeta (y_0). \label{eq:Cauchy_for_F}
\end{multline}
But by Cauchy's integral theorem, it is sufficient to show that
\begin{equation}
\sup_{\zeta (y_0) = \pm R, \, |s| \leq 1} \big| e^{-\frac{1}{2} (\zeta + irz +  i s\psi, C^{-1} (\zeta + irz + is \psi))} F (X, \varphi' + \zeta + i r \sigma_q \xi_{<s}+ is\psi) \big| \rightarrow 0 \quad \text{as } R\rightarrow \infty. \label{eq:Cauchy_for_F_bound_condition}
\end{equation}
To see this, first note that
\begin{align}
\big| e^{-\frac{1}{2} (\zeta + irz + is\psi, C^{-1} (\zeta + irz + is \psi))} \big| \leq e^{-\frac{1}{2} (\zeta, 
C^{-1} \zeta)} e^{r^2 z^2 (1,C^{-1} 1) + s^2 (\psi, C^{-1} \psi) } 
\end{align}
while, by \eqref{eq:complex_extension_via_Taylor}, we have
\begin{align}\label{eq:F-bound-id}
|F (X, \varphi' + \zeta + i r \sigma_q\xi_{<s}+ is \psi)| \leq \norm{F(X)}_{ {h}, T_{j}(X)} G_j (X, \varphi' + \zeta) \leq \norm{F(X)}_{ {h}, T_{j}(X)} G_{j+s} (X, \varphi' + \zeta)
\end{align}
and by Lemma~\ref{lemma:G_change_of_scale},
\begin{align}
G_{j+s} (X, \varphi' +  \zeta ) \leq  g_{j+s} (X, \zeta)  G_{j+s'} (X, \varphi' )  .
\end{align}
  Now by Lemma~\ref{lemma:g_j+s_bound_by_quadratic},
  $g_{j+s} (X, \zeta) \leq e^{\frac{1}{2} Q_{j+s} (X,\zeta)}$ where $Q_{j+s}(X,\zeta)$ is a quadratic form in $\zeta$
  and $e^{\frac{1}{2} Q_{j+s}(X,\zeta)}$ is integrable with respect to $\E^\zeta$.
This implies
\begin{align} \label{eq:g-o1}
e^{-\frac{1}{2} (\zeta, C^{-1} \zeta)} g_{j+s} (X, \zeta) \leq e^{-\frac{1}{2} (\zeta, C^{-1} \zeta)} e^{\frac{1}{2} Q_{j+s}(X,\zeta)} = o(1) \quad \text{as } |\zeta| \rightarrow \infty
\end{align}
and proves \eqref{eq:Cauchy_for_F_bound_condition}.
(Note that the last step would not have worked if we had directly used $G_j$ instead of $g_{j+s}$ because we do not have a quadratic form $Q'$ such that $G_j (X, \zeta) \leq e^{\frac{1}{2} Q'(\zeta)}$ and $\E[e^{\frac{1}{2} Q'(\zeta)}] < + \infty$ at the same time.)

To compute $R_0 (r)$, consider $R_0( r'+ \delta r')$ for sufficiently small $\delta r'$. Another application of \eqref{eq:Cauchy_for_F} shows that
\begin{equation}
R_0(r'+\delta r') = \int_{\R^{X'}} e^{-\frac{1}{2}(\zeta + i r'z , C^{-1}(\zeta + ir'z))} F (X, \varphi' + \zeta + i r \sigma_q\xi_{<s} - i(\delta r') z)\, d\zeta.
\end{equation}
But by \eqref{eq:complex_shift_of_charged_activity}, $F (X, \varphi' + \zeta + ir \sigma_q\xi_{<s} - i(\delta r') z ) = e^{\sqrt{\beta} q  z \delta r'} F (X, \varphi' + \zeta + ir \sigma_q \xi_{<s} )$ and hence $\frac{d}{dr'} R_0(r') = \sqrt{\beta} q z R_0(r')  = {\beta} |q| C(0)R_0(r')$, cf.~below \eqref{e:R_psi} regarding $z$. Solving this differential equation yields
\begin{equation}
R_0(r) = e^{|q| r \beta C (0)} \int_{\R^{X'}} e^{-\frac{1}{2} (\zeta,C^{-1} \zeta)} F (X, \varphi' + \zeta + ir\sigma_q \xi_{<s})  d\zeta ,
\end{equation}
thus proving \eqref{e:shift-red-step}. 
\end{proof}

This identity leads to the following contraction mechanism for charge-$q$ polymer activities.

\begin{lemma} \label{lemma:contraction_of_charge_q_term}
Let $r \in (0,1]$ 
, $h \geq r c_h \rho_J^{-2} \sqrt{\beta}$ and $L\geq 2C_g$.
There exists $C>0$ such that for $X\in \mathcal{S}_j$,
and any
charge-$q$ polymer activity $F$
with $|q| \geq 1$ and $\|F (X)\|_{h,T_j (X)} < \infty$, and all $\varphi' \in \R^{\Lambda_N}$,
\begin{equation} \label{e:contraction_of_charge_q_term-1}
  \norm{\Eplus [ F (X,  \varphi' + \zeta)]}_{h, T_{j+1}(X, \varphi')}
  { \leq C e^{\sqrt{\beta} |q| h} e^{-(|q|-1/2) r \beta \Gamma_{j+1} (0)} \norm{ F(X)}_{h, T_j (X)} G_{j+1} (\bar{X},  \varphi' ).
    }
\end{equation}
\end{lemma}

\begin{proof}
We will hide the dependence of $F(X,\varphi)$ on the polymer $X$ for brevity
and assume that $X$ is a small set.
Let us start from \eqref{eq:exponential_contraction} with $ r \in (0,1]$. Then
\begin{equation}\label{e:charge-q-start}
D^n \Eplus [F(\varphi' + \zeta)] = e^{-\frac{2 |q | r- r^2}{2} \beta \Gamma_{j+1} (0)} \Eplus [e^{-i\sqrt{\beta} r {\sigma_q} \zeta_{x_0}} D^n F(\varphi' + \zeta + i r { \sigma_q} \xi) ]
\end{equation}
where $\xi (x) = \sqrt{\beta} (\Gamma_{j+1} (x-x_0) - \Gamma_{j+1}(0))$.
By our assumptions and Lemma~\ref{lemma:choice_of_c_h} (with the choice $M=1$), we have $C_d L^{-1}h + r \|\xi\| < h$, where $\|\cdot\| =\|\cdot\|_{C_j^2} $.
Thus by Proposition~\ref{prop:complex_extension_of_polymer_activity},  
$F$ is analytic in the strip $S_{C_d L^{-1} h + r\norm{\xi}}$, and hence the Taylor expansion
\begin{align}
D^n F(\varphi' + \zeta + i  r { \sigma_q}\xi) = \sum_{k=0}^{\infty} \frac{1}{k!} D^{k}_{\varphi'} (D^nF)(\varphi' + \zeta) ( (i  r { \sigma_q} \xi)^{\otimes k} )
\end{align}
is convergent and so, combining with \eqref{e:charge-q-start}, and since $|\sigma_q|=1$,
\begin{align}
\norm{D^n_{\varphi'} \Eplus [F(\varphi' + \zeta)]}_{n, T_{j} (X, \varphi')} &\leq e^{-\frac{2{ |}q { |}-1}{2}  r \beta \Gamma_{j+1} (0)} \Eplus \Big[ \sum_{k=0}^{\infty} \frac{\norm{ r \xi}^{k}}{k!} \norm{D^{n+k} F(\varphi' + \zeta)}_{n+k, T_{j}(X, \varphi')} \Big] .
\end{align}
Therefore, for $h'>0$ left to be chosen,
\begin{align}
\norm{\Eplus [F(\varphi' + \zeta)]}_{h', T_{j}(X, \varphi')} &\leq e^{-\frac{2{ |}q { |}-1}{2}  r \beta \Gamma_{j+1}(0)} \Eplus\Big[ \sum_{n,k} \frac{h'^n \norm{ r \xi}^k}{n! k!} \norm{D^{n+k} F(\varphi' + \zeta)}_{n+k, T_{j} (X, \varphi')}  \Big]  \nnb 
&\leq e^{-\frac{2{ |}q { |}-1}{2}  r \beta \Gamma_{j+1}(0)} \Eplus \Big[ \sum_{n=0}^{\infty} \frac{(h' + \norm{ r \xi} )^n }{n!} \norm{D^{n} F(\varphi' + \zeta)}_{n, T_{j} (X, \varphi')}  \Big] \nnb 
&= e^{-\frac{2{ |}q { |}-1}{2}  r \beta \Gamma_{j+1}(0)} \Eplus [\norm{F (\varphi' + \zeta)}_{h' +  r \norm{\xi} , T_{j} (X, \varphi')  }].
\end{align}
To complete the lemma, one is just left to compare $\norm{F (\varphi' + \zeta)}_{h , T_{j+1} (X, \varphi') }$ with a quantity in a lower scale.
This is where Lemma~\ref{lemma:charged_term_under_rescaling} comes in, yielding the bound
\begin{align}
\norm{\Eplus F(\varphi' + \zeta)}_{h, T_{j+1} (X, \varphi')} \leq  e^{\sqrt{\beta} |q| h} \norm{\Eplus [F(\varphi' + \zeta) ]}_{C_g L^{-1} h, T_{j}(X, \varphi')}
\end{align}
and we see that the choice $h' = C_g L^{-1} h$ gives
\begin{align}
\norm{ \Eplus [F(\varphi' + \zeta)]  }_{h, T_{j+1} (\varphi' + \zeta)} \leq e^{\sqrt{\beta} |q| h} e^{-\frac{2{ |}q { |}-1}{2}  r \beta \Gamma_{j+1}(0)} \Eplus [\norm{F(\varphi' + \zeta)}_{C_g L^{-1} h +  r \norm{\xi}, T_{j}(X, \varphi')}].
\end{align}
Now invoking Proposition~\ref{prop:E_G_j},
\begin{align}
\norm{\Eplus [F(\varphi' + \zeta)]}_{h, T_{j+1}(\varphi', X)} & \leq e^{-\frac{2{ |}q { |}-1}{2}  r \beta \Gamma_{j+1}(0)} e^{\sqrt{\beta} |q| h} \Eplus [ G_{j} (X, \varphi' + \zeta)] \norm{F}_{C_g L^{-1} h +  r \norm{\xi}, T_{j} (X)} \nnb 
&\leq e^{-\frac{2{ |}q { |}-1}{2}  r \beta \Gamma_{j+1}(0)} e^{\sqrt{\beta} |q| h} 2^{|{X}|_j}  G_{j+1} (\bar{X}, \varphi') \norm{F}_{C_g L^{-1} h +  r \norm{\xi}, T_{j} (X)} .
\end{align}
Since $\norm{F}_{C_g L^{-1} h +  r \norm{\xi}, T_{j} (X)}
\leq \norm{F}_{h, T_{j} (X)}$ and $X$ is a small set, the proof is complete.
\end{proof}

Finally, we conclude the Proposition~\ref{prop:Loc-contract} for charged $F$.

\begin{proof}[Proof of Proposition~\ref{prop:Loc-contract}: charged part]
  Assume that $F$ is charged, i.e.,
  \begin{equation}
    F = \sum_{q\neq 0} \hat F_q.
  \end{equation}
  The triangle inequality, \eqref{e:norm-monot} and Lemma~\ref{lemma:contraction_of_charge_q_term} give
\begin{equation}
  \norm{\Eplus F (X,  \varphi' + \zeta)}_{h, T_{j+1}(\bar X, \varphi')}
  \leq C \qa{\sum_{q\geq 1} e^{\sqrt{\beta} |q| h} e^{-(|q|-1/2) r \beta \Gamma_{j+1} (0)}} \norm{ F(X  )}_{h, T_j (X)} G_{j+1} (\bar{X},  \varphi' ).
  \end{equation}
  If this sum is not convergent, one uses the alternative bound
\begin{align}
  \norm{ \Eplus F (X, \varphi'+\zeta) }_{h, T_{j+1} (\bar{X},\varphi')} \leq C \norm{F(X)}_{h, T_j (X)} G_{j+1} (\bar{X}, \varphi').
\end{align}
  This implies the claim since $\Loc_X \Eplus F(X,\varphi'+\zeta)=0$ when $F$ is charged.
\end{proof}

\subsection{Proof of Proposition~\ref{prop:Loc-contract}: neutral part}
\label{sec:pf-Loc-neutral}

For neutral $F$, the contraction in \eqref{e:Loc-contract-full} does not rely on the expectation,
but instead uses that gradients contract under change of norm. 
In all of the following lemmas, we assume that $h\geq  \rho_J^{-1}$ as appearing in the assumptions of Proposition~\ref{prop:Loc-contract}, and we also suppose that all remaining assumptions of Proposition~\ref{prop:Loc-contract} are in force. We will also frequently abbreviate
$\Eplus F(X)=\Eplus F(X, \varphi'+\zeta)$.

In order to bound $\Loc_X \Eplus F(X)- \Eplus F(X)$ (cf.~\eqref{e:Loc-contract-full}) for neutral $F$, our starting point is to split it into two terms whose norms will be bounded separately in Lemmas~\ref{lemma:Loc_minus_Tay} and~\ref{lemma:gaussian_contraction} below. The proof of Proposition~\ref{prop:Loc-contract} then quickly follows. It appears at the end of this section, and combines these two ingredients, along with Lemma~\ref{lemma:Loc-bounded}, which will account for \eqref{e:Loc-bounded}. Thus, let
\begin{align} \label{e:Loc-split}
  \Loc_X \Eplus F(X)- \Eplus F(X)
  &= (\Loc_X \Eplus F(X)-\bar\Tay_2 \Eplus F(X))+(\bar\Tay_2 \Eplus F(X)- \Eplus F(X))
    \nnb
  &= (\Loc_X \Eplus F(X)-\bar\Tay_2 \Eplus F(X))-\bar\Rem_2 \Eplus F(X)
    ,
\end{align}
where the Taylor approximation  and its remainder are defined as follows:
for $F(X) \in \cN_{j}(X)$ with $\|F\|_{h,T_j(X)}<\infty$,
define the Taylor approximation and remainder of degree $n$ (around $0$) by
\begin{align}
& \operatorname{Tay}_n F (X, \varphi ) = \sum_{k=0}^{n} \frac{1}{k !} \sum_{x_1, \cdots, x_k \in X^*} \frac{\partial^k F( X, \psi)}{ \partial \psi (x_1) \cdots \partial \psi (x_k)} \Big|_{\psi \equiv 0} \varphi(x_1) \cdots \varphi(x_k) \\
& \operatorname{Rem}_n F (X, \varphi ) = F( X, \varphi ) - \operatorname{Tay}_n F ( X, \varphi ) .
\end{align}
For $F(X) \in {\cN}_j(X)$ neutral, define
\begin{align}
& \bar{\Tay}_2 F(X, \varphi ) = \frac{1}{|X|} \sum_{x_0 \in X} \Tay_2 F(X, \delta \varphi ) \\
& \bar{\Rem}_2 F (X, \varphi ) = F(X, \varphi ) - \bar{\Tay}_2 F(X, \varphi )
\end{align}
where $\delta \varphi (x) := \varphi(x) - \varphi(x_0)$ is dependent on the choice of $x_0 \in X$.  
Thus, $\bar{\Tay}_2$ corresponds to a second-order Taylor approximation around the origin for the increment $\delta \varphi $, averaged over the base point $x_0$.

We first collect two auxiliary results that will be used to bound the first term in \eqref{e:Loc-split}.

\begin{lemma}
For $\varphi \in \R^{\Lambda_N}$, $X\in \cS_j$ and $x_0, y_0 \in X$,
\begin{multline}
    \norm{\nabla^{e_1} \varphi (x_0)\nabla^{e_2} \varphi (x_0) - \nabla^{e_1} \varphi (y_0)\nabla^{e_2} \varphi (y_0)}_{h, T_{j+1} (X, \varphi)}
    \\
    \leq C L^{-2j -3} (h + \norm{\nabla_{j+1} \varphi}_{L^{\infty}(X)} + \norm{\nabla_{j+1}^2 \varphi}_{L^{\infty} (X)})^2 \label{eq:nfe2}
\end{multline}
and for any $\mu \in \hat{e}$ and $x\in X$ (see below \eqref{e:innerprod} for notation),
\begin{equation}
  \norm{\nabla^{\mu} \varphi(x) \nabla^{(\mu, - \mu)} \varphi(x) }_{h, T_{j+1} (X, \varphi)} \leq C L^{-3j-3} (h + \norm{\nabla_{j+1} \varphi}_{L^{\infty}(X)} + \norm{\nabla_{j+1}^2 \varphi}_{L^{\infty} (X)})^2.      \label{eq:nfe3}
\end{equation}
\end{lemma}
\begin{proof}
To see the first inequality, observe that
\begin{align}\label{e:taylor-gradphi1}
& |\nabla^{\mu} \varphi (x_0) - \nabla^{\mu} \varphi (y_0)| \leq C L^{-j-2} \norm{\nabla^2_{j+1} \varphi}_{L^{\infty}(X)}, 
\end{align}
hence
\begin{align}\label{e:taylor-gradphi2}
& |D_{\varphi} (\nabla^{\mu} \varphi (x_0) - \nabla^{\mu} \varphi (y_0)) (f)| = |\nabla^{\mu} f(x_0) - \nabla^{\mu} f(y_0)| \leq C L^{-j-2} \norm{f}_{C^2_{j+1} (X^*)}.
\end{align}
Since $\nabla^\mu\varphi(x_0)-\nabla^\mu\varphi(y_0)$ is linear in $\varphi$, all but the first two terms in the series expansion \eqref{eq:seminorm} of $\norm{\nabla^{\mu} \varphi (x_0) - \nabla^{\mu} \varphi (y_0)}_{h, T_{j+1} (\varphi,X)}$ vanish and therefore, using \eqref{e:taylor-gradphi1} and \eqref{e:taylor-gradphi2},
\begin{equation}
\norm{\nabla^{\mu} \varphi (x_0) - \nabla^{\mu} \varphi (y_0)}_{h, T_{j+1} (X, \varphi)} \leq CL^{-j-2} (h + \norm{\nabla^2_{j+1}\varphi}_{L^{\infty}(X)}).
\end{equation}
Analogously,
\begin{equation}
  \norm{\nabla^{\mu} \varphi (x_0)}_{h, T_{j+1} (X, \varphi)}
  + \norm{\nabla^{\mu} \varphi (y_0)}_{h, T_{j+1} (X, \varphi)}
  \leq CL^{-j-1}(h + \norm{\nabla_{j+1} \varphi}_{L^{\infty} (X)}),
\end{equation}
and \eqref{eq:nfe2} follows using the submultiplicativity of the norm. The second inequality \eqref{eq:nfe3} follows from similar direct computations:
\begin{align}
& \norm{\nabla^{\mu} \varphi(x) \nabla^{(\mu, - \mu)} \varphi(x) }_{0, T_j (X, \varphi)} \leq C L^{-3j-3} \norm{\nabla_{j+1} \varphi}_{L^{\infty} (X)} \norm{\nabla_{j+1}^2 \varphi}_{L^{\infty} (X)} \\
& \norm{D \nabla^{\mu} \varphi(x) \nabla^{(\mu, - \mu)} \varphi(x) }_{1, T_j (X, \varphi)} \leq C L^{-3j-3} (\norm{\nabla_{j+1} \varphi}_{L^{\infty} (X)} + \norm{\nabla_{j+1}^2 \varphi}_{L^{\infty} (X)} ) \\
& \norm{D^2 \nabla^{\mu} \varphi(x) \nabla^{(\mu, - \mu)} \varphi(x) }_{2, T_j (X, \varphi)} \leq C L^{-3j-3},
\end{align}
and higher-order derivatives vanish.
\end{proof}

\begin{lemma} \label{lemma:bound_D^2F(x_1,x_2)}
Let $F\in \cN_{j} (X)$ with $\|F\|_{h,T_j(X)}<\infty$, and let $X\in \cS_j$. Choose any $x_0 \in X$ and denote $\delta x_1 = x_1 - x_0$, $\delta x_2 = x_2 -x_0$. Then
\begin{equation}
  \absa{\sum_{x_1,x_2  \in X^*} \partial_{\varphi(x_1)}\partial_{\varphi(x_2)} \Eplus F(X,\zeta)
    \delta x_1^\mu \delta x_2^\nu} \leq C h^{-2}L^{2j} \|F(X)\|_{h, T_{j}(X)} 
  . \label{eq:bound_D^2F(x_1,x_2)}
\end{equation}
\end{lemma}
\begin{proof}
{By definition of $\norm{\cdot}_{h, T_j (X)}$ followed by \eqref{eq:linearity_of_expectation1} with $G(\bar X,0)=1$,
\begin{equation}
\begin{split}
  h^2 \absa{\sum_{x_1,x_2} \partial_{\varphi(x_1)}\partial_{\varphi(x_2)} \Eplus F(X,\zeta)
    \delta x_1^\mu \delta x_2^\nu}
 &\leq \| \Eplus F(X,\zeta)\|_{h, T_{j}(X,0)}
 \|\delta x_1^\mu\|_{C_{j}^2(X^*)}\|\delta x_2^\nu\|_{C_{j}^2(X^*)}
      \\
    &\leq C L^{2j} \|F(X)\|_{h, T_{j}(X)} 
\end{split}
\end{equation}}
where we used $\|\delta x_1^\mu\|_{C_{j}^2(X^*)}, \|\delta x_2^\nu\|_{C_{j}^2(X^*)} = O(L^{j})$.
\end{proof}

\begin{lemma} \label{lemma:Loc_minus_Tay}
	Let $h \geq \rho_J^{-1}$ and $\kappa_L = c_{\kappa} \rho_J^2 (\log L)^{-1}$. 
    Then for all $X \in \cS_j$ and neutral $F(X) \in \cN_j(X)$ such that $F(X, \varphi) = F(X, -\varphi)$,
  \begin{equation} \label{eq:Loc_minus_Tay_2}
    \| \Loc_X \Eplus F(X, \varphi' +\zeta) - \bar\Tay_2   \Eplus F(X, \varphi' +\zeta)\|_{h,T_{j+1}(\bar{X}, \varphi')} \leq
    C L^{-3}(\log L) A^{-|X|_j} 
    \|F\|_{h,T_j} G_{j+1} (\bar{X}, \varphi').
  \end{equation}
\end{lemma}
\begin{proof}
  By definition, see \eqref{e:Loc-def} and \eqref{eq:Loc_decomp}, denoting by $B_j(x_0)$ the block $B \in \cB_j$ such that $x_0 \in B$,
  \begin{multline} \label{eq:loc-pf-neutral}
    \Loc_{X} \Eplus F(X)
    =
    \sum_{B\in \cB_j(X)}\Loc_{X,B} \Eplus F(X)
    = \Eplus \hat{F}_0(X,\zeta)
    \\
    +
    \frac{1}{|X|} \sum_{x_0\in X} \frac{1}{|B|} \sum_{y_0 \in B_j(x_0)} \sum_{x_1,x_2 \in X^*} \frac12 \partial_{\varphi (x_1)} \partial_{\varphi (x_2)} \Eplus \hat{F}_0(X,\zeta) \avg{\nabla \varphi'(y_0), \delta x_1, \nabla \varphi'(y_0), \delta x_2},
  \end{multline}
where $\delta x_1=x_1-x_0, \delta x_2=x_2-x_0$ and, following the notation of Lemma~\ref{lemma:bound_D^2F(x_1,x_2)} (cf.~\eqref{eq:doublebracket-def}), 
\begin{align}\label{eq:new-bracket}
\avg{\nabla \varphi'(y_0), \delta x_1, \nabla \varphi'(y_0), \delta x_2} := \frac{1}{4} \sum_{\mu, \nu \in \hat{e}} (1 + \delta_{\mu, \nu} - \delta_{\mu, -\nu})  \nabla^{\mu} \varphi'(y_0) \delta x_1^{\mu} \, \nabla^{\nu} \varphi'(y_0) \delta x_2^{\nu}  .
\end{align}
We firstly replace $\avg{\nabla \varphi'(y_0), \delta x_1, \nabla \varphi'(y_0), \delta x_2}$ by $\langle \nabla \varphi'(y_0), \delta x_1 \rangle \langle \nabla \varphi'(y_0), \delta x_2 \rangle$ and secondly replace $y_0$ by $x_0$ in \eqref{eq:loc-pf-neutral} where
$\langle \nabla \varphi'(x), y \rangle = \frac{1}{2} \sum_{\mu \in \hat{e}} \nabla^{\mu} \varphi'(x) y^{\mu}$. 
This gives
  \begin{equation} \label{e:Loc-contract-pf1}
    \frac{1}{|X|} \sum_{x_0\in X} \sum_{x_1,x_2 \in X^*}\frac12 \partial_{\varphi (x_1)} \partial_{\varphi (x_2)} \Eplus \hat{F}_0(X,\zeta) \langle \nabla \varphi'(x_0), \delta x_1 \rangle \langle \nabla \varphi'(x_0), \delta x_2 \rangle
  \end{equation}
  and, as we now explain, an error term bounded in the $\norm{\cdot}_{h,T_{j+1}(X,\varphi')}$-norm by
  \begin{equation}\label{e:neutral_err0-bound}
    C L^{-3} 2^{|X|_j} \norm{F(X)}_{h,T_j(X)}   h^{-2}(h + \norm{\nabla_{j+1} \varphi'}_{L^{\infty}(X^*)} + \norm{\nabla_{j+1}^2 \varphi'}_{L^{\infty} (X^*)})^2.
  \end{equation}
  Indeed, to obtain this error bound, we proceed as follows:
  observing that
  \begin{equation}
     \langle \nabla \varphi'(y_0), \delta x_1 \rangle \langle \nabla \varphi'(y_0), \delta x_2 \rangle - \avg{\nabla \varphi'(y_0), \delta x_1, \nabla \varphi'(x_0), \delta x_2} = \frac{1}{4} \sum_{\mu \in \hat{e}} \nabla^{\mu} \varphi'(y_0) \nabla^{(\mu, -\mu)} \varphi'(y_0) \delta x_1^{\mu} \, \delta x_2^{\mu},
  \end{equation}
  the claimed bound for the first replacement is justified by \eqref{eq:nfe3} and \eqref{eq:bound_D^2F(x_1,x_2)},
  whereas the claimed bound for the second replacement follows from \eqref{eq:nfe2} and \eqref{eq:bound_D^2F(x_1,x_2)}. 
The factor $2^{|X|_j}$ appearing in \eqref{e:neutral_err0-bound} follows hereby from an application of Proposition~\ref{prop:E_G_j}. Rather than including full details here, we refer to \eqref{e:neutral_err1-bound}-\eqref{e:neutral_err2-bound} below, which estimate a similar but slightly more involved error term, yielding the bound \eqref{neutral_err3-bound}. The bound \eqref{e:neutral_err0-bound} is readily obtained by adapting these arguments.
 
Next we replace $\avg{\nabla \varphi'(x_0),\delta x_i}$ in \eqref{e:Loc-contract-pf1} by $\delta\varphi'(x_i)=\varphi'(x_i)-\varphi'(x_0)$.
For $X\in \mathcal{S}_j$, one has
\begin{align}
\norm{\delta \varphi' }_{C^2_{j}(X^*)} , \,\, \norm{\langle \nabla \varphi'(x_0) , \delta x \rangle}_{C^2_{j}(X^*)}
  & \leq C L^{-1} \norm{\nabla_{j+1} \varphi'}_{L^{\infty}(X^*)} \label{eq:Loc_contract_dpdx1}
  \\
  \norm{\delta \varphi' (x) - \langle \nabla \varphi'(x_0) , \delta x \rangle}_{C^2_{j}(X^*)}
  &\leq C L^{-2} \norm{\nabla^2_{j+1} \varphi'}_{L^{\infty}(X^*)}  \label{eq:Loc_contract_dpdx2}
\end{align}
where the objects above are all functions of $x \in X^*$, measured in $\norm{\cdot}_{C^2_{j}}$-norm.
Using again the definition of the norm \eqref{eq:seminorm}, we may thus replace  \eqref{e:Loc-contract-pf1} by
\begin{equation}\label{eq:loc-pf-neutral-2ndorder}
  \frac{1}{2 |X|}
  \sum_{x_0 \in X}\sum_{x_1, x_2 \in X^*} \partial_{\varphi (x_1)} \partial_{\varphi (x_2)} \Eplus \hat{F}_0 (X,\zeta) (\varphi' (x_1) - \varphi' (x_0)) (\varphi' (x_2) - \varphi' (x_0))
\end{equation}
with an error in the $\|\cdot\|_{h,T_{j+1}(X,\varphi')}$-norm 
bounded by
\begin{equation}\label{neutral_err3-bound}
  CL^{-3} 2^{|X|_j}
  \|F(X)\|_{h,T_j(X)}  h^{-2}(h + \|\nabla_{j+1}\varphi'\|_{L^\infty(X^*)}+\|\nabla_{j+1}^2\varphi'\|_{L^\infty(X^*)})^2 . 
\end{equation}
Indeed,
\begin{align}\label{e:neutral_err1-bound}
&{ \Big |} \sum_{x_1, x_2 \in X^*} \partial_{\varphi (x_1)} \partial_{\varphi (x_2)} \Eplus \hat{F}_0 (X,\zeta) ( \delta \varphi' (x_1) - ( \nabla \varphi' (x_0), \delta x_1 )  ) \delta \varphi'(x_2) {\Big |} \nnb
& \leq \Eplus \norm{D^2 \hat{F}_0 (X, \zeta) }_{2, T_{j} (X, \zeta)} \norm{\delta \varphi' (x_1) - ( \nabla \varphi' (x_0), \delta x_1 ) }_{C^2_{j} (X^*)} \norm{\delta \varphi'}_{C_{j}^2 (X^*)} \nnb
& \leq C h^{-2} \norm{F (X)}_{h, T_{j} (X)} \Eplus[ G_j (X, \zeta) ] L^{-3} \norm{\nabla_{j+1} \varphi'}_{L^{\infty} (X^*)} \norm{\nabla_{j+1}^2 \varphi'}_{L^{\infty} (X^*)} \nnb
& \leq C L^{-3} 2^{|X|_j} \norm{F(X)}_{h, T_j{(X)}} h^{-2} \|\nabla_{j+1}\varphi'\|_{L^\infty(X^*)} \|\nabla_{j+1}^2\varphi'\|_{L^\infty(X^*)}
\end{align}
where $\delta \varphi' (x_1) - ( \nabla \varphi' (x_0), \delta x_1 )$ is a function of $x_1 \in X^*$,
using
\eqref{eq:linearity_of_expectation2}, \eqref{eq:Loc_contract_dpdx1} and \eqref{eq:Loc_contract_dpdx2} for the second inequality
and Proposition~\ref{prop:E_G_j} in the last step,  
and since each $\delta \varphi'(x_1) - (\nabla \varphi'(x_0), \delta x_1)$ and $\delta \varphi'(x_1)$ are linear in $\varphi'$, we immediately see { (see around \eqref{e:taylor-gradphi2} for a similar reasoning)} that
\begin{align}\label{e:neutral_err2-bound}
& \norm{\sum_{x_1, x_2 \in X^*} \partial_{\zeta (x_1)} \partial_{\zeta (x_2)} \Eplus \hat{F}_0 (X,\zeta) ( \delta \varphi' (x_1) - ( \nabla \varphi' (x_0), \delta x_1 )  )  \delta \varphi'(x_2)  }_{h, T_{j+1} (X, \varphi')}  \nnb
& \leq CL^{-3} 2^{|X|_{j}} \norm{F(X)}_{h, T_j{(X)}} h^{-2} ( h+ \|\nabla_{j+1} \varphi'\|_{L^\infty(X^*)} ) ( h + \|\nabla_{j+1}^2 \varphi'\|_{L^\infty(X^*)}) 
\end{align}
A similar bound holds for $\sum_{x_1, x_2 \in X^*} \frac{\partial^2 \Eplus \hat{F}_0 (X,\zeta)}{\partial \varphi (x_1) \partial \varphi (x_2)}  ( \delta \varphi' (x_1) - ( \nabla \varphi' (x_0), \delta x_1 )  ) (\nabla \varphi'(x_0), \delta x_2)$ and hence the claim follows.

Recognizing $\Eplus \hat{F}_0(X,\zeta)$ in \eqref{eq:loc-pf-neutral} together with \eqref{eq:loc-pf-neutral-2ndorder} as $\bar{\Tay}_2 \Eplus F(X)$
(the first order term in the Taylor expansion vanished due to the assumption $F(X, \varphi) = F(X, -\varphi)$)
and collecting the errors, 
we have thus overall shown 
\begin{multline}
  \|  \Loc_X \Eplus F(X, \varphi' + \zeta) - \bar{\Tay}_2   \Eplus F(X , \varphi' + \zeta) \|_{h,T_{j+1}(\bar{X},\varphi')} \nnb
  \leq
  C  L^{-3}
  \|F(X)\|_{h,T_{j} (X)}
  (1 + h^{-1} \max_{n=1,2}\norm{\nabla_{j+1}^n \varphi'}_{L^{\infty}(X^{*})})^2
  .
\end{multline}
The claim now follows from Lemma~\ref{lemma:bound_of_Gj_by_Gjplus1},  along with \eqref{e:norm-monot}, using that $h^{-2}\kappa_L^{-1}=O(\log L)$
which holds since $h^{-2} = O(\rho_J^{2})$ by our assumption $h \geq \rho_J^{-1}$
and since $\kappa_L^{-1} = O(\rho_J^{-2}\log L)$. 
\end{proof}

\begin{lemma} \label{lemma:gaussian_contraction}
Under the setting of Lemma~\ref{lemma:Loc_minus_Tay},  
\begin{equation}
  \norm{ 
    \bar{\Rem}_2 \Eplus F 
    ( {X} ,  \varphi' + \zeta)}_{h, T_{j+1} (\bar{X}, \varphi')}  \leq C 
  L^{-3}(\log L)^{3/2} { A^{-|X|_j} 
    \|F\|_{h,T_j}} 
    G_{j+1} (\bar{X}, \varphi').
\end{equation}
\end{lemma}
\begin{proof} Recall that $F(X,  \varphi' + \zeta) = F(X,  \delta \varphi' + \zeta)$ with $\delta \varphi' (x) = \varphi' (x) - \varphi' (x_0)$ for neutral $F$ and any $x_0$ by \eqref{e:chargedecompx0}. Thus, $\bar{\Rem}_2 \Eplus F (X, \varphi') = \frac{1}{|X|} \sum_{x_0 \in X} \Rem_2 \Eplus F (X, \delta \varphi')$ with $\delta \varphi' (x)$ defined for varying $x_0$'s, we just need to prove the statement for a fixed $x_0 \in X$ and $\bar{\Rem}_2$ replaced by $\Rem_2$.

We need to estimate $\norm{D^n \Rem_2 \Eplus F (X) }_{n, T_{j+1} ({ X}, \varphi')}$. We will consider the cases $n \geq 3$ and $0\leq n \leq 2$. Writing $\varphi = \varphi' + \zeta$, using that $\Rem_2 \Eplus F (X, \varphi)$ is neutral, the estimate for $n\geq 3$
follows simply from Lemma~\ref{lemma:neutral_term_under_rescaling} and the fact that $D^n \Rem_2=D^n$ for $n \geq 3$:
\begin{align} 
	\norm{D^n \Rem_2 \Eplus F (X,\varphi' + \zeta) }_{n, T_{j+1} ({ X}, \varphi')} \leq (C_g^{-1}L)^{-n} \norm{D^n \Eplus F (X,  \varphi' + \zeta)}_{n,T_j(X, \varphi')}
 	.
\end{align}
Multiplying by $h^n/n!$, summing over $n$, and combining with Lemma~\ref{lemma:linearity_of_expectation}, noting that $2^{|X|_{j}}  \leq C$ since $X$ is small, this readily yields
 \begin{equation}\label{eq:small-neutral-deriv-estim0}
\sum_{n\geq 3} \frac{h^n}{n!} \norm{D^n \Rem_2 \Eplus F (X,\varphi) }_{n, T_{j+1} (X, \varphi')} \leq C L^{-3}   A^{-|X|_j} 
    \|F\|_{h,T_j}
    G_{j+1} (\bar{X}, \varphi').
\end{equation}
The cases $n=0,1, 2$ require a bit of effort and represent in fact the dominant contributions.
We use Taylor's theorem and neutrality of $F$ to write
\begin{align}\label{eq:small-neutral-deriv-estim}
D^n \Rem_2 \Eplus F (X, \varphi) (f_1, \cdots, f_n) = \sum_{k=n}^2 \frac{1}{(k-n)!} D^k \Rem_2 \Eplus F (X, \zeta) (f_1, \cdots, f_n, (\delta \varphi')^{\otimes k-n}  ) \nnb
\quad +\int_0^1 dt \, \frac{(1-t)^{3-n}}{(2-n)!}  D^3 \Rem_2 \Eplus F (X, \zeta + t\varphi') (f_1, \cdots, f_n, (\delta \varphi')^{\otimes 3-n}) .
\end{align}
But since $D^k \operatorname{Rem}_2 \Eplus F (X, \zeta) = D^k \Eplus F (X, \zeta)$ for $k\geq 3$, 
applying successively \eqref{eq:nTj_norm_definition}, \eqref{eq:neutral_contract1} and \eqref{eq:linearity_of_expectation1.1}, one sees that 
\begin{align}
&|D^3 \operatorname{Rem}_2 \Eplus F (X, \zeta + t\varphi') (f_1, \cdots, f_n, (\delta \varphi')^{\otimes 3-n})|\nnb[0.5em]
& \qquad\qquad\qquad \leq \norm{D^3  \Eplus F (X, \zeta + t\varphi') }_{3, T_{j+1} ({X}, t\varphi')} \norm{\delta \varphi'}_{C_{j+1}^2(X^*)}^{3-n} \prod_{1\leq l \leq n} \norm{f_l}_{C^2_{j+1}(X^*)}\nnb
& \qquad\qquad\qquad \leq (C_g^{-1}L)^{-3} \norm{D^3  \Eplus F (X, \zeta + t\varphi') }_{3, T_{j} ({X}, t\varphi')} \norm{\delta \varphi'}_{C_{j+1}^2(X^*)}^{3-n} \prod_{1\leq l \leq n} \norm{f_l}_{C^2_{j+1}(X^*)}\nnb
&\qquad\qquad\qquad \leq C' h^{-3} L^{-3} \norm{F (X) }_{3, T_{j} (X)}  G_{j+1} (\bar{X}, t\varphi')  \norm{\delta \varphi'}_{C_{j+1}^2(X^*)}^{3-n} \prod_{1\leq l \leq n} \norm{f_l}_{C^2_{j+1}(X^*)}.
\end{align}
Moreover, since $D^k \operatorname{Rem}_2 \Eplus F (X, \zeta) = 0$ for $k\in \{0,1,2\}$, whenever $\norm{f_l}_{C^2_{j+1}(X^{*})} \leq 1$ for each $l \in \{1, \dots, n \}$, one obtains that 
the left-hand side of \eqref{eq:small-neutral-deriv-estim} is bounded in absolute value by
\begin{align}
(C_g^{-1}L h)^{-3} \int_0^1 dt \, \frac{(1-t)^{3-n}}{(2-n)!} \norm{F { (X)}}_{3, T_{j}{(X)}} 
 G_{j+1} (\bar{X}, t\varphi') \norm{\delta \varphi'}_{C_{j+1}^2(X^*)}^{3-n} .
\end{align}
Now {since $G_{j+1}(\bar{X}, \varphi')=G_{j+1}(\bar{X}, t\varphi')G_{j+1}(\bar{X}, \sqrt{1-t^2}\varphi')$ by definition of $G_{j+1}$ in \eqref{eq:def_large_field_regulator},}
and then using Lemma~\ref{lemma:bound_of_Gj_by_Gjplus1} applied with $\varphi =  \sqrt{1-t^2} \varphi'$, we obtain, for $n=0,1,2,$
\begin{align}
  \sup_{\varphi'} \frac{ G_{j+1} (\bar{X}, t\varphi' ) \norm{\delta \varphi'}_{C_{j+1}^2(X^*)}^{3-n} }{G_{j+1} (\bar{X}, \varphi')}
  = \sup_{\varphi'} \frac{\norm{\delta \varphi'}_{C_{j+1}^2(X^*)}^{3-n} }{G_{j+1} (\bar{X}, \sqrt{1-t^2} \varphi')}
  \leq O( \kappa_L^{-(3-n)/2} ) (1-t^2)^{-(3-n)/2}. \label{eq:G_j_with_t_bounded_by_G_j}
\end{align}
All in all, since 
\begin{equation}
  \int_0^1 (1-t)^{3-n} (1-t^2)^{-(3-n)/2} \, dt
  \leq C_n \int_0^1 t^{(3-n)/2} < \infty, \quad n \leq 2,
\end{equation}
this implies, for $n =0,1, 2$,
\begin{equation}
\norm{D^n \Rem_2 \Eplus F (X, \varphi  )}_{n, T_{j+1} (\bar{X}, \varphi')} \leq O( L^{-3} \kappa_L^{-(3-n)/2} ) h^{-3}  \norm{F (X)}_{3, T_j(X)} G_{j+1} (\bar{X}, \varphi') . 
\end{equation}
Multiplying by $h^n/n!$, summing over $n$, 
using that $\sum_{0\leq n \leq 2}  \frac{h^{n-3}}{n!} \kappa_L^{-(3-n)/2} \leq C(\log L)^{3/2} h^{-3}$ by assumption on $\kappa_L$ and $h$ (the dominant term being $n=0$), 
it follows that
\begin{equation} \label{eq:small-neutral-deriv-estim2}
	\sum_{ n=0,1,2}\frac{h^n}{n!} \norm{D^n \Rem_2 \Eplus F (X, \varphi  )}_{n, T_{j+1} (\bar{X}, \varphi')} \leq C L^{-3} (\log L)^{3/2} A^{-|X|_j} \norm{F}_{h, T_j} G_{j+1} (\bar{X}, \varphi') . 
\end{equation}

The claim follows immediately by combining the estimates \eqref{eq:small-neutral-deriv-estim0} (with \eqref{e:norm-monot}) and \eqref{eq:small-neutral-deriv-estim2}. 
\end{proof}

The next result will be used below to deduce~\eqref{e:Loc-bounded}.

\begin{lemma} \label{lemma:Loc-bounded}
Let $h$ and $\kappa_L$ be as in Proposition~\ref{prop:Loc-contract} and
$F$ be a neutral scale-$j$ polynomial activity, $B\in \cB_j$ and $X\in\cS_j$.
Then
\begin{equation} \label{e:Loc-bounded-bis}
  \norm{\Loc_{X, B} \Eplus F (X, \varphi' +\zeta)}_{h, T_j (X,\varphi')} \leq C (\log L) \norm{F(X)}_{h, T_j (X)} e^{c_w \kappa_L w_j (B, \varphi')^2}
  .
\end{equation}
\end{lemma}
    
\begin{proof}
Using \eqref{e:Loc-def} and \eqref{eq:new-bracket}, write
\begin{align}
\Loc_{X, B} \Eplus F(X) &= { \frac{1}{|X|_j} } \Eplus \hat{F}_0 (X, \zeta) \nnb
& \quad + \frac{1}{|X|} \sum_{x_0, y_0 \in B} \frac{1}{|B|} \sum_{x_1,x_2 \in X^*} \frac12 \partial_{\varphi (x_1 )}\partial_{\varphi (x_2) } \Eplus\hat{F}_0(X,\zeta) \avg{\nabla \varphi' (y_0), \delta x_1, \nabla \varphi' (y_0), \delta x_2} . \label{eq:LocXB_E_F_explicit}
\end{align}
The term in the first line is bounded using \eqref{eq:linearity_of_expectation1.1} with $\varphi'\equiv 0$ by
\begin{align}
|\Eplus \hat{F}_0 (X, \zeta)| \leq 2^{|X|_j} \norm{F(X)}_{h, T_j (X)} \leq C \norm{F(X)}_{h, T_j (X)},
\end{align}
since $X$ is small. We now consider the term in the second line of \eqref{eq:LocXB_E_F_explicit}.
For $\mu, \nu \in \hat{e}$, Lemma~\ref{lemma:bound_D^2F(x_1,x_2)} provides a bound for
$\sum_{x_1, x_2 \in X^*} \partial_{\varphi (x_1 )}\partial_{\varphi (x_2) } \Eplus \hat{F}_0(X,\zeta) (\delta x_1)^{\mu} (\delta x_2)^{\nu}$.
Moreover since $y_0 \in B$,
\begin{align}
L^j \norm{\nabla^{\mu} \varphi' (y_0) 1_{y_0 \in B} }_{h, T_j (X, \varphi')} \leq h + \norm{\nabla_j \varphi'}_{L^{\infty} (B)}.
\end{align}
Putting these together, using the submultiplicativity of the norm and recalling the definition of $w_j$ from \eqref{eq:wj_def}, the $\norm{\cdot}_{h, T_j (X, \varphi)}$-norm of the second term of \eqref{eq:LocXB_E_F_explicit} is readily seen to be bounded by
\begin{equation}
C h^{-2} (h + \norm{\nabla_j \varphi'}^2_{L^{\infty}(B)} )^2 \norm{F}_{h, T_j (X)} \leq C' h^{-2} \kappa_L^{-1} \norm{F (X)}_{h, T_j (X)} e^{c_w \kappa_L w_j (B, \varphi')^2} .
\end{equation}
The claim again follows from the fact that $h^{-2}{\kappa_L}^{-1}=O(\log L)$.
\end{proof}

\begin{proof}[Proof of Proposition~\ref{prop:Loc-contract}]

The Fourier decomposition \eqref{def:chargedecomp} yields, since $\Loc_X \Eplus \hat{F}_q (X)=0$ whenever $q\neq 0$ (cf.~\eqref{e:Loc-def}),
\begin{align}
\Loc_X \Eplus F(X) - \Eplus F(X) = \Loc_X \Eplus \hat{F}_0 (X) - \Eplus \hat{F}_0(X) - \sum_{q
{\neq 0}} \Eplus \hat{F}_{q} (X),
\end{align}
which allows to prove \eqref{e:Loc-contract-full} by bounding the terms of different charge separately.
The last sum is the charged part of $F$ and was already bounded at the end of Section~\ref{sec:pf-Loc-charged}.
In order to bound $\norm{ \Loc_X \Eplus \hat{F}_0 (X, \varphi' +\zeta) - \Eplus \hat{F}_0(X, \varphi' +\zeta)}_{h,T_{j+1}({X}, \varphi')}$, 
one applies Lemmas~\ref{lemma:Loc_minus_Tay} and~\ref{lemma:gaussian_contraction} (with the choice $F=\hat{F}_0$), which yield suitable estimates for $ \norm{ \Loc_X \Eplus \hat{F}_0 (X, \varphi'+\zeta) - \bar{\Tay}_2   \Eplus \hat{F}_0 (X, \varphi'+\zeta) }_{h,T_{j+1}({X}, \varphi')}$ and  $\norm{ (1- \bar{\textnormal{Tay}}_2 )
  \Eplus \hat{F}_0
  ( {X} ,  \varphi' + \zeta)}_{h, T_{j+1} ({X}, \varphi')}$, respectively, from which \eqref{e:Loc-contract-full} readily follows.
The bound \eqref{e:Loc-bounded} is a direct result of Lemma~\ref{lemma:Loc-bounded}.
Finally, the continuity in $s$ again follows from Lemma~\ref{lemma:stability_of_expectation_singleX},
similarly as in the proof of Proposition~\ref{prop:Loc-coupling}.
\end{proof}

\subsection{Reblocking}
\label{sec:reblocking}

The final contraction mechanism states that the contribution to a polymer activity
from large sets contracts under so-called reblocking for entropic reasons.

\begin{definition} \label{def:reblocking}
For a scale-$j$ polymer activity $F$, define the reblocking operator
\begin{equation}
  \mathbb{S} F (X, \varphi) := \sum_{Y \in \mathcal{P}_j^c}^{\bar{Y}=X} F(Y, \varphi)
  \;\;\;\; \text{for} \;\;  X\in \cP_{j+1}^c.
  \label{eq:reblocking_operator_definition}
\end{equation}
\end{definition}

Note that $\mathbb{S}$ is a linear map taking a scale-$j$ polymer activity $F$ to a scale-$(j+1)$ polymer activity
since all polymers are connected.
The following proposition shows that reblocking is contracting when acting on
polymer activities $F$ supported on large sets, i.e., {$F(X)=0$ holds for all $X\in\cS_j$.}
Contrary to the previous mechanisms, this does not use periodicity of $F$ nor any structure of the $F(X)$.

\begin{proposition} 
\label{prop:largeset_contraction-v2}
There exists a geometric constant $\eta >0$ such that the following holds
{when $L\geq 2^{d}+1=5$.}
For $A^{\frac{\eta}{2}} \geq 2^{\frac{2+\eta}{2}} e L^3 $, $X\in \cP_{j+1}^{c}$,
and any scale-$j$ polymer activity $F$ with $\|F\|_{h,T_j}<\infty$,
\begin{equation}
  \norm{\mathbb{S} ( \Eplus [ F 1_{Y \not\in \cS_j} ] ) (X) }_{h,  T_{j+1}(X)} \leq  (L^{-1} A^{-1})^{|X|_{j+1}} \norm{F}_{h, T_j}.
\end{equation}
\end{proposition}

The factor $L^{-1}$ will compensate the loss of the factor of $2$ in the $A/2$ factor in Lemma~\ref{lemma:linearity_of_expectation}.
The proof is a consequence of the following combinatorial lemma.
%

\begin{lemma}[Lemmas 6.14--15 of \cite{MR2523458}] \label{lemma:setsizes} 
There exists a geometric constant $\eta > 0$ such that the following holds when $L \geq 2^{d}+1=5$.
For every $X\in \mathcal{P}_j$,
\begin{equation} 
(1+ \eta) |\bar{X}|_{j+1} \leq |X|_j + 8 (1+ \eta)|\operatorname{Comp}_j(X)|. \label{eq:setsizes}
\end{equation}
 Moreover, if $X$ is connected but not a small set, then
\begin{equation}
(1+ \eta) |\bar{X}|_{j+1} \leq |X|_{j} . \label{eq:setsizes_largeconnected}
\end{equation}
\end{lemma}

\begin{proof}[Proof of Proposition~\ref{prop:largeset_contraction-v2}]
By \eqref{eq:setsizes_largeconnected}, we have $|Y|_j \geq (1+\eta) |X|_{j+1}$ if $\bar{Y} = X$ and $Y\in \cP_j^c \backslash \cS_j$, and so applying successively \eqref{eq:reblocking_operator_definition}, \eqref{eq:NORm}-\eqref{e:norm-monot} and \eqref{eq:E_G_j}, one obtains 
\begin{align}
&\norm{ \mathbb{S} \Eplus [F 1_{Y \not\in \cS_j} ] (X, \varphi') }_{h, T_{j+1}(X, \varphi')} \nnb
& \leq \sum_{Y\in \cP_j^c \backslash \cS_j}^{\bar{Y} = X} A^{-|Y|_j} \Eplus[G_j (Y, \varphi' + \zeta)] \norm{F}_{h,T_j} \nnb
& \leq \sum_{Y\in \cP_j^c \backslash \cS_j}^{\bar{Y} = X} (A/2)^{-|Y|_j /2} (A/2)^{-(1+\eta) |X|_{j+1} /2} G_{j+1} (X, \varphi') \norm{F}_{h,T_j}, \label{eq:reblock-bd1}
\end{align}
using the lower bound on $|Y|_j$ in the last step. { Next, observe that for any $z >0$, decomposing a polymer $Y \in \cP_j$ with $\bar{Y} = X$ over $(j+1)-$blocks constituting $X$, one can rewrite
\begin{align}
\label{eq:reblock-bd2}
\sum_{Y\in \cP_j}^{\bar{Y} = X} z^{|Y|_j} =  \prod_{B\in \mathcal{B}_{j+1} (X)} \,  \sum_{Y' \in \mathcal{P}_j}^{\overline{Y'} = B} z^{|Y'|_j}= \big((1+z)^{L^2} - 1 \big)^{|X|_{j+1}}.
\end{align}
Returning to \eqref{eq:reblock-bd1}, using \eqref{eq:reblock-bd2} with the choice $z= (A/2)^{-1/2}$, one obtains that the quantity $\norm{ \mathbb{S} \Eplus [F 1_{Y \not\in \cS_j} ] (X, \varphi') }_{h, T_{j+1}(X, \varphi')}$ is bounded by}
\begin{align}
&  (A/2)^{-\frac{1+\eta}{2}  |X|_{j+1}} \big( (1+(A/2)^{-\frac{1}{2}} )^{L^2} -1 \big)^{|X|_{j+1}} G_{j+1} (X, \varphi') \norm{F}_{h,T_j} \nnb 
& \leq (A/2)^{-\frac{1+\eta}{2}|X|_{j+1}} \big( e (A/2)^{-\frac{1}{2}} L^2  \big)^{|X|_{j+1}}
G_{j+1} (X, \varphi') \norm{F}_{h,T_j} 
\end{align}
under the assumption $(A/2)^{-\frac{1}{2}} L^2 \leq 1$ 
where we use $(1 + b)^c - 1  \leq \exp(bc) -1 \leq ebc$ for any $b,c \geq 0$, $bc\leq 1$ to obtain the last inequality.
If we assume further that $A$ is large enough so that $e(A/2)^{-(2+\eta)/2 } L^2 \leq L^{-1} A^{-1}$, 
 then this is bounded by
\begin{equation}
   (e (A/2)^{-\frac{2+\eta}{2}}L^2)^{|X|_{j+1}}
  G_{j+1} (X, \varphi') \norm{F}_{h,T_{j}} \leq (L^{-1} A^{-1})^{|X|_{j+1}} G_{j+1} (X, \varphi') \norm{F}_{h,T_j} ,
\end{equation}
giving the desired bound.
\end{proof}

We also have the following lemma which is of a slightly different flavour, and has its use in various places related to large sets.

\begin{lemma} \label{lemma:setsizes_2}
Let $X\in \mathcal{P}_{j+1}$, $0\leq x \leq \epsilon_{rb} = A^{-16}$, $\eta$ be as in Lemma~\ref{lemma:setsizes} and $L^2 A^{-\eta/(1+\eta)} \leq 1$. Then 
\begin{align}
& \sum_{Y\in \mathcal{P}_j : | \operatorname{Comp}_j (Y)|=1}^{\bar{Y}=X} 1_{Y\not\in \mathcal{S}_j}  x A^{-|Y|_j}  \leq (e L^2 A^{-(1+2\eta)/(1+\eta)} )^{|X|_{j+1}} x
\end{align}
and
\begin{align}
& \sum_{Y\in \mathcal{P}_j : | \operatorname{Comp}_j (Y)| \geq 2}^{\bar{Y}=X} 1_{Y\not\in \mathcal{S}_j} x^{| \operatorname{Comp}_j (Y)|} A^{-|Y|_j}  \leq A^{16} (e L^2 A^{-(1+2\eta)/(1+\eta)} )^{|X|_{j+1}} x^2 .
\end{align}
\end{lemma}

\begin{proof}
For the first estimate, for $x \leq 1$, \eqref{eq:setsizes_largeconnected} 
  implies $|Y|_{j} = \frac{1}{1+ \eta} |Y|_j + \frac{\eta}{1+ \eta} |Y|_j \geq |X|_{j+1} +  \frac{\eta}{1+ \eta} |Y|_j$ so that
\begin{equation}
\sum_{Y\in \mathcal{P}_j : | \operatorname{Comp}_j (Y)|=1}^{\bar{Y} = X} 1_{Y\not\in \mathcal{S}_j} x^{| \operatorname{Comp}_j (Y)|} A^{-|Y|_j} 
\leq A^{- |X|_{j+1}}  \sum_{Y\in \mathcal{P}^c_j}^{\bar{Y}=X}  A^{- \frac{\eta}{1+\eta} |Y|_j} x \label{eq:SF_estimate1}
\end{equation}
For the second estimate, observe that \eqref{eq:setsizes} implies $|Y|_{j} = \frac{1}{1+ \eta} |Y|_j + \frac{\eta}{1+ \eta} |Y|_j \geq |X|_{j+1} - 8 |\operatorname{Comp}_j (Y)| +  \frac{\eta}{1+ \eta} |Y|_j$ so that
\begin{align}
\sum_{Y\in \mathcal{P}_j : | \operatorname{Comp}_j (Y)| \geq 2}^{\bar{Y} = X} x^{|\operatorname{Comp}_j(Y)|} A^{-|Y|_j} 
& \leq \sum^{\bar{Y} = X}_{Y: |\operatorname{Comp}_j (Y)| \geq 2} A^{-\frac{\eta}{1+\eta} |Y|_{j}} A^{-|X|_{j+1} + 8 |\operatorname{Comp}_j (Y)|} x^{|\operatorname{Comp}_j (Y)|}  \nnb
& \leq A^{16} A^{-|X|_{j+1}} \sum_{Y\in \mathcal{P}_{j}}^{\bar{Y} = X} A^{-\frac{\eta}{1+\eta} |Y|_{j}} x^{2}. \label{eq:SF_estimate2}
\end{align}
where the final line follows under the assumption $x \leq \epsilon_{rb} = A^{-16}$.
Now \eqref{eq:reblock-bd2} implies
\begin{align}
 A^{-|X|_{j+1}} \sum_{Y\in \mathcal{P}_{j}}^{\bar{Y} = X} A^{-\frac{\eta}{1+ \eta}|Y|_j} 
& = A^{-|X|_{j+1}} \Big[ (1+A^{-\eta/(1+\eta)})^{L^2} - 1 \Big]^{|X|_{j+1}} \nnb 
& \leq A^{-|X|_{j+1}} \Big( e^{A^{-\eta/(1+\eta)} L^2} - 1 \Big)^{|X|_{j+1}} ,
\end{align}
If $A$ is chosen so that $A^{-\eta/(1+\eta)} L^2 \leq 1$, then this can be bounded by
\begin{equation}
( e L^2 A^{-(1+2\eta) /(1+\eta)})^{|X|_{j+1}} ,
\end{equation}
completing the proof of the second estimate.
\end{proof}

\section{The renormalisation group map}
\label{sec:rg_generic_step}

The present section is at the heart of the argument. We define a suitable renormalisation group map $\Phi_{j+1}$ from scale $j$ to scale $j+1$, which corresponds to integrating out the covariance $\Gamma_{j+1}$, and exhibit in Theorems~\ref{thm:general_RG_step_consistent}--\ref{thm:local_part_of_K_j+1} its key algebraic and analytical properties. These are the only features which will be needed in the sequel and, roughly speaking, will allow to perform a suitable fixed-point argument in the next section. The map $\Phi_{j+1}$ has two components, one describing the evolution of coupling constants, and one describing that of the remainder coordinate. The latter is an evolution on polymer activities, whose growth will be controlled in terms of the norms introduced in Section~\ref{sec:norms}. The estimates corresponding to these two components appear separately in Theorems~\ref{thm:H_j_E_j_estimate} and~\ref{thm:local_part_of_K_j+1}. The actual definition of the remainder coordinate (Definition~\ref{def:evolution_of_remainder}) involves the localisation operator introduced in Section~\ref{sec:Loc}, which is used to extract the relevant terms.
The most involved part, which occupies most of this section, is to obtain the relevant bounds for the resulting remainder coordinate, and in particular for its non-linear part, cf.~\eqref{eq:bound_for_N_j_K_j}--\eqref{eq:bound_for_derivative_of_Nj} below.

Our study of the Discrete Gaussian model proceeds through its mass regularised version \eqref{eq:DG-withmass};
the original model will then be recovered by applying Lemma~\ref{lemma:m2to0} to take $m^2\to 0$ at the end of the analysis (in Section~\ref{sec:integration_of_zero_mode}).
The starting point for the renormalisation group analysis of the mass regularised version of model is the reformulation
in Lemma~\ref{lemma:reformulation}, which involves the Gaussian measure
with covariance $C(s, m^2)$.
The renormalisation group is defined in terms of the finite-range decomposition
of this covariance defined in Section~\ref{sec:scales} (see~\eqref{eq:frd_of_C^Lambda_N}): 
\begin{equation}
  C (s,m^2) = \sum_{j=1}^{N-1} \Gamma_j(s,m^2) +  \Gamma^{\Lambda_N}_{N}(s,m^2) + t_N (s, m^2) Q_N.
\end{equation}
Since we will eventually take the limit $m^2\to 0$,
by Lemma~\ref{lemma:m2to0_with_frd},
we may actually directly set $m^2=0$
in the covariances $(\Gamma_{j+1} (s, m^2) : 0 \leq j \leq N-2 )$ and $\Gamma_{N}^{\Lambda_N} (s, m^2)$.
We will do this and thus replace $C (s,m^2)$ by
\begin{equation} \label{eq:barC^Lambda_N}
  \bar C (s,m^2)=
  \sum_{j=1}^{N-1} \Gamma_j(s) + \Gamma^{\Lambda_N}_{N}(s) + t_N (s, m^2) Q_N
\end{equation}
where $\Gamma_j(s) = \Gamma_j(s,0)$ and $\Gamma_N^{\Lambda_N} (s) = \Gamma_N^{\Lambda_N} (s,0)$.
The parameter $s$ is arbitrary in Lemma~\ref{lemma:reformulation} (provided $|s| \leq \epsilon_s \theta_J$).
A careful choice will be necessary in the analysis of the stable manifold of
the renormalisation group map (in Section~\ref{sec:stable_manifold_theorem}), but in the present section
the parameter does not play an important role.
We will therefore usually leave the $s$-dependence implicit in our notation.
Thus all definitions in this section do implicitly depend on $s$,
but all estimates will be uniform in $|s| \leq \epsilon_s\theta_J$.
Thoughout this section, the distribution $J$ is allowed to be any finite-range step distribution that is invariant under
lattice symmetries (cf.~above \eqref{eq:Delta_J_definition}) and we assume \eqref{eq:frd_ulbds}, which is no loss of generality.

\subsection{Coordinates for the renormalisation group map}

The initial condition for the renormalisation group map is the interaction function
$Z_0(\varphi|\Lambda_N )$. This function will eventually be chosen as in \eqref{eq:Z_0_definition}
with $s_0=s$ and $s$ chosen carefully, but
we allow it to be more general for the moment.
Given such a function $Z_0(\varphi|\Lambda_N )$, the renormalisation group map parametrises the successive integration
\begin{equation}
  Z_{j+1}(\varphi | \Lambda_N  ) = \Eplus Z_j(\varphi+ \zeta | \Lambda_N  ),  \qquad (j<N-1, \,  \varphi \in \R^{\Lambda_N }
),
  \label{eq:general_RG_step1}
\end{equation}
(recall that $\Eplus $ integrates the Gaussian field $\zeta$ with covariance $\Gamma_{j+1}= \Gamma_{j+1}(s)$)
as
\begin{equation} 
  Z_j (\varphi | \Lambda_N  ) =
  e^{-E_{j} |\Lambda_N |}
  \sum_{X\in \mathcal{P}_j (\Lambda_N )} e^{U_j ( \Lambda \backslash X, \varphi)} K_j (X,\varphi). \label{eq:general_RG_step2}
\end{equation}
A careful inductive choice of $E_j$, $U_j$ and $K_j$ for the representation \eqref{eq:general_RG_step2} will later constitute the renormalisation group flow.
For the remainder of this section, we merely specify general conditions that we impose on the form of $U_j$ and $K_j$ and how to measure their size.
The coordinate $U_j$ is an explicit leading part that is defined in terms of coupling constants $(s_j,z_j)$ as follows.

\begin{definition}\label{def:U_space}
  The coordinate $U_j$ is parametrised in terms of the coupling constants $(s_j,z_j)$
  where $s\in \R$ and $z_j=(z_j^{(q)})_{q\geq 1}$ is itself a sequence of real coupling constants as
\begin{equation}
  \begin{split}
U_j ( X,\varphi) &= \frac{1}{2} s_j |\nabla \varphi|_X^2 + W_j (  X, \varphi) \\
W_j( X, \varphi)& = \sum_{x\in X} \sum_{q\geq 1} L^{-2j} z^{(q)}_j \cos( \beta^{1/2} q \varphi(x)), \label{eq:general_RG_step3}
\end{split}
\end{equation}
where we recall the notation $|\nabla \varphi|_X^2$ from \eqref{eq:nablaf^2_definition}.
We will always identify $U_j$ with the coupling constants $(s_j,z_j)$ and use the norm
\begin{equation} \label{eq:U_norm}
	\|U_j\|_{\Omega_j^U} = 
	A
   \max\ha{ |s_j|, \, \sup_{q \geq 1} e^{c_f\beta q}|z_j^{(q)}| },
  \qquad c_f = \frac14 \gamma,
\end{equation}
where the constant $\gamma$ is the one from Proposition~\ref{prop:decomp}.
Let $\Omega_j^U$ be the {Banach} space of such $U_j$ (with finite $\|\cdot\|_{\Omega_j^U} $-norm). 
\end{definition}
In particular, note for later purposes that $ \|W_j\|_{j} \bydef \|W_j\|_{\Omega_j^U}$ is also defined by \eqref{eq:U_norm} and corresponds to $s_j=0$.
The quantity $K_j$ is a remainder coordinate on whose form we only impose
the following generic conditions. Note that this includes in particular the important component factorisation property \eqref{eq:factorization}
which is implied by Definition~\ref{def:polymeractivity}.

\begin{definition} \label{def:K_space}
  The coordinate $K_j$ is a polymer activity (see Definition~\ref{def:polymeractivity}),
  satisfying 
  the periodicity condition $K_j ( X, \varphi) = K_j ( X, \varphi + 2\pi \beta^{-1/2} n)$ for any $n \in \mathbb{Z}$,
invariance under the lattice symmetries and evenness (see Definition~\ref{def:latticesym}).
For such polymer activitives $K_j$ we use the norm \eqref{eq:NORM}, i.e.,
\begin{equation}\label{eq:omega-K-norm_finitevol}
	\|K_j\|_{\Omega_j^K} = \|K_j\|_{h,T_j},
\end{equation}
with
\begin{equation} \label{eq:h_def}
  h = \max\{ c_f^{1/2}, r c_h \rho_J^{-2} \sqrt{\beta}, \rho_J^{-1} \},
\end{equation}
where $r\in (0,1]$, $c_f$ is as in \eqref{eq:U_norm} and $c_h$ is chosen by \eqref{eq:choice_of_c_h}.
Let 
$\Omega_j^K$ 
be the Banach space of polymer activies $K_j$ with finite 
$\norm{\cdot}_{h, T_j}$-norm. 
\end{definition}

Finally, we define the norm on the product space of $(U_j,K_j)$ as follows.
\begin{definition}\label{eq:def-Omega-j}
Let $\Omega_j = \Omega_j^U \times \Omega_j^K$, i.e.,
\begin{equation}
  \Omega_j = \{ \omega_j = (U_j, K_j) : \norm{\omega_j}_{\Omega_j} < + \infty \},
  \qquad
  \norm{\omega_j}_{\Omega_j} = \max\{ \norm{U_j}_{\Omega_j^U}, \, \norm{K_j}_{\Omega_j^K} \} .
\end{equation}
\end{definition}
Ultimately we will choose
  $W_0(X,\varphi)=\sum_{x\in X}\tilde U(\varphi_x)$ with $\tilde U$ as in \eqref{e:tildeF},
  i.e., with $z_0^{(q)} = \tilde{z}^{(q)}$ as in Lemma~\ref{lemma:Fourier_repn_of_V}.
Then Lemma~\ref{lemma:Fourier_repn_of_V} implies
$\norm{W_0}_{\Omega_0^U} \leq A e^{-\frac{1}{4}\gamma\beta}$ for $c_f = \frac{1}{4} \gamma$.

We close this section with the following lemma which shows that $\|W_j(B)\|_{h,T_j(B)}$
is bounded in terms of the $\|W_j\|_{\Omega_j^U}$ norm.


\begin{lemma}
  \label{lem:W_norm}
  Let $h$ be as in \eqref{eq:h_def}, and
  assume $\beta \geq 2  \max\{ c_f^{-2},  c_f^{-1} \}$ 
  and $\rho_J^2 \geq \sqrt{2} r c_h c_f^{-1}$.
  Then for any $B \in \mathcal{B}_j$,
    \begin{equation}
    \norm{W_{j}(B, \varphi)}_{h, T_j (B, \varphi) }
    \leq
    C
	A^{-1}
    \norm{W_{j}}_{\Omega_j^U}. \label{eq:T_norm_cf_norm_comparison}
  \end{equation}
\end{lemma}

\begin{proof}
By \eqref{eq:general_RG_step3}, \eqref{e:normbd-eiphi} and the triangle inequality 
it follows that for $\beta \geq 2h^2 c_f^{-2}$ and $h\geq c_f^{1/2}$, 
\begin{align}
  \norm{W_{ j}(B, \varphi)}_{h, T_j (B, \varphi) }
  &\leq  
  2 { A^{-1}}      
  \sum_{q\geq 1} \norm{e^{i\sqrt{\beta} q \varphi (x_0)}}_{h, T_{j} (B, \varphi)} |z_j^{(q)}|\nnb
  &\leq 
  2 { A^{-1}}      
  \sum_{q\geq 1} e^{-q(c_f \beta -\sqrt{\beta} h)} \norm{W_{j}}_{\Omega_j^U} 
  \leq 
  C { A^{-1}}    
  \norm{W_{j}}_{\Omega_j^U} 
\end{align}
for any $B\in \mathcal{B}_j$.
With $h$ as in \eqref{eq:h_def},  both conditions hold when $\beta \geq 2 \max\{ c_f^{-2},  c_f^{-1} \}$ and $\rho_J^2 \geq \sqrt{2} r c_h c_f^{-1}$.
Indeed, when $h = c_f^{1/2}$, then $2h^2 c_f^{-2} = 2 c_f^{-1} \leq \beta$. 
When $h = r c_h \rho_J^{-2} \sqrt{\beta}$, then $2h^2 c_f^{-2} =  \beta (2r c_h^2 \rho_J^{-4} c_f^{-2}) \leq \beta$ by the second condition. 
When $h = \rho_J^{-1}$, then $2h^2 c_f^{-2} = 2\rho_J^{-2} c_f^{-2} \leq 2c_f^{-2} \leq \beta$. This completes the proof. 
\end{proof}

Since $c_h$ and $c_f$ are  absolute constants, the conditions on $\beta$ and $\rho_J$ appearing in Lemma~\ref{lem:W_norm} can be achieved either by taking $r$ small enough with $\rho_J$ fixed or $\rho_J$ large enough with $r=1$. Note that by Proposition~\ref{prop:Loc-contract} (observe that all of its assumptions hold) and the discussion below its statement, in particular the second term in the definition of $\alpha_{\Loc}$ in \eqref{e:Loc-contract-kappa}
indicates that the price to pay for having $r$ small is to take $\beta$ sufficiently large so that $e^{-\frac{1}{2} r \beta \Gamma_{j+1} (0) } < L^{-2}$ (which we will later need).
 We will eventually impose one of these choices of parameters;
  this choice occurs in the proof of Corollary~\ref{cor:tuning_s}.

\subsection{Estimates for the renormalisation group map} \label{sec:thms-RGmap}

There are many choices of maps that act on the renormalisation group coordinates
$(E_j,s_j,z_j,K_j) \mapsto (E_{j+1},s_{j+1},z_{j+1}, K_{j+1})$
such that \eqref{eq:general_RG_step1}--\eqref{eq:general_RG_step3} hold.
The renormalisation group map corresponds to a
careful choice in which the remainder coordinates $K_j$ contract from scale to scale
in an appropriate sense
(i.e., are \emph{irrelevant}),
while the evolution of the coordinates $U_j$ can be analysed explicitly.
Such a choice of the renormalisation group map
\begin{equation}
  \Phi_{j+1} : (E_j,s_j,z_j,K_j) \mapsto (E_{j+1},s_{j+1},z_{j+1}, K_{j+1}) \label{eq:Phi_j+1_definition}
\end{equation}
is explicitly given in Definitions~\ref{def:evolution_of_U}--\ref{def:evolution_of_remainder} below.
Note that, throughout Section~\ref{sec:rg_generic_step}, $ \Phi_{j} $ depends implicitly on $\Lambda_N$ and $0\leq j < N-1$ (but see Section~\ref{sec:infvol}, in particular Proposition~\ref{prop:inf_vol_RG}, for its infinite-volume extension).
The precise choice of the definition of $\Phi_{j+1}$ is not significant for later sections,
however, save for certain key properties that follow from this definition, which we gather in the next three theorems.
Any definition that implies these properties would have been equally good.

We briefly set up some convenient notation. In what follows,
we either denote the components of the map $\Phi_{j+1}$
by $(E_j + \mathcal{E}_{j+1},\mathfrak{s}_{j+1},\mathfrak{z}_{j+1},\mathcal{K}_{j+1})$
or by $(E_j+\mathcal{E}_{j+1},\mathcal{U}_{j+1},\mathcal{K}_{j+1})$
where $\mathcal{U}_{j+1}=(\mathfrak{s}_{j+1},\mathfrak{z}_{j+1})$.
Note that the coupling constant $E_j$ contributes to \eqref{eq:general_RG_step2} only by a $\varphi$-independent factor, 
and therefore its influence on \eqref{eq:general_RG_step1} is trivial. 
As indicated above, we will thus assume that $E_j=0$ is assumed in the
definition of $\Phi_{j+1}$, and that the definition is then extended to general $E_j$ by setting
$\mathcal{E}_{j+1}(E_j,s_j,z_j,K_j)=\mathcal{E}_{j+1}(0,s_j,z_j,K_j)$,
$\mathfrak{s}_{j+1}(E_j,s_j,z_j,K_j)=\mathfrak{s}_{j+1}(0,s_j,z_j,K_j)$,
and analogously for the other components.
To emphasise the dependence on $\Lambda_N$,
we will sometimes write $\Phi_{j+1}^{\Lambda_N}$ and $\cK_{j+1}^{\Lambda_N}$
instead of $\Phi_{j+1}$ and $\cK_{j+1}$.
{Whenever we write only a subset of the arguments $(E_j,s_j,z_j,K_j)$ below, we implicitly mean that the given map is a function of these arguments alone. For instance, $\mathfrak{s}_{j+1}(s_j,K_j)$ means that $\mathfrak{s}_{j+1}$ is a function of $(s_j,K_j)$.}

The following three theorems refer to the map $\Phi_{j+1}$ introduced below in Definitions~\ref{def:evolution_of_U}--\ref{def:evolution_of_remainder} and exhibit its salient features.
We start with the algebraic property of the renormalisation group map.

\begin{theorem}[Algebraic properties] \label{thm:general_RG_step_consistent}
  The renormalisation group map $\Phi_{j+1}$ is consistent with \eqref{eq:general_RG_step1}--\eqref{eq:general_RG_step2},
  i.e., if $Z_j$ has the form \eqref{eq:general_RG_step2} at scale $j$ with parameters $(E_j,s_j,z_j,K_j)$
  then $Z_{j+1}$ defined by \eqref{eq:general_RG_step1} has this form at scale $j+1$ with
  $(E_{j+1},s_{j+1},z_{j+1},K_{j+1})= \Phi_{j+1}(E_j,s_j,z_j,K_j)$. 
 Moreover, if $K_j$ is a scale-$j$ polymer activity
 (see Definition~\ref{def:polymeractivity})
 that is even,  invariant under lattice symmetries (see Definition~\ref{def:latticesym})
 and satisfies the periodicity condition $K_{j} ( X, \varphi) = K_{j} ( X, \varphi + 2\pi \beta^{-1/2} n)$ for any $n \in \mathbb{Z}$,
 then $K_{j+1}$ is a scale-($j+1$) polymer activity with the same properties.
\end{theorem}

Next we state the simple estimates for the $U$-component of the renormalisation group map.

\begin{theorem}[Estimate for coupling constants] 
  \label{thm:H_j_E_j_estimate}
For any choice of $L >1$, $A>1$, and $h>0$,
one has $\mathfrak{z}_{j+1}^{(q)}(z_j) = L^2 e^{-\frac12 \beta q^2\Gamma_{j+1}(0)}z_j^{(q)}$ for all $q\geq 1$ and
the following estimates hold:
\begin{align}
  |\mathfrak{s}_{j+1}(s_j, K_j)-s_j|
  &\leq C h^{-2} A^{-1} \norm{K_j}_{\Omega_j^K} \label{eq:H_j+1_bound} \\
 |\mathcal{E}_{j+1}(s_j,K_j) + s_j \nabla^{(e_1, -e_1)} \Gamma_{j+1} (0) |
  &\leq C L^{-2j} A^{-1} \norm{K_j}_{\Omega_j^K} .	
  \label{eq:E_j+1_bound}
\end{align}
Moreover,  all maps above are continuous in the implicit parameter $s$ for fixed $(s_j, z_j, K_j) \in \Omega_j$.
\end{theorem}

The final theorem concerns the evolution of the remainder coordinate and shows that it contracts.
Stating it requires a suitable notion of derivative.
Let $\bbX$ and $\bbY$ be Banach spaces with norms $\norm{\cdot}_{\bbX}$ and $\norm{\cdot}_{\bbY}$, and let $F: \bbX \to \bbY$.
The directional derivative of $F$ at a point $x \in \bbX$, in direction $\dot{x}$ is denoted by $DF (x, \dot{x})$, i.e., when the limit exists,
\begin{equation}\label{eq:Fre-deriv}
D F (x, \dot{x}) = \lim_{t\rightarrow 0} \frac{1}{t} (F (x + t\dot{x}) - F (x)) ,
\end{equation}
and if $F$ is Fr\'echet-differentiable the norm of the derivative is the operator norm
\begin{equation}
\norm{D F (x, \cdot)} := \sup\{  \norm{D F (x, \dot{x})}_{\bbY} \, : \, \norm{\dot{x}}_{\bbX}  \leq 1  \}.
\end{equation}
We also say that a family of maps $F_N: I \times D_N \to \bbY_N$, where $I$ is an interval,
the $D_N \subset \bbX_N$ are domains in a normed space $\bbX_N$, and $\bbY_N$ are normed spaces,
is equicontinuous in the first variable if for every $\epsilon>0$ there exists $\delta>0$ such that
$\norm{F_N(s_1,x)-F_N(s_2,x)}_{\bbY_N} <\epsilon$ for all $N$, any $s_1,s_2 \in I$ with $|s_1-s_2|< \delta$ and $x \in D_N$.
Note that all of $\Omega_j, \Omega_j^U, \Omega_j^K$ implicitly depend on the underlying torus $\Lambda_N$. We usually keep this dependence implicit in our notation; emphasise the dependence we will write $\Omega_j \equiv \Omega_j^{\Lambda_N}$.
\begin{theorem}[Estimate for remainder coordinate] 
  \label{thm:local_part_of_K_j+1}
  The map $\cK_{j+1}$ 
  admits a decomposition
\begin{equation}
\mathcal{K}_{j+1}(U_j,K_j)
= \mathcal{L}_{j+1} (K_j) + \mathcal{M}_{j+1} (U_j,K_j)
\end{equation}
{into polymer activities at scale $j+1$} such that the following holds for any $r\in (0,1]$, $\beta \geq 2c_f^{-1}$
and with $h$ given by \eqref{eq:h_def},
provided $L \geq L_0$, 
$A \geq A_0(L)$:
\begin{itemize}
\item[(i)] The map $\cL_j$ is linear in $K_j$
and there is a constant $C_1 >0$ independent of $\beta$, $\rho_J$, $A$, $L$ and $r$
such that, with $\alphaLoc$ as in \eqref{e:Loc-contract-kappa},
\begin{equation}
  \norm{\mathcal{L}_{j+1} (K_j)}_{\Omega_{j+1}^K} \leq C_1
  L^2\alphaLoc 
  \norm{K_j}_{\Omega_j^K} \label{eq:bound_for_L_j_K_j}.
\end{equation}
\item[(ii)] The remainder maps $\mathcal{M}_{j+1}$ satisfy $\mathcal{M}_{j+1} = O( \norm{(U_j,K_j)}_{\Omega_j}^2 )$ 
in the sense that there exist $\epsilon_{nl} \equiv \epsilon_{nl} ( \beta, A, L) >0$  (only polynomially small in its arguments) and $C_2=C_2 (\beta,A, L) >0$ (only polynomially large in its arguments) such that $\mathcal{M}_{j+1} (U_j,K_j)$ is continuously Fr\'echet-differentiable and, for $\norm{(U_j,K_j)}_{\Omega_j} \leq \epsilon_{nl}$,
\begin{align}
& \norm{\mathcal{M}_{j+1} (U_j,K_j) }_{ \Omega_{j+1}^K} \leq C_2 ( \beta, A, L)  \norm{(U_j,K_j)}_{ \Omega_j}^2 \label{eq:bound_for_N_j_K_j} \\
& \norm{D \mathcal{M}_{j+1}  (U_j,K_j) }_{ \Omega_{j+1}^K} \leq C_2 ( \beta, A, L)  \norm{(U_j,K_j)}_{ \Omega_j}. \label{eq:bound_for_derivative_of_Nj}
\end{align}
\item[(iii)] 
  The family $(\mathcal{K}^{\Lambda_N}_{j+1})_{N}$ with $\mathcal{K}^{\Lambda_N}_{j+1}: D_N \times [-\epsilon_s\theta_J,\epsilon_s\theta_J] \to (\Omega_{j+1}^K)^{\Lambda_N}$ and $D_N =\{ \|(U_j,K_j)\|_{\Omega_j} \leq \epsilon_{nl} \} \subset \Omega_j^{\Lambda_N}$ and $\epsilon_{nl}$ as in (ii),
  is equicontinuous as a function of the implicit parameter~$s$.
\end{itemize}
\end{theorem}

The remainder of this section is concerned with the definition of the renormalisation group map
and the proof of the above three theorems.
More specifically, 
the renormalisation group map is defined in Section~\ref{sec:def_rg_map} and we 
prove Theorems~\ref{thm:general_RG_step_consistent} and \ref{thm:H_j_E_j_estimate}.
In Section~\ref{subsec:the_expression_Lj}, we prove Theorem~\ref{thm:local_part_of_K_j+1}~(i), 
where the contraction mechanisms of Section~\ref{sec:inequalities} are combined into one. 
Theorem~\ref{thm:H_j_E_j_estimate} and Theorem~\ref{thm:local_part_of_K_j+1}~(i) are the key to understanding the construction of the stable manifold in Section~\ref{sec:stable_manifold_theorem}.
In Sections~\ref{subsec:M_j+1_decomposition}--\ref{sec:rgmap-continuitys},
we prove Theorem~\ref{thm:local_part_of_K_j+1}~(ii) and (iii).
These estimates are of rather technical nature, 
and may be skipped on the first read.

\subsection{Definition of the renormalisation group map}
\label{sec:def_rg_map}

The first definition concerns the coupling constants $(s_j,z_j,E_j)$. These are given by first order perturbation theory,
plus a correction from the remainder coordinate $K_j$, which involves its localisation as introduced in Section~\ref{sec:Loc}.

\begin{definition}
\label{def:evolution_of_U}
For $U_j$ of the form  \eqref{eq:general_RG_step3}, define $(\cE_{j+1},\cU_{j+1}): (U_j,K_j)\mapsto (E_{j+1},U_{j+1})$ to be the unique solution of
\begin{equation}
  \label{eq:evolution_of_U}
  - \cE_{j+1}(U_j,K_j)|B| + \cU_{j+1}(U_j,K_j,B,\varphi')
  = \Eplus U_j(B,\varphi'+\zeta) + \sum_{X\in \cS_j: X \supset B} \Loc_{X,B} \Eplus K_j (X,\varphi'+\zeta),
\end{equation}
where $B \in \cB_{j}$ is any scale-$j$ block, $\cE_{j+1}(U_j,K_j) \in \R$ and $\cU_{j+1}$ is of the same form as in Definition~\ref{def:U_space}.
For general $Y \in \cP_j$,  the definition extends by setting $\cU_{j+1}(Y)= \sum_{B\in \cB_j(Y)} \cU_{j+1}(B)$.
\end{definition}


That \eqref{eq:evolution_of_U} well-defines $(\cE_{j+1}, \cU_{j+1})$,
i.e., that the right-hand side of
\eqref{eq:evolution_of_U} can be uniquely written
in the form of the left-hand side,
follows by explicitly evaluating the Gaussian expectation in the first term and
by Proposition~\ref{prop:Loc-coupling} for the sum over $\Loc_{X,B}$, 
as will become apparent in the proof of Theorem~\ref{thm:H_j_E_j_estimate} below.
Although $U_j$ and $\cU_{j+1}$ are defined as polymer activities in scale $j$, they can easily be extended to scale $j+1$ by simply letting $U_j (X) = \sum_{B \in B_{j} (X)} U_j (B)$ for $X\in \cP_{j+1}$ and likewise for $\cU_{j+1}$. 
Thus we may say $\cU_{j+1} \in \Omega_{j+1}^U$ in this sense. 
These are the polymer activities that are used for the definition and the proofs below.

The following definition gives the evolution of the remainder coordinate $K_j$.
The explicit formula is somewhat involved, but it arises from simple algebraic principles
developed in \cite[Section 5]{MR2523458},
with the small difference that the order of expectation and reblocking reversed, following the set-up of \cite{MR2917175}.
The proof of Theorem~\ref{thm:general_RG_step_consistent} will shed some light on this definition.

\begin{definition} \label{def:evolution_of_remainder}
  The map $\cK_{j+1} : (U_j,K_j)\mapsto K_{j+1}$
  is defined by (suppressing the dependence on
  $\varphi'$
  on the right-hand side,  or writing it as $\cdot$
),  for $X \in \mathcal{P}_{j+1}^c$,
\begin{multline}
   \label{eq:expression_for_K_j+1}
  \mathcal{K}_{j+1} (U_j, K_j, X,\varphi')
  = \sum_{X_0, X_1, Z, (B_{Z''})}^{*} 
  e^{ \cE_{j+1} |T|} e^{\cU_{j+1} (X \backslash T)} \\
  \times \Eplus \Big[ (e^{U_j ( \varphi' + \zeta) } - e^{ -\cE_{j+1} |B| + \cU_{j+1}})^{X_0} (\bar{K}_j (\varphi' + \zeta) - \mathcal{E} K_j)^{[X_1]}  \Big] \prod_{Z'' \in \operatorname{Comp}_{j+1} (Z)} J_j (B_{Z''}, Z''),
\end{multline}
where $*$ refers to the constraints $X_0 \cup X_1 \cup (\cup_{Z''} B^*_{Z''} )=X$,
for $(j+1)$-polymers $X_0, X_1, Z$
such that
$X_1 \not\sim Z$,  $B_{Z''} \in \cB_{j+1} (Z'')$ for each $Z'' \in \operatorname{Comp}_{j+1} (Z) $, 
$T = X_0 \cup X_1 \cup Z$, 
and with the shorthand notation 
\begin{align}
  (e^{U_j} - e^{ -\cE_{j+1} |B| + \cU_{j+1}})^{X_0} &= \prod_{B\in \cB_{j+1} (X_0)} (e^{U_j (B)} - e^{ -\cE_{j+1} |B| + \cU_{j+1} (B)} ) \label{eq:polymer_power_convention1} \\
  (\bar{K}_j - \mathcal{E} K_j)^{[X_1]} &= \prod_{Y \in \operatorname{Comp}_{j+1} (X_1)}(\bar{K}_j - \mathcal{E} K_j)(Y) \label{eq:polymer_power_convention2}
\end{align}
(the right-hand side is equal to $1$ by convention when $X_0$ or $X_1$ are empty)
and the definitions of $J_j$, $\bar K_j$, and $\mathcal{E}K_j$ are as follows:
\begin{align}
\mathcal{E} K_j ( X, \varphi' ) &= \sum_{B\in \mathcal{B}_{j+1} (X)} J_j (B, X, \varphi') \label{eq:cEK_j_definition} \\
Q_j ( D, Y, \varphi') &= 1_{Y\in \cS_j} \Loc_{Y, D} \Eplus [  K_{j} ( Y, \varphi' + \zeta) ]   \label{eq:Q_j_definition}
\\
J_j (B, X,\varphi')  &= 1_{B\in \cB_{j+1} (X)} \sum_{D\in \cB_j (B)} \sum_{Y\in \cS_j}^{D\in \cB_j (Y)} Q_j (D, Y, \varphi') ( 1_{\bar{Y} = X} - 1_{B=X} ) \label{eq:J_j_definition} \\
\bar{K}_j (X,\varphi'+\zeta) &= \sum_{Y\in \mathcal{P}_j}^{\bar{Y}=X} e^{U_j (X \backslash Y,\varphi'+\zeta)} K_j (Y,\varphi'+\zeta), \label{eq:K_bar_definition}
\end{align}
where on the left-hand sides $D\in \cB_j$, $B\in \cB_{j+1}$, $Y\in \cP_j$ and $X\in \cP_{j+1}$. 
These are all, sometimes implicitly, functions of $(U_j, K_j)$.
\end{definition}

In \eqref{eq:expression_for_K_j+1}, note that since $J_j (B_{Z''}, Z'')$ vanishes if  $Z'' \in \cP_{j+1} \backslash \cS_{j+1}$, the summation $\sum^*$ does not vanish only if $Z'' \in \cS_{j+1}$ and in particular $Z'' \subset B_{Z''}^* $. 
We henceforth always assume this when we write $\sum^*$. We now proceed to give the proofs of Theorems~\ref{thm:general_RG_step_consistent} and~\ref{thm:H_j_E_j_estimate}.

\begin{proof}[Proof of Theorem~\ref{thm:general_RG_step_consistent}]
  The proof 
  is similar to that of \cite[Proposition~5.1]{MR2523458},
  except that the order of expectation and reblocking reversed as in \cite{MR2917175}, as mentioned above. 
Thoughout the proof, we write
\begin{equation}
  \varphi = \varphi' + \zeta
  \label{eq:phi_is_phiprime+zeta}
\end{equation}
with $\zeta \sim \Gamma_{j+1}$ and $\varphi',  \zeta$ independent,
and the fluctuation integral $\Eplus$ acts on the variable $\zeta$.
As explained at the beginning of Section~\ref{sec:thms-RGmap}, there is no loss of generality in setting $E_j=0$, which we henceforth assume. Suppose now that \eqref{eq:general_RG_step2} holds and let $\Lambda\equiv \Lambda_N$, $0<j< N-1$.
The first step is the reblocking 
\begin{align}
  Z_j 
\stackrel{\eqref{eq:general_RG_step2}}{=}
  \sum_{X\in \mathcal{P}_{j}} e^{U_j (\Lambda \backslash X)} K_j (X ) 
  =
  \sum_{X'\in \mathcal{P}_{j+1}} e^{U_j (\Lambda \backslash X')} \bar{K}_j (X')  \label{eq:reblock-K-1}
\end{align}
with $\overline{K}_j$ as defined in \eqref{eq:K_bar_definition}, where the second equality follows from the additivity of $U_j$ given by~\eqref{eq:general_RG_step3} upon writing $U_j (\Lambda \backslash X) = U_j (\Lambda \backslash \overline{X}) + U_j (\overline{X} \backslash X)$ for $X\in \mathcal{P}_{j}$. We will repeatedly use additivity in the sequel.
In the next step, $e^{U_j}$ is replaced by $e^{-E_{j+1}|B| + U_{j+1}}$ using the identity, valid for all $X' \in  \mathcal{P}_{j+1}$,
\begin{align}
e^{U_j (\Lambda \backslash X', \varphi)} & = \prod_{B\in \mathcal{B}_{j+1} (\Lambda \backslash X')} \Big( \big( e^{U_j ( B, \varphi)} - e^{-E_{j+1}|B| +  U_{j+1} (B, \varphi')} \big) + e^{-E_{j+1} |B| +  U_{j+1} (B, \varphi')} \Big)  \nnb 
& = \sum_{Y \in \mathcal{P}_{j+1} (\Lambda \backslash X')} e^{-E_{j+1}|\Lambda \backslash (X'\cup Y)| + U_{j+1} (\Lambda \backslash (X'\cup Y), \varphi')}  \big( e^{U_j (\varphi)} - e^{- E_{j+1}|B| +  U_{j+1} (\varphi')} \big)^{Y}  \label{eq:reblock-K-2}
\end{align}
(cf.~\eqref{eq:polymer_power_convention1} for notation).
Similarly, observing that $\bar{K}_j$ in \eqref{eq:K_bar_definition} inherits from $U_j$ and $K_j$ a factorisation property at scale $j+1$,
one replaces $\bar{K}_j$ in \eqref{eq:reblock-K-1} by $\bar{K}_j - \cE K_j$ (recall $\cE K_j$ from \eqref{eq:cEK_j_definition}) using the identity, for $X' \in  \mathcal{P}_{j+1}$,
\begin{align}
\bar{K}_j (X', \varphi)&={ \prod_{Z' \in \operatorname{Comp}_{j+1} (X')} \bar{K}_j (Z', \varphi)} \nnb
&= \prod_{Z' \in \operatorname{Comp}_{j+1} (X')} \big( \mathcal{E} K_j (Z', \varphi') + ( \bar{K}_j (Z', \varphi ) - \mathcal{E} K_j (Z', \varphi') ) \big) \nnb 
&= \sum_{Z \in \mathcal{P}_{j+1} (X')}^{ Z \not\sim X' \backslash Z} \mathcal{E} K_j (\varphi')^{[Z]} (\bar{K}_j (\varphi)  - \mathcal{E} K_j (\varphi')  )^{[X' \backslash Z]}, \label{eq:reblock-K-3}
\end{align}
with the polymer powers following the convention \eqref{eq:polymer_power_convention2}. Using the specific form of $\cE K_j$ given by \eqref{eq:cEK_j_definition} the right-hand side in the previous display can be rewritten as
\begin{equation}
  \mathcal{E} K_j (\varphi')^{[Z]}
  = \sum_{(B_{Z''})_{Z''}} \prod_{Z''} J_j (B_{ Z''}, Z'', \varphi')
\end{equation}
where the last sum $(B_{Z''})_{Z''}$ runs over the collections of blocks $B_{Z''} \in \cB_{j+1} (Z'')$ for all $Z'' \in \operatorname{Comp}_{j+1}(Z)$. Thus, returning to \eqref{eq:reblock-K-1}, substituting \eqref{eq:phi_is_phiprime+zeta}, \eqref{eq:reblock-K-2} and \eqref{eq:reblock-K-3}, taking expectations and rewriting $X'' = X'\cup Y$, the partition function $Z_{j+1}(\varphi')$ can be written as
\begin{align}
  Z_{j+1}(\varphi')
  &=
    \Eplus [Z_j (\varphi' + \zeta)]
    \nnb
  &= e^{-E_{j+1} |\Lambda|}
    \Eplus \Bigg[
    \sum_{X'' \in \mathcal{P}_{j+1}} e^{U_{j+1} (\Lambda \backslash X'')} e^{E_{j+1} |X''|}
 \sum_{X'\subset X''} (e^{U_j} - e^{-E_{j+1} |B| + U_{j+1}})^{X''\backslash X'} \label{eq:reblock-K-4}\\
  &\qquad\qquad\qquad \qquad \qquad  \times \sum_{Z \subset X'} (\bar{K}_j - \mathcal{E} K_j)^{[X' \backslash Z]} \sum_{(B_{Z''})} \prod_{Z'' \in \operatorname{Comp}_{j+1} (Z)} J_j (B_{Z''}, Z'') \Bigg];
    \notag
\end{align}
{above the sums over $X' (\subset  X'')$ and $Z (\subset X')$ are over elements in $\mathcal{P}_{j+1}$ and $Z$ satisfies the additional constraint $Z \not\sim X' \backslash Z$, i.e., $Z$ runs over all unions of subsets of $ \operatorname{Comp}_{j+1} (X')$.}
The final result is obtained after performing another resummation: we write
$X_0 = X'' \backslash X'$, $X_1 = X' \backslash Z$, $T=X_0 \cup X_1 \cup Z = X''$ and define, { summing over $X_0,X_1,Z \in \mathcal{P}_{j+1}$ with the constraint $Z\not\sim X_1$ and $(B_{Z''})$ as above, for all $X \in \mathcal{P}_{j+1}$,}
\begin{multline}
K_{j+1} (X, \varphi') = \sum_{X_0, X_1, Z, (B_{Z''})} {1_{ (\cup_{Z''}B^*_{Z''} \cup X_0 \cup X_1)=X} }e^{E_{j+1} |T|} e^{U_{j+1} (X \backslash T)} \\
\times \Eplus \Big[ (e^{U_j} - e^{-E_{j+1} |B| + U_{j+1}})^{X_0} (\bar{K}_j - \mathcal{E} K_j)^{[X_1]}  \Big] \prod_{Z'' \in \operatorname{Comp}_{j+1} (Z)} J_j (B_{Z''}, Z'') . \label{eq:expression_for_K_j+1_in_appendix}
\end{multline}
(We remark that the particular arrangement of the sum with $B_{Z''}^*$ in the indicator function
  will allow to exhibit the important cancelation \eqref{eq:cancellation_J}, i.e., to sum over all $Z$  while keeping $B$ and $X$ fixed.)

Note that only $T\subset X$ contribute because, by definition of $\cE K_{j}$, the whole expression vanishes when $Z \not\in \cS_{j+1}$.
With this definition, it follows that \eqref{eq:reblock-K-4} can be recast as
\begin{equation}
Z_{j+1} (\varphi') = e^{-E_{j+1} |\Lambda| } \sum_{X \in \mathcal{P}_{j+1}} e^{U_{j+1} (\Lambda \backslash X,\varphi')} K_{j+1} (X, \varphi')
\end{equation}
which has the desired form.

If we assume that $K_j$ obeys the evenness,  lattice symmetries and the periodicity condition, then it is also apparent from the expressions that $K_{j+1}$ has the same properties, since $U_{j}$ and $U_{j+1}$ also satisfy them.
Similarly, the factorisation property is inherited from those of $e^{U_j}$, $e^{U_{j+1}}$ and $K_j$.
\end{proof}

\begin{proof}[Proof of Theorem~\ref{thm:H_j_E_j_estimate}]

  Evaluating the expectation $\Eplus U_j$ on the right-hand side of
\eqref{eq:evolution_of_U} explicitly gives, using \eqref{eq:general_RG_step3}, the fact that $\zeta$ is centered and invariant under lattice rotations and \eqref{eq:integral_of_charge_q}, 
\begin{align}
\Eplus U_j(B,\varphi'+\zeta) &=  \frac{1}{2} s_j \big (|\nabla \varphi'|_B^2 + \Eplus |\nabla \zeta|_B^2 \big) +    \sum_{x\in B} \sum_{q\geq 1} L^{-2j} z^{(q)}_j \cos( \beta^{1/2} q \varphi'(x)) \Eplus[e^{i \sqrt{\beta} q \zeta (x)}]\nnb
&=  \frac{1}{2} s_j \Big( |\nabla \varphi'|_B^2 + |B| \sum_{\sigma = \pm}\Eplus\big[ ( \zeta (x_0+\sigma e_1)-  \zeta (x_0))^2  \big] \Big) + W_{j+1} (B, \varphi') \label{eq:U-exp}
\end{align}
for any reference point $x_0 \in \Lambda_N$, with ${z}_{j+1}^{(q)}=\mathfrak{z}_{j+1}^{(q)}(z_j)$ implicit in $W_{j+1}$ given by
\begin{equation}
 \mathfrak{z}_{j+1}^{(q)}(z_j^{(q)})
      = L^{2} e^{-\frac{1}{2}\beta q^2 \Gamma_{j+1} (0)} z_j^{(q)}  \label{eq:RG_step_parameter_choice4}
\end{equation}
as declared in Theorem~\ref{thm:H_j_E_j_estimate}. Hence, combining \eqref{eq:U-exp} with \eqref{e:Loc-coupling-bds}--\eqref{e:Loc-coupling}, it follows that the right-hand side of \eqref{eq:evolution_of_U} corresponds to the change of coupling constants
$(\mathfrak{s}_{j+1},\mathcal{E}_{j+1}, \mathfrak{z}_{j+1}): (s_j,z_j,K_j) \mapsto (s_{j+1},E_{j+1}, z_{j+1})$
given by \eqref{eq:RG_step_parameter_choice4} and
\begin{align}
  \mathfrak{s}_{j+1}(s_j,K_j)
  &= s_j
    + O \big(A^{-1}  h^{-2} \norm{K_j}_{ \Omega_{j}^K}  \big)     \label{eq:RG_step_parameter_choice1}
  \\
  \mathcal{E}_{j+1}(s_j, K_j)
  &= -s_j \nabla^{(e_1, -e_1)} \Gamma_{j+1} (0) +
    O (L^{-2j}A^{-1} \norm{K_j}_{ \Omega_{j}^K} )
    \label{eq:RG_step_parameter_choice3}
 \end{align}
where $\nabla^{(e_1, -e_1)} \Gamma_{j+1} (0) = {  \frac{1}{2} \sum_{\sigma}\Eplus[ ( \zeta (x_0+\sigma e_1)-  \zeta (x_0))^2  ]=}  - \Gamma_{j+1} (e_1) - \Gamma_{j+1} (-e_1) + 2\Gamma_{j+1} (0)$, giving the bounds \eqref{eq:H_j+1_bound} and \eqref{eq:E_j+1_bound}.

Finally, we argue that the asserted continuity properties in the implicit parameter $s$ hold. With regards to $ \mathfrak{z}_{j+1}^{(q)}$, this is immediate by \eqref{eq:RG_step_parameter_choice4} and the continuity of $s \mapsto \Gamma_{j+1}(s)$, cf.~Proposition~\ref{prop:decomp},(ii). Next, referring to Proposition~\ref{prop:Loc-coupling}, we have $\mathfrak{s}_{j+1}(s_j,K_j)
= s_j + \mathfrak{s}_{j+1}(0, K_j)$ 
and $ \mathcal{E}_{j+1}(s_j, K_j)  = -s_j \nabla^{(e_1, -e_1)} \Gamma_{j+1} (0) + \cE_{j+1}(0, K_j)$, whereby $\cE_{j+1}(0, K_j)=-\bar{E}(K_j)$ and $\mathfrak{s}_{j+1} = \bar{s} (K_j)$. Thus, Proposition~\ref{prop:Loc-coupling} immediately yields that $\cE_{j+1}(0, K_j)$ and $\mathfrak{s}_{j+1}(0, K_j)$ are both continuous in the implicit parameter $s$ whenever $\norm{K_j}_{\Omega_j^K}< \infty$. The claim follows.
\end{proof}

The proof of Theorem~\ref{thm:local_part_of_K_j+1} occupies the remainder of Section~\ref{sec:rg_generic_step}.
More precisely, in Section~\ref{subsec:the_expression_Lj} we find the explicit expression of $\cL_{j+1}$ and prove its bound,
in Sections~\ref{subsec:M_j+1_decomposition}--\ref{subsec:estimate_Nj} the bound on the nonlinear part $\cM_{j+1}$,
and finally in Section~\ref{sec:rgmap-continuitys} the continuity of all maps in the parameter $s$.

\subsection{Proof of Theorem~\ref{thm:local_part_of_K_j+1}: bound of linear part}
\label{subsec:the_expression_Lj}

The constant terms and the terms linear in $U_j$, $K_j$ can be identified directly from \eqref{eq:expression_for_K_j+1}
by (1) only keeping the terms with 
\begin{equation}
  \# (X_0, X_1, Z) := |X_0|_{j+1} + |\operatorname{Comp}_{j+1}  (X_1)| + |\operatorname{Comp}_{j+1}  (Z)|
  \leq 1 , 
\end{equation}
(2) replacing exponentials by $1$ outside the expectation,
(3) replacing exponentials by their linearisations inside the expectation,
and (4) replacing $\bar{K}_j$ by $\mathbb{S} K_j$. 
This gives (see also \eqref{eq:M_decomp} below for the expression for $\cK_{j+1}-\cL_{j+1}$):
for $X\in \mathcal{P}_{j+1}^c$,
\begin{align}    \label{eq:L_j_K_j_0}
\mathcal{L}_{j+1} (K_j) ( X, \varphi') 
& := \sum_{Y : \bar{Y} = X} \Big( 1_{Y\in \cP_j^c} \Eplus K_j (Y, \varphi'+\zeta ) - 1_{Y\in \cS_j} \sum_{D\in \cB_j (Y)} Q_j (D, Y, \varphi') \Big) \\
&\quad + \sum_{D \in \mathcal{B}_{j}}^{\bar{D}=X}  \Big( \Eplus [ U_j (D, \varphi'+\zeta)] + \cE_{j+1} |D| - \cU_{j+1} (D, \varphi') + \sum_{Y\in \cS_j}^{D\in \cB_j (Y)}  Q_j (D, Y, \varphi') \Big).\nonumber
\end{align}
In more detail, the terms in the first line above and the $Q_j$-terms in the second line come from $X=T=X_1$
(replacing $e^{U_j(X\setminus Y)}$ by $1$ in $\bar{K}_j$,
which corresponds to replacing $\bar{K}_j$ by $\mathbb{S} K_j$, cf.~\eqref{eq:K_bar_definition} and \eqref{eq:reblocking_operator_definition}, and keeping only connected polymers $Y$),
the remaining terms in the second line are due to $X=T=X_0$ (and linearising the exponentials),
and finally the terms with $T=Z$ (and thus $X=B_{Z}^*$) actually vanish by the construction of $J_j$
(after the replacement of the exponential outside
the expectation, i.e., in the first line of \eqref{eq:expression_for_K_j+1}, by $1$).
Indeed, to see that the contribution from $X = B^*_Z$ cancels,
note that for any  $B \in \cB_{j+1}$,
\begin{align}
\sum_{Z\in \cS_{j+1}}^{B\in \cB_{j+1} (Z)} J_j (B, Z)
& = \sum_{Z\in \cS_{j+1}}^{B\in \cB_{j+1} (Z)} \sum_{D\in \cB_j (B)} \sum_{Y\in \cS_j}^{D\in \cB_j (Y)} Q_j (D, Y) ( 1_{\bar{Y} = Z} - 1_{B=Z} ) \nnb
  & = \sum_{D\in \cB_j (B)} \sum_{Y\in \cS_j}^{D\in \cB_j (Y)} Q_j (D, Y) \sum_{Z\in \cS_{j+1}}^{B\in \cB_{j+1} (Z)} (1_{\bar{Y} = Z} - 1_{B=Z} ) = 0.
     \label{eq:cancellation_J}
\end{align}
In obtaining the last equality, we have implicitly used that the closure of a small set is again small.
Using the choice of $\cU_{j+1}$, cf.~\eqref{eq:evolution_of_U}, the second line in \eqref{eq:L_j_K_j_0} cancels,
and with \eqref{eq:Loc_decomp} the first line simplifies to
\begin{equation}
\mathcal{L}_{j+1} (K_j) ( X, \varphi') 
= \sum_{Y : \bar{Y} = X} \Big( 1_{Y\in \cP_j^c} \Eplus K_j ( Y, \varphi'+\zeta) - 1_{Y\in \cS_j} [ \Loc_{Y} \Eplus K_{j} ({ Y}, \varphi' + \zeta)  ] \Big).
\label{eq:L_j_K_j}
\end{equation}

\begin{proof}[Proof of Theorem~\ref{thm:local_part_of_K_j+1},(i)]

One may decompose $\mathcal{L}_{j+1}$ further as
\begin{equation}
\mathcal{L}_{j+1} (K_j) (X, \varphi') 
=  \sum_{Y : \bar{Y} = X} 1_{Y\in \cS_j} (1- \Loc_{Y} )  \Eplus [ K_j (Y, \varphi'+\zeta)  ] 
 +  \mathbb{S} \big( \Eplus [K_j 1_{Y \not\in \cS_j}] \big) (X, \varphi') , \label{eq:L_j+1_decomposition}
\end{equation}
where $\mathbb{S}$ is the reblocking operator defined by \eqref{eq:reblocking_operator_definition}. The first term is bounded at once using Proposition~\ref{prop:Loc-contract},
\begin{equation}
  \norm{ (1-\operatorname{Loc}_Y) \Eplus [ K_j (Y, \varphi' + \zeta) ] }_{h, T_{j+1} (\bar{Y}, \varphi')}
  \leq  \alphaLoc   A^{-|Y|_j} 
  \norm{K_j}_{\Omega_j^K}
   G_{j+1}(\bar Y, \varphi'),
\end{equation}
and since the number of $Y \in \cS_j$ with $\bar Y=X$ is $O(L^2)$, this gives
\begin{equation}
  \norm{\sum_{Y : \bar{Y} = X} 1_{Y\in \cS_j} (1- \Loc_{Y} ) \Eplus [ K_j ( Y, \varphi'+\zeta)  ]  }_{h, T_{j+1} ( X, \varphi')}
    \leq CL^2\alphaLoc A^{-|Y|_j}  \norm{K_j}_{\Omega_j^K}  G_{j+1} (X, \varphi')
\end{equation}
for $L$ sufficiently large.
 The bound on the second term is a result of 
Proposition~\ref{prop:largeset_contraction-v2}, 
\begin{align}
\norm{\mathbb{S} \big( \Eplus [K_j 1_{Y\not\in \cS_j}] \big) (X) }_{h, T_{j+1} (X)} \leq C L^{-1} A^{-|X|_{j+1}} \norm{K_j}_{\Omega_j^K}
\end{align}
for $L \geq C$ and $A \geq C'(L)$. Since $L^{-1} \leq L^2 \alphaLoc$ for sufficiently large $L$, cf.~\eqref{e:Loc-contract-kappa},
this yields the desired bound.
\end{proof}

\subsection{Proof of Theorem~\ref{thm:local_part_of_K_j+1}: bound of non-linear part}
\label{subsec:M_j+1_decomposition}

Below we write $s_{j+1}, W_{j+1}, U_{j+1}$ in place of $\mathfrak{s}_{j+1}, \cW_{j+1}, \cU_{j+1}$ {(with arguments of these functions clear from the context)} for simplicity of notation.
In what follows, it will be convenient to have a shorthand notation for the collection $(\cE_{j+1} |X|, U_j, \bar{U}_{j+1} , K_j, \bar{K}_j , \cE K_j  , J_j  ) $, where we view $X\mapsto \cE_{j+1}|X|$ as a polymer activity and define $\bar{U}_{j+1}$ as
 \begin{equation} \label{eq:Ubar_def}
\bar{U}_{j+1} (X, \varphi') = -\cE_{j+1} |X| + U_{j+1} ( X, \varphi').
\end{equation}
Accordingly, we introduce the map
\begin{equation}
\label{eq:k_j-def}
\omega_j = (U_j,K_j) \mapsto \bar{\mathfrak{K}}_j (\omega_j) \equiv (\cE_{j+1} |X|, U_j, \bar{U}_{j+1} , K_j, \bar{K}_j , \cE K_j  , J_j  ) (\omega_j).
\end{equation}
By definition, $\mathcal{M}_{j+1}$ is just $\mathcal{K}_{j+1}$ without its local part $\cL_{j+1}$, so by Definition~\ref{def:evolution_of_remainder} and \eqref{eq:L_j_K_j_0}
one may decompose $\cM_{j+1}$ into four terms as follows: for $X\in \mathcal{P}_{j+1}^c$, using the notation \eqref{eq:k_j-def}, 
\begin{equation} \label{eq:M_decomp}
  \mathcal{M}_{j+1}  ({ U_j,} K_j,  X, \varphi')
  =  \sum_{k=1}^4 \MM_{j+1} ^{(k)} (\bar{\mathfrak{K}}_j (\omega_j) , X, \varphi'), 
\end{equation}  
 where the $\MM_{j+1} ^{(k)}$ are given as follows:
\begin{align}
\MM_{j+1} ^{(1)} (\bar{\mathfrak{K}}_j (\omega_j) , X, \varphi') 
&= \sum_{X_0, X_1, Z, (B_{Z''})}^{*} 1_{\# (X_0, X_1, Z)  \geq 2} e^{\cE_{j+1} |X|} e^{ \bar{U}_{j+1} (X \backslash T)} \nonumber \\
	& \quad \times \Eplus \Big[ (e^{U_j} - e^{\bar{U}_{j+1}})^{X_0} (\bar{K}_j - \mathcal{E} K_j)^{[X_1]}  \Big] \prod_{Z'' \in \operatorname{Comp}_{j+1} (Z)} J_j (B_{Z''}, Z'') \label{eq:M^1_j+1} \\
\MM_{j+1} ^{(2)} (\bar{\mathfrak{K}}_j (\omega_j) , X, \varphi')  
&= \sum_{X_0, X_1, Z, (B_{Z''})}^* 1_{ \# (X_0, X_1, Z)  \leq 1 }  ( e^{\cE_{j+1} |X|} e^{ \bar{U}_{j+1} (X \backslash T)} - 1) \nonumber \\
	& \quad \times \Eplus \Big[ (e^{U_j} - e^{\bar{U}_{j+1}})^{X_0} (\bar{K}_j - \mathcal{E} K_j)^{[X_1]}  \Big] \prod_{Z'' \in \operatorname{Comp}_{j+1} (Z)} J_j (B_{Z''}, Z'') \label{eq:M^2_j+1}  \\
\MM_{j+1} ^{(3)} (\bar{\mathfrak{K}}_j (\omega_j) , X, \varphi')  
&= \sum_{|X_0|_{j+1} = 1}^{X_0 = X} \Eplus \Big[ \Big( e^{U_j} - e^{\bar{U}_{j+1}} - U_j + \bar{U}_{j+1} \Big)^{X_0} \Big] \label{eq:M^3_j+1}  \\
\MM_{j+1} ^{(4)} (\bar{\mathfrak{K}}_j (\omega_j) , X, \varphi')  
&= \Eplus \Big[ \sum_{Y \in \cP_j}^{\bar{Y} = X} e^{U_j (Y)} K_j (X \backslash Y) - \mathbb{S} [ K_j ] (X) \Big] \label{eq:M^4_j+1}
\end{align}
 and, as in Definition~\ref{def:evolution_of_remainder}, we are letting  and $T = X_0 \cup X_1 \cup Z$. Each $\mathfrak{M}_{j+1}^{(1)}$ -- $\mathfrak{M}_{j+1}^{(4)}$ arise from the linearisation process (1) -- (4) described above \eqref{eq:L_j_K_j_0}.

The bound of $\mathcal{M}_{j+1} ({ U_j}, K_j)$ will follow by bounding each $\MM_{j+1}^{(k)} (\bar{\mathfrak{K}}_j)$ separately.
Although the above is not the most efficient way to express $\cM_{j+1}$, writing it in this way will make it easier to generalise the estimate in our companion paper \cite{dgauss2}.
Indeed, one may deduce a bound on each $\MM_{j+1}^{(k)}$ that only depends on the estimates on $\bar{\mathfrak{K}}_j$.
The next definition collects a list of bounds on various terms that appear in the above formulas.
These estimates are sufficient to imply the desired bounds on the $\MM_{j+1}^{(k)}$,
as asserted in Lemma~\ref{lemma:bound_on_M^k}.
The validity of these ``building block estimates'' under the assumptions of Theorem~\ref{thm:local_part_of_K_j+1} is shown separately in Lemma~\ref{lemma:Ujbound-summary}.

In the following definition, our main application uses the case when $\mathbb{Y}$ is a closed ball of $\Omega_j$, but it will be useful
   in the proof of the continuity in $s$ to have the additional flexibility of the space $\mathbb{Y}$.
   
\begin{definition} \label{def:derivativebds}
Let $(\mathbb{Y}, |\cdot |)$ be a closed subset of a normed vector space.  Given $\delta, \eta >0$, define $\cX_j^{\mathfrak{K}} (\mathbb{Y})$, to be the set of functions $\mathfrak{K}_j : x \mapsto \mathfrak{K}_j(x) =(\cE_{j+1} |X|, U_j, \bar{U}_{j+1},  K_j, \bar{K}_j, \cE K_j, J_j) (x)$,
where each component takes polymer activity value as in the right-hand side of \eqref{eq:k_j-def},
such that $x \mapsto K_j (x)$ is linear and bounded as a function $\mathbb{Y} \rightarrow \cN_j$,   
and satisfies the following estimates for all  $B \in \cB_{j+1}$, $Z \in \cP_{j+1}$ and $\varphi \in \R^{\Lambda_N}$: 
for $k\in \{0,1,2\}$,
\begin{align}
& \norm{\mathfrak{U} (B, \varphi)}_{h, T_j (B, \varphi)} \leq 
C(\delta, L) ( 1+  \delta c_w \kappa_L w_j (B, \varphi)^2 ) |x| \label{eq:Ujbound_1}  \\  
& \norm{e^{\mathfrak{U} (B, \varphi )} - \sum_{m=0}^{k} \frac{1}{m!} (\mathfrak{U}(B, \varphi ))^m }_{h, T_{j} (B, \varphi)} \leq 
C(\delta, L)  e^{\delta c_w \kappa_L w_j (B, \varphi)^2} |x|^{k+1} \label{eq:Ujbound},
\end{align}
for $\mathfrak{U}\in \{ U_j, \bar{U}_{j+1}\}$ and some $C(L)$, and the same inequalities hold with $\mathfrak{U} (B)$ and 
$C(\delta, L)$ replaced by $\cE_{j+1} |B|$ and $C(L)$, respectively, but $\delta$ set to $0$. 
Moreover {(with $D$ the derivative in $x$, cf.~\eqref{eq:Fre-deriv}),}
\begin{align}
& \norm{D e^{\mathfrak{U}' (B, \varphi)} }_{h, T_j (B, \varphi)} \leq C(L) e^{ c_w \kappa_L w_j (B, \varphi)^2} \label{eq:derivatives1-v2}, \\
& \norm{D^2 e^{\mathfrak{U}' (B, \varphi)} }_{h, T_j (B, \varphi)} \leq C(L) e^{ c_w \kappa_L w_j (B, \varphi)^2} \label{eq:derivatives2-v2}, \\
& \norm{D J_j (B, Z, \varphi)}_{h, T_j (B, \varphi)} \leq C(L) A^{-1} e^{ c_w \kappa_L w_{j} (B, \varphi)^2} \label{eq:derivatives3-v2}, \\
& \norm{D  \bar{K}_j (Z, \varphi)  }_{h, T_j (Z, \varphi)} \leq C(A, L) A^{- (1+ \eta)|Z|_{j+1}} G_j (Z, \varphi)  \label{eq:derivatives4-v2}, \\
& \norm{D \mathcal{E} K_j (Z, \varphi)  }_{h, T_j (Z, \varphi)} \leq C(A, L) A^{- (1+ \eta)|Z|_{j+1}} e^{ c_w \kappa_L w_j (Z, \varphi)^2}, \label{eq:derivatives5-v2}
\end{align}
$\mathfrak{U}' \in \{ U_j, \bar{U}_{j+1}, \cE_{j+1} |B|\}$, and in the case of $ \cE_{j+1} |B|$, the factor $e^{c_w \kappa_L w_j (B, \varphi)}$ can be omitted.
Moreover the derivatives exist in the space of polymer activities with finite $\norm{\cdot}_{h, T_j (B)}$-norm for $e^{\mathfrak{U}'}$, $D e^{\mathfrak{U}'}$, $J_j$ and finite $\norm{\cdot}_{h, T_j (Z)}$-norm for $\bar{K}_j$, $\cE K_j$. 
\end{definition}

\begin{lemma} \label{lemma:Ujbound-summary} 
  Under the assumptions of Theorem~\ref{thm:local_part_of_K_j+1},
  for any $\delta>0$ and $\beta \geq 2c_f^{-1}$,
  there exists $\epsilon(L) >0$ only polynomially small in $L$, and constants 
  $C(\delta,L) \equiv C(\delta, \beta,L)$, $C(L)\equiv C (\beta,L)$, 
  $C(A,L)$, 
  $\epsilon (\delta, L) \equiv \epsilon (\delta, \beta, L)$ and $\eta >0$
  such that
  if $\cX_j^{\mathfrak{K}} (\cdot)$ is defined with these $\delta$, $\eta$, $C(\delta,L)$, $C(L)$, $C(A,L)$ 
  then $\bar{\mathfrak{K}}_j$ is in $\cX_j^{\mathfrak{K}} (\{ \omega_j \in \Omega_j : \norm{\omega_j}_{\Omega_j} \leq \epsilon (\delta, L) \})$.
\end{lemma}

\begin{lemma} \label{lemma:bound_on_M^k}
Let $(\mathbb{Y}, |\cdot|)$ be a closed subset of a normed space.  Under the assumptions of Theorem~\ref{thm:local_part_of_K_j+1} 
and if $\mathfrak{K}_j$
is in $\cX_j^{\mathfrak{K}} (\mathbb{Y})$, there exists $\epsilon_{nl} > 0$ such that 
each $\MM_{j+1}^{(k)} (\mathfrak{K}_j (x))$ 
is continuously differentiable on $\{x \in \mathbb{Y} :  |x| \leq \epsilon_{nl} \}$ for $k \in \{1,2,3,4\}$ and satisfies
\begin{equation}\label{lem-bd-D-nl}
  \norm{D \MM_{j+1}^{(k)} (\mathfrak{K}_j (x)) }_{h,T_{j+1}}
  \leq C_2 (A, L) |x|
\end{equation}
for some $C_2 (A, L) >0$.
\end{lemma}

In Section~\ref{subsec:derivative_of_N_j} we will prove Lemma~\ref{lemma:Ujbound-summary},
and in Section~\ref{subsec:estimate_Nj} then Lemma~\ref{lemma:bound_on_M^k}.
Assuming these lemmas to hold, the proof of the bounds on $\cM_{j+1}$ in Theorem~\ref{thm:local_part_of_K_j+1} is immediate, as we now explain.

\begin{proof}[Proof of Theorem~\ref{thm:local_part_of_K_j+1},(ii)] 
  The continuous differentiability of $\cM_{j+1}$ together with the bound \eqref{eq:bound_for_derivative_of_Nj}
  are a direct consequence of Lemma~\ref{lemma:bound_on_M^k} applied with $\delta >0$ sufficiently small, 
  $\mathbb{Y} = \{ \omega\in \Omega_j : \norm{\omega_j}_{\Omega_j} \leq \epsilon (\delta,  \beta,L) \}$, 
  $\mathfrak{K}_j = \bar{\mathfrak{K}}_j $,
  and the decomposition $\cM_{j+1} = \sum_{k=1}^4 \MM_{j+1}^{(k)}$ 
  from  \eqref{eq:M_decomp},
  with the assumptions of Lemma~\ref{lemma:bound_on_M^k} being verified by Lemma~\ref{lemma:Ujbound-summary}.  
  The bound \eqref{eq:bound_for_N_j_K_j} is simply obtained from \eqref{eq:bound_for_derivative_of_Nj} by integration, as $\cM_{j+1} (0) = 0$.
\end{proof}

\subsection{Proof of Lemma~\ref{lemma:Ujbound-summary}}
\label{subsec:derivative_of_N_j}

In this section we prove Lemma~\ref{lemma:Ujbound-summary},
i.e., that
$\bar{\mathfrak{K}}_j(\omega_j)$ defined in \eqref{eq:k_j-def} satisfies $\bar{\mathfrak{K}}_j \in \cX_j^{\mathfrak{K}} (\Omega_j)$ whenever $\omega_j=(U_j,K_j)$ is sufficiently small.
Indeed,
in Lemma~\ref{lemma:Ujbound} we show 
that \eqref{eq:Ujbound_1} and \eqref{eq:Ujbound} hold,
and in Lemmas~\ref{lemma:derivative_of_components_v2_1}--\ref{lemma:derivative_of_components_v2} we prove
\eqref{eq:derivatives1-v2}--\eqref{eq:derivatives5-v2}.

To control the term $\frac{1}{2} |\nabla \varphi|^2_B$ that appears in the expressions to be bounded (cf.~for instance \eqref{eq:general_RG_step3}),
the expression \eqref{eq:wj_def} will appear repeatedly, i.e.,
\begin{equation}
w_j ( X, \varphi )^2 = \sum_{D \in \mathcal{B}_j (X)} \max_{n=1,2} \norm{\nabla^n_j \varphi}_{L^{\infty}(D^*)}^2.
\end{equation}
We recall that $w_j$ is related to the large field regulator $G_j$ by the inequalities \eqref{eq:strong_regulator1} and \eqref{eq:strong_regulator2}.
We start with the following simple lemma.

\begin{lemma}
For $D \in \mathcal{B}_j$, $\mu, \nu \in \hat{e}$,
\begin{align}
& \norm{ (\nabla^{\mu} \varphi, \nabla^{\nu} \varphi  )_D}_{h, T_j ( D,\varphi)} \leq  4(h^2 + w_j (D, \varphi)^2 )\label{eq:example1-1} 
.
\end{align}
\end{lemma}
\begin{proof}
One has the following exact derivatives of $\frac{1}{2} |\nabla \varphi|^2_D$: 
\begin{align}
& D_{\varphi}( (\nabla^{\mu} \varphi, \nabla^{\nu} \varphi  )_D )(f) = \sum_{y \in D} \partial^{\mu} f_y \partial^{\nu} \varphi (y) + \partial^{\nu} f_y \partial^{\mu} \varphi (y)  \\
& D_{\varphi}^2 ( (\nabla^{\mu} \varphi, \nabla^{\nu} \varphi  )_D )(f,g) = \sum_{y \in D} \partial^{\mu} f_y \partial^{\nu} g_y + \partial^{\mu} g_y \partial^{\nu} f_y
\end{align}
and hence $\norm{ (\nabla^{\mu} \varphi, \nabla^{\nu} \varphi  )_D }_{h, T_j (D, \varphi)} \leq 2(h + \norm{\nabla_j \varphi}_{L^{\infty} (D)})^2$, from which
 \eqref{eq:example1-1} follows. 
\end{proof}

\begin{lemma} \label{lemma:Ujbound} 
  Under the assumptions of Theorem~\ref{thm:local_part_of_K_j+1},
  there exists $\epsilon(\delta,  \beta, L) >0$ only polynomially small in $L$ and $\beta$ such that the following holds:
for any  $\delta >0$, suppose $\norm{\omega_j}_{\Omega_j} := \norm{(U_j, K_j)}_{\Omega_j} \leq \epsilon(\delta,  \beta,L)$.
Then \eqref{eq:Ujbound_1}, \eqref{eq:Ujbound} hold,
and the same holds when $\mathfrak{U}$ is replaced by $\cE_{j+1}$ but with $C(\delta,  \beta,L)$ and $\delta$ set to $C(L)$ and $0$, respectively.
\end{lemma}
\begin{proof}
By Theorem~\ref{thm:H_j_E_j_estimate},  \eqref{eq:Gammaj_bd},  \eqref{eq:T_norm_cf_norm_comparison} and \eqref{eq:example1-1},  
for $j^* \in \{j, j+1\}$ 
and $B\in \cB_{j+1}$, 
\begin{align}
& |\cE_{j+1}| |B| , \;\; \norm{ W_{j^*} (B, \varphi) }_{h, T_{j} (B, \varphi)} \leq C L^2 \norm{\omega_j}_{\Omega_j} , \label{eq:Ujbound0} \\
  & \norm{ \frac{1}{2} s_{j^*} |\nabla \varphi|^2_{B} }_{h, T_{j} (B, \varphi)}  \leq \sum_{D \in \cB_j (B)} 2 (h^2 + w_j (D, \varphi)^2 )
  \norm{\omega_j}_{\Omega_j} \leq 2 (L^2 h^2 + w_j (B, \varphi)^2) \norm{\omega_j}_{\Omega_j} .
\end{align}
Since $h = \max \{ c_f^{1/2}, rc_h \rho_J^{-2} \sqrt{\beta}, \rho_{J}^{-1} \}$, by taking $L \geq c_f^{-1/2}$ we have $L^2 h^2 \geq 1$.
Hence for each choice of $\mathfrak{U}  \in \{ U_j, \bar{U}_{j+1}\}$, we obtain in view of \eqref{eq:general_RG_step3} and \eqref{eq:Ubar_def}, 
\begin{equation}
\norm{\mathfrak{U} (B, \varphi )}_{h, T_{j} (B, \varphi)} 
\leq C(\delta) \kappa_L^{-1} L^2 h^2  \big( 1+ \delta c_w \kappa_L w_j (B, \varphi)^2 \big) \norm{\omega_j}_{\Omega_j} \label{eq:Ujbound11}.
\end{equation}
Also since $\beta \geq 2 c_f^{1/2}$, 
there exists $C >0$ such that $h \leq C \sqrt{\beta}$, 
so for some $C(\delta, \beta, L)$ only polynomially large in $\beta$ and $L$, 
\begin{equation}
\norm{\mathfrak{U} (B, \varphi )}_{h, T_{j} (B, \varphi)} 
\leq  C(\delta,  \beta, L) \big( 1+ \delta c_w \kappa_L w_j (B, \varphi)^2 \big) \norm{\omega_j}_{\Omega_j} \label{eq:Ujbound1}.
\end{equation}
Also using the trivial fact that $1+x \leq e^x$ for $x\geq 0$,
\begin{equation}
\norm{\mathfrak{U} (B, \varphi )}_{h, T_{j} (B, \varphi)}
\leq  C(\delta,  \beta,L) 
e^{ \delta c_w \kappa_L w_j (B, \varphi)^2} \norm{\omega_j}_{\Omega_j} . \label{eq:Ujbound2}
\end{equation}
This shows \eqref{eq:Ujbound_1}.
To deduce \eqref{eq:Ujbound},
assume $\norm{\omega_j}_{\Omega_j} \leq \epsilon (\delta,  \beta, L) = \frac{1}{ C(\delta,  \beta, L) }$.
Together with the submultiplicativity \eqref{eq:prodprop} of the norm, \eqref{eq:Ujbound1} then implies that
\begin{equation}\label{eq:bound_exp-U}
  \|e^{\mathfrak{U}}\|_{h,T_j(B, \varphi)} \leq e^{\norm{\mathfrak{U}}_{h,T_j(B, \varphi)}}
  \leq e^{1+\delta c_w\kappa_L w_j(B,\varphi)^2}
  \leq C e^{\delta c_w\kappa_L w_j(B,\varphi)^2}
\end{equation}
and furthermore, using \eqref{eq:Ujbound2} to bound $(\mathfrak{U})^{k+1}$ for $k\in\{0,1,2\}$,
\begin{align}
\norm{ e^{\mathfrak{U}} - \sum_{m=0}^k \frac{1}{m!} (\mathfrak{U})^m }_{h, T_{j} (B, \varphi)} 
& \leq \frac{1}{(k+1)!} \norm{\mathfrak{U}}_{h, T_{j} (B, \varphi)}^{k+1} e^{\norm{\mathfrak{U}}_{h, T_{j} (B, \varphi)}} \nnb
& \leq  C(\delta,  \beta, L)  \norm{\omega_j}_{\Omega_j}^{k+1} \exp \big( 4\delta c_w \kappa_L w_j (B, \varphi)^2 \big),
\end{align}
which is equivalent to the claim, by replacing $4\delta$ by $\delta$.
The remark about $\cE_{j+1}$ follows from the same computations starting just from \eqref{eq:Ujbound0}.
\end{proof}

\begin{lemma} \label{lemma:derivative_of_components_v2_1}
  Under the assumptions of Theorem~\ref{thm:local_part_of_K_j+1}, there exist
  $\epsilon \equiv \epsilon ( \beta,  L) >0$ (only polynomially small in $\beta$),  $C \equiv C(c_w,  \beta,  L)$, and $C_{A} \equiv C_A(c_w,   \beta,  L, A)$ such that the assumption of Lemma~\ref{lemma:strong_regulator} holds and that the bounds \eqref{eq:derivatives1-v2}, \eqref{eq:derivatives2-v2}, \eqref{eq:derivatives3-v2} and \eqref{eq:derivatives5-v2} hold  whenever $\norm{\omega_j}_{\Omega_j} \leq \epsilon$,
  the derivatives exist in the asserted spaces of polymer activities (cf.~below \eqref{eq:derivatives5-v2}),
  and,
  for $D\in \cB_j (Y)$, $Y\in \cS_j$,
\begin{equation}
\norm{Q_j (D, Y, \varphi')}_{h, T_j (Y, \varphi')} \leq C (\log L) \norm{K_j (Y)}_{h, T_j (Y)} e^{c_w \kappa_L w_j (D, \varphi')^2}. \label{eq:Q_j_bound}
\end{equation}
\end{lemma}

\begin{proof} 
Recall that $\mathfrak{U}'$ is $U_j$ or $\bar{U}_{j+1}$ or $\cE_{j+1} |B|$. 
The twice differentiability of $e^{\mathfrak{U}'}$ is a consequence of Lemma~\ref{lemma:Ujbound}, as we will show in detail below.
In the proof, we make the $\omega_j$-dependence of $\mathfrak{U}'$ explicit by writing $\mathfrak{U}' (\omega_j, B, \varphi)$,
write $D F(\dot{\omega}_j)$ for the $\omega_j$-derivative of $F$ in direction $\dot{\omega}_j$
and write similarly for the second derivative $D^2 F (\dot{\omega}_j, \ddot{\omega}_j)$.
Let $\norm{\dot{\omega}_j}_{\Omega_j} \leq \epsilon (\delta,  \beta, L)$ for small $\delta >0$, where $\epsilon(\delta,  \beta, L)$ is as in Lemma~\ref{lemma:Ujbound}. By Lemma~\ref{lemma:Ujbound} {and \eqref{eq:bound_exp-U}},
\begin{equation} \label{eq:easy-bd1}
	\norm{ e^{\mathfrak{U}'(\omega_j + \dot{\omega}_j, B, \varphi)} -  (1+\mathfrak{U}'(\dot{\omega}_j, B, \varphi)) e^{\mathfrak{U}'(\omega_j , B, \varphi )} }_{h, T_j (\varphi, B)} \leq C(\delta,  \beta, L) \norm{\dot{\omega}_j}_{\Omega_j}^2 e^{2\delta c_w \kappa_L w_j (B, \varphi)^2}
	,
\end{equation}
so $D e^{\mathfrak{U}'(\omega_j, B, \varphi)} (\dot{\omega}_j) = e^{\mathfrak{U}'(\omega_j, B, \varphi)} \mathfrak{U}'(\dot{\omega}_j, B, \varphi)$.
Moreover, as asserted, the differentiability is uniform in $\varphi$ after dividing by $G_j(B,\varphi)$
by Lemma~\ref{lemma:strong_regulator}, i.e., the derivatives exist in the space of polymer activities.
Similarly, for $\norm{\ddot{\omega}_j}_{\Omega_j} \leq \epsilon(\delta,  \beta, L)$,
\begin{multline} \label{eq:easy-bd2}
  \norm{ D e^{\mathfrak{U}'(\omega_j + \ddot{\omega}_j, B, \varphi)} (\dot{\omega}_j) -  (1+\mathfrak{U}'(\ddot{\omega}_j, B, \varphi)) D e^{\mathfrak{U}'(\omega_j , B, \varphi)} (\dot{\omega}_j) }_{h, T_j (\varphi, B)}
  \\
  \leq C(\delta,  \beta, L)  \norm{\dot{\omega}_j}_{\Omega_j} \norm{\ddot{\omega}_j}_{\Omega_j}^2 e^{2 \delta c_w \kappa_L w_j (B, \varphi)^2}
\end{multline}
so $D^2 e^{\mathfrak{U}'(\omega_j, B, \varphi)} (\dot{\omega}_j, \ddot{\omega}_j) = e^{\mathfrak{U}'(\omega_j, B, \varphi)} \mathfrak{U}'(\dot{\omega}_j, B, \varphi) \mathfrak{U}'(\ddot{\omega}_j, B, \varphi)$.
It follows that $e^{\mathfrak{U}'}$, $D e^{\mathfrak{U}'}$ are differentiable and $D e^{\mathfrak{U}'(\omega_j, B, \varphi)}, D^2 e^{\mathfrak{U}'(\omega_j, B, \varphi)}$ satisfy the desired bounds again using Lemma~\ref{lemma:Ujbound} and \eqref{eq:bound_exp-U}.

Since $J_j$, $\mathcal{E} K_j$ are linear functions of $K_j$, their differentiabilities follow from boundedness. 
 To obtain a bound for the derivative of $J_j$, first consider $Q_j (D, 
 {Y}, \varphi')$ for $D\in \mathcal{B}_j$, $Y \in \mathcal{S}_j$, cf.~\eqref{eq:Q_j_definition}. 
But because of \eqref{e:Loc-bounded},
\begin{equation}
\norm{\Loc_{Y,D} \Eplus { K_j} 
 (Y, \varphi'+\zeta)}_{h, T_j (Y, \varphi')} \leq C (\log L) \norm{
 { K_j}(Y)}_{h, T_j (Y)} e^{c_w \kappa_L w_j (D, \varphi')^2} ,
\end{equation}
hence $Q_j$ satisfies \eqref{eq:Q_j_bound} and is differentiable with the desired bound, i.e.,~its derivative is bounded in $\norm{\cdot}_{h, T_j (Y)}$-norm by $C(L)e^{c_w \kappa_L w_j (D, \varphi')^2} $. In view of \eqref{eq:J_j_definition}, it follows from this that
\begin{equation}
\norm{D  J_j (B, Z, \varphi') }_{h, T_j (B, \varphi')} \leq C(L) A^{-1} e^{c_w \kappa_L w_j (B, \varphi')^2} ,
\end{equation}
i.e.,~\eqref{eq:derivatives3-v2} holds.
The final inequality, \eqref{eq:derivatives5-v2}, is a direct result of \eqref{eq:derivatives3-v2}, but just using the fact that $\mathcal{E} K_j (\varphi', Y) =0$ whenever $Y\not\in \mathcal{S}_{j+1}$.

\end{proof}

\begin{lemma} \label{lemma:derivative_of_components_v2}
  Under the assumptions of Theorem~\ref{thm:local_part_of_K_j+1}, there exist $\epsilon \equiv \epsilon ( \beta,  L) >0$ (only polynomially small in $\beta$) and $C(A, L)$ such that \eqref{eq:derivatives4-v2}
holds whenever $\norm{\omega_j}_{\Omega_j} \leq \epsilon$.
\end{lemma}

\begin{proof}
Recall \eqref{eq:K_bar_definition} and first rewrite, for $X \in \cP^c_{j+1}$,
\begin{align}
\bar{K}_j (X, \varphi) &= \sum_{Y\in \cP_j^c}^{\bar{Y} =X} e^{U_j (X \backslash Y, \varphi)} K_j (Y, \varphi) + \sum_{Y\in \cP_j}^{\bar{Y} =X} e^{U_j (X \backslash Y, \varphi)} K_j^{(n)} (Y, \varphi)   \label{eq:bar_K_j_decomposition}
\end{align}
where
\begin{align}\label{eq:K-j-n}
K_j^{(n)}(Y,  \varphi) &= 1_{Y \in \cP_j \backslash \cP_j^c} K_j (Y, \varphi).
\end{align}
We will bound the two terms in \eqref{eq:bar_K_j_decomposition} separately.
Observe that, for $Y \in \mathcal{P}_j^c$, $\bar{Y}=X$ and any $\delta >0$, {applying submultiplicativity,} 
Lemma~\ref{lemma:Ujbound} (also see \eqref{eq:bound_exp-U}) implies 
\begin{align}
	\|  e^{U_j(X\backslash Y, \varphi)}\|_{h,T_j(X, \varphi)}  
  \leq e^{|X\backslash Y|_j+\delta c_w\kappa_L w_j(X\backslash Y,\varphi)^2 \norm{\omega_j}_{\Omega_j}}
\end{align}
whenever $\norm{\omega_j}_{\Omega_j} \leq \epsilon(\delta,   \beta,  L)$ for suitable $\epsilon(\delta,  \beta, L) $.
Using this bound, together with \eqref{eq:strong_regulator2}, Lemma~\ref{lemma:setsizes}, and estimating $|X\backslash Y|_j \leq L^2 |X|_{j+1}$, one obtains that 
\begin{align}
\norm{e^{U_j (X\backslash Y, \varphi)} K_j (Y, \varphi)}_{h, T_j (X, \varphi)} & \leq e^{|X\backslash Y|_j + \delta c_w \kappa_L \epsilon w_j (X\backslash Y, \varphi)^2 } G_j (Y, \varphi) A^{-|Y|_j} \norm{K_j}_{\Omega_j^K} \nnb
& \leq A^{8(1+\eta)} e^{L^2 |X|_{j+1}} G_j (X, \varphi) A^{-(1+\eta) |X|_{j+1}} \norm{K_j}_{\Omega_j^K} \label{eq:e^U_j_K_j_bound}
\end{align}
for some $\eta > 0$ and $\norm{\omega_j}_{\Omega_j} \leq \epsilon(\delta, L)$.
Hence for the first term of \eqref{eq:bar_K_j_decomposition},
\begin{equation}
\norm{\sum_{Y \in \cP_j^c}^{\bar{Y}=X} e^{U_j (X\backslash Y, \varphi)} K_j (Y, \varphi)}_{h, T_j (X,\varphi)} \leq C(A) e^{L^2 |X|_{j+1}} G_j (X, \varphi) \sum_{Y: \bar{Y}= X} \norm{K_j}_{\Omega_j^K} A^{-(1+\eta)|X|_{j+1}}
\end{equation}
but $\sum_{Y: \bar{Y}= X} 1 \leq  2^{|X|_j} \leq 2^{L^2 |X|_{j+1}}$ so this is bounded by $C(A) A^{-(1+\frac{\eta}{2}) |X|_{j+1} } G_{j} (X, \varphi) \norm{K_j}_{\Omega_j^K}$ for $A \geq C(L)$ sufficiently large.
Now by the linearity of the map $K_j \mapsto \sum_{Y\in \cP_j^c}^{\bar{Y}=X} e^{U_j (X\backslash Y)} K_j (Y)$, we immediately have,
for $\eta' = \eta/2$,
\begin{equation}\label{eq:nl-bdK1}
\Big\| \partial_{K_j} \Big[ \sum_{Y\in \cP_j^c}^{\bar{Y}=X} e^{U_j (X\backslash Y, \varphi)} K_j (Y, \varphi) \Big]  (\dot{K}_j) \Big\|_{h, T_j (X, \varphi)} \leq C (A) A^{-(1+\eta') |X|_{j+1}} \norm{\dot{K}_j}_{\Omega_j^K} G_j (X, \varphi) .
\end{equation}
Next, for $Y\in \cP_j \backslash \cP_j^c$ and $\bar{Y} = X$, we have by \eqref{eq:factorization} that
\begin{equation}
(K_j + \dot{K}_j) (Y) - K_j (Y) = \prod_{Z \in \operatorname{Comp}_j (Y)} (K_j (Z) + \dot{K}_j (Z)) - \prod_{Z \in \operatorname{Comp}_j (Y)} K_j (Z) 
\end{equation}
so, denoting by $\bar{K_j^{(n)}}$ the object defined by \eqref{eq:K_bar_definition} with $K_j^{(n)} $ from \eqref{eq:K-j-n} in place of $K_j$, we obtain
\begin{align}
& \Big\| \bar{(K_j + \dot{K}_j)^{(n)}} - \bar{K_j^{(n)}}  - \sum_{Y\not\in \cP_j^c}^{\bar{Y}=X} \sum_{Z\in \operatorname{Comp}_j (Y)} e^{U_j (X\backslash Y)} \dot{K}_j (Z) \prod_{Z' \in \operatorname{Comp}_j (Y\backslash Z)} K_j (Z')  \Big\|_{h, T_j (X)} \nnb
& 
\leq  \sum_{Y\not\in \cP_j^c}^{\bar{Y}=X} e^{|X\backslash Y|_j}  A^{-|Y|_j}  \Big( \big(\epsilon + \norm{\dot{K}_j}_{\Omega_j^K} \big)^{|\operatorname{Comp}_j (Y)| } -  \epsilon^{|\operatorname{Comp}_j (Y)|}  -  |\operatorname{Comp}_j (Y)| \norm{\dot{K}_j}_{\Omega_j^K} \epsilon^{|\operatorname{Comp}_j (Y)| - 1} \Big) 
\nnb
& \leq C \sum_{Y\not\in \cP_j^c}^{\bar{Y}=X} e^{|X\backslash Y|_j}  A^{-|Y|_j} |\operatorname{Comp}_j (Y)|^2   \norm{\dot{K}_j}_{\Omega_j^K}^2 \, \epsilon^{|\operatorname{Comp}_j (Y)| -2} \nnb
& \leq C' e^{L^2 |X|_{j+1}} \sum_{Y\not\in \cP_j^c}^{\bar{Y}=X} e^{- \frac{1}{2} |Y|_j}  A^{-|Y|_j}  \norm{\dot{K}_j}_{\Omega_j^K}^2 \, \epsilon^{|\operatorname{Comp}_j (Y)| -2}
\end{align}
where the second inequality holds under the assumption $\norm{\dot{K}_j}_{\Omega_j^K} \leq \frac{1}{2} \epsilon$.
By Lemma~\ref{lemma:setsizes_2}, this is bounded by $C(A)( L^2 e^{L^2} A^{-(1+\eta'')}  )^{|X|_{j+1}} \norm{\dot{K}_j}^2_{\Omega_j^K}$ for some $\eta''>0$,
and hence $\bar{K_j^{(n)}}$ is differentiable in $K_j$. The derivative satisfies a similar bound:
\begin{equation}\label{eq:nl-bdK2}
\Big\| \sum_{Y\not\in \cP_j}^{\bar{Y}=X} \sum_{Z\in \operatorname{Comp}_j (Y)} e^{U_j (X\backslash Y)} \dot{K}_j (Z)K_j (Y \backslash Z)  \Big\|_{h, T_j (X)} \leq C(A) A^{-(1+\eta''/2) |X|_{j+1}}  \norm{\dot{K}_j}_{\Omega_j^K} \, \epsilon
\end{equation}
when $A$ is chosen sufficiently large.
So only the derivative in $U_j$ is left to be studied. But 
\begin{align}
& \Big\| \partial_{U_j} \Big[\sum_{Y\in \cP_j}^{\bar{Y}=X} e^{U_j (X \backslash Y, \varphi)} K_j (Y, \varphi))\Big] (\dot{U}_j) \Big\|_{h, T_j (X,  \varphi)} \nnb
&\leq \sum_{Y\in \cP_j}^{\bar{Y}=X} \norm{\dot{U}_j (X\backslash Y, \varphi)}_{h, T_j (X\backslash Y, \varphi)} \norm{ e^{U_j (X \backslash Y)} K_j (Y, \varphi) }_{h, T_j (X, \varphi)} \nnb
& \leq C(\beta, L) \sum_{Y\in \cP_j}^{\bar{Y}=X} e^{L^2 |X|_{j+1}}  e^{-|Y|_j} G_j (X, \varphi) A^{-|Y|_j} \norm{K_j}_{\Omega_j^K}^{|\operatorname{Comp}_{j} (Y)|}  \norm{\dot{U}_j}_{\Omega_j^U} \nnb
& \leq C(\beta, L)  e^{L^2|X|_{j+1}} G_{j} (X, \varphi) \norm{\dot{U}_j}_{\Omega_j^U} (e L^2 A^{-(1+2\eta) / (1+\eta)} )^{|X|_{j+1}} \norm{K_j}_{\Omega_j^K}
\end{align}
where the final inequality follows again by Lemma~\ref{lemma:setsizes_2} assuming $\norm{K_j}_{\Omega_j^K} \leq \epsilon_{rb}$. 
Also, since $C(\beta, L)$ is a constant only polynomially large in $\beta$,  we obtain
\begin{align}
\Big\| \partial_{U_j} \Big[\sum_{Y\in \cP_j}^{\bar{Y}=X} e^{U_j (X \backslash Y, \varphi)} K_j (Y, \varphi))\Big] (\dot{U}_j) \Big\|_{h, T_j (X,  \varphi)}
\leq C (L)  A^{-(1+ \eta''') |X|_{j+1}} G_{j} (X, \varphi) \norm{\dot{U}_j}_{\Omega_j^U}
\label{eq:nl-bdK3}
\end{align}
after choosing $A$ large in $L$ and $\norm{K_j}_{\Omega_j^K}$ polynomially small in $\beta$.
Hence we have the bound for $\partial_{U_j} \bar{K}_j$ when $A$ is sufficiently large and together, \eqref{eq:nl-bdK1}, \eqref{eq:nl-bdK2} and \eqref{eq:nl-bdK3} yield \eqref{eq:derivatives4-v2}.
\end{proof}

\subsection{Product rule for polymer activities}

In preparation of the proof of Lemma~\ref{lemma:bound_on_M^k},
we first prove a product rule for polymer activities defined as in \eqref{eq:M^1_j+1}--\eqref{eq:M^2_j+1}.
For general polymer activities $K$, $K'$ with $\norm{K}_{\Omega_j^K}, \norm{K'}_{\Omega_j^K} < \infty$, the polymer activity defined by $K'' (X) = \sum_{Y\in \cP_j (X)} K_j (Y) K'_j (X\backslash Y)$ is not necessarily differentiable.
There are obstacles related both to the large field and the large set regulators.
The first obstacle is that it is not true that $G_j (X) G_j (Y) = G_{j} (X\cup Y)$ for general disjoint $X, Y \in \cP_j$.
The second obstacle is that summing over all $Y\in \cP_j (X)$ would create a combinatorial factor $2^{|X|_j}$ in the end, so taking the supremum over $X\in \cP_j^c$ would make $\norm{K''}_j$ diverge. Fortunately, we can circumvent these problems in \eqref{eq:M^1_j+1}--\eqref{eq:M^2_j+1}
due to the specific form of the polymers involved.
Sufficient conditions for the former operations are implied by the following conditions: 
\begin{itemize}
\item[$\textnormal{(Q)}$] Let $(\mathbb{X}, |\cdot|)$ be a normed space and $B^{\mathbb{X}}_{\epsilon}$ be the open ball with radius $\epsilon >0$.
Let $\varphi', \zeta$ be fields taking value in $\mathbb{R}^{\Lambda}$, and $\zeta \sim \mathcal{N} (0, \Gamma_{j+1})$.
Let $\mathcal{Q}$ be a partition of $\mathcal{B}_{j+1}$ 
and let $F_x (\, \cdot, \varphi', \zeta)$ and $f^y_x (\,  \cdot, \varphi', \zeta )$ be polymer activities supported on $\mathcal{Q}$ and labelled by $x \in B^{\mathbb{X}}_{\epsilon}$, $y\in \mathbb{X}$.
Assume that $F_0 (Q, \varphi', \zeta) \equiv 1_{Q = \emptyset}$, that $f^y_x$ is linear in $y$,
and that there are $C >0$, $\vec{\eta} (Q) \geq 0$ and a function $\psi: \mathbb{X}\to \R$ with $\psi(z) = o(1)$ as $z\rightarrow 0$ such that
\begin{itemize}
\item[(i)] (Boundedness) $\norm{f^y_x (Q, \varphi', \zeta)}_{h, T_{j} (Q, \varphi')} \leq C A^{- (1+\vec{\eta}(Q)) |Q|_{j+1}}   \mathcal{G}_{j} (Q, \varphi', \zeta)  |y|$;
\item[(ii)] (Continuity) $\norm{ (f^y_{x+z} - f^y_x )(Q, \varphi', \zeta)}_{h, T_{j} (Q, \varphi')} \leq C A^{- (1+\vec{\eta}(Q)) |Q|_{j+1}} \mathcal{G}_{j} (Q, \varphi', \zeta) |y| \psi(z)$;
\item[(iii)] (Derivative) $\norm{ (F_{x+y} - F_x - f_x^y )(Q, \varphi', \zeta)}_{h, T_{j} (Q, \varphi')} \leq C A^{- (1+\vec{\eta}(Q)) |Q|_{j+1}} \mathcal{G}_{j} (Q, \varphi', \zeta)  |y| \psi (y)$
\end{itemize}
where $\mathcal{G}_{j} (Q, \varphi' , \zeta)$ satisfies $\Eplus [ \prod_{Q \in \cQ (X) }  \mathcal{G}_{j} (Q, \varphi' , \zeta)] \leq C_G 2^{|X|_j} G_{j+1} (\bar{X},\varphi')$
and $Q\in \cQ(X)$ means $Q\in \cQ$ and $Q\subset X$.
Further assume that $\vec{\eta}(Q)$ takes value $0$ or $\eta_0 >0$ and if $\vec{\eta}(Q) =0$, then $Q\in \mathcal{S}_{j+1}$.
\end{itemize}
The flexibility of $\cG_j$ will save us from the problem of regulators  and the extra decay due to $\vec{\eta}(Q)$ will save us from the problem of combinatorial factor $2^{|X|_j}$. 
In practice, $\cG_j$ will be either $G_j$ or $e^{c_w \kappa_L (w_j )^2}$.

\begin{proposition}[Product rule] \label{prop:differentiability_of_product_v2}
Let $\mathbb{X}$, $\cQ$, $\cQ(X)$, $f$, $F$, $\psi$ and $\vec{\eta}$ be as in $\textnormal{(Q)}$ for $X\in \cP_{j+1}$.
Given collection of parameters $\vec{x} = \{ x(Q) \in \mathbb{X} : Q \in \mathcal{Q} \}$, define for $X\in \cP_{j+1}$,
\begin{align}
\mathcal{L}_{\vec{x}} (X, \varphi') =
\begin{array}{ll}
 \begin{cases}
 \Eplus \Big[ \prod_{Q\in \mathcal{Q}(X) } F_{x(Q)} (Q, \varphi', \zeta) \Big] & \textnormal{if} \;\; |\mathcal{Q} (X)| \geq 2, \; X = \cup_{Q\in \mathcal{Q}(X)} Q \\
 0  & \textnormal{otherwise} .
\end{cases}
\end{array}
\end{align}
Then for $A$ sufficiently large and $\epsilon$ sufficiently small (polynomially in $L$, $A$, $C$), the partial derivatives of $\vec{x} \mapsto \cL_{\vec{x}}$ exist (as a map from $\mathbb{X}$ to the space of polymer activities of finite $\norm{\cdot}_{h, T_{j+1} ({X})}$-norm),
the partial derivatives in directions $P \in \cQ$ are given by
\begin{equation}
d^{\, y}_{P} := \partial_{y, P} \mathcal{L}_{\vec{x}} := \Eplus \Big[ f^y_{x(P)} (P, \varphi', \zeta) \prod_{Q \in \mathcal{Q}(X \backslash P) } F_{x(Q)} (Q,\varphi',\zeta) \Big]  \quad \textnormal{if} \;\; |X|_{j+1} \geq 2,
\end{equation}
and they are continuous
in the domain $\{ \vec{x} : |x(Q)| < \epsilon, \,\, \forall Q \in \mathcal{Q} \}$. Moreover, in the case $x(Q) \equiv x$, if we let $\mathcal{L}_x = \mathcal{L}_{\vec{x}}$, then $\mathcal{L}_x$ is differentiable in $\{ |x| < \epsilon \}$, the derivative satisfies the bound
\begin{equation}
  \norm{D \mathcal{L}_x (X)}_{h, T_{j+1} ({X})} \leq C A^{-(1+\eta_0 /2)|X|_{j+1}} |x|^{\max\{1, \frac{ |\mathcal{Q}(X)| -1 }{2} \}},
\end{equation}
and $D \mathcal{L}_x (X)$ is continuous in $x$.
\end{proposition}
\begin{proof}
Since $C_G$ only contributes as a constant factor, we may assume that $C_G =1$.
Also, all $X$ used below are assumed to satisfy $|\mathcal{Q}(X)| \geq 2$ and $\cup_{Q\in \mathcal{Q}(X)} Q = X$
which is sufficient by the definition of $\cL_{\vec{x}}$.
We first show that $d^{\, y}_{P}$ has finite norm.  Indeed,
{the assumption on $\cG_j$ gives}
\begin{align}
\norm{d^{\, y}_{P} (\varphi', X)}_{h, T_{j+1} ({X}, \varphi')} 
& \leq C^{|\mathcal{Q} (X)|} A^{-\sum_{Q\in \mathcal{Q}(X)} (1+ \vec{\eta}(Q))|Q|_{j+1}  } \Eplus \Big[ \prod_{Q \in \cQ (X)} \cG_j (Q, \varphi' , \zeta) \Big] \, |x|^{|\mathcal{Q}(X)| -1} |y| \nnb
& \leq C^{|\mathcal{Q}(X)|} A^{\eta_0 \sum_{Q\in \mathcal{Q}(X)}^{\vec{\eta}(Q)=0} |Q|_{j+1} } (2^{L^2}  A^{-(1+\eta_0)})^{|X|_{j+1}} G_{j+1} ({X}, \varphi') |x|^{|\mathcal{Q}(X)| -1}|y|  \nnb
& \leq C^{|\mathcal{Q}(X)|} A^{4 \eta_0 |\cQ (X)| } (2^{L^2}  A^{-(1+\eta_0)})^{|X|_{j+1}} G_{j+1} ({X}, \varphi') |x|^{|\mathcal{Q}(X)| -1}|y|  .
\end{align}
Thus $(y\mapsto d^{\, y}_P)$ is a bounded linear map from $\mathbb{X}$ to the polymer activities of finite $\norm{\cdot}_{h, T_{j+1} ({X})}$-norm.

To show that $d^{\, y}_P$ is the derivative of $\cL_{x}$,
let $\delta \mathcal{L}_{\vec{x}, y} := \mathcal{L}_{\vec{x} + y \delta_{P, Q}} - \mathcal{L}_{\vec{x}} - d^{\, y}_{P}$.
Then by essentially the same computation as above,
\begin{align}
\norm{ \delta \mathcal{L}_{\vec{x}, y} (X, \varphi') }_{h, T_{j+1} ({X}, \varphi')} &= \norm{ \Eplus \Big[ ( F_{x(P) + y} - F_{x(P)} - f^y_{x(P)} ) (P) \prod_{Q \in \mathcal{Q} (X \backslash P) } F_{x(Q)} (Q)  \Big]  }_{h, T_{j+1} (X, \varphi')} \nnb
& \leq (CA^{4\eta_0})^{|\mathcal{Q}(X)|} |x|^{|\mathcal{Q}(X)| -1} A^{-(1+ \eta_0) |X|_{j+1}} |y| \Eplus \Big[ \prod_{Q\in \cQ (X)} \cG_j (Q, \varphi' , \zeta) \Big] \psi (y) \nnb
& \leq (CA^{4\eta_0})^{|\mathcal{Q}(X)|} |x|^{|\mathcal{Q}(X)| -1} (2^{L^2} A^{-(1+ \eta_0)} )^{|X|_{j+1}} G_{j+1} ({X}, \varphi') |y| \psi(y),  \label{eq:delta_L_estimate}
\end{align}
proving the existence of the partial derivatives of $\cL_{\vec{x}}$ as a function from $\mathbb{X}$ to the space of polymer activities of finite $\norm{\cdot}_{h, T_{j+1} ({X})}$-norm.
To see the differentiability of $\mathcal{L}_x$, let
\begin{align}
d^{\, y} (X, \varphi') := D \mathcal{L}_x (X, \varphi') := \sum_{P\in \mathcal{Q} (X)} \partial_{y, P} \mathcal{L}_{\vec{x}} |_{x(Q) \equiv x}.
\end{align}
Then
\begin{align}
\label{eq:d^y_bound}
\norm{d^{\, y} (X)}_{h, T_{j+1} ({X})} \leq |\mathcal{Q}(X)| (CA^{4\eta_0})^{|\mathcal{Q}(X)|} |x|^{|\mathcal{Q}(X)| -1} (2^{L^2} A^{-(1+ \eta_0)} )^{|X|_{j+1}} G_{j+1} ({X}, \varphi') |y|
\end{align}
and hence $(y\mapsto d^{\, y} )$ is bounded linear from $\mathbb{X}$ to the space of polymer activities with finite $\norm{\cdot}_{h, T_{j+1}({X})}$-norm. 
Also applying \eqref{eq:delta_L_estimate} multiple times shows that
\begin{align}
\norm{(\mathcal{L}_{x+y} - \mathcal{L}_x -  d^{\, y} ) (X)}_{h, T_{j+1}({X})} \leq (CA^{4\eta_0})^{|\mathcal{Q}(X)|} |\mathcal{Q}(X)| \, |x|^{|\mathcal{Q}(X)| -1} (2^{L^2} A^{- (1+ \eta_0)})^{|X|_{j+1}} |y| \psi(y)
\end{align}
proving differentiability of $\mathcal{L}_x$. 
The bound for the derivative is obtained from \eqref{eq:d^y_bound} once we choose $|x| < \epsilon$ small and $A$ large so that
\begin{align}
A^{|X|_{j+1} }\norm{D \mathcal{L}_x (X) }_{h, T_{j+1} ({X})} \leq (CA^{4\eta_0})^{|\mathcal{Q}(X)|} |\mathcal{Q}(X)| A^{-\frac{\eta_0}{2} |X|_{j+1}} |x|^{|\mathcal{Q}(X)| -1} \leq C' A^{-\frac{\eta_0}{2} |X|_{j+1}} |x|^{\frac{|\mathcal{Q}(X)| -1}{2}}.
\end{align}
But for the case $|\mathcal{Q}(X)| =2$, one could just have bounded the left-hand by $C' A^{-\frac{\eta_0}{2} |X|_{j+1}} |x|$ instead.

The continuity of the derivative follows from the assumption on the continuity of $f$.
\end{proof}

\subsection{Proof of Lemma~\ref{lemma:bound_on_M^k}}
\label{subsec:estimate_Nj}

In this subsection, $\cE_{j+1} |X|, U_j, \bar{U}_{j+1}, K_j, \bar{K}_j,\cE K_{j}$ and $J_j$ will always be a function of $x$ implicitly.
 For brevity, we define the following expressions which appear as part of the definitions of the $\MM_{j+1} ^{(k)}$:
\begin{align}
  F(\mathfrak{K}_j , T, X) &= e^{ \cE_{j+1} |X| + \bar{U}_{j+1} (X \backslash T) } \label{eq:F_omega_T_X} \\
   H(\mathfrak{K}_j  , X_0, X_1, Z, (B_{Z''})) &= \Eplus \Big[ (e^{U_j} - e^{\bar{U}_{j+1}})^{X_0} (\bar{K}_j - \cE K_j)^{[X_1]}  \Big]
    \nnb
    &\qquad\qquad \times \prod_{Z'' \in \operatorname{Comp}_{j+1} (Z)} J_j (B, Z'') \label{eq:G_omega_X0_X1_Z_BZ}.
\end{align}

\begin{lemma} \label{lemma:N_j_derivative_primitive_bound}
Under the same assumptions as in Lemma~\ref{lemma:bound_on_M^k}, 
\begin{align}
\mathcal{A}_1 (\mathfrak{K}_j (x), X) & = 1_{|X|_{j+1} = 1}  \Eplus [ e^{U_j (X)} - U_j (X) - e^{ \bar{U}_{j+1} (X) } + \bar{U}_{j+1} (X) ]  \\
\mathcal{A}_2 (\mathfrak{K}_j(x), T) & = \sum_{X_0, X_1, Z, (B_{Z''})}^{\# (X_0, X_1, Z) \geq 2} H(x, X_0, X_1, Z, (B_{Z''}))
\end{align}
with $T = X_0 \cup X_1 \cup (\cup_{Z'' \in \operatorname{Comp}_{j+1} (Z)} B_{Z''}^*)$
are differentiable in $x$ with
\begin{align}
& \norm{D H (\mathfrak{K}_j (x),  X_0, X_1, Z, (B_{Z''}))}_{h, T_{j+1} (T )} \leq C_A A^{-(1+\frac{\eta}{4} ) |X|_{j+1}} \xnorm^{\max\{1, \frac{\#(X_0, X_1, Z) -1}{2} \}} \label{eq:DH_bound} \\
& \norm{D \mathcal{A}_l (\mathfrak{K}_j (x), X)}_{h, T_{j+1} (T)} \leq C A^{- (1+ \eta ) |X|_{j+1}} \xnorm,   \quad l=1,2
\end{align}
for some $\eta >0$,  $C_A \equiv C_A(A, L)$, $C \equiv C(L)$ and $H$ defined by \eqref{eq:G_omega_X0_X1_Z_BZ}.
Moreover, each derivative is continuous in $x$.
\end{lemma}
\begin{proof}
The differentiability of $\cA_1$ follows from \eqref{eq:Ujbound} and \eqref{eq:derivatives1-v2}. To see its bound, let $X=B \in \mathcal{B}_{j+1}$.
We have $\Eplus [ e^{U_j (B)} - U_j(B) - e^{ \bar{U}_{j+1} (B) } + \bar{U}_{j+1} (B) ] = \Eplus \big[ \big( (e^{U_j}-1- U_j) + (e^{\bar{U}_{j+1}} -1 - \bar{U}_{j+1} ) \big) (B)\big]$,
and \eqref{eq:derivatives2-v2} implies
\begin{equation}
\norm{ D \Eplus [ (e^{U_j}  - 1- U_j)] (B, \varphi') }_{h, T_{j+1} (B, \varphi')} \leq C \Eplus [ e^{c_w \kappa_L w_j (B, \varphi' + \zeta )^2 } ] \xnorm \leq C' G_{j+1} ( B, \varphi') \xnorm
\end{equation}
where the second inequality follows from $\Eplus [e^{c_w \kappa_L w_j (B, \varphi ' + \zeta)^2}] \leq \Eplus [ G_j (B, \varphi ' + \zeta) ] \leq 2^{L^2} G_{j+1} (B, \varphi')$,
see Lemma~\ref{lemma:strong_regulator} and Proposition~\ref{prop:E_G_j}.
The same estimate applies to $e^{\bar{U}_{j+1}} - 1 - \bar{U}_{j+1}$ and hence
\begin{align}
\norm{D \mathcal{A}_1 (\mathfrak{K}_j (x), B)}_{h, T_{j+1} (B)} \leq C \xnorm.
\end{align}
To show the differentiability of $\mathcal{A}_2$, we can apply Proposition~\ref{prop:differentiability_of_product_v2}. To see this, expand
\begin{equation} \label{eq:H_expansion}
\begin{split}
H (\mathfrak{K}_j (x)  , X_0, & X_1, Z, (B_{Z''})) = \sum_{ Y_0, Y_1} (-1)^{|Y_0|_{j+1} + |\operatorname{Comp}_{j+1} (Y_1)|} \\
& \times \Eplus \Big[ (e^{U_j} - 1)^{X_0 \backslash Y_0 } (e^{\bar{U}_{j+1}} -1)^{Y_0} ( \bar{K_j} )^{ [X_1 \backslash Y_1 ]} (\mathcal{E}K_j)^{[Y_1]} \Big] \prod_{Z'' \in \operatorname{Comp}_{j+1} (Z)} J_j (B, Z'')
\end{split}
\end{equation}
where the sum runs over $Y_0 \in \mathcal{P}_{j+1} (X_0)$ and $\operatorname{Comp}_{j+1} (Y_1) \subset \operatorname{Comp}_{j+1} (X_1)$.
For fixed $X_0$, $X_1$, $Z$, $Y_0 \subset X_0$ and $Y_1 \subset X_1$, let
$\cQ = \mathcal{B}_{j+1} (X_0) \cup \operatorname{Comp}_{j+1} (X_1) \cup \operatorname{Comp}_{j+1} (Z) \cup \mathcal{B}_{j+1} ( T^c)$ where $T = X_0 \cup X_1 \cup Z$ and define
\begin{equation}
F_{x} (Q, \varphi', \zeta) =  \begin{array}{ll}
\begin{cases}
e^{U_j (Q, \varphi' + \zeta)} -1 & \text{if} \;\; Q \in \mathcal{B}_{j+1} (X_0 \backslash Y_0) \\
e^{\bar{U}_{j+1} (Q, \varphi' )} -1 & \text{if} \;\; Q \in \mathcal{B}_{j+1} (Y_0) \\
\bar{K}_j (Q, \varphi' + \zeta) & \text{if} \;\; Q \in \operatorname{Comp}_{j+1} (X_1 \backslash Y_1) \\
\mathcal{E} K_{j+1} (Q, \varphi') & \text{if} \;\; Q \in \operatorname{Comp}_{j+1} (Y_1) \\
J_j (B_Q, Q, \varphi') & \text{if} \;\; Q \in \operatorname{Comp}_{j+1} (Z) \\
1 & \text{if} \;\; Q \in \mathcal{B}_{j+1} (T^c),
\end{cases}
\end{array} 
\end{equation}
and
\begin{equation}
\cG_{j} (Q, \varphi', \zeta) = \begin{array}{ll}
\begin{cases}
e^{c_w \kappa_L w_j (Q, \varphi' + \zeta )} & \text{if} \;\; Q \in \cB_{j+1} (X_0 \backslash Y_0) \\
e^{c_w \kappa_L w_j (Q, \varphi')} & \text{if} \;\; Q \in \mathcal{B}_{j+1} (Y_0) \cup \operatorname{Comp}_{j+1} (Y_1 ) \cup \operatorname{Comp}_{j+1} (Z) \\
G_j (Q, \varphi' + \zeta) & \text{if} \;\; Q \in \operatorname{Comp}_{j+1} (X_1 \backslash Y_1) \\
1 & \text{if} \;\; Q \in \cB_{j+1} (T^c)  .
\end{cases} 
\end{array} 
\end{equation}
Then Proposition~\ref{prop:differentiability_of_product_v2} with the assumption that $\mathfrak{K}_j \in \cX_j^{\mathfrak{K}} (\mathbb{X})$
(i.e., it satisfies the bounds \eqref{eq:Ujbound_1}--\eqref{eq:derivatives5-v2}) shows $\cA_2$ is differentiable and
\begin{align}
\norm{D \mathcal{A}_2 (\mathfrak{K}_j(x), X)}_{h, T_{j+1} (X)} & \leq \sum_{X_0, X_1, Y_0, Y_1, Z, (B_{Z''})}^{\# (X_0, X_1, Z) \geq 2} E(X_0, X_1, Y_0, Y_1, Z, (B_{Z''}))  \\ 
E(X_0, X_1, Y_0, Y_1, Z, (B_{Z''})) &= C_A  A^{-(1+\eta/2)|T|_{j+1} } \xnorm^{\max\{1, \frac{\# (X_0, X_1, Z) -1}{2} \} }
\end{align}
First consider the cases $\#(X_0, X_1, Z) \geq 4$ and note that $\max\{1, \frac{\# (X_0, X_1, Z) -1}{2} \} = \frac{\# (X_0, X_1, Z) -1}{2}$.
Since $X_0, X_1, Y_0, Y_1, (B_{Z''}), Z \backslash (\cup_{Z''} B_{Z''})$ and $X \backslash T$ partition $X$, one may bound the sum by a sum running over partitions of $X$ partitioned into 7 subsets. This gives a crude combinatorial bound
\begin{align}
 \sum_{X_0, X_1, Y_0, Y_1, Z, (B_{Z''})}^{\# (X_0, X_1, Z) \geq 4} E(X_0, X_1, Y_0, Y_1, Z, (B_{Z''}))  \leq C_A 7^{|X|_{j+1}} \sup A^{-(1+ \frac{\eta}{2}) |T|_{j+1}} \xnorm^{\frac{\# (X_0, X_1, Z) -1}{2} }
\end{align}
where the supremum also runs over the choices of $X_0, X_1, Y_0, Y_1, Z, (B_{Z''})$.
Also with the assumption $7 A^{-\eta/4} \leq 1$, this can also be bounded by
\begin{equation}
C_A A^{-(1+\frac{\eta}{4})|X|_{j+1}}  \sup A^{(1+\frac{\eta}{2}) |X\backslash T|_{j+1} } \xnorm^{\frac{\# (X_0, X_1, Z) -1}{2} }.
\end{equation}
But $|X\backslash T|_{j+1} = |\cup_{Z''} (B^*_{Z''} \backslash Z'' ) |_{j+1} \leq 48 |Z|_{j+1}$, so $A^{|X\backslash T|_{j+1}}  \leq A^{48|Z|_{j+1}}$.
Since each connected component of $Z$ is a small set, it follows that $|Z|_{j+1} \leq 4 |\operatorname{Comp}_{j+1} (Z)|$,
and hence the condition $A^{192(1+\eta/4)} \xnorm^{1/8} \leq 1$ gives
\begin{align}
A^{(1+\frac{\eta}{4}) |X|_{j+1}}  \sum^{\#(X_0, X_1, Z) \geq 4}_{X_0, X_1, Y_1, Y_2, Z, (B_{Z''})} E(X_0, X_1, Y_0, Y_1, Z, (B_{Z''})) \leq C \xnorm .
\end{align}
For the cases $\#(X_0, X_1, Z) \in \{2,3\}$, we have $|X_0|_{j+1} \leq 3$ and $|\cup_{Z''} B_{Z''}^*|_{j+1} \leq 3\times 49$ so
\begin{equation}
  A^{(1+\frac{\eta}{4})|X|_{j+1}} \sum_{X_0, X_1, Y_1, Y_2, Z, (B_{Z''})}^{\# (X_0, X_1, Z) \in \{2,3 \}} E(X_0, X_1, Y_0, Y_1, Z, (B_{Z''})) \leq C_A \xnorm
\end{equation}
by just setting $C_A$ sufficiently large depending on $A$.

The continuity of $D G$ and $D \cA_l$ is a result of the continuity of the derivative in Proposition~\ref{prop:differentiability_of_product_v2}.
\end{proof}

\begin{proof}[Proof of Lemma~\ref{lemma:bound_on_M^k}, case $k\in \{1,2,3\}$]
Consider the function
\begin{equation}
 M_{j+1} (x,  x') = \sum_{X_0, X_1, Z, (B_{Z''})}^{\# (X_0, X_1, Z) \geq 2} F(\mathfrak{K}_j (x) , T, X) H(\mathfrak{K}_j ( x' ) , X_0, X_1, Z, (B_{Z''})) 
\end{equation}
and recall that $F$ and $H$ are defined by \eqref{eq:F_omega_T_X} and \eqref{eq:G_omega_X0_X1_Z_BZ},
and we emphasise that above $F$ uses $x$ to define $E_{j+1}$ and $U_{j+1}$ while
$H$ uses $x'$, so that $\MM_{j+1}^{(1)}(\mathfrak{K}_j (x)) = M_{j+1} (x, x)$.
By Lemmas~\ref{lemma:derivative_of_components_v2} and~\ref{lemma:N_j_derivative_primitive_bound}, $M_{j+1} (x, x')$, $\MM_{j+1}^{(2)} (\mathfrak{K}_j (x'))$ and $\MM_{j+1}^{(3)}(\mathfrak{K}_j (x'))$ are differentiable in $x'$ with the desired bounds.
For the $x$ derivative of $M_{j+1} (x, x')$, we justify the differentiability more carefully: let
\begin{align}
& f^{\dot{x}}_{x} (T, X, \varphi') = (D \cE_{j+1} (\dot{x}) |X| + D \bar{U}_{j+1} (X \backslash T) (\dot{x}) ) F(\mathfrak{K}_j (x), T, X) \\
& m^{\dot{x}}_{x} (X, \varphi') = \sum_{X_0, X_1, Z, (B_{Z''})}^{\# (X_0, X_1, Z) \geq 2}  f^{\dot{x}}_{x} (T, X) H(\mathfrak{K}_j (x') , X_0, X_1, Z, (B_{Z''})) .
\end{align}
Letting $\delta_{\dot{x}} M_{j+1} (x, x') = M_{j+1} (x + \dot{x}, x') - M_{j+1} (x, x') - m^{\dot{x}}_{x}$, the bounds \eqref{eq:Ujbound} and \eqref{eq:DH_bound} give 
\begin{align}
& \norm{ \delta_{\dot{x}} M_{j+1} (x, x')(X, \varphi') }_{h, T_{j+1} (X, \varphi')} \nnb
& \leq C_A \sum_{X_0, X_1, Z, (B_{Z''})}^{\# (X_0, X_1, Z) \geq 2} C^{|X|_{j+1}} (A^{-1-\eta})^{|X|_{j+1}}  e^{c_w \kappa_L w_j (X \backslash T, \varphi')^2} G_{j+1} (T, \varphi')  |\dot{x}|^2 |x'|^{\max\{2, \frac{\# (X_0, X_1, Z) +1}{2} \}} \nnb
&\leq C_A \sup (5C)^{|X|_{j+1}} (A^{-1-\eta})^{|X|_{j+1}}  G_{j+1} (X, \varphi')  |\dot{x}|^2 |x'|^{\max\{2, \frac{\# (X_0, X_1, Z) +1}{2} \}}
\end{align}
where the supremum ranges over $X_0, X_1, Z, (B_{Z''})$ with $\#(X_0, X_1, Z) \geq 2$.
Choosing $5CA^{-\eta} \leq 1$,
\begin{align}
\norm{ \delta_{\dot{x}} M_{j+1} (x, x')(X, \varphi') }_{h, T_{j+1} (X, \varphi')} \leq C_A A^{-|X|_{j+1}} G_{j+1} (X, \varphi') |\dot{x}|^2 |x'|^2.
\end{align}
Therefore $M_{j+1} (x,x') (X)$ is differentiable in $x$ and the same computation gives the bound 
\begin{align}
\norm{\partial_{x} M_{j+1}(x, x')(X)}_{h, T_{j+1} (X)} \leq C_A |x'|^2.
\end{align}
The continuity of the derivatives are results of continuity of derivatives in Lemma~\ref{lemma:N_j_derivative_primitive_bound}.
\end{proof}

\begin{proof}[Proof of Lemma~\ref{lemma:bound_on_M^k}, case $k=4$]
We may alternatively write $\mathfrak{M}_{j+1}^{(4)} = \Eplus M^{(4)}_{-} := \Eplus [ M^{(4,1)}_{-} + M^{(4,2)}_{-}]$ where
\begin{align} 
& M^{(4,1)}_{-} (\mathfrak{K}_j (x), X, \varphi', \zeta) = \sum_{Y\in \cP_j}^{\bar{Y}=X} 1_{Y\in \cP_j^c} (e^{U_j (X\backslash Y,\varphi' + \zeta)} - 1 ) K_j (Y, \varphi' + \zeta) \\
& M^{(4,2)}_{-} (\mathfrak{K}_j (x), X, \varphi', \zeta) = \sum_{Y\in \cP_j}^{\bar{Y} = X} 1_{Y\not \in \cS_j} 1_{Y\not\in \cP_j^c} e^{U_j (X\backslash Y,\varphi' + \zeta)} \prod_{Z\in \operatorname{Comp}_j (Y)} K_j (Z, \varphi' + \zeta)
\end{align}
as $1_{Y\not \in \cS_j} 1_{Y\not\in \cP_j^c} = 1_{Y\not\in \cP_j^c}$.
By \eqref{eq:Ujbound},
\begin{align}
\norm{ D e^{U_j (X \backslash Y , \varphi)} (\dot{x}) }_{h, T_j (X, \varphi)} & \leq C(\delta, L) e^{  (1+ C(\delta, L) \xnorm ) ( |X\backslash Y|_j + \delta c_w \kappa_L w_j (X \backslash Y, \varphi)^2  )  } |\dot{x}| \nnb
& \leq C(L)  e^{2|X\backslash Y|_j} e^{c_w\kappa_L w_j (X \backslash Y, \varphi)} |\dot{x}| \label{eq:M^4_bound_1}
\end{align}
for $\delta < 1/2$ and $\xnorm \leq \frac{1}{C(\delta, L)}$ and then the mean value theorem gives
\begin{align}
\norm{ e^{U_j (X\backslash Y, \varphi)} -1 }_{h, T_j (X, \varphi)} \leq C(L) e^{2|X\backslash Y|_j} e^{c_w\kappa_L w_j (X \backslash Y, \varphi)} \xnorm. \label{eq:M^4_bound_2}
\end{align}
So using \eqref{eq:M^4_bound_1} to bound $\partial_{U_j} M_{-}^{(4)}$ and \eqref{eq:M^4_bound_2} to bound $\partial_{K_j} M_{-}^{(4)}$,  and since $x \mapsto K_j$ is linear and bounded, we see that
\begin{multline}
\norm{ D M_{-}^{(4,1)} (\mathfrak{K}_j (x), X, \varphi', \zeta) }_{h, T_j (X, \varphi')} \\
\leq C(L) \sum_{Y \in \mathcal{P}_j}^{\bar{Y}=X} 1_{Y\in \cP_j^c}  e^{2 |X\backslash Y|_{j+1}} e^{c_w \kappa_L w_j (X\backslash Y)^2} \xnorm A^{-|Y|_j} G_j (Y, \varphi' + \zeta).
\end{multline}
If $Y\in \cS_j$, then $X\in \cS_{j+1}$ and $|X|_{j+1} \leq |Y|_j$ so the summand on the right-hand side is bounded by $C'(A, L) A^{-(1+ \eta)|X|_{j+1}} \xnorm$ where $C'(A, L) = C(L) A^{4\eta}$.
If $Y \not\in \mathcal{S}_j$, then Lemma~\ref{lemma:setsizes} implies $|Y|_{j} \geq \frac{\eta}{2(1+\eta)} |Y|_j + \frac{2+\eta}{2} |X|_{j+1}$ so that
\begin{align}
C(L) \sum_{Y \in \mathcal{P}_j}^{\bar{Y}=X} 1_{Y\not\in\mathcal{S}_j} (2e^2)^{|X|_{j+1}} A^{-|Y|_j} \xnorm \leq C(L) (2e^2 A^{-\frac{2+\eta}{2}})^{|X|_{j+1}} \xnorm \sum_{Y\in \cP_j}^{\bar{Y}=X} A^{-\frac{\eta}{2(1+\eta)}|Y|_j} .
\end{align}
But for $L^2 \leq  A^{\frac{\eta}{2(1+\eta)}}$,
\begin{align}
\sum_{Y : \bar{Y}=X} A^{-\frac{\eta}{2(1+\eta)}|Y|_j} \leq (1+A^{-\frac{\eta}{2(1+\eta)}})^{|X|_{j}} \leq e^{L^2 A^{-\frac{\eta}{2(1+\eta)}} |X|_{j+1} } \leq e^{|X|_{j+1}}
\end{align}
so we may conclude
\begin{align} \label{eq:DM_-^4_bound}
\norm{ D M_{-}^{(4,1)} ( \mathfrak{K}_j (x), X, \varphi') }_{h, T_j (X, \varphi')}  \leq C(A, L) (e^3 A^{-(1+ \frac{\eta}{2})})^{|X|_{j+1}} \xnorm G_j (X, \varphi' + \zeta)
\end{align}
For $M_-^{(4,2)}$, we have
\begin{align}
& \norm{ D M_-^{(4,2)} (\mathfrak{K}_j (x ), X, \varphi', \zeta)  }_{h, T_{j+1} (X, \varphi')} \nnb
& \leq C(L) \sum_{Y\in \cP_j \backslash \cP_j^c}^{\bar{Y}=X} 1_{Y\not\in \cS_j} |\operatorname{Comp}_j (Y)| A^{-|Y|_j}  e^{2|X\backslash Y|_j }  e^{c_w \kappa_L w_j (X\backslash Y, \varphi)} G_j (Y, \varphi' + \zeta) \xnorm^{|\operatorname{Comp}_j (Y)| - 1}  \nnb
& \leq C(L) G_{j} (X, \varphi' + \zeta) e^{2L^2 |X|_{j+1}} \sum_{Y\in \cP_j \backslash \cP_j^c}^{\bar{Y}=X} 1_{Y\not\in \cS_j}  (e^2 A/2)^{-|Y|_j} \xnorm^{|\operatorname{Comp}_j (Y)| - 1} .
\end{align}
But by Lemma~\ref{lemma:setsizes_2}, this is bounded by
\begin{align}
C(L) G_{j} (X, \varphi' + \zeta) e^{2L^2 |X|_{j+1}} (2 e^{-1} L^2 A^{-1-\eta})^{|X|_{j+1}} \xnorm
\end{align}
for some $\eta >0$. Hence for sufficiently large $A$, we have
\begin{align} \label{eq:DM_-^4,2_bound}
\norm{ D M_{-}^{(4,2)} ( \mathfrak{K}_j (x), X, \varphi',\zeta) }_{h, T_{j+1} (X, \varphi')}  \leq C( L) A^{-(1+\eta/2)|X|_{j+1}} \xnorm G_j (X, \varphi' + \zeta)
\end{align}
and the same bounds also imply the differentiability of $D \MM_{j+1}^{(4)}$ with bound
\begin{align}
\norm{ D \mathfrak{M}_{j+1}^{(4)} ( \mathfrak{K}_j (x) ) }_{h, T_{j+1} (X, \varphi')}  \leq C(A, L) A^{-(1+ \frac{\eta}{3}) |X|_{j+1} } \xnorm G_{j+1} (X, \varphi') .
\end{align}
The continuity of the derivative is a consequence of continuity of derivatives in Lemma~\ref{lemma:N_j_derivative_primitive_bound}.
\end{proof}

\subsection{Proof of Theorem~\ref{thm:local_part_of_K_j+1}: continuity in $s$}
\label{sec:rgmap-continuitys}

The proof of continuity in $s$ of the renormalisation group map uses the following lemma which
extends Lemma~\ref{lemma:linearity_of_expectation}.

\begin{lemma} \label{lemma:stability_of_expectation}
  For any $C>0$ and any scale-$j$ polymer activity $F$
  that is invariant under translations and satisfies $\|F\|_{h,T_j}\leq C$,
  for $|s|, |s'| < \theta_J \epsilon_s$,
\begin{equation}
\lim_{s' \rightarrow s}  \sup_{X\in \mathcal{P}_j^c} \Big( \frac{A}{3}  \Big)^{|{X}|_{j}}
\norm{ \E_{\Gamma_{j+1}(s')} [F(X,  \cdot + \zeta )] - \mathbb{E}_{\Gamma_{j+1} (s)} [F( X, \cdot+ \zeta)] }_{h, T_{j+1} (\bar{X})} = 0 \label{eq:stability_of_expectation1}
\end{equation}
and the limit is uniform in $F$ satisfying $\|F\|_{h,T_j} \leq C$ (and in particular in the size of $\Lambda_N$).
An analogous statement holds if we assume
\begin{equation} \label{eq:stability_of_expectation1_j+1}
\sup_{X \in \cP^c_{j+1}} A^{|X|_{j+1}} \sup_{\varphi} G_j (X, \varphi)^{-1} \norm{F(X, \varphi)}_{h, T_{j+1} (X, \varphi)} \leq C
\end{equation}
with the conclusion now being
\begin{equation}
\lim_{s' \rightarrow s}  \sup_{X\in \mathcal{P}_{j+1}^c} \Big(\frac{2 A}{3 \cdot 2^{L^2}} \Big)^{|{X}|_{j+1}}
\norm{ \E_{\Gamma_{j+1}(s')} [F(X,  \cdot + \zeta )] - \mathbb{E}_{\Gamma_{j+1} (s)} [F( X, \cdot+ \zeta)] }_{h, T_{j+1} ({X})} = 0. \label{eq:stability_of_expectation2}
\end{equation}
\end{lemma}

\begin{proof}
  We first claim that any scale-$j$ polymer activity $F$ with $\|F\|_{h,T_j}\leq C$ can be approximated by polymer activities that are supported on polymers consisting of a bounded number of blocks.
  Indeed,
  $\|F\|_{h,T_j} = \sup_{X \in\cP_j^c} A^{|X|_j} \norm{F(X)}_{h,T_j(X)} \leq C$
  implies that
  $(2A/3)^{|X|_j} \norm{F(X)}_{h, T_j (X)} \rightarrow 0$ as $|X|_j \rightarrow \infty$.
  More precisely, for any $\delta >0$, there exists $M>0$ only depending on $C$ such that
\begin{equation}
  \sup_{X\in \cP_j^c} (2A /3)^{|X|_j} \norm{F(X) 1_{|X|_j \leq M} - F(X)}_{h, T_j (X)} \leq \delta .
\end{equation}
By Lemma~\ref{lemma:linearity_of_expectation}, then also
\begin{equation} \label{eq:stability_of_expectation_conclusion}
  \sup_{X\in \cP^c_j} (A /3)^{|X|_{j}} \norm{\E_{\Gamma_{j+1}(s)} [F(X,\cdot+\zeta) 1_{|X|_j \leq M} - F(X,\cdot+\zeta)]}_{h, T_{j+1} (\bar X)} \leq \delta .
\end{equation}
Since $\E_{\Gamma_{j+1}(s)} F(X,\cdot+\zeta)1_{|X|_j\leq M}$
is continuous in $s$ by Lemma~\ref{lemma:stability_of_expectation_singleX}
uniformly {in $X \in \cP^c_j$ and $F$ with $\norm{F(X)}_{h,T_j(X)} \leq C$}
(by translation invariance there are only a bounded number of polymers $X$ with $|X|_j\leq M$ to consider), the claim follows.
For the case \eqref{eq:stability_of_expectation1_j+1}, the conclusion follows from the same argument and \eqref{eq:stability_of_expectation_conclusion} replaced by
\begin{align}
  \sup_{X\in \cP^c_{j+1}} (3^{-1} 2^{-L^2 + 1} A)^{|X|_{j+1}} \norm{\E_{\Gamma_{j+1}(s)} [F(X,\cdot+\zeta) 1_{|X|_j \leq M} - F(X,\cdot+\zeta)]}_{h, T_{j+1} (X)} \leq \delta 
\end{align}
because $\Eplus[G_j (X, \zeta)] \leq 2^{|X|_j} = 2^{L^2 |X|_{j+1}}$.
\end{proof}

We begin with the continuity of the maps $\cL_{j+1}$. To make their $s$-dependence explicit
we write $\cL_{j+1}^s$ for $\cL_{j+1}$ defined with $\Eplus = \E_{\Gamma_{j+1}(s)}$.

\begin{lemma} \label{lemma:continuity_of_L_in_s}
Under the assumptions of Theorem~\ref{thm:local_part_of_K_j+1} and $s,s' \in [-\epsilon_s \theta_J, \epsilon_s \theta_J]$, we have
\begin{equation}
\lim_{s' \rightarrow s} \norm{\cL_{j+1}^{s} (K_j) - \cL_{j+1}^{s'} (K_j) }_{ \Omega_{j+1}^K} = 0
\end{equation}
and the limit is uniform in $\Lambda_N$.
\end{lemma}
\begin{proof}
By \eqref{eq:L_j+1_decomposition}, for $X \in \cP^c_{j+1}$,
 \begin{equation}
 \cL_{j+1}^s (K_j) (X, \varphi') = \cL_{j+1}^{s} (K_j 1_{Y\in \cS_j}) (X,\varphi') + \mathbb{S} \big[ \E_{\Gamma_{j+1} (s)} [ K_j 1_{Y \not\in \cS_j}   ] \big] (X, \varphi')
\end{equation}
where $\cL_{j+1}^{s} (K_j 1_{Y\in \cS_j})$ is generated by $K_j(Y)$ for $Y\in \cS_j$ and we recall the reblocking operator  $\bbS$ from \eqref{eq:reblocking_operator_definition}.
Since by translation invariance the norm effectively only uses a bounded number of $Y \in\cS_{j}$,
Lemma~\ref{lemma:stability_of_expectation_singleX}
and the continuity statement of Proposition~\ref{prop:Loc-contract}
directly imply the continuity of $\cL_{j+1}^{(\cS_{j}), s} (K_j)$ in $s$, uniformly in $\Lambda_N$.
Concerning the continuity of the second term,  \eqref{eq:stability_of_expectation1} shows that
\begin{equation}
y(s,s') :=\sup_{Y \in \cP_j^c} (A/3)^{|Y|_j} \norm{ \E_{\Gamma_{j+1} (s)} [ K_j 1_{Y \not\in \cS_j} (Y, \cdot + \zeta)] -\E_{\Gamma_{j+1} (s')} [ K_j 1_{Y \not\in \cS_j} (Y, \cdot + \zeta) ]  }_{h, T_{j+1} (\bar{Y}) }
\end{equation}
tends to 0 as $s'\rightarrow s$ and
\begin{equation}
\norm{ \mathbb{S} \big[ \big(\E_{\Gamma_{j+1} (s)} - \E_{\Gamma_{j+1} (s')} \big) [ K_j 1_{Y \not\in \cS_j}  ] \big] (X)  }_{h, T_{j+1} (X)} \leq \sum_{Y\in \cP_j }^{\bar{Y}=X} 1_{Y \in \cP_j^c \backslash \cS_{j}} (A/3)^{-|Y|_j} y(s,s')^{|\operatorname{Comp}_j (Y)|}.
\end{equation}
But then Lemma~\ref{lemma:setsizes_2} directly implies, whenever $y(s,s') \leq (A/3)^{-8}$,
\begin{equation}
\sum_{Y\in \cP_j }^{\bar{Y}=X} 1_{Y \in \cP_j^c \backslash \cS_{j}} (A/3)^{-|Y|_j} y(s,s')^{|\operatorname{Comp}_j (Y)|} \leq (eL^2 (A/3)^{-(1+2\eta)/(1+\eta)} )^{|X|_{j+1}} y(s,s') .
\end{equation}
By setting $eL^2 (A/3)^{-(1+2\eta)/(1+\eta)} \leq A^{-1}$, 
we have that $\mathbb{S} \big[ \E_{\Gamma_{j+1} (s)} [ K_j 1_{Y \not\in \cS_j}  ] \big]$ is continuous in $s$ in a way that is uniform in $\Lambda_N$.
\end{proof}

In the definition of the maps $\cM_{j+1}$, there are two sources of dependence on $s$,
the first one coming from $\cE_{j+1}$, $\bar{U}_{j+1}$, $\cE K_j$ and $J_j$, and the second one coming from the expectation
$\Eplus = \E_{\Gamma_{j+1}(s)}$ written explicitly in \eqref{eq:expression_for_K_j+1}.
Concerning the first dependence,
by the continuity statement of Proposition~\ref{prop:Loc-contract} and Theorem~\ref{thm:H_j_E_j_estimate},
we have that
\begin{equation}
\bar{\mathfrak{K}}_j (\omega_j) = (\cE_{j+1} |X|,  U_j,  \bar{U}_{j+1} , K_j,  \bar{K}_j,  \cE K_j,  J_j  ) (\omega_j)
\end{equation}
is continuous in the implicit parameter $s$, so if we can show that $\mathfrak{M}_{j+1}^{(k)} (\mathfrak{K}_j)$ depends `continuously' on $\mathfrak{K}_j$, then the dependence on $s$ coming from the first source is continuous. 
Indeed, this will be
shown in the following corollary.
 For given $\eta >0$, define ${\Omega}_{j, \eta}^{\mathfrak{K}}$ 
to be the linear space of coordinates $(\cE_{j+1} |X|, U_j, \bar{U}_{j+1} , K_j, \bar{K}_j, \cE K_j, J_j  )$ where the following norm takes finite value:
\begin{equation}
\begin{split}
& \norm{ (\cE_{j+1} |X|, U_j, \bar{U}_{j+1} , K_j, \bar{K}_j, \cE K_j, J_j  )  }_{j, \eta, \mathfrak{K}} \\
& = \max \Big\{ L^{2j} |\cE_{j+1}| , \norm{U_j}_{\Omega_j^{U}},  \norm{  \bar{U}_{j+1} + \cE_{j+1} |X|  }_{\Omega_j^U} , \norm{K_j}_{\Omega_j^K}, \\
 & \qquad \qquad \sup_{X\in \cP_{j+1}^c, \; \varphi \in \R^{\Lambda_N}} A^{ (1+ \eta) |X|_{j+1}} G_{j} (X, \varphi)^{-1} \norm{\bar{K}_j (X, \varphi)}_{h, T_j (X, \varphi)},  \\
 & \qquad \qquad \sup_{X\in \cP_{j+1}^c, \; \varphi \in \R^{\Lambda_N}} A^{ (1+ \eta) |X|_{j+1}} e^{-c_w \kappa_L w_j (X, \varphi)^2} \norm{\cE K_j (X, \varphi)}_{h, T_j (X, \varphi)}, \\
 & \qquad \qquad \sup_{Z\in \cS_{j+1}, \; B\in \cB_{j+1} (Z)} \sup_{\varphi \in \R^{\Lambda_N}} A e^{-c_w \kappa_L w_j (B, \varphi)^2} \norm{J_j (B, Z, \varphi)}_{h, T_j (B, \varphi)} 
\Big \} .
\end{split} 
\end{equation}
Then $(\Omega_{j, \eta}^{\mathfrak{K}},  \norm{\cdot}_{j, \eta, \mathfrak{K}}  )$ forms a normed space.
Note that this norm is essentially defined by the conditions in Definition~\ref{def:derivativebds}.

\begin{corollary} \label{cor:K_j+1_continuity_in_s_hidden_dependence}
  Let $\eta, \delta >0$ and $B_a^{\mathfrak{K}} = \{ x\in \Omega_{j, \eta}^{\mathfrak{K}} : \norm{x}_{j,  \eta, \mathfrak{K}} \leq a  \}$.
  Then there exists $a \equiv a (\delta, \beta,L) > 0$ (independent of $j$ and $N$)
  such that the identity map $\id |_{B_{a}^{\mathfrak{K}}}$ is in $\cX_j^{\mathfrak{K}} (B_{a}^{\mathfrak{K}})$. 
  In particular, if we set $\mathfrak{K}_j (x) = x$ for $x \in B_{a}^{\mathfrak{K}}$,
  then each $\mathfrak{M}_{j+1}^{(k)} (\mathfrak{K}_j (x))$ ($k = 1,2,3,4$) is differentiable in $x \in B_{a}^{\mathfrak{K}}$ with the derivative uniformly bounded in $j$ and $N$.
\end{corollary}

\begin{proof}
The first statement is obvious because $\id : \Omega_{j, \eta}^{\mathfrak{K}} \rightarrow \Omega_{j, \eta}^{\mathfrak{K}}$ is a linear function with norm 1.  For the second statement, we just need to apply Lemma~\ref{lemma:bound_on_M^k} with $(\mathbb{Y}, |\cdot|) = (B_a^{\mathfrak{K}} ,  \norm{\cdot}_{j, \eta, \mathfrak{K}})$.
\end{proof}

Note that by  Lemma~\ref{lemma:Ujbound-summary}, there exist $\epsilon (\delta, \beta, L)$ and $C(\delta, \beta, L)$ such that  $\norm{(U_j, K_j)}_{\Omega_j} \leq \epsilon(\delta,\beta, L)$ gives $\norm{\bar{\mathfrak{K}}_j (U_j, K_j)}_{j, \eta, \mathfrak{K}} \leq C(\delta, \beta, L) \epsilon (\delta, \beta, L)$.  
So if we set $\norm{(U_j, K_j)}_{\Omega_j} \leq \epsilon (\delta, \beta,L) \leq a (\delta, \beta,L) / C(\delta, \beta,L)$, then this corollary implies that each $\mathfrak{M}_{j+1}^{(k)} (\bar{\mathfrak{K}}_j (U_j, K_j))$ is continuous in $s$ coming from the first source described above.

For the second source of $s$-dependence of $\cM_{j+1}$, we will make the dependence due to $\Eplus = \E_{\Gamma_{j+1} (s)}$ visible in 
\eqref{eq:M_decomp} and \eqref{eq:M^1_j+1}--\eqref{eq:M^4_j+1} by writing $\cM_{j+1}^{s}$ and $\mathfrak{M}_{j+1}^{(k),s}$
for $\cM_{j+1}$ and $\mathfrak{M}_{j+1}^{(k)}$ evaluated by taking the  expectation over $\zeta \sim \cN (0, \Gamma_{j+1} (s))$.
This dependence will be studied in the next lemma.

\begin{lemma} \label{lemma:continuity_of_M_in_s}
Under the assumptions of Theorem~\ref{thm:local_part_of_K_j+1} and $s,s' \in [-\epsilon_s \theta_J, \epsilon_s \theta_J]$, we have
\begin{equation}
\lim_{s'\rightarrow s} \norm{\cM^{s}_{j+1} (U_j, K_j) - \cM^{s'}_{j+1} (U_j, K_j)   }_{\Omega_{j+1}^K} = 0
\end{equation}
and the limit is uniform in $\Lambda_N$.
\end{lemma}
\begin{proof}
Since we have \eqref{eq:M_decomp} and Lemma~\ref{lemma:Ujbound-summary}, we only have to verify
\begin{equation}
\lim_{s'\rightarrow s} \norm{\mathfrak{M}^{(k),s}_{j+1} (\bar{\mathfrak{K}}_j(\omega_j) ) - \mathfrak{M}^{(k),s'}_{j+1} (\bar{\mathfrak{K}}_j (\omega_j))  }_{\Omega_{j+1}^K} = 0
\end{equation}
for each $k\in \{1,2,3,4\}$ and the limit is uniform in $\Lambda_N$.
Define
\begin{equation}
H_{-} 
(\mathfrak{K}_j, X_0, X_1, Z, (B_{Z''}), \varphi', \zeta) = (e^{U_j} - e^{\bar{U}_{j+1}})^{X_0} (\bar{K}_j - \cE K_j)^{[X_1]} \prod_{Z'' \in \operatorname{Comp}_{j+1} (Z)} J_j (B_{Z''} , Z'')
\end{equation}
and, as in \eqref{eq:G_omega_X0_X1_Z_BZ},
\begin{equation}
H^s (\bar{\mathfrak{K}}_j X_0, X_1, Z, (B_{Z''}), \varphi' ) = \E_{\Gamma_{j+1} (s)} H_{-}
(\bar{\mathfrak{K}}_j X_0, X_1, Z, (B_{Z''}), \varphi', \zeta) .
\end{equation}
Expanding \eqref{eq:H_expansion}, i.e.,
\begin{equation}
\begin{split}
(e^{U_j} - e^{\bar{U}_{j+1}})^{X_0} (\bar{K}_j - \cE K_j)^{[X_1]} 
&= \sum_{Y_0, Y_1} (e^{U_j} -1)^{Y_0} (-e^{\bar{U}_{j+1}} + 1)^{X_0 \backslash Y_0} (\bar{K}_j)^{[Y_1]} (-\cE K_j)^{[X_1 \backslash Y_1]}
 \end{split}
\end{equation}
where $Y_0, Y_1$ run over $Y_0 \in \cP_{j+1} (X_0)$, $Y_1 \in \cP_{j+1} (Y_1)$, $Y_1 \not\sim X_1 \backslash Y_1$,
the bounds \eqref{eq:derivatives1-v2}--\eqref{eq:derivatives5-v2} imply, 
for $T = X_0 \cup X_1 \cup Z$,
\begin{equation}
\begin{split}
 \norm{ H_{-}
  (\bar{\mathfrak{K}}_j, X_0, & X_1, Z, (B_{Z''}), \varphi', \zeta) }_{h, T_j (T, \varphi')} \\
& \leq  \sum_{Y_0, Y_1} \big( C(A,L) \norm{\omega_j}_{\Omega_j} \big)^{\# (X_0, X_1, Z)} A^{- (1+ \eta) |X_0 \cup X_1 |_{j+1}   } G(X_0, Y_0, X_1, Y_1, Z, \varphi', \zeta)
\end{split}
\end{equation}
for some $\eta >0$ where
\begin{align}
& G(X_0, Y_0, X_1, Y_1, Z, \varphi', \zeta) = e^{c_w \kappa_L \big( w_j ( (X_0 \backslash Y_0) \cup (X_1 \backslash Y_1) \cup Z, \varphi' ) + w_j (Y_0, \varphi' + \zeta) \big)  } G_j (Y_1, \varphi' + \zeta).
\end{align}
Choosing $C(A,L) \norm{\omega_j}_{\Omega_j}^{1/4} \leq 1$ and $(C(A,L) \norm{\omega_j}_{\Omega_j})^{1/196} \leq A^{-(1+\eta)}$, since $49 |\operatorname{Comp}_{j+1} (Z)| \leq |\cup_{Z''} B^*_{Z''} |_{j+1}$, we have 
\begin{align}
\big(4 C(A,L) \norm{\omega_j}_{\Omega_j} \big)^{\# (X_0, X_1, Z)} A^{- (1+ \eta) |X_0 \cup X_1 |_{j+1}   } \leq 4^{-\# (X_0, X_1, Z)} \norm{\omega_j}_{\Omega_j}^{\frac{\# (X_0, X_1, Z)}{2}} A^{- (1+ \eta) |X|_{j+1}   }
\end{align}
where $X = X_0 \cup X_1 \cup (\cup_{Z''} B_{Z''}^*)$. Therefore
\begin{multline}
\norm{ H_{-} 
 (\bar{\mathfrak{K}}_j X_0,  X_1, Z, (B_{Z''}), \varphi', \zeta) }_{h, T_j (T, \varphi')} \\
\leq \norm{\omega_j}_{\Omega_j}^{\frac{\# (X_0, X_1, Z)}{2}} A^{- (1+ \eta) |X|_{j+1}   } \sup_{Y_0, Y_1} G(X_0, Y_0, X_1, Y_1, Z, \varphi', \zeta) . 
\end{multline}
since $H_- (\cdot, \varphi', \zeta)$ is a function of two field variables, Lemma~\ref{lemma:stability_of_expectation} does not apply directly.
Nevertheless, since $G$ serves the role of the regulator satisfying $$\Eplus[G(X_0, Y_0, X_1, Y_1, Z, \varphi', \zeta) ] \leq 2^{|X|_{j}} G_{j+1} (X, \varphi'),$$ the proof of \eqref{eq:stability_of_expectation2} shows that, defining
\begin{equation}
H^{s,s'} 
 (\bar{\mathfrak{K}}_j, X_0, X_1, Z, (B_{Z''}) , \varphi') = \Big( \E_{\Gamma_{j+1} (s)} - \E_{\Gamma_{j+1} (s')} \Big)  \big[ H_- 
 (\bar{\mathfrak{K}}_j, X_0, X_1, Z, (B_{Z''}) , \varphi', \zeta) \big],
\end{equation}
in the limit $s'\rightarrow s$, one has
\begin{equation}
|H^{s,s'} 
|_{j+1} := \sup_{T\in \cP_{j+1}} \Big( \frac{2}{3 \cdot 2^{L^2}} A^{1+\eta} \Big)^{|T|_{j+1}} \norm{ H^{s,s'} 
(\bar{\mathfrak{K}}_j, X_0, X_1, Z, (B_{Z''}) )  }_{h, T_{j+1} (T)} \rightarrow 0 ,
\end{equation}
In particular each $\norm{ H^{s,s'} (\bar{\mathfrak{K}}_j X_0, X_1, Z, (B_{Z''}) )  }_{h, T_{j+1} (T)}$ is finite. 
Hence
\begin{align}
& \norm{ (\mathfrak{M}^{(1), s}_{j+1} - \mathfrak{M}^{(1), s'}_{j+1})  (\bar{\mathfrak{K}}_j, X, \varphi') }_{h, T_{j+1} (X, \varphi')} \nnb 
& = \Big\| \sum_{X_0, X_1, Z, (B_{Z''})}^{\# (X_0, X_1, Z) \geq 2} F(\bar{\mathfrak{K}}_j,T, X, \varphi') H^{s,s'} (\bar{\mathfrak{K}}_j,X_0, X_1, Z, (B_{Z''}),  \varphi' ) \Big\|_{h, T_{j+1} (X, \varphi')} \nnb
& \leq \sum_{X_0, X_1, Z, (B_{Z''})}^{\# (X_0, X_1, Z) \geq 2} C^{|X|_{j+1}} e^{c_w \kappa_L w_j (X\backslash T, \varphi' )} |H^{s,s'}|_{j+1} \Big( \frac{2}{3 \cdot 2^{L^2}} A^{(1+\eta)} \Big)^{ - |X|_{j+1}} G_{j+1} (T, \varphi' ) \nnb
& \leq C^{|X|_{j+1}} |H^{s,s'}|_{j+1} 5^{|X|_{j+1}} \Big( \frac{2}{3 \cdot 2^{L^2}} A^{(1+\eta)} \Big)^{ - |X|_{j+1}} G_{j+1} (X, \varphi') ,
\end{align}
where $5$ in the last line is a combinatorial factor arising from choices of $X_0, X_1, Z$ and $(B_{Z''})$. Taking $A^{\eta} \geq 15 C 2^{L^2 -1}$, we see 
\begin{align}
\norm{ (\mathfrak{M}^{(1), s}_{j+1} - \mathfrak{M}^{(1), s'}_{j+1})  (\bar{\mathfrak{K}}_j (\omega_j) ) }_{ \Omega_{j+1}^K} \leq C |H^{s,s'}|_{j+1} \rightarrow 0 \;\; \text{as} \;\; s' \rightarrow s .
\end{align}
A similar but simpler computations shows the same for $\mathfrak{M}^{(2), s}_{j+1}$. The continuity of $\mathfrak{M}_{j+1}^{(3), s}$ in $s$ is implied directly by Lemma~\ref{lemma:stability_of_expectation_singleX} because it only allows the case $|X|_{j+1} =1$.

To see the same for $\mathfrak{M}_{j+1}^{(4), s}$, recall from \eqref{eq:DM_-^4_bound} and \eqref{eq:DM_-^4,2_bound} that
\begin{align}
\norm{M_-^{(4)} (\bar{\mathfrak{K}}_j(\omega_j), X, \varphi)  }_{h, T_j (X, \varphi')} \leq C(A, L) A^{-(1+\eta) |X|_{j+1}} G_j (X, \varphi' + \zeta) \norm{\omega_j}_{\Omega_j}^2
\end{align} 
for some $\eta >0$. Since $\mathfrak{M}_{j+1}^{(4), s,s'} = (\E_{\Gamma_{j+1}(s)} - \E_{\Gamma_{j+1}(s')} )M_{-}^{(4)}$, \eqref{eq:stability_of_expectation2} implies
\begin{align}
\lim_{s' \rightarrow s} \sup_{X\in \cP_{j+1}^c} \Big( \frac{2}{3 \cdot 2^{L^2}} A^{1+\eta} \Big)^{ - |X|_{j+1}} \norm{\mathfrak{M}_{j+1}^{(4), s,s'} (\bar{\mathfrak{K}}_j (\omega_j), X)}_{h, T_{j+1} (X)} =0.
\end{align}
Just taking $A^{\eta} \geq 3 \cdot 2^{L^2 -1}$, this implies continuity of $\mathfrak{M}_{j+1}^{(4), s}$ in $s$. To see the final remark of the lemma, notice that the limits $\lim_{s\rightarrow s'}$ are uniform in $\Lambda_N$ because the limit was uniform in Lemma~\ref{lemma:stability_of_expectation}.
\end{proof}

\begin{proof}[Proof of Theorem~\ref{thm:local_part_of_K_j+1},(iii)]
  The final continuity statement of $\cK_{j+1} = \cL_{j+1} + \cM_{j+1}$ is now a direct consequence of Lemma~\ref{lemma:continuity_of_L_in_s},
  Corollary~\ref{cor:K_j+1_continuity_in_s_hidden_dependence}, and Lemma~\ref{lemma:continuity_of_M_in_s}.
  Note that the equicontinuity of  $(\cK_{j+1}^{\Lambda_N})_N$ in $s$ follows from the fact
that the limits in the two previous lemmas are uniform in $\Lambda_N$ and that the Corollary yields an upper bound on the derivative that is uniform in $j$ and $N$. 
\end{proof}

\section{Stable manifold}
\label{sec:stable_manifold_theorem}

\subsection{Statement of result}

In Section~\ref{sec:rg_generic_step}, we defined a renormalisation group map
\begin{align}
\Phi^{\Lambda_N}_{j+1} : (s_j,z_j,K_j) \mapsto (\mathfrak{s}_{j+1} (s_j, K_j), \mathfrak{z}_{j+1} (z_j), \mathcal{K}^{\Lambda_N}_{j+1}(s_j, z_j, K_j)  ), \label{eq:rgflow_on_torus}
\end{align}
for $j+1 < N$, 
which by iteration constructs a renormalisation group flow $(s_j, z_j, K_j)_{j < N}$, defined by
\begin{equation}
  (s_{j+1},z_{j+1},K_{j+1}) = \Phi^{\Lambda_N}_{j+1} (s_j,z_j,K_j), \label{eq:abstract_RG_map}
\end{equation}
provided that $(s_j,z_j,K_j)$ remains in the domain of the renormalisation group maps.
Compared to the definition in Section~\ref{sec:rg_generic_step}, we have dropped the $E$-coordinate from the renormalisation
group map as it does not influence its dynamics and thus does not play a role in this section.

Our goal is now to show that for appropriate initial conditions $(s_0,z_0,K_0)$,
independent of $\Lambda_N$ (in the sense explained below),
the renormalisation group flow exists indefinitely.
Moreover, we will address the point that
our renormalisation group map actually depends on a parameter $s$
(mostly suppressed in our notation so far),
which we ultimately need to set equal to $s_0$
(see Lemma~\ref{lemma:reformulation}), but which has been arbitrary so far.
Thus a renormalisation group flow depends both on the parameter $s$ and the initial condition $(s_0,z_0,K_0)$, but we will show that it is possible to choose $s=s_0$.

Given a finite-range step distribution $J$,
recall the definition of the reference temperature $\betafree(J)$
from \eqref{eq:betafree_def}:
\begin{equation}
\betafree (J) = 8\pi v_{J}^2
\end{equation}
and recall that $\gamma$ is given by Proposition~\ref{prop:decomp} (see also below \eqref{eq:U_norm}).
In the sequel we frequently write $K_0=0$ to denote the zero element in the linear space of polymer activities, i.e., the polymer activity given by $K_0(X)=0$, $X \in \mathcal{P}_j^c$, whence $K_0(X)= 1_{\emptyset}(X)$, cf.~below Definition~\ref{def:polymeractivity}.

\begin{proposition} \label{prop:stable_manifold}
\begin{itemize}
\item[(i)]  For any finite-range step distribution $J$ (as always invariant under lattice symmetries and satisfying \eqref{eq:frd_ulbds})
  there exist
  $r \in (0,1]$ and  
   $\beta_0(J) \in (0,\infty)$ such that the following holds for $\beta \geq \beta_0(J)$.
   There exist $s_0^c(J,\beta) = O(e^{-\frac{1}{4}\gamma \beta})$ and $\alpha = \alpha(J,\beta) >0$,
   and positive integers $L=L(J)$ and $A=A(J)$
  such that the solution to \eqref{eq:abstract_RG_map} with parameter $s=s_0^c(J,\beta)$
  and initial conditions $s_0=s=s_0^c(J,\beta)$, $z_0 = \tilde z(\beta)$ as in Lemma~\ref{lemma:Fourier_repn_of_V}, and $K_0=0$ satisfies
  for $0\leq j < N$ and $N> 1$,
  \begin{equation} \label{eq:stable_manifold_bounds}
    \|U_j\|_{\Omega_j^U} = O(e^{-\frac14 \gamma \beta}L^{-\alpha j}), \qquad \|K_j\|_{\Omega_j^K} = O(e^{-\frac14 \gamma \beta}L^{-\alpha j}),
  \end{equation}
  where the norms are as in Definitions~\ref{def:U_space}--\ref{def:K_space} (and thus depend on $A,L,r,\beta,\rho_J$).
    
 \item[(ii)] If $\cJ$ is a family of finite-range step distributions
   and \eqref{eq:frd_ulbds} holds with the same constant for all $J \in \cJ$, then
   there exists $C(\cJ) >0$ such that for any $\delta >0$ and $J\in \cJ$ with $v_{J}^2 \geq C(\cJ)  |\log \delta|$,  one may take $\beta_0(J) = (1+\delta)\betafree (J)$ in (i).
\end{itemize}
\end{proposition}
%

We remark that in terms of the function $s_0^c(J,\beta)$ of the proposition, the effective temperature in Theorem~\ref{thm:highbeta}
will be defined by (cf.~the discussion around \eqref{e:sl-beta_eff})
\begin{equation} \label{eq:betaeff_def}
  \betaeff(J,\beta) 
  = (1+s_0^c(J,\beta)v_J^{-2})^{-1}\beta.
\end{equation}

To prove Proposition~\ref{prop:stable_manifold}, we first extend the renormalisation
group map to an infinite-volume version, in Section~\ref{sec:infvol} below.
In Section~\ref{sec:stable_manifold_proof},
we then apply a version of the stable manifold theorem,
and finally use the intermediate value theorem to solve the constraint $s=s_0$.

\subsection{Infinite-volume limit}
\label{sec:infvol}

In Section~\ref{sec:rg_generic_step}, we considered $\Lambda_N$ fixed and corresponding scales $j < N-1$.
In particular the renormalisation group map \eqref{eq:abstract_RG_map} also depends on $\Lambda_N$.
However, in order to talk about the convergence of the flow $(s_j, z_j, K_j)$ as $j\rightarrow\infty$, we now introduce notions of polymer activities and renormalisation flow that is free of this dependence
by being defined in infinite volume.

To distinguish polymer activities that depend on the torus from those defined in $\mathbb{Z}^2$,
we now write $K^{\Lambda_N}$ for the former and $K^{\Z^2}$ or $K$ (without index) for the latter.

For each $\Lambda_N$, fix an origin $0 \in \Lambda_N$ and let $\pi_N : \mathbb{Z}^2 \rightarrow \Lambda_N$ be the quotient map such that $\pi_N (0) = 0$. Define $R_N = \mathbb{Z}^2 \cap  [-\frac{L^N -1}{2} , \frac{L^N -1}{2} ]^2 \subset  \mathbb{Z}^2$, so $\pi_N |_{R_N}: R_N \rightarrow \Lambda_N$ is the canonical bijection with inverse $\iota_N : \Lambda_N \rightarrow R_N$. 

\begin{definition}
Given the sequence of discrete tori $(\Lambda_N)_{N \geq 1}$ and a sequence of scale-$j$ polymer activities $((K_j^{\Lambda_N} (X ))_{X \in \mathcal{P}_j^c (\Lambda_N)}: j \geq 1 )$ for each $N$, the polymer activities have a local (infinite-volume) limit $(K_j(X))_{X\in \mathcal{P}_j^c (\mathbb{Z}^2)}$ if there exist integers $N_X$ such that
\begin{equation}
K_j(X,\varphi) = K_j^{\Lambda_N}  ( \pi_N (X), (\pi_N)_{\#}  \varphi ) = K_j^{\Lambda_{N'}} (  \pi_{N'} (X), (\pi_N)_{\#} \varphi )
\end{equation}
for $X \in \mathcal{P}_j^c ( \mathbb{Z}^d)$, $j < N_{X} < N, N'$ and any $\varphi\in \mathcal{V}_N$. $N_X$ is called the localising scale of $X$.
\end{definition}

When the local limit exists, it is an element of
the infinite-volume analogue of the space $\Omega_j^K = \Omega_j^K(\Lambda_N)$ of Definition~\ref{def:K_space},
which we denote by $\Omega_j^K(\Z^2)$.
Thus the norm on this space is defined by
\begin{equation} \label{e:fullnorm_infvol}
  \norm{K_j}_{\Omega_j^K} =  \norm{K_j}_{h, T_j} 
  = \sup_{X\in \mathcal{P}^c_j (\Z^d)} A^{|X|_j} \norm{K_j(X)}_{h, T_j (X)}
  .
\end{equation}
It follows from Appendix~\ref{app:completeness} that this space is complete.

\begin{proposition}[Infinite volume RG map] \label{prop:inf_vol_RG}
  There exist maps
  $\Phi_{j+1}^{\Z^2} = (\cU^{\Z^2}_{j+1}, \mathcal{K}^{\Z^2}_{j+1} )$ and $\cE_{j+1}^{\Z^2}$ such that the following hold, when $\Phi_{j+1}^{\Z^2}  (U^{\Z^2}_j, K^{\Z^2}_j)=(U^{\Z^2}_{j+1}, K^{\Z^2}_{j+1})$.

\begin{itemize}
\item[(i)] If $K^{\Z^2}_j$ is even,  respects lattice symmetries (in the sense of Definition~\ref{def:latticesym} with $\Z^2$ in place of $\Lambda_N$),
 and $K^{\Z^2}_j (X, \varphi) = K^{\Z^2}_j (X, \varphi + 2\pi \beta^{-1/2} \textbf{c})$ for any constant field $\textbf{c}$ taking integer value, 
then $K^{\Z^2}_{j+1}$ satisfies the same.

\item[(ii)] The estimates of Theorem~\ref{thm:H_j_E_j_estimate} and Theorem~\ref{thm:local_part_of_K_j+1} also hold for this
  infinite-volume renormalisation group map $\Phi^{\Z^2}_{j+1}$
  when measured in $\norm{\cdot}_{\Omega_j^K}$-norm of \eqref{e:fullnorm_infvol}.

\item[(iii)] $\Phi^{\Z^2}_{j+1}$ is continuous in $s$.

\item[(iv)] Suppose $U_j^{\Lambda_N} = U_j^{\Z^2}$ for $j<N$
$K^{\Z^2}_j$ is a local limit of $(K_j^{\Lambda_N} )_{N}$
and $(U_{j+1}^{\Lambda_N},  K_{j+1}^{\Lambda_N} ) = \Phi^{\Lambda_N}_{j+1} (U_j^{\Lambda_N}, K_j^{\Lambda_N})$ for each $N > j+1$.
Then $\cE_{j+1}^{\Z^2} = \cE_{j+1}^{\Lambda_N}$ and $U_{j+1}^{\Z^2} = U_{j+1}^{\Lambda_N}$ for $j+1 < N$
and $K_{j+1}^{\Z^2} $ is a local limit of $(K_{j+1}^{\Lambda_N} )_{N}$.
\end{itemize}

\end{proposition}
\begin{proof}
We may define $\mathcal{K}^{\mathbb{Z}^2}_{j+1}$ using the formula obtained from an infinite-volume analogue of \eqref{eq:expression_for_K_j+1}, 
and analogously for $\cE_{j+1}^{\mathbb{Z}^2}$ and $\mathcal{U}_{j+1}^{\mathbb{Z}^2}$.
Other way to think about $\Phi_{j+1}^{\Z^2}$ is to simply think of it as a local limit of $\Phi_{j+1}^{\Lambda_N}$ as $N\rightarrow \infty$.
Then (i), (ii), (iv) are direct consequences of the fact that $\mathcal{E}^{\Lambda_N}_{j+1}$ and $\cU^{\Lambda_N}_{j+1}$  only depend on $(K_j^{\Lambda_N}  (Z ) : Z\in \mathcal{S}_j (\Lambda_N))$,
$\mathcal{K}^{\Lambda_N}_{j+1} (X)$ only depends on the $K_j^{\Lambda_N}  (Y )$ for $Y \in \mathcal{P}_j (X^*)$,
and the same hold for the analogous objects on $\Z^2$. 
Also (iii) is a consequence of the fact that the family $(\cK^{\Lambda_N}_{j+1})_N$ is equicontinuous in $s$ by Theorem~\ref{thm:local_part_of_K_j+1}, 
hence this continuity extends to that of $\cK_{j+1}^{\Z^2}$, and continuity of $\cU_{j+1}^{\Z^2}$ follows from the continuity statement in Theorem~\ref{thm:H_j_E_j_estimate}.
\end{proof}

\subsection{Stable manifold for the infinite volume RG flow}
\label{sec:stable_manifold_proof}


In this section, we prove an analogue of Proposition~\ref{prop:stable_manifold} for the infinite volume RG flow. 

It is somewhat more convenient to represent $z_j=(z_j^{(q)})$ and its evolution in terms of $W_j$ as defined in Definition~\ref{def:U_space}. This is mainly so that so that we can
use the notation $\|W_j\|_{\Omega_j^U}$ from that definition
(and do not need to introduce further notation). 
Thus given the map $\cU_{j+1} = (\mathfrak{s}_{j+1}, \mathfrak{z}_{j+1})$, we define  
\begin{align}
	\cW_{j+1} (\omega_j) (B, \varphi) = \sum_{q\geq 1} \sum_{x\in B} L^{-2(j+1)} \mathfrak{z}_{j+1}^{(q)}(z_j) \cos( \sqrt{\beta} q \varphi(x) )
	.
\end{align}
Then by Proposition~\ref{prop:inf_vol_RG} (with estimates of Theorems~\ref{thm:H_j_E_j_estimate}--\ref{thm:local_part_of_K_j+1}),
the infinite-volume renormalisation flow is given by
\begin{align}
	& s_{j+1} = \mathfrak{s}_{j+1} (s_j, K_j) =  s_j + \mathcal{H}_{j+1} (K_j) \label{eq:rgflow4} \\
	& W_{j+1}(B,\varphi') = \mathcal{W}_{j+1} (W_j) (B,\varphi') = \mathbb{E}_{\Gamma_{j+1}} [W_j (B,\varphi' + \zeta)] \label{eq:rgflow5} \\
	& K_{j+1} = \mathcal{K}^{\Z^2}_{j+1} (s_j, W_j, K_j)  = \mathcal{L}^{ \Z^2}_{j+1} (K_j) + \mathcal{M}^{\Z^2}_{j+1} (s_j, W_j , K_j) \label{eq:rgflow6}
\end{align}
where $\mathcal{H}_{j+1} (K_j)$ is given by Definition~\ref{def:evolution_of_U} (whose extension to $\Z^2$ is clear, as it only uses small polymers)
and $\mathcal{L}^{\Z^2}_{j+1}$, $\mathcal{M}^{\Z^2}_{j+1}$ are given by
Theorem~\ref{thm:local_part_of_K_j+1}, extended to $\Z^2$ by Proposition~\ref{prop:inf_vol_RG}.
We omitted index $\Z^2$ for \eqref{eq:rgflow4} and \eqref{eq:rgflow5} since we have seen in Proposition~\ref{prop:inf_vol_RG}~(iv) that they do not depend on the volume of the system. 
Our goal is to apply the stable manifold theorem in the form stated in \cite[Theorem~2.16]{MR2523458} to show the existence of $s_0^c$ explained earlier.
For this it is essential that the maps $\mathcal{K}^{\Z^2}_{j+1}$ contract.
According to \eqref{eq:bound_for_L_j_K_j}
and the definition of $\alphaLoc$ in \eqref{e:Loc-contract-kappa},
this requires control of the  lower bound on $\Gamma_{j+1} (0)$.
The covariance estimate \eqref{eq:Gammaj0_asymp} implies a good lower bound on $\Gamma_{j+1}(0) / \log L$
once $j$ is larger than a \emph{critical scale} $j_0$, defined precisely by the next lemma.
In the following we will write
(note the extra argument $s$ compared to \eqref{eq:betaeff_def}):
\begin{equation}
  \betaeff(J, \beta, s) = (1 + s v_{J}^{-2})^{-1} \beta.
\end{equation}

\begin{proposition} 
  \label{prop:j0_definition}
  For given $r \in (0,1]$ and $\delta >0$,
  assume $\beta$ is such that $r \betaeff(J,\beta,s) \geq (1+\delta ) \betafree (J)$.
Then there exists $j_0 \equiv j_0 (\rho_J, L, \delta)$ such that
\begin{align}
L^{j_0} = O\Big( L\rho_J(1+ \delta^{-1}) \Big)
\end{align}
and that, for $j\geq j_0$,
\begin{equation} \label{eq:W-contract}
L^{2} e^{-\frac{1}{2} r \beta \Gamma_{j+1} (0 ; s)}  \leq L^{- \delta}.
\end{equation}
\end{proposition}
\begin{proof}
By \eqref{eq:decomp5} {and \eqref{eq:frd_ulbds}}, there exists $c_0 \geq 1$ such that
$| 2\pi t (v_J^2+s) \dot{{ D}}_{t} (0,0 | s ) -1 | \leq c_0 \rho_J / t$ for all $t \geq \rho_J$.
Hence define
\begin{align}
t_0 := c_0 \Big( \frac{1}{4} - \frac{1}{4 (1+ \delta)} \Big)^{-1} \rho_J \geq c_0 \Big( \frac{1}{4} - \frac{\betafree(J)}{4 r \betaeff(J,\beta,s)} \Big)^{-1} \rho_J
\end{align}
and $j_0 := \lceil \log_L (2 t_0) \rceil$. Then for $j\geq j_0$, 
\begin{align}
\frac{r \beta}{2 \log L} \Gamma_{j+1}(0 ; s) - 2 \geq \frac{2r \betaeff(J,\beta,s)}{\betafree(J)} \Big( 1- \frac{c_0 \rho_J }{2 t_0} \Big) -2 \geq \frac{7}{4} \delta
\end{align}
so the claim holds. 
\end{proof}


We explain some terminologies for the following theorem. 
We assume that $r\in (0,1]$,  $\beta>0$, $\rho_J \geq 1$ satisfy the assumptions of Lemma~\ref{lem:W_norm} and that $r \beta > \betafree (J)$. 
Let $L$ and $A$ be at least those given in Theorem~\ref{thm:local_part_of_K_j+1},  
$j_0 (\rho_J, L, \delta)$ be as in Proposition~\ref{prop:j0_definition},
and recall \eqref{eq:theta_def},  the definition of $\theta_J$. 
We use various $\epsilon$'s.
Given $\delta > 0$, we let $\epsilon_{\delta} > 0$ be such that $r \betaeff (s, J) \geq (1+\delta) \betafree (J)$ for $|s| \leq \epsilon_{\delta}$.
Let $\epsilon_{nl} \equiv \epsilon_{nl} (\beta, A, L)$, a rational function of its arguments,  be as in Theorem~\ref{thm:local_part_of_K_j+1},
$\epsilon_s$ be as in Lemma~\ref{lemma:stability_of_expectation}
and let
\begin{align}
	\epsilon'_{\delta} = \min \{ \epsilon_{\delta} ,  \theta_J \epsilon_s ,  {\textstyle \frac{1}{4} } \}	,
	\qquad 
	\epsilon'_{nl} = \min \{ \epsilon_{nl},  (2 L)^{-1} C_3 (\beta, A,L)^{-1}\}	.
\end{align}
with $C_3$ as in \eqref{eq:bound_for_derivative_of_Nj}).
Thus $\epsilon'_{\delta}$ is a bound for the parameter $s$ and $\epsilon'_{nl}$ is a bound for various polymer activities. 
Also, let $\epsilon_0 = L^{-3j_0(\rho_J, L, \delta)}  \epsilon'_{nl}( \beta, A, L)$ and $\theta_0 =  \frac{1}{8}  \min \{ 1, \delta  \} > 0$.


\begin{theorem} \label{thm:stable_manifold}
Let $\ell$ be sufficiently large and $r ,\delta >0$.
Then for $L\geq L_0(\theta_0)$ of form $L = \ell^{N'}$,
$A \geq A_0 (L)$,  $|s| \leq \epsilon'_{\delta}$
and $\norm{W_0}_{\Omega_0^U} \leq \epsilon_0$
there exists $\mathfrak{s}_0^c(\beta,s) = O(\norm{W_0 }_{\Omega_0^U})$
such that $(s_j, W_j, K_j) \to 0$ as $j\rightarrow \infty$
satisfying the flow equations \eqref{eq:rgflow4}--\eqref{eq:rgflow6} with initial conditions $s_0 = \mathfrak{s}_0^c (\beta, s)$, 
$W_0$ given as above, and $K^0_0=0$.
Moreover, $\mathfrak{s}_0^c$ is continuous in $s$ and
\begin{align}
	|s_j|, \; \norm{W_j}_{\Omega_j^U}, \; \norm{K_j}_{\Omega_{j}^K} \leq O(\norm{W_0}_{\Omega_0^U}) L^{-\alpha j}
\end{align}
for some $\alpha >0$ satisfying $C L^2 \alphaLoc \leq L^{-\alpha}$ for sufficiently large $C$.

\end{theorem}

\begin{proof}
We drop $\Z^2$ in the proof. 
The proof is an application of the stable manifold theorem in the form of \cite[Theorem~2.16]{MR2523458}, only with smoothness
replaced by continuous differentiability in its assumption and conclusion.
To obtain the continuity in $s$ we will work with spaces of continuous functions in $s$.
For this application, we begin by defining
Banach spaces $(I_j)_j$, $(F_j)_j$ for $j \in \mathbb{N}_{\geq 0}$ by
\begin{align}
	& I_j = \big\{ s_j (s) \in C([-\epsilon'_{\delta},  \epsilon'_{\delta}],\R) \, : \, \norm{s_j}_{I_j} < +\infty  \big\}, \\
	& F_j = \big\{ (W_j, K_j) (s) \in C \big( [-\epsilon'_{\delta},  \epsilon'_{\delta}],\Omega^W_j \times \Omega^K_{j}  \big) \, : \,  \norm{(W_j, K_j)}_{F_j} < +\infty \big\},
\end{align}
where $\Omega_j^W\subset \Omega_j^U$ is the (closed) subspace of elements with $s$-component equal to $0$,
\begin{align}
	\norm{s_j}_{I_j} &= \tau (j) \sup_{s\in [-\epsilon'_{\delta},  \epsilon'_{\delta}] } |s_j(s)|, \nnb
	\norm{(W_j, K_j)}_{F_j} &= \tau(j) \sup_{s\in [-\epsilon'_{\delta},  \epsilon'_{\delta}]} \max\{ \norm{W_j(s)}_{\Omega_j^U}, \norm{K_j (s)}_{\Omega_{j}^K} \},
\end{align}
and
\begin{equation}
	\tau (j) = L^{ 3(j_0 - j)_+} = L^{ 3 \max\{ j_0 - j , 0 \}}.
\end{equation}
The weight $\tau(j)$ will ensure contractiveness of the map for scales $j \leq j_0$ where it is not guaranteed that $\Gamma_{j+1}(0)$ is not bounded below.
Since $\Omega_j^U$ and $\Omega_{j}^K (\Z^2)$ are Banach spaces, $I_j$ and $F_j$ are Banach spaces.
Also let $B^{\mathbb{X}}_{a}$ be the open ball in normed space $\mathbb{X}$ centred at 0 with radius $a>0$.
Define
\begin{align}
\begin{split}
	T_{j+1} : & B^{I_j}_{\epsilon_0} \times B^{F_j}_{\epsilon_0} \rightarrow I_{j+1} \times F_{j+1}, \\
	& (s_j, W_j,  K_j)  \mapsto ( \mathfrak{s}_{j+1}(s_j, K_j) , \mathcal{W}_{j+1}(W_j), \mathcal{K}_{j+1} (s_j, W_j, K_j)  ) .
\end{split} \label{eq:RGmap_reform1} 
\end{align}
Since $\mathcal{H}_{j+1}, \mathcal{W}_{j+1}, \mathcal{L}_{j+1}$ are bounded linear functions and $\mathcal{M}_{j+1}$ is a continuously differentiable function, $T_{j+1}$ is also continuously differentiable. 
Also, the operators $T_{j+1}$ are uniformly invertible in  a neighbourhood of $(0,0)$ in the following sense:
by Proposition~\ref{prop:inf_vol_RG} (and using estimates of Theorems~\ref{thm:H_j_E_j_estimate},~\ref{thm:local_part_of_K_j+1}),
there are constants $C_1, C_2$ independent of $j$ such that
\begin{align*}
\begin{array}{ll}
	\text{(C1)} & \;\; \sup_j \{ | \mathcal{H}_{j+1} (K_j)| : \norm{K}_{\Omega_{j}^K} \leq 1 \} < + \infty \; ; \\
	\text{(C2)} & \;\; \sup_j \{ \norm{\mathcal{L}^0_{j+1} (K_j)}_{\Omega_{j+1}^K} : \norm{K}_{\Omega_{j}^K} \leq 1  \} \leq C_1 L^{2} \alphaLoc \; ;  \\
	\text{(C3)} & \;\; \sup_j \{ \norm{\mathcal{W}_{j+1}(W_j)}_{\Omega_{j+1}^U} : \norm{W_j}_{\Omega_j^U} \leq 1  \} \leq L^2 e^{-\frac{1}{2} \beta \Gamma_{j+1}(0)} \; ; \\
	\text{(C4)} & \;\;  (s_j, W_j, K_j) \mapsto \cM_{j+1} \text{ is continuously differentiable} \; ; \\
	\text{(C5)} & \;\;  \norm{D \cM_{j+1} (s_j, W_j, K_j)}_{\Omega_{j+1}^K} \leq C_2 \norm{(s_j, W_j, K_j)}_{\Omega_{j}}
			{ \text{ for } (s_j, W_j, K_j) \in B^{I_j}_{\epsilon_0} \times B^{F_j}_{\epsilon_0} }, \\
	& \;\; \text{ and }   \cM_{j+1} (0,0,0) = 0 
	.
\end{array}
\end{align*}
Note that Proposition~\ref{prop:j0_definition} implies,
for $e^{2 \sqrt{\beta} h} \leq (e^{\frac{1}{2} r\beta \Gamma_{j+1} (0 ; s)})^{\theta_0}$ (which is implied always possible by choosing $L \geq L_0 (\theta_0)$ sufficiently large),
\begin{align}
	L^2 \alphaLoc \leq  C \big( L^{-1} (\log L)^{3/2} +  L^2 \sum_{q \geq 1}  L^{-(2+\delta)(2q-1) (1-\theta_0)}  \big) \leq C' ( L^{-1}  ( \log L)^{3/2} + L^{-\delta/2}   ) .
\end{align}
Together with (C2), (C3), and \eqref{eq:W-contract}, this implies 
\begin{align} \label{eq:W-L-contract}
	\sup_j \norm{ (\mathcal{W}_{j+1} , \mathcal{L}_{j+1} ) }_{F_j \rightarrow F_{j+1}} < 2 C_1 L^2 \alphaLoc 
		\leq L^{-\alpha} <  1
\end{align}
for some $\alpha$ when $L$ is chosen sufficiently large. 
Then, by (C1), (C4), (C5), and \eqref{eq:W-L-contract}, 
$T_{j+1}$ is as required for the proof of \cite[Theorem~2.16]{MR2523458} (with smoothness of $\cM_{j+1}$ replaced by continuous differentiability) to apply, thus yielding the existence of a continuously differentiable function $S^{(s)} : B^{F_0}_{\epsilon_0} \rightarrow I_0$ 
such that the initial condition $(S^{(s)}(W_0, K_0), W_0, K_0)$ solves the
flow equations \eqref{eq:rgflow4}--\eqref{eq:rgflow6} with the final condition $(s_j, W_j, K_j) \rightarrow (0,0,0)$ exponentially.
The rate of the exponential decay also follows from the proof of the cited theorem. 

Then $\mathfrak{s}_0^c = S^{(s)} (W_0, 0)$ is as desired: indeed,
\begin{equation}
	|\mathfrak{s}_0^c (\beta, s)| \leq \sup_{(W'_0, 0) \in B_{\epsilon_0}^{F_0}} \norm{D_{(W_0, K_0)} S^{(s)}(W'_0, 0)}_{F_0 \rightarrow I_0} \, \norm{W_0}_{\Omega_0^U} = O( \norm{W_0}_{\Omega_0^U} ) ,
\end{equation}
Finally, $\mathfrak{s}_0^c (s)$ is continuous in $s$ by construction, 
as $I_0$ is a space of functions continuous in $s$.
\end{proof}


\begin{corollary} \label{cor:tuning_s}
Let $\mathfrak{s}_0^c (\beta, s)$, $L_0$, and $A_0$ be as in Theorem~\ref{thm:stable_manifold} applied with ${W}_0 = \tilde U$ as in \eqref{e:tildeF}, 
and set $N'_0 = \lceil \log_{\ell} L_0 \rceil$.  The following hold for $L =\ell^{N'_0}$ and $A = A_0 (L)$.
\begin{itemize}
\item[(i)] If $J$ is fixed and $\beta$ is sufficiently large, there exists $s_0^c (J, \beta)$ such that $\mathfrak{s}_0^c (\beta, s_0^c (J, \beta)) = s_0^c(J, \beta)$.
\item[(ii)] {Let $\cJ$ be a family of finite-range step distributions and suppose
that \eqref{eq:frd_ulbds} holds with the same constants for all $J\in \cJ$.}
Then for any $\delta>0$,  there exists $C>0$ such that whenever $J\in \cJ$, $v_{J}^2 \geq C |\log \delta|$ and $\beta \geq (1+\delta)\betafree (J)$, there exists $s_0^c(J, \beta)$ such that $\mathfrak{s}_0^c(\beta,s_0^c(J, \beta)) =s_0^c(J, \beta)$.
\end{itemize}
\end{corollary}

The proof of the corollary is an application of the intermediate value theorem.

\begin{proof}
To see \emph{(i)}, first choose $r>0$ small enough and $\beta>0$ large enough so that the assumption of Lemma~\ref{lem:W_norm} is satisfied and $r\beta \geq 2 \betafree$.
Also choose $\epsilon_0 >0$ as in Theorem~\ref{thm:stable_manifold} and fix $\delta=1/2$. Then $\epsilon_{\delta} >0$ is chosen to be less than $1/10$.

Now note that
Lemma~\ref{lemma:Fourier_repn_of_V} implies that $\norm{W_0}_{\Omega_0^U} \leq C A e^{-\frac14\gamma\beta}$.
By Theorem~\ref{thm:local_part_of_K_j+1} and Proposition~\ref{prop:j0_definition},
$\epsilon_0 = L^{-3j_0 (\rho_J, L, \delta)} \epsilon_{nl} (\beta, A,L)$ is only polynomially decaying in $\beta$.
Therefore $\norm{W_0}_{\Omega_0^U} \leq C A e^{-\frac14\gamma\beta}  <\epsilon_0$
for sufficiently large $\beta$, and the assumption concerning $W_0$ of
Theorem~\ref{thm:stable_manifold} is satisfied with ${W}_0 = \tilde{U}$.
Also by the choice of $|s| \leq \epsilon'_{\delta} = \min\{\epsilon_{\delta},  \theta_J \epsilon_s, \frac{1}{4} \}$ above and because
$v_J^2 \geq 1/2$, it is also true that $r \betaeff (s, J) = r (1 + s v_J^{-2})^{-1} \beta  \geq \frac{10}{12} r \beta \geq \frac{20}{12} \betafree$, verifying the other assumption of Theorem~\ref{thm:stable_manifold}.
Hence by the theorem, there is $\mathfrak{s}^c_0 (\beta, s) = O(e^{-\frac{1}{4}\gamma \beta})$ so taking $\beta$ sufficiently large so that $|\mathfrak{s}_0^c(\beta,s)| \leq \epsilon'_{\delta} /2$ for all $|s|\leq \epsilon'_{\delta} $ then \emph{(i)} follows from continuity:
  if $f(s)=s-\mathfrak{s}_0^c(\beta,s)$ then $f(+\epsilon'_{\delta} ) \geq \epsilon'_{\delta} /2$ and $f(-\epsilon'_{\delta} )<- \epsilon'_{\delta} /2$.
  By the intermediate value theorem there is $s$ such that $f(s)=0$ which is the claim.
  
To see \emph{(ii)}, first fix $r=1$ and $\rho_J$ large enough to satisfy the assumptions of Lemma~\ref{lem:W_norm}.
Having $v_J^2$ sufficiently large and $\beta \geq (1+\delta)\betafree (J) = 8\pi (1+\delta) v_J^2$ is again sufficient to obtain 
$\norm{W_0}_{\Omega_0^U} \leq C {A} e^{-\frac{1}{4} \gamma \beta} \leq \epsilon_0$.
Then we choose $\epsilon''_{\delta} = \min\{\epsilon_{\delta},  \theta_{\cJ} \epsilon_s, \frac{1}{4} \}$ (in place of $\epsilon'_{\delta}$)
so we have a common domain $[-\epsilon''_{\delta} , \epsilon''_{\delta}]$ of $s$ on which Theorem~\ref{thm:stable_manifold} is satisfied for all $J\in \cJ$.
Moreover, whenever $|s| \leq \epsilon''_{\delta} \leq \frac{\delta}{4}$,
\begin{align}
	\betaeff (s, J) = (1 + s v_J^{-2})^{-1} \beta \geq (1+\delta/2)^{-1} \beta \geq (1+\delta)(1+\delta/2)^{-1} \betafree (J)
\end{align}  
hence one has uniform lower bound of $\betaeff(s, J)/\betafree (J)$ greater than $1$.
Since $\mathfrak{s}_0^c (\beta, c) = O(e^{-\frac{1}{4}\gamma \beta}) = O(e^{-2\pi \gamma v_J^2 })$ by Theorem~\ref{thm:stable_manifold},
taking $v_J^2 \geq C |\log \delta|$ for sufficiently large $C$ gives $|\mathfrak{s}_0^c|\leq \epsilon''_{\delta} / 2$.
The same continuity argument as in \emph{(i)} then applies to give the conclusion.
\end{proof}

\begin{proof}[Proof of Proposition~\ref{prop:stable_manifold}]
  The claims are a consequence of Corollary~\ref{cor:tuning_s} and Proposition~\ref{prop:inf_vol_RG}. 
To be more specific, 
we first tune the initial condition to $(s_0, W_0, K_0) = (s_0^c  , \tilde{U}, 0)$ and assume as an induction hypothesis that the flow of $(s_k, W_k, K_k^{\Lambda_N})$ determined by $( \Phi_k^{ \Lambda_N} )_{k\leq j}$ exists up to $k \leq j \leq N-2$. 
Then by Proposition~\ref{prop:inf_vol_RG}~(iv),  
they have the same coupling constants $s_j$ and $W_j$ as the flow defined by $(\Phi^{\Z^2}_k)_{k \leq j}$ with the same initial conditions, thus in particular $\norm{U_j}_{\Omega_j^U}  \leq O(\norm{W_0}_{\Omega_0^U} ) L^{-\alpha j}$.
Now by \eqref{eq:bound_for_L_j_K_j} and \eqref{eq:bound_for_derivative_of_Nj}, 
since $2 C_1 L^2 \alphaLoc \leq L^{-\alpha}$,
\begin{align}
	\norm{K_{j+1}^{ \Lambda_N}}_{\Omega_{j+1}^K}
		\leq \frac{1}{2} L^{-\alpha} \norm{K^{ \Lambda_N}_j}_{\Omega_j^K} + C_2 ( \norm{K^{ \Lambda_N}_j}_{\Omega_j^K} + \norm{U_j}_{\Omega_j^U}  )^2
		,
\end{align}
for $j\leq N-2$. The flow of $( K_{j}^{\Lambda_N}) )_{0\leq j \leq N-1}$ is thus dominated by an exponentially converging sequence uniformly in $N$, i.e.,  
if $(k_j)_{j \geq 0}$ solves $k_0 = 0$ and
\begin{align}
	k_{j+1} =  \frac{1}{2} L^{-\alpha} k_j + C_2 ( k_j + \norm{U_j}_{\Omega_j^U}  )^2	,
\end{align}
then $\norm{K_{j}^{ \Lambda_N}}_{\Omega_{j}^K} \leq k_j \leq O( \norm{W_0}_{\Omega_0^U} L^{-\alpha j} )$ for $\norm{W_0}_{\Omega_0^U}$ sufficiently small. 
Thus the flow of the renormalisation group coordinates exist up to scale $j+1$, completing the induction. 
The estimates \eqref{eq:stable_manifold_bounds} are by-products of the induction. 

\end{proof}

\section{Torus scaling limit}
\label{sec:integration_of_zero_mode}

We assume that the conclusions of Proposition~\ref{prop:stable_manifold} hold.
In particular, we will fix $s=s_0^c(J,\beta)$ and the renormalisation group flow $(E_j,U_j,K_j)$ satisfies \eqref{eq:stable_manifold_bounds} for $j<N-1$.
We now consider the final renormalisation group steps corresponding to the 
covariances $\Gamma_{N}^{\Lambda_N}(s)$ and $t_N(s,m^2)Q_N$ -- as was done in \eqref{eq:barC^Lambda_N}, $m^2$ is set to be 0 in $\Gamma_N^{\Lambda_N}$.

\subsection{Final renormalisation group steps}
\label{sec:finalsteps}

We first consider the $N$-th renormalisation group step corresponding to the covariance $\Gamma_N^{\Lambda_N} = \Gamma_N^{\Lambda_N} (s)$. 
 
\begin{proposition}[Integration with respect to the bounded covariance] \label{prop:final_RG_v2}
Let 
\begin{align}
\Phi_{N}^{\Lambda_N} : (E_{N-1}, s_{N-1}, W_{N-1}, K_{N-1}) \mapsto (E'_{N}, s'_{N}, W'_{N}, K'_{N})
\end{align}
be defined according to Definition~\ref{def:evolution_of_U} and Definition~\ref{def:evolution_of_remainder} but with $\Gamma_{j+1}$ replaced by $\Gamma_{N}^{\Lambda_N}$. 
Then
\begin{equation}
Z_{N} (\varphi') := e^{-E'_{N}|\Lambda_N|} ( e^{U'_{N} (\Lambda_N, \varphi') } + K'_{N} (\Lambda_N, \varphi') ) = \E^{\zeta}_{\Gamma^{\Lambda_N}_{N}}[ Z_{N-1}(\varphi' + \zeta)] \label{eq:Z_N,N_v2}
\end{equation}
where $U'_N (\Lambda_N, \varphi') = \frac{1}{2} s'_N |\nabla\varphi'|^2_{\Lambda_N} + W'_N (\Lambda_N, \varphi')$
and $E'_{N}, s'_{N}, W'_{N}, K'_{N}$ satisfy the estimates of Theorem~\ref{thm:H_j_E_j_estimate} and Theorem~\ref{thm:local_part_of_K_j+1}.

In particular, there are $C > 0$,  $\epsilon \equiv \epsilon(\beta, A,L)$ (only polynomially small in $\beta$) such that whenever $\norm{(U_{N-1}, K_{N-1})}_{\Omega_{N-1}} \leq \epsilon$, then
\begin{align}  \label{eq:U_N_K_N_bound}
\norm{(U'_N, K'_N)}_{\Omega_N} \leq C L^2 \norm{(U_{N-1}, K_{N-1})}_{\Omega_{N-1}}
\end{align}
\end{proposition}

\begin{proof}
  The identity \eqref{eq:Z_N,N_v2} is true by construction since $\cB_N(\Lambda_N)$ only consists of the empty polymer and $\Lambda_N$ itself.
  Also the estimates of Theorem~\ref{thm:H_j_E_j_estimate} and Theorem~\ref{thm:local_part_of_K_j+1} hold because 
$\Gamma_{N}^{\Lambda_N}$ satisfies the same estimates as $\Gamma_{N}$,
  cf. Corollary~\ref{cor:Gammaj} and Lemma~\ref{lemma:fine_Gamma_estimate}
  for the covariance estimates and Proposition~\ref{prop:E_G_j} for the corresponding regulators.
  
  To see the final remark, notice that the analogue of Theorem~\ref{thm:H_j_E_j_estimate} bounds $s_N$. Also since
  \begin{align}
  W'_N (\Lambda_N, \varphi) = L^{-2N} \sum_{q=1}^{\infty} L^2 e^{-\frac{1}{2} \beta q^2 \Gamma_N^{\Lambda_N} (0)} z_{N-1}^{(q)} \sum_{x\in B} \cos(q \sqrt{\beta} \varphi(x) )
  \end{align}
and $\Gamma_{N}^{\Lambda_N} (0)\geq 0$, we have $\norm{U'_N}_{\Omega_N^U} \leq C L^2 \norm{U_{N-1}}_{\Omega_N^U}$. Also \eqref{eq:bound_for_L_j_K_j} and \eqref{eq:bound_for_N_j_K_j} bound $K'_N$, but now
\begin{align}
\alphaLoc = CL^{-3}(\log L)^{3/2}+ C\min\ha{1,\sum_{q\geq 1}
      e^{\sqrt{\beta}qh}
      e^{-(q-1/2)r\beta\Gamma_{N}^{\Lambda_N}(0)}} \leq 2 C ,
\end{align}
so $\norm{K'_N}_N \leq 2C L^2 \norm{(U_{N-1}, K_{N-1})}_{\Omega_{N-1}}$.
\end{proof}

Finally, we consider the integration of the zero mode, i.e., the last covariance $t_NQ_N$.
Since $Q_N$ is the orthogonal projection onto the constant vectors in $\R^\Lambda$,
\begin{align}
  \tilde Z_{N}(\varphi;m^2)
  := \E_{\bar{C}(s,m^2)} Z_0(\varphi+\zeta)
  &= \E_{t_N(s,m^2)Q_N} Z_{N}(\varphi+\zeta)
    \nnb
    &= \sqrt{\frac{|\Lambda_N|}{2\pi t_N}} \int_{\R} e^{-\frac{|\Lambda_N|}{2t_N}\zeta^2} Z_{N} (\varphi+\zeta)\,  d\zeta
\end{align}
where $\bar{C}$ is defined by \eqref{eq:barC^Lambda_N}.
 For constant $\zeta$, using $G_N(\Lambda,\varphi+\zeta)=G_N(\Lambda,\varphi)$ for such $\zeta$,
 see \eqref{eq:def_large_field_regulator},
\begin{align}
e^{E'_N |\Lambda_N|}  Z_{N}(\varphi+\zeta) & =e^{U'_{N} (\Lambda_N, \varphi+\zeta)} + K'_{N} (\Lambda_N, \varphi+\zeta) \nnb
  & = e^{\frac{1}{2}s'_N |\nabla \varphi|^2_{\Lambda_N}} \big( 1+ O( \norm{W'_N}_{\Omega_{N}^U} ) \big) + O(\|K'_{N} \|_{\Omega_N^K}) G_N(\Lambda_N, \varphi).
\end{align}
whenever $\norm{W'_N}_{\Omega_{N}^U} \leq 1$.
The last right-hand side is independent of $\zeta$, so uniformly in $m^2>0$,
\begin{equation}
  \tilde Z_N(\varphi;m^2) 
  = e^{-E'_N |\Lambda_N|} \Big( e^{\frac{1}{2} s'_{N}|\nabla \varphi|_{\Lambda_N}^2 } \big( 1 + O( \norm{W'_N}_{\Omega_{N}^U} ) \big) + O(\|K'_{N} \|_{\Omega_N^K}) G_N(\Lambda_N, \varphi) \Big).
  \label{eq:Z_N,N}
\end{equation}

\subsection{Proof of Theorem~\ref{thm:highbeta}}

To prove the theorem, we will apply   \eqref{eq:Z_N,N} with
\begin{equation} \label{eq:uN_def}
	\varphi=u_N:=\tilde C(s,m^2)(1 + s\gamma \Delta)^{-1}f_N
\end{equation}
where the covariance $\tilde C(s,m^2)$
is the one from Lemma~\ref{lemma:reformulation} and $f_N$ is as in Theorem~\ref{thm:highbeta}.
The next elementary lemma shows that the exponential term and the regulator of the above
are bounded for this choice.

\begin{lemma} \label{lemma:integral_of_zero_mode}
  Let $J \subset \Z^2\setminus 0$ be any finite-range step distribution as in Section~\ref{sec:finite_range_decomposition},
  and assume that $\theta_J$ is bounded below (see \eqref{eq:theta_def}).
  Let $f\in C^\infty(\T^2)$ with $\int_{\T^2} f\, dx=0$, let $f_N$ be given by \eqref{eq:f_N_definition},
  and define $u_N$ by \eqref{eq:uN_def}.
  Then there are constants $C,c > 0$ uniform in $m^2 \geq 0$ and $N \in \mathbb{N}$
  such that for $|s|\leq c$,
\begin{equation}
  | \nabla u_N|^2_{\Lambda_N} \leq C, \qquad 
    G_N(\Lambda_N, u_N) \leq C . \label{eq:last_regulator_has_unif_bound}
  \end{equation}
Further,
  \begin{equation}
\lim_{N\to\infty} \lim_{m^2\downarrow 0}(f_N,\tilde C(s,m^2)f_N) = \frac{1}{s+v_{J}^2} (f,(-\Delta_{\T^2})^{-1}f)_{\T^2} . \label{eq:fCf_limit}
  \end{equation}
\end{lemma}

The proof of the lemma is given after the following conclusion of
the proof of our main theorem.

\begin{proof}[Proof of Theorem~\ref{thm:highbeta}]
By assumption, the conditions of Proposition~\ref{prop:stable_manifold}(i) hold
and $Z_{N}$ and $\tilde Z_{N}$ are then defined as above.
By Lemma~\ref{lemma:m2to0} and Lemma~\ref{lemma:reformulation},
\begin{equation}
  \langle e^{(f_N, \varphi)} \rangle_{J, \beta}^{\Lambda_N}
  = \lim_{m^2 \downarrow 0} \langle e^{(f_N, \varphi )} \rangle_{\beta, m^2}^{\Lambda_N}
  =  \lim_{m^2 \downarrow 0} e^{\frac{1}{2} (f_N, \tilde{C}(s, m^2) f_N)} \frac{\E_{C(s, m^2)} [Z_0 (\zeta + u_N)]}{\E_{C(s, m^2)} [Z_0 (\zeta)]} ,
\end{equation}
and by Lemma~\ref{lemma:m2to0_with_frd},
\begin{equation}
  \lim_{m^2 \downarrow 0} \E_{C(s,m^2)} [Z_0 (\zeta + u_N)]
  = \lim_{m^2 \downarrow 0}
  \E_{t_N(s,m^2)Q_N} Z_{N} ( \zeta + u_N ).
\end{equation}
But by \eqref{eq:Z_N,N} and \eqref{eq:last_regulator_has_unif_bound}, 
\begin{align}
  \frac{\mathbb{E}_{t_NQ_N} [Z_{N} (\zeta + u_N)]}{\mathbb{E}_{t_NQ_N}  [Z_{N} (\zeta)]}
  & = e^{\frac{1}{2}s'_{N} (u_N, -\Delta u_N)} \big(1+ O(\norm{W'_N}_{\Omega_N^U})  \big) + O(\norm{K'_{N}}_{\Omega_N^K}) G_N (\Lambda_N,u_N) \nnb
   & = e^{O(s'_{N})} \big(1+ O(\norm{W'_N}_{\Omega_N^U})  \big) + O(\norm{K'_{N}}_{\Omega_N^K})
\end{align}
while Proposition~\ref{prop:stable_manifold}(i) and \eqref{eq:U_N_K_N_bound} 
implies that 
$|s'_{N}|+ \norm{W'_N}_{\Omega_N^U} + \norm{K'_{N}}_{\Omega_N^K} = O(L^{-\alpha N})$,
provided that $s_0$ and $s$ are tuned to the correct initial value $s_0^c (J, \beta)$.
Therefore the limit in $N\rightarrow \infty$ converges to $1$, uniformly in $m^2 >0$, hence in particular
\begin{equation}
  \lim_{N\rightarrow \infty} \lim_{m^2 \downarrow 0}
  \frac{\E_{C(s_0^c(J, \beta), m^2)} [Z_0 (\zeta+u_N)]}{\E_{C(s_0^c(J, \beta), m^2)} [Z_0 (\zeta)]} =1.
\end{equation}
Also, by \eqref{eq:fCf_limit},
\begin{equation}
\lim_{N\rightarrow \infty} \lim_{m^2 \downarrow 0}  e^{\frac{1}{2} (f_N, \tilde{C}(s_0^c(\rho, \beta), m^2) f_N )} = \exp\pa{ \frac{(f,  (-\Delta_{\mathbb{T}^2})^{-1} f)}{2 (v_{J}^2 + s_0^c (J, \beta))}  }.
\end{equation}
In view of the rescaling discussed around \eqref{eq:Omega_rescaled},
this proves  the main conclusion
\eqref{e:highbeta-convergence}
with $\beta_{\rm eff}(J, \beta) = \beta v_J^2/(v_{J}^2 + s_0^c (J, \beta))$.
\end{proof}

\begin{proof}[Proof of Remark~\ref{rk:highbeta-rho}]
Let $\cJ = \{ J_{\rho} : \rho \in \N \}$ be the family of range-$\rho$ step distributions.
Then by Lemma~\ref{lem:lambda_rho_error}, if we let $\theta_{\cJ} = \frac{1}{3^2} = \frac{1}{9}$, then $\theta_J \geq \theta_{\cJ}$ for each $J\in \cJ$ and $v_{J_{\rho}}^2 \sim \frac{1}{6} \rho^2$. 
Hence $\cJ$ satisfies the assumptions of Proposition~\ref{prop:stable_manifold}(ii),
so there exists $C >0$ such that for any $\delta >0$, 
$\norm{U_j}_{\Omega_j^U}$ and $\norm{K_j}_{\Omega_j^K}$ both decay exponentially in $j$ and uniformly in $\Lambda_N$ whenever $\rho^2 \geq C | \log \delta |$ and $\beta \geq \beta_0 (J_{\rho}) = (1+ \delta) \betafree (J_{\rho}) \sim \frac{4\pi(1+\delta)}{3} \rho^2$ as $\rho \rightarrow \infty$. 
Together with \eqref{eq:U_N_K_N_bound}, this implies $|s'_N| + \norm{W'_N}_{\Omega_N^U} + \norm{K'_N}_{\Omega_N^K} \rightarrow 0$ as $N\rightarrow \infty$ for $\beta \geq \beta_0 (J_{\rho})$.
Therefore we may follow exactly the same proof as that of Theorem~\ref{thm:highbeta}, but in the temperature range $\beta \geq (1+ \delta) \betafree (J_{\rho})$.
\end{proof}

\subsection{Proof of Lemma~\ref{lemma:integral_of_zero_mode}}

The proof of Lemma~\ref{lemma:integral_of_zero_mode} uses the following standard estimates for the Fourier coefficients
of the functions $f_N$.

\begin{lemma} \label{lemma:hat_f_N_bound}
  For $f\in C^{\infty} (\mathbb{T}^2)$, let $f_N$ be given by \eqref{eq:f_N_definition}. Then there exist constants $C_a$ for $a \geq 0$
  independent of $\Lambda_N$ such that,
  for any $p \in \Lambda_N^* \subset [-\pi, \pi]^2$,
\begin{equation}
  | \hat{f}_N (p) | \leq C_{a} \norm{\nabla^{2a} f}_{L^{\infty} (\T^2)} L^{-2N a}  |p|^{-2 a}.
\end{equation}
\end{lemma}
\begin{proof}
Define two components of the Fourier multiplier
\begin{align}
\lambda_1 (p) = 2 - 2\cos(p_1), \quad \lambda_2 (p) = 2- 2\cos (p_2)
\end{align}
for $p = (p_1, p_2)$ so that $\lambda (p) = \lambda_1 (p) + \lambda_2 (p)$.
One has
\begin{align}
\hat{f}_N (p) = 
\begin{array}{ll}
\begin{cases}
\frac{1}{|\Lambda_N|} \sum_{x\in \Lambda_N} f(L^{-N} x) e^{-i p\cdot x} &  \text{if} \;\; p \neq 0 \\
0 & \text{if} \;\; p = 0
\end{cases}
\end{array}
\end{align}
hence for $k\in\{1,2\}$,
\begin{equation}
\lambda_k^a (p) |\hat{f}_N (p) | = \Big|  \frac{1}{|\Lambda_N|} \sum_{x\in \Lambda_N} e^{-ip \cdot x} ( \partial_{\Lambda_N}^{(e_k, -e_k )} )^a f(L^{-N}x) \Big| 
\leq \sup_{x \in \Lambda_N} | ( \partial_{\Lambda_N}^{(e_k, -e_k)} )^a f(L^{-N}x)| \label{eq:discrete_Laplacian_bound}
\end{equation}
where $\partial_{\Lambda_N}^{(e_k, -e_k)} f(x/L^N) = -f( (x+e_k) / L^N ) - f((x-e_k)/L^N) + 2 f(x/L^N)$ for $x\in \Lambda_N$. But since $\lambda_k (p) \geq \frac{4}{\pi^2} p_k^2$ by Lemma~\ref{lem:lambda_error}, we are just left to bound $| ( \partial^{(e_k, -e_k)} )^a f(x/L^N)|$.
We now claim that
\begin{align}
( \partial_{\Lambda_N}^{(e_k, -e_k)} )^a f (z) =  \int_{[0, L^{-N}]^{2a}} \prod_{l=1}^a ds_l \, dt_l \, \partial_{x_k}^{2a} f \big( z + \sum_{l=1}^{a} (s_l+t_l -L^{-N}) e_l   \big) . \label{eq:derivative_integral_representation}
\end{align}
To see this, start from the elementary observation
\begin{multline}
\partial_{\Lambda_N}^{(e_k, -e_k)}  f (z)= \\2f(z) - f(z+L^{-N}e_k) - f(z-L^{-N}e_k)  = - \int_{[0, L^{-N}]^2} ds \, dt \, \partial_{x_k}^2 f \big( z + (s+t-L^{-N}) e_k \big)
\end{multline}
and proceed by induction. Now by \eqref{eq:discrete_Laplacian_bound} and \eqref{eq:derivative_integral_representation},
\begin{align}
|\hat{f}_N (p)| \leq C_a |p_k|^{-2a} L^{-2Na} \norm{\nabla^{2a} f}_{L^{\infty} (\T^2)}
\end{align}
for $k=1,2$, which concludes the proof.
\end{proof}

\begin{proof}[Proof of Lemma~\ref{lemma:integral_of_zero_mode}]
We defined
\begin{align}
& u_N = (1 + s\gamma \Delta)^{-1} \tilde{C} (s, m^2) f_N \label{eq:varphi_N_definition} \\
& \tilde{C}(s,m^2)= \gamma(1+ s\gamma\Delta) + (1+s\gamma\Delta) C(s,m^2)(1+s\gamma\Delta).
\end{align}
We claim a bit stronger statement than the first inequality in \eqref{eq:last_regulator_has_unif_bound}: for any $f\in C^\infty(\T^2)$ and all $a\in \{1,2,3,4\}$, the norm $\norm{\nabla^a_N u_N}_{L^2_N (\Lambda_N)}$ is bounded uniformly in $N$.
Indeed, in Fourier space, and recalling that $\lambda$ (resp.~$\lambda_J$) denote the Fourier multipliers of $-\Delta$ (resp.~$-\Delta_J$),
\begin{align}
	\hat{u}_N (p) = \frac{(\lambda_J (p) + m^2)^{-1}}{1 + s \lambda (p) ( (\lambda_J (p) + m^2)^{-1} - \gamma )} \hat{f}_N (p)
	.
\end{align}
Since $\lambda (p) \in [0,2]$ and
$(\lambda_J (p) + m^2)^{-1} \lambda (p) \leq \theta_J^{-1}$,
for $|s|$ small,
\begin{align}
|\hat{u}_N(p)| \leq C \theta_J^{-1} \lambda(p)^{-1} |\hat{f}_N (p)|
\end{align}
and for $\Lambda_N^* = 2\pi L^{-N} \Lambda_N$,
\begin{align}
  \norm{\nabla^a_N u_N}_{L^2_N (\Lambda_N)} := L^{2Na - 2N} \norm{\nabla^{a} u_N}^2_{L^2 (\Lambda_N)}
  &\leq \frac{C \theta_J^{-2} |\Lambda_N|^{a - 2} }  {4\pi^2} \sum_{p \in \Lambda_N^*} \lambda (p)^{a - 2} |\hat{f}_N (p)|^2  \\
&= \frac{C \theta_J^{-2}}{4\pi ^2} \sum_{k \in \Lambda_N} \big( |\Lambda_N| \lambda( 2\pi L^{-N} k ) \big)^{a-2}  | \hat{f}_N (2\pi L^{-N} k)  |^2 .\nonumber
\end{align}
By Lemma~\ref{lem:lambda_error} and the lower bound $\cos(x) \geq 1 - x^2/2$,
\begin{align}
16 |k|^2 \leq |\Lambda_N| \lambda (2\pi L^{-N} k) = 2 L^{2N} (2 - \cos ( 2\pi L^{-N} k_1 ) - \cos( 2\pi L^{-N} k_2) ) \leq 4\pi^2 |k|^2
\end{align}
and together with Lemma~\ref{lemma:hat_f_N_bound}, we have
\begin{align}
\sum_{k \in \Lambda_N} \big( |\Lambda_N| \lambda( 2\pi L^{-N} k ) \big)^{a-2}  | \hat{f}_N (2\pi L^{-N} k)  |^2 \leq C_a (1 + \sum_{k\in \Lambda_N \backslash \{ 0\} } |k|^{2(a-2)} |k|^{-2a} ) < \infty
\end{align}
for $a\in \{1,2,3,4\}$.
The case $a=1$ concludes the proof of the first inequality \eqref{eq:last_regulator_has_unif_bound}.
Moreover, by the lattice Sobolev inequality (Lemma~\ref{lemma:lattice_sobolev_inequality}) there exists $c'>0$ such that
\begin{align}
\log G_N (\Lambda_N,u_N) \leq c' \kappa_L \sum_{a=1}^4 \norm{\nabla^a_N u_N}_{L^2_N (\Lambda_N)},
\end{align}
also giving the second inequality in \eqref{eq:last_regulator_has_unif_bound}.
For the final claim \eqref{eq:fCf_limit}, recalling $\hat{f}_N (0) = 0$, one has
\begin{align}
	\lim_{m^2\rightarrow 0} (f_N, \tilde{C}(s,m^2) f_N)
	= \frac{1}{4\pi^2 |\Lambda|} \sum_{k \in 2\pi \Lambda_N \backslash \{0\} } 
		\frac{ \lambda_{J}( L^{-N}k)^{-1} (1- s \gamma \lambda (L^{-N} k) )}
		{1+ s  \lambda ( L^{-N}k)  \big(  \lambda_{J} (L^{-N}k)^{-1} - |\Lambda_N|^{-1} \gamma  \big) } |\hat{f}_N ( L^{-N}k)|^2 .
\end{align}
Since $f\in C^{\infty} (\mathbb{T}^2)$, we have $\hat{f}_N ( L^{-N} k) \rightarrow \hat{f} (k)$ as $N\rightarrow \infty$ for each $k \in (2\pi \mathbb{Z})^2 \backslash\{0\}$.
By Lemma~\ref{lem:lambda_error},
\begin{equation}
  \lim_{N\rightarrow \infty} L^{2N} \lambda (L^{-N} k) = |k|^2,
  \qquad \lim_{N\rightarrow \infty} L^{2N} \lambda_{J} (L^{-N} k) = v_{J} ^2|k|^2 ,
\end{equation}
where $v_{J}^2$ is defined by \eqref{eq:vJ2_variance}.
Also by Lemma~\ref{lemma:hat_f_N_bound}, the sum is dominated by $C \sum_{k \in (2\pi \mathbb{Z})^2 \backslash\{0\}} |k|^{-4}$ for some $C > 0$, and therefore the Dominated convergence theorem implies
\begin{equation}
  \lim_{N\rightarrow \infty} \lim_{m^2 \rightarrow 0} (f_N, \tilde{C}(s, m^2) f_N)
  = \frac{1}{4\pi^2} \sum_{k \in (2\pi \mathbb{Z})^2 \backslash \{0\}} \frac{1}{v_{J}^2 + s} |k|^{-2} |\hat{f}(k)|^2
  = \frac{1}{v_{J}^2 +s} (f, (-\Delta_{\mathbb{T}^2})^{-1} f)
\end{equation}
as needed.
\end{proof}

\appendix

\section{Properties of the regulator}
\label{app:regulator}

In this appendix, we prove the properties of the regulator as introduced in Definition~\ref{def:G_j}.
The choice of this weight is essentially that from \cite{MR2523458}, and the estimates we derive
follow that reference, incorporating the improvements from \cite[Appendix~D]{MR2917175}.
In our presentation, we pay particular care to obtain estimates
with the correct dependence on the range of the step distribution $J$.
Some simplifications in our presentation result from the use of the continuous scale
decomposition from Section~\ref{sec:finite_range_decomposition} and the use of the
discrete Sobolev trace theorem.

\subsection{Lattice Sobolev estimates}
\label{sec:sobolev}

The proof of Lemma~\ref{lemma:G_change_of_scale}
heavily depends on lattice versions of the Sobolev inequality and the
trace theorem for Sobolev spaces. We include the versions we need here.
To simplify notation, from now on we fix $d=2$.

\paragraph{Trace theorem}
Consider a block $B\in \{1, \dots, R\}^2$ with $l_1 = \{0\} \times [1, R]$, $l_2 = [1 , R] \times \{ R +1 \}$, $l_3 = \{ R + 1\} \times [1, R ]$, $l_4 = [1, R ] \times \{0\}$. Note that  $l_i$'s
are the outer boundary,
which are different from $\partial B$, which is the inner boundary.

\begin{lemma} \label{lemma:trace_theorem}
  For any $u : \{0, \cdot, R+1\}^2 \rightarrow \R$,
  if $(k, \mu_k) \in \{ (1, -e_1), (2, e_2 ), (3, e_1), (4, -e_2)  \}$,
\begin{align}
R^{-1} \sum_{x\in l_k} u(x)^2 \leq R^{-2} \sum_{x\in B}  \Big(  u(x)^2 + R |\nabla^{\mu_k} u(x)^2| \Big).
\end{align}
\end{lemma}
\begin{proof}
Without loss of generality, fix $k=1$. Define a function $\xi : B \cup (\cup_{k=1}^4 l_k) \rightarrow \mathbb{R}$ by
\begin{align}
\xi ( ae_1 + be_2 ) = ( R -a) / R. 
\end{align}
Then
\begin{align}
R^{-1} \sum_{x\in l_1} u(x)^2 &=  R^{-1} \sum_{x\in l_1} \xi(x) u(x)^2 \nnb
&= R^{-1} \sum_{k=1}^{L^{j}} \nabla^{-e_1} \Big[ \sum_{b=1}^{L^j} \xi(ke_1 + be_2) u(ke_1 + be_2)^2 \Big] \nnb
&= R^{-1} \sum_{x\in B}  \, \nabla^{-e_1}(\xi(x) u(x)^2) \nnb
&= R^{-2} \sum_{x\in B} \Big( ( R \, \nabla^{-e_1} \xi(x)) u(x)^2  + \xi(x-e_1) R \, \nabla^{-e_1} (u(x)^2) \Big).
\end{align}
But $|\xi(x)|,  R | \nabla^{-e_1} \xi(x)| \leq 1$ for $x \in B$ and hence the result follows.
\end{proof}

In particular, this lemma can be applied to the control the field on the boundary of the box by the field inside the box.

\begin{corollary} \label{cor:discrete_trace_theorem}
  Let $B$ be a box with outer boundary $\cup_k l_k$ as above and diameter $R \geq 10$.
  Then for $\varphi: \{0, \cdot, R+1\}^2 \to \R$,
\begin{align} 
R^{-1} \sum_{x\in \cup_k l_k} |\nabla \varphi(x)|^2 \leq 10  \Big( R^{-2} \sum_{x\in B } |\nabla \varphi(x)|^2 + \norm{\nabla^2 \varphi}_{L^{\infty}(B)} \Big). \label{eq:discrete_int_by_parts_rec_step}
\end{align}
\end{corollary}
\begin{proof}
In Lemma~\ref{lemma:trace_theorem}, set
\begin{align}
u(x)^2 = |\nabla^{\mu} \varphi(x)|^2.
\end{align} 
Then for $\nu \in \hat{e} = \{\pm e_1, \pm e_2\}$,
\begin{align}
R |u(x+\nu)^2 - u(x)^2| &= R|\nabla^{\mu} \varphi(x + \nu) + \nabla^{\mu}\varphi(x) | \, |\nabla^{\nu} \nabla^{\mu} \varphi(x) | \nnb
&\leq \frac{1}{2}  |\nabla^{\mu} \varphi(x + \nu) + \nabla^{\mu} \varphi(x) |^2 + R^2 \norm{\nabla^2 \varphi}_{L^{\infty}(B)}.
\end{align} 
so summation over $x\in B$ gives
\begin{align}
R \sum_{x\in B} |\nabla^{\nu} u(x)^2 | \leq \sum_{x\in B \cup l_k}  |\nabla^{\mu} \varphi (x)|^2 + R^4 \norm{\nabla^2 \varphi}_{L^{\infty}(B)}. 
\end{align}
Therefore the lemma gives
\begin{align}
R^{-1} \sum_{x\in l_k} |\nabla \varphi(x)|^2 \leq 2R^{-2} \sum_{x\in B \cup l_k} |\nabla \varphi(x)|^2 +  R^2 \norm{\nabla^2 \varphi}_{L^{\infty}(B)}.
\end{align}
If $R \geq 10$, we may send the $l_k$ part in the sum $\sum_{x\in B \cup l_k}$ to the left-hand side to obtain
\begin{align} 
R^{-1} \sum_{x\in l_k} |\nabla \varphi(x)|^2 \leq \frac{10}{4}  \Big( R^{-2} \sum_{x\in B } |\nabla \varphi(x)|^2 + \norm{\nabla^2 \varphi}_{L^{\infty}(B)} \Big). 
\end{align}
\end{proof}

\paragraph{Sobolev inequality}

While the large field regulator $G_j$ contains $\exp ( \norm{\nabla^2 \varphi}^2_{L^{\infty}(B^*)})$ for $B\in \mathcal{B}_j$, we have a nice estimate of Gaussian integration only for exponentials of quadratic forms. Hence it is desirable to bound $\norm{\nabla^2 \varphi}^2_{L^{\infty}(B^*)}$ in terms of $\norm{\nabla^a \varphi}^2_{L^{2} (B^*)}$, $a\geq 2$.
This follows from the following Sobolev inequality.
Here, we are using the convention $\norm{f}_{L^2 (X)} = \sum_{x\in X} |f(x)|^2$.

\begin{lemma} \label{lemma:lattice_sobolev_inequality}
  For $B$ be square of diameter $R$ as above.
  There exists a constant $C >0$ uniform in $R$ such that
  for all $f : \{ x\in \Lambda : d_1 (x, B) \leq 2 \} \rightarrow \mathbb{R}$, 
\begin{equation}
  \norm{f}^2_{L^{\infty} (B)} \leq  C \sum_{a=0}^2 R^{2a-2} \norm{\nabla^a f}^2_{L^2 (B)} . \label{eq:lattice_sobolev1} 
\end{equation}
\end{lemma}
\begin{proof}
Take $x\in [1, \frac{R+1}{2}]^2 \cap B$. By symmetry, the conclusion follows if we bound $f(x)^2$ in terms of $\norm{\nabla^a f}^2_{L^2 (B)}$, $a=0,1,2$. Take
\begin{align}
\xi_x (x+ ae_1 + be_2) = \begin{array}{ll}
\begin{cases}
(1 - 3 R^{-1} a)(1 - 3R^{-1} b) &\quad \text{if } 0\leq a,b \leq \frac{1}{3} R  \\
0 & \quad \text{otherwise}
\end{cases}
\end{array}. 
\end{align}
Also let $D_x = \{x+ ae_1 + be_2 : 0\leq a,b \leq \frac{1}{3} R +2 \}$. Then
\begin{align}
f(x)^2 = f(x)^2 \xi_x (x) = \sum_{a,b=0}^{\lfloor R/3 \rfloor} \nabla^{e_1} \nabla^{e_2} ( f(x + ae_1 + be_2)^2 \xi_x(x+ae_1 +be_2) ) \nnb
= \sum_{a,b=0}^{\lfloor R/3 \rfloor} \big( \nabla^{e_1} \nabla^{e_2} f(x + ae_1 + be_2)^2 \big) \xi_x(x + (a+1)e_1 + (b+1)e_2) \nnb
+ \big( \nabla^{e_1} f(x + ae_1 + be_2)^2 \big) \big( \nabla^{e_1} \xi_x (x + ae_1 + (b+1)e_2) \big) \nnb
+ \big( \nabla^{e_1} f(x + ae_1 + be_2)^2 \big) \big( \nabla^{e_2} \xi_x (x + ae_1 + (b+1)e_2) \big) \nnb
+ f(x + ae_1 + be_2)^2 \nabla^{e_2} \nabla^{e_1} \xi_x (x + ae_1 + be_2) \Big).
\end{align}
Noticing that $\norm{\nabla^a_j \xi_x}_{L^{\infty}(B)} \leq 3^a$ for $a=0,1,2$,
\begin{align}
f(x)^2 \leq 9 \sum_{a=0}^2 R^{a-2} \norm{ \nabla^a f^2 }_{L^1 (D_x)}
\end{align}
but also the Young's inequality implies
\begin{align}
\norm{\nabla^2 f^2}_{L^1 (D_x)} & \leq \frac{1}{2}  \norm{\nabla f}^2_{L^2(B)} + R^{-1}\norm{f}^2_{L^2 (B)} + R \norm{\nabla^2 f}^2_{L^2(B)} \\
\norm{\nabla f^2}_{L^1 (D_x)} & \leq R^{-1} \norm{f}_{L^2 (B)} + R \norm{\nabla f}^2_{L^2 (B)}
\end{align}
which completes the proof of the inequalities.
\end{proof}

\subsection{Proof of Lemma~\ref{lemma:G_change_of_scale}}
\label{app:lem_G_change_of_scale}

The proof of the lemma heavily depends on formulas derived from
the lattice versions of the Sobolev estimates, see Section~\ref{sec:sobolev}.

\begin{proof}[Proof of Lemma~\ref{lemma:G_change_of_scale}]
For brevity, $s + M^{-1}$ will be denoted $s'$ and $X_{s'}$ will be denoted $X'$ (which we recall that it is the smallest $j+s'$-polymer containing $X$).
We will bound the terms in $\log G_{j+s} (\varphi + \xi, X)$ one by one,
see \eqref{eq:def_large_field_regulator}.
First, $\norm{\nabla \varphi}^2_{L^2(X)}$ will be isolated from $\norm{\nabla (\varphi + \xi)}^2_{L^2 (X)}$. Let $B\in \mathcal{B}_{j+s} (X)$ and without loss of generality, let $B$, $l_i$ ($i=1,2,3,4$) be as above but $B = [1, L^{j+s}]^2$. Then by discrete integration by parts,
\begin{align}
\sum_{x\in B} \nabla^{e_1} \varphi(x) \nabla^{e_1} \xi(x) = -\sum_{x\in l_3} \xi( x) \nabla^{-e_1} \varphi (x)
- \sum_{x\in l_1} \xi(x+ e_1) \nabla^{e_1} \varphi(x)
+ \sum_{x\in B} \xi (x) \nabla^{e_1} \nabla^{-e_1}  \varphi(x).
\end{align}
Hence in particular, summing this over each direction $\pm e_1, \pm e_2$, $B\in \mathcal{B}_{j+s} (X)$, and using the Young's inequality,
\begin{align}
( \nabla \varphi, \nabla \xi )_X & \leq  \tau \norm{\xi}^2_{L^2_{j+s} (X)} +  \tau^{-1} \norm{\nabla^2_{j+s} \varphi}^2_{L^2_{j+s} (X)} + \tau \norm{\xi}^2_{L^2_{j+s} (\partial X)} + \tau^{-1} \norm{\nabla_{j+s} \varphi}^2_{L^2_{j+s} (\partial X)}  \nnb
& \leq  2\tau W_{j+s} (\xi, X)^2 +  \tau^{-1} \big( \norm{\nabla_{j+s} \varphi}^2_{L^2_{j+s} (\partial X)} + W_{j+s}(\nabla^2_{j+s} \varphi, X)^2 \big)
\end{align}
for any $\tau >0$, and hence
\begin{equation}
\begin{split}
\norm{\nabla_{j+s} (\varphi + \xi)}^2_{L^2_{j+s} (X)} & \leq \norm{\nabla_{j+s'} \varphi}^2_{L^2_{j+s'} (X)} + \norm{\nabla_{j+s} \xi}^2_{L^2_{j+s} (X)} \\
& \quad + 2\tau W_{j+s} (\xi, X)^2 +  \tau^{-1} \big( \norm{\nabla_{j+s} \varphi}^2_{L^2_{j+s} (\partial X)} + W_{j+s}(\nabla^2_{j+s} \varphi, X)^2 \big).
\end{split} \label{eq:G_change_of_scale1}
\end{equation}
Next, we will use the following rather trivial bounds on the other two terms of $\log G_{j+s}$:
\begin{align}
& \norm{\nabla_{j+s} (\varphi + \xi)}_{L^2_{j+s} (\partial X)}^2 \leq 2\norm{\nabla_{j+s} \varphi}^2_{L^2_{j+s} (\partial X)} + 2 W_{j+s} (\nabla_{j+s} \xi , X)^2 \label{eq:G_change_of_scale2} \\ 
& W_{j+s} (\nabla^2_{j+s} (\varphi + \xi), X)^2 \leq 2W_{j+s} (\nabla^2_{j+s} \varphi , X)^2 + 2 W_{j+s} (\nabla^2_{j+s} \xi, X)^2 . \label{eq:G_change_of_scale3}
\end{align}
By \eqref{eq:G_change_of_scale1}, \eqref{eq:G_change_of_scale2}, \eqref{eq:G_change_of_scale3} and setting $c_4 = \max \{ 2 \cwone , 2\tau \cwone , 2 c_2 \}$,
\begin{align}
\begin{split}
\frac{1}{\kappa_L} \log G_{j+s} (\varphi + \xi , X) \leq \cwone \norm{\nabla_{j+s} \varphi}^2_{L^2_{j+s}(X)} + (2c_2 + \cwone \tau^{-1}) \norm{\nabla_{j+s} \varphi}^2_{L^2_{j+s} (\partial X)} \\
+ 2 \cwone (1+\tau^{-1}) W_{j+s} (\nabla^2_{j+s} \varphi, X) + \frac{1}{\kappa_L} \log g_{j+s} (\xi, X).
\end{split}
\end{align}
Now by repeated application of \eqref{eq:discrete_int_by_parts_rec_step}, the discrete trace theorem,
\begin{equation}
  \norm{\nabla_{j+s} \varphi}^2_{L^2_{j+s} (\partial X)} \leq \norm{\nabla_{j+s} \varphi}_{L^2_{j+s} (\partial X')}^2 + 10 \norm{\nabla_{j+s} \varphi}^2_{L^2_{j+s} (X' \backslash X)} + 10 W_{j+s} (\nabla^2_{j+s} \varphi, X' \backslash X)
  .
\end{equation}
Hence by choosing $\tau = \cwone c_2^{-1}$ and $30c_2 \leq 1 \cwone$,
\begin{align}
& \frac{\log( G_{j+s} (\varphi + \xi, X) / g_{j+s} (\xi, X)  )}{\kappa_L} \nnb
& \leq \cwone \norm{\nabla_{j+s} \varphi}^2_{L^2_{j+s} (X')} +3c_2 \norm{\nabla_{j+s} \varphi}^2_{L^2_{j+s} (\partial X')} + 2 \cwone (1+ \tau^{-1}) W_{j+s} (\nabla^2_{j+s} \varphi, X') \nnb
& \leq \cwone \norm{\nabla_{j+s'} \varphi}^2_{L^2_{j+s'} (X')} + 3 \, \ell^{-1} c_2 \norm{\nabla_{j+s'} \varphi}^2_{L^2_{j+s'} (\partial X')} + 2 \, \ell^{-2} \cwone (1+ \tau^{-1}) W_{j+s'} (\nabla^2_{j+s'} \varphi, X') . 
\end{align}
The inequality \eqref{eq:G_change_of_scale} follows upon taking $\ell$ large enough.
\end{proof}

\subsection{Proof of Lemma~\ref{lemma:g_j+s_bound_by_quadratic}}
\label{app:g_j+s_bound_by_quadratic}

\begin{proof}[Proof of \eqref{eq:g_j+s_bound} of Lemma~\ref{lemma:g_j+s_bound_by_quadratic}]
By the Sobolev inequality, Lemma~\ref{lemma:lattice_sobolev_inequality}, for each $a=0,1,2$, we have
\begin{align}
W_{j+s} (X, \nabla^a_{j+s} \zeta) \leq C_{a,d} \norm{\nabla^a_{ j+s} \zeta}_{L^{\infty} (X^*)} \leq C_{a,d}' \sum_{b=0,1,2} \norm{\nabla^{a+b}_{j+s} \zeta  }_{L_{j+s}^2 (X^*)}.
\end{align}
Plugging this into the definition of $g_{j+s} (X, \zeta)$ with scaled coefficients give the desired result.
\end{proof}

For the proof of \eqref{eq:g_j+s_expectation},
we will need the following general estimate for Gaussian fields; 
see \cite[Lemma~6.28]{MR2523458} for a proof.

\begin{lemma} \label{lemma:expectation_of_exp_quad}
Let $\zeta$ be a centered real-valued Gaussian field on a finite set $X$ with covariance matrix $C$ and suppose that the largest eigenvalue of $C$ is smaller or equal to $\frac{1}{2}$. Then
\begin{align}
\mathbb{E} \big[ \exp ( \frac{1}{2}\sum_{x\in X} \zeta (x)^2 ) \big] \leq e^{\operatorname{Tr} C}.
\end{align}
\end{lemma}

Applying this lemma to gradients of the slices $\xi_k$ (see Section~\ref{sec:subscale}) gives the following lemma.

\begin{lemma} \label{lemma:E[exp_grad_Gaussian]}
  For any $j \in \{1, \dots, N\}$, $k \in \{1, \dots, M \}$, $s= \frac{k-1}{M}$, $s' = \frac{k}{M}$, let $Y\in \mathcal{P}^c_{j+s}$,
  and let $\xi$ be a centered Gaussian field with covariance $\Gamma_{j + s, j+ s'}$.
  For a multiindex $(\mu) = (\mu_1, \dots, \mu_a) \in \{\pm e_1, \pm e_2 \}^a$ for $a\in \{0,1,2,3,4\}$,
  then let $\eta_{a, (\mu)}$ be the Gaussian field defined by
  \begin{equation} \label{eq:eta_def}
  \eta_{a, (\mu)} (x) = \rho_J^{2} (\log L)^{-1} L^{-(j+s)} \nabla^{\mu_1 \cdots \mu_a}_{j+s} \xi (x)
  .
\end{equation}
Then there is a constant $C'>0$ such that if $t \leq (2C'\ell^2)^{-1}$, then (recall $L = \ell^M$)
\begin{equation}
\mathbb{E}\big[ e^{\frac{t}{2} \sum_{x\in Y} \eta(x)^2} \big] \leq e^{tC' M^{-1} |Y|_{j+s}} \label{eq:E[exp_grad_Gaussian]}
\end{equation}
where $|Y|_{j+s}$ denotes the number of $L^{j+s}$-blocks contained in $Y$.
\end{lemma}

\begin{proof}
By Lemma~\ref{lemma:fine_Gamma_estimate},
for all $x,y \in \Lambda_N$ and $(\mu)$ as above, defining $\alpha=(\mu, -\mu)$ to be the concatenated multi-index (of length $2a$), 
\begin{equation}\label{eq:xi_k_cov}
   \big| \E^{\xi_k}\big[  (\nabla^{(\mu)}_{j+s} \xi_k)(x) (\nabla^{(\mu)}_{j+s}\xi_k)(y)\big]\big|
   \leq  C_{\alpha}\rho_J^{-2} \log \ell,
\end{equation}
which follows by considering the (worst) case estimate $|\alpha|=0$ in Lemma~\ref{lemma:fine_Gamma_estimate}.
Letting $H_{a, (\mu)}  (x,y) = \operatorname{Cov}(\eta_{a, (\mu)} (x), \eta_{a, (\mu)} (y))$, it follows from \eqref{eq:eta_def} and \eqref{eq:xi_k_cov} that
there is a constant $C'>0$ such that for all $x,y \in Y$ and $a=0,1,\dots,4$,
\begin{align} \label{eq:H_abound}
  | H_{a, (\mu)}  (x,y) | \leq 
  C' 
  (\log L)^{-1} L^{-2(j+s)} \log \ell.
\end{align}
Since also $H_{a, (\mu)}  (x,y) = 0$ for $|x-y|_{ \infty} \geq \frac12 L^{j+s+M^{-1}} = \frac12 \ell L^{j+s}$
by the finite-range property \eqref{eq:Dt_range},
it follows from \eqref{eq:H_abound} that 
\begin{align}
  t \sup_{ a\in \{0,\dots,4 \}} \norm{H_{a, (\mu)}}_{\operatorname{op}}
  \leq t C' \ell^2 (\log \ell) (\log L)^{-1}
  = \frac{C't}{M} \ell^2 \leq \frac{1}{2M} \leq \frac{1}{2},
  \label{eq:H_op_small_condition}
\end{align}
whenever $t \leq (2C'\ell^{2})^{-1}$.
Thus $\sqrt{t}\eta_{a, (\mu)}$ satisfies the assumption of Lemma~\ref{lemma:expectation_of_exp_quad}
so that with \eqref{eq:H_abound},
\begin{equation}
\mathbb{E}[ e^{\frac{t}{2} \sum_{x\in Y} \eta(x)^2} ] \leq e^{tC' M^{-1} |Y|_{j+s}}
\end{equation}
as claimed.
\end{proof}

\begin{proof}[Proof of \eqref{eq:g_j+s_expectation} of Lemma~\ref{lemma:g_j+s_bound_by_quadratic}]

Let $\eta_{a, (\mu)}$ be as in Lemma~\ref{lemma:E[exp_grad_Gaussian]}
and write $\kappa_L= c_\kappa \rho_J^{2} (\log L)^{-1}$ with $c_\kappa>0$.
Then by \eqref{eq:g_j+s_bound}
there is a constant $C>0$ such that
\begin{equation}
  g_{j+s} (Y, \xi) \leq \prod_{a=0}^4 \prod_{(\mu)}
  \exp \Big( \frac{Cc_4 c_\kappa}{2} (5 \cdot 2^{a})^{-1} \norm{\eta_{a, (\mu)}}^2_{L^2 (Y)} \Big)
  ,
\end{equation}
and hence by H\"older's inequality, 
\begin{equation}
\E[ g_{j+s} (Y, \xi)  ] \leq \prod_{a=0}^4  \prod_{(\mu)} \E\Big[ \exp \Big( \frac{Cc_4c_\kappa}{2} \norm{\eta_{a, (\mu)}}^2_{L^2 (Y)} \Big) \Big]^{1/(5\cdot 2^a)}.
\end{equation}
Applying
\eqref{eq:E[exp_grad_Gaussian]} with $t=Cc_4c_\kappa$, the right-hand side is bounded by
$e^{Cc_4c_\kappa C' M^{-1} |Y|_{j+s}}
\leq 2^{M^{-1} |Y|_{j+s}}$
when $c_\kappa \leq (2CC'c_4\ell^2)^{-1}$ is chosen small enough.

For the analogous conclusion for the last step with $\Gamma_{N}^{\Lambda_N}$ instead of $\Gamma_{j+1}$,
we just need to use the decomposition \eqref{eq:gamma^b_js} instead of \eqref{eq:gamma_js}
and recall from Lemma~\ref{lemma:fine_Gamma_estimate} that $\Gamma^{\Lambda_N}_{N-1+s,N-1+s'}$ satisfies
the same estimates as $\Gamma_{N-1+s,N-1+s'}$.
\end{proof}

\subsection{Proof of Proposition~\ref{prop:E_G_j}} \label{sec:A4}

\begin{proof}[Proof of Proposition~\ref{prop:E_G_j}]
Assuming that $c_2$ is sufficiently small 
so that Lemma~\ref{lemma:G_change_of_scale} applies, fix
(with the right-hand sides as in the conclusion of Lemma~\ref{lemma:G_change_of_scale})
\begin{equation}\label{eq:lchoice}
\ell = \ell_0(\cwone c_2),
\qquad
c_4=c_4(\cwone c_2).
\end{equation}
By the subdecomposition $\Gamma_{j+1} = \Gamma_{j,j+1/M}+ \cdots +\Gamma_{j+1-1/M,j+1}$
and the corresponding decomposition of the field
$\zeta = \sum_{k=1}^M {\xi}_{k} \sim \cN(0, \Gamma_{j+1})$, by repeated application of \eqref{eq:G_change_of_scale}, 
for all $\varphi' \in \R^{\Lambda_N}$, $X \in \cP^c_j$,
it follows that
\begin{equation}\label{eq:g_j+s_initial}
\Eplus [ G_j (X, \varphi' + \zeta )  ] \leq \prod_{k=1}^M \mathbb{E}^{{\xi}_{k}} \big[g_{j + \frac{(k-1)}{M}} (X_{\frac{k}{M}} , {\xi}_{k}) \big]  G_{j+1} (\overline{X}, \varphi')	,
\end{equation}
where we recall that $X_{k/M}$ is the smallest $j+ \frac{k}{M}$-polymer containing $X$.
Now by Lemma~\ref{lemma:g_j+s_bound_by_quadratic}, 
since $|X_{k/M}|_{j+k/M} \leq |X|_j$, we obtain the claim:
\begin{equation}
\label{eq:kappa_Lchoice}
\Eplus [ G_j (X, \varphi' + \zeta )  ]
\leq \big( 2^{M^{-1}  |X|_j } \big)^{M} G_{j+1} (\overline{X}, \varphi')
\leq 2^{|X|_j } G_{j+1} (\overline{X}, \varphi').
\end{equation}
The proof of the analogous conclusion for the last step with $\Gamma_{N}^{\Lambda_N}$ instead of $\Gamma_{j+1}$
is analogous.
\end{proof}

\subsection{Proof of Lemmas~\ref{lemma:bound_of_Gj_by_Gjplus1} and \ref{lemma:strong_regulator}} 
\label{sec:A.5}

\begin{proof}[Proof of Lemma~\ref{lemma:bound_of_Gj_by_Gjplus1}]
We first collect the following elementary but fundamental inequality. For any function $f:{\Lambda_N} \to \R$, any connected polymer $X \in \cP^c_j$ (not necessarily small) and $x_0 \in X$,
\begin{equation}
\label{eq:gradL_infty}
\max_{x\in X} |f(x)-f(x_0)| \leq 2 |X|_j L^j \norm{\nabla f}_{L^\infty(X)} = 2  |X|_j \norm{\nabla_j f}_{L^\infty(X)}.
\end{equation}
Observing that $2 |X^*|_j  \leq C$ for some $C > 0$ when $X \in \mathcal{S}_j$, this gives
\begin{align}
& \norm{\delta \varphi}_{L^{\infty} (X^*)} \leq C \norm{\nabla_j \varphi}_{L^{\infty}(X^*)}, \quad X\in \mathcal{S}_j. \label{eq:deltaphi_sup} 
\end{align}
Similarly, applying \eqref{eq:gradL_infty} with the choice $f= \nabla_j^{e} \varphi$ for $e \in \{\pm e_1, \pm e_2\}$ to obtain
that $|\nabla_j^{\mu} \varphi(x)| \leq |\nabla_j^{\mu} \varphi(y)| + C \norm{\nabla_j^2 \varphi}_{L^{\infty}(X^*)}$ for any $x, y\in X^*$, averaging over $y$ in $X$ and $\mu$ in $\hat{e}$, taking squares and using that $(a+b)^2 \leq 2(a^2+ b^2)$, one obtains for any $X \in \cP_j^c$ that
\begin{align}\label{eq:gradphi_sup}
& \norm{\nabla_j \varphi}_{L^{\infty}(X^*)}^2 \leq C \big(|X|_j^{-1} \norm{\nabla_j \varphi}_{L^2_j (X)}^2 + W_j (X, \nabla_j^2 \varphi)^2 \big),
\end{align}
where we also used that
\begin{align} \label{eq:grad2phi_sup}
& \norm{\nabla_j^2 \varphi}_{L^{\infty} (X^*)} \leq W_j (X, \nabla_j^2 \varphi),
\end{align}
which follows from \eqref{eq:W_j}.
Recalling $\norm{\cdot}_{C^2_j (X^*)}$ from \eqref{eq:normC2}, combining \eqref{eq:deltaphi_sup}, \eqref{eq:gradphi_sup} and \eqref{eq:grad2phi_sup} while noting that $\nabla_j \delta \varphi = \nabla_j \varphi$, one readily infers that
\begin{equation}
\norm{\delta \varphi}_{C^2_j (X^*)}^2 \leq C \big(\norm{\nabla_j \varphi}_{L^2_j (X)}^2 + W_j (X, \nabla_j^2 \varphi)^2 \big), \quad X \in \mathcal{S}_j,
\end{equation}
from which \eqref{eq:deltaphi_C2} follows in view of \eqref{eq:def_large_field_regulator} by means of the elementary inequality $t^k \leq C(k)e^{t^2}$, valid for all $t \geq 0 $.
\end{proof}

\begin{proof}[Proof of Lemma~\ref{lemma:strong_regulator}]
  The bound \eqref{eq:strong_regulator1} is a direct consequence of the first estimate in (6.100) of \cite[Lemma~6.21]{MR2523458} upon taking a product over $B \in \mathcal{B}_j(X)$ (for the reader's orientation, the quantity $e^{c_w \kappa_L w_j (\varphi, B)}$ for $B \in \mathcal{B}_j$ corresponds to $G^2_{\text{strong}, \varphi}(B)$ in the notation of \cite{MR2523458}).
  In particular, the presence of the factors $2^{-n}$ in \eqref{eq:normL2} and \eqref{eq:normL2boundary}, absent in \cite[(6.67)]{MR2523458}, is inconsequential for the validity of these results. The same applies to further references to \cite{MR2523458} in the sequel.

Note also that, while the value of $c_w$ is fixed in \cite{MR2523458} as $c_w=2$ 
and there is an extra parameter $c_1$ in $G_j$ chosen large enough, 
we take $c_1 = 1$,  $c_w$ small,
which is equivalent. 
Finally, note that \eqref{eq:strong_regulator1} does not rely on the presence of the $\norm{\cdot}_{L^2_j (\partial X)}^2$-term in $G_j$, i.e., \eqref{eq:strong_regulator1} holds with $c_2=0$ in \eqref{eq:def_large_field_regulator}.

The inequality \eqref{eq:strong_regulator2} is the content of (6.103) in \cite{MR2523458}. Here $G_j$ and $c_2$ corresponds to $G_{\varphi'}$ and $c_3$, respectively, in the notation of \cite{MR2523458}. 
Conditions on $c_2, c_w$ above \eqref{eq:strong_regulator2} follows by inspection of the proof of \cite[Lemma~6.22]{MR2523458}, see in particular (6.105) therein.

\end{proof}

\section{Completeness of the space of polymer activities}
\label{app:completeness}

Variations of the space of polymer activities have been defined and used by various different authors in similar contexts but we could not find reference
for its completeness with this specific norm, so we include a proof in this appendix.

\begin{proposition} \label{prop:completeness_of_N_j}
For any $h>0$, the space $\{F\in \mathcal{N}_{j} (X): \norm{F}_{h, T_j (X)}<\infty\}$ is a Banach space.
\end{proposition}
\begin{proof}
Suppose $(F_k)_{k\geq 1}$ is a Cauchy sequence in the norm $\norm{\cdot}_{h, T_j (X)}$. Without loss of generality, we will assume $\norm{F_k - F_{k+1}}_{h, T_j (X)} \leq 2^{-k}$. In particular, 
\begin{align}
\frac{h^n}{n!} \norm{D^n (F_k - F_{k+1})(\varphi)}_{n, T_j (X,\varphi)} \leq 2^{-k} G_j (X,\varphi)
\end{align}
for each $n \geq 0$. Therefore the pointwise limit exists for $(F_k)_{k\geq 1}$, say $F$.
{From the completeness of the spaces $C^k(\R^{X^*})$, it is also clear that $F$ is smooth.}
In fact, if we define another normed space
\begin{align}
&\mathcal{N}'_j = \{ K \text{ polymer activity } : \; \norm{K}'_{h, T_j (X)} < \infty  \}, \\
&\norm{K}'_{h, T_j (X)} = \sup_{\varphi \in \mathbb{R}^{X^*}} G_j (X,\varphi)^{-1} \Big( \sup_{n\geq 0} \frac{h^n}{n!} \norm{D^n K(\varphi)}_{n, T_j (X,\varphi)} \Big)
\end{align}
then the pointwise limit satisfies $F \in \mathcal{N}'_j$. Now suppose $\norm{F}_{h, T_j (X)} = +\infty$. Then for each $M > 0$, there exists $\varphi_M \in \mathbb{R}^{X^*}$ and $N_M \in \mathbb{Z}$ such that
\begin{equation}
\sum_{n=0}^{N_M} \frac{h^n}{n!} \norm{D^n F (\varphi_M)}_{n, T_j (X,\varphi_M)} \geq M G_j (X,\varphi_M).
\end{equation}
But $\sum_{n=0}^{N_M} \frac{h^n}{n!} \norm{D^n (F_{k} - F_{k+1}) (\varphi_M)}_{n, T_j (X,\varphi_M)} \leq 2^{-k} G_j (X,\varphi_M)$ for each $k$ so if we set $M > 1+ \norm{F_1}_{h, T_j (X)}$, this gives a contradiction. This proves $\|F\|_{h,T_j(X)}<\infty$. 

Finally, we have to prove $F_k \rightarrow F$ as $k\rightarrow \infty$ in the $\|\cdot\|_{h,T_j(X)}$ norm. 
To see this, let $\bar{F}_k = F_k - F$, and notice that $(\bar{F}_k)_k$ is still Cauchy in the $\|\cdot\|_{h,T_j(X)}$ norm. 
Suppose $\bar{F}_k$ does not converge to $0$ as $k\rightarrow \infty$. By scaling and taking a subsequence if necessary, this means there is $\varphi \in \mathbb{R}^{X^*}$ such that
\begin{equation}
\sum_{n=0}^{\infty} \frac{h^n }{n!} \norm{D^n \bar{F}_k (\varphi_k)}_{n, T_j (X,\varphi_k )} \geq G_j (X,\varphi_k) .
\end{equation}
But also since $\norm{D^n \bar{F}_k (\varphi, k)}_{n, T_j (X,\varphi)} \rightarrow 0$ as $k\rightarrow \infty$,
up to a subsequence, there exist sequences $(N_k)_{k\geq 0}$, $(M_k)_{k\geq 0}$ such that $N_k < M_k < N_{k+1}$ and
\begin{align}
\sum_{n=N_k}^{M_k} \frac{h^n }{n!} \norm{D^n \bar{F}_k (\varphi_k)}_{n, T_j (X,\varphi_k)} & \geq \frac{2}{3} G_j (X,\varphi_k)  \\
\sum_{n \in \mathbb{N} \backslash [N_k, M_k]} \frac{h^n }{n!} \norm{D^n \bar{F}_k (\varphi)}_{n, T_j (X,\varphi)} & \leq \frac{1}{3} G_j (X,\varphi)  \quad \text{for all } \varphi \in \R^{X^*} .
\end{align}
But this implies $\norm{(\bar{F}_k - \bar{F}_{k+1} ) (\varphi_{k+1})}_{h, T_j (X,\varphi_{k+1})} \geq \frac{1}{3} G_j (X,\varphi_{k+1})$ which contradicts that $\bar{F}_k$ is a Cauchy sequence. Therefore $F_k \rightarrow F$ as $k\rightarrow \infty$.
\end{proof}

\section{Fourier transform of the standard bump function}
\label{app:bump}

In the proof of Lemma~\ref{lemma:basic_frd}, the decay rate of the Fourier transform of the standard bump function $\kappa$ was used.
Since we were unable to locate a reference, we include the elementary proof here.

\begin{proposition} \label{prop:kappa_hat_decay}
Define $\kappa (x) = e^{-\frac{1}{1- 4x^2}} 1_{|x| < 1/2}$ for $x\in \R$ and $\hat{\kappa}(p)$ to be its Fourier transform. Then $\hat{\kappa}(p) = O(e^{-|p|^{1/2}})$.
\end{proposition}
\begin{proof}
Letting $\tau (x) = \kappa(x/2) = e^{-1/(1-x^2)} 1_{|x| < 1}$, it is sufficient to prove $\hat{\tau} (p) = O(e^{-|2p|^{1/2}})$. One has $\hat{\tau}(p) = \int_{(-1,1)} e^{-ipx - \frac{1}{1-x^2}} dx$.  Since $\tau$ is analytic and bounded on the rectangle $S = \{ z \in \mathbb{C} : \operatorname{Re}(z) \in (-1, 1), \; \operatorname{Im}(z) \in (-2,2) \}$, one may write alternatively
\begin{align}
\hat{\tau} (p) = \int_{\Gamma_- \cup \Gamma_+} e^{-ipz - \frac{1}{1-z^2}} dz  = 2 \operatorname{Re} \Big[ \int_{\Gamma_+} e^{-ipz - \frac{1}{1-z^2}} dz  \Big]
\end{align}
where $\Gamma_{\pm} = \{ \pm 1+( \mp 1+i) t  \in \mathbb{C} : t\in (0,1] \}$ (with orientations as appropriate). Without loss of generality, take $p >0 $. Then change of parameter $v =  (\frac{1+i}{\sqrt{2}})^{-1} \sqrt{2p}(1-z)$ gives
\begin{align}
G(p) := \int_{ \Gamma_+} e^{-ipz - \frac{1}{1-z^2}} dz  = \frac{1}{\sqrt{2p}} e^{-ip-\frac{1}{4}} \int_0^{2\sqrt{p}} e^{- \sqrt{p} \frac{1-i}{\sqrt{2}} (v + v^{-1}) - g(\frac{1+i}{2\sqrt{p}} v)} dv
\end{align}
where $g(x) = \frac{x}{4 (2-x)}$. Since $g(\frac{1+i}{2\sqrt{p}} v)$ is bounded uniformly on $v \in [0,  2\sqrt{p})$,
there is $C>0$ such that
\begin{align}
|G(p)| \leq \frac{C}{2\sqrt{p}} \int_0^{2\sqrt{p}} e^{-\sqrt{\frac{p}{2}} (v+v^{-1}) } dv \leq C e^{-\sqrt{2p}}
\end{align}
utilising $v+v^{-1} \geq 2$.

\end{proof}

\section*{Index of notation}
\label{sec:index_of_notations}
\addcontentsline{toc}{section}{Index of notation}

The following list defines the model.

\begin{itemize}
\item $L, \ell, M, N$: (integer) parameters determining the volume of the two-dimensional discrete torus $\Lambda_N$. 
The side length of $\Lambda_N$ is $L^N$ and $L = \ell^{M}$ for sufficiently large $\ell$ and $M$. 
$\ell$ is determined in Section~\ref{sec:norm_properties} while $M$ is chosen only in Section~\ref{sec:stable_manifold_theorem}. 
$N$ is arbitrarily large and tends to $\infty$.

\item $J$: finite-range distribution, subset of $\Z^d$, with associated quantities $\rho_J, v_J^2, \theta_J$ as in Section~\ref{subsec:FRresults}.

\item $\beta$: the temperature, chosen $\beta \geq \beta_0 (J)$ for $\beta_0 (J)$ determined in Section~\ref{sec:stable_manifold_theorem}. 
Related quantities include $\betaeff$ and $\betafree$.
$\betaeff$ is the effective temperature in the scaling limit defined by \eqref{eq:betaeff_def} and $\betafree$ is a crude lower bound on $\beta_0$ given by \eqref{eq:betafree_def}. 

\item $q \geq 0$: (integer) index of Fourier modes of an even periodic function.

\item $\varphi, \varphi', \zeta, \zeta',\dots$:  (Gaussian) fields on $\Lambda_N$. The notation $\varphi + \zeta$ etc.~typically refers to $\zeta$ being currently integrated over while retaining $\varphi$ as a parameter.

\item $\delta f= f(x)-f(x_0)$, similarly $\delta \varphi,\dots$: increment of a function at a point $x_0 \in \Lambda_N$.

\end{itemize}

The following notations show up in the proof of the main theorem, in relation with the renormalisation group. 

\begin{itemize}
\item $C (s,m^2)$: modification of the Green's function $(-\Delta_J)^{-1}$ after $0$-range part ($\gamma$) extraction, 
mass ($m^2$) regularisation and stiffness ($s$) renormalisation.
See Section~\ref{subsec:FRresults} for the definition and related objects.

\item $\Gamma_j$, $\Gamma_N^{\Lambda_N}$, $t_N Q_N$: decomposed convariances of $C(s,m^2)$, defined in Section~\ref{subsec:limit_m_to_0}.

\item $\cB_j$, $\cP_j$,  $\cP_j^c$,  $\cS_j$: set of $j$-blocks, $j$-polymers, connected $j$-polymers and small set polymers  introduced in Section~\ref{sec:polymersdef0}.

\item $\bar{X}$, $\operatorname{Comp}_j (X)$, $X^*$: operations on polymer $X \in \cP_j$ defined in Section~\ref{sec:polymersdef0}.

\item $\cN_j (X)$, $\cN_j$, $\Omega_j^K$, $\Omega_j^U$: set of polymer functions (functions that depend on polymer $X$ and field $\varphi \in \R^{\Lambda_N})$.
$\cN_j (X)$ and $\cN$ are given in Definition~\ref{def:polymeractivity}, 
$\Omega_j^K$ is given in Definition~\ref{def:K_space} and $\Omega_j^U$ is given in Definition~\ref{def:U_space}.
Functions in $\Omega_j^K$ and $\Omega_j^U$ are required to satisfy strong restrictions.

\item $K(X), F(X),\dots$: typical polymer activity (at scale $j \geq0$), element of  $\cN_j$ (Definition~\ref{def:polymeractivity}).

\item $\norm{\cdot}_{n, T_j (X, \varphi)}$, $\norm{\cdot}_{h, T_j (X, \varphi)}$, $\norm{\cdot}_{h, T_j (X)}$, $\norm{\cdot}_{h, T_j}$: (semi-)norms on polymer activities at scale $j$ with associated radius of convergence $h$, large field regulator $G_j$ and large set regulator $A$ as in Section~\ref{sec:polymersdef}.
One may also consult the end of Section~\ref{sec:polymersdef},
Remark~\ref{R:choice-parameters}, Definition~\ref{def:K_space} and Lemma~\ref{lem:W_norm} for choices of parameters in the definition of the norms.

\item  $\Loc_F$, $ \Loc_{F,B}$: localisation operators (Definition~\ref{def:Loc}).

\item $\Tay_n$,  $\Rem_n$ and their averages $\bar{\Tay}_n, \bar{\Rem}_n$: Taylor expansion and remainder of a polymer $F$ to order $n$ about point $0$ (Section~\ref{sec:pf-Loc-neutral}).

\item $\Phi_{j+1}$: the renormalisation group map with coordinates $(\cE_{j+1}, \cU_{j+1}, \cK_{j+1})$, defined in Section~\ref{sec:rg_generic_step} (Definitions~\ref{def:evolution_of_U}--\ref{def:evolution_of_remainder}).  Calligraphic notation $\cE, \cU, \cK$ etc.~refers to an actual map (on a suitable function space), roman notation e.g.~$E, U, K,\dots$ (possibly with subscripts $j$) to a point in such a space; for instance in writing $\mathcal{K}_{j+1}(U_j,K_j)
$ we mean evaluate the map $\mathcal{K}_{j+1}$ at point $(U_j,K_j)$. 
$\cL_{j+1}$: linear part, $\cM_{j+1}$: non-linear part of $\cK_{j+1}$ (Theorem~\ref{thm:local_part_of_K_j+1}).

\end{itemize}

One may also see Section~\ref{sec:notation} for further notation. 


\section*{Acknowledgements}

R.B.\ was supported by the European Research Council under the European Union's Horizon 2020 research and innovation programme
(grant agreement No.~851682 SPINRG). He also acknowledges the hospitality of the Department of Mathematics at McGill University
where part of this work was carried out.
J.P.\ was supported by the Cambridge doctoral training centre Mathematics of Information.

\bibliography{all}
\bibliographystyle{plain}

\end{document}